\documentclass[11pt,a4paper]{amsart}[2010]
\usepackage{tikz}
\usepackage[inline]{enumitem}
\usepackage{comment}
\usepackage{quiver}

\title{Long thin covers and nuclear dimension}

\author{Ilan Hirshberg}

\address{Department of Mathematics, Ben Gurion University of the Negev, \phantom{----------------}\linebreak\text{}\hspace{3.5mm}
P.O.B. 653, Be'er Sheva 84105, Israel}
\email{ilan@math.bgu.ac.il}

\author{Jianchao Wu}

\address{Shanghai Center for Mathematical Sciences, Fudan University,
	\phantom{----------------}\linebreak\text{}\hspace{3.5mm}
	2005 Songhu Rd., Shanghai 200438, China}

\email{jianchao\_wu@fudan.edu.cn}

\thanks{The first author was partly supported by Israel Science Foundation grant no. 
476/16. 
The second author was partly supported by NSFC Key Program No.~12231005 and National
Key R\&D Program of China 2022YFA100700. 
Part of the research was conducted at the Mittag-Leffler Institute, at the CRM in Barcelona and at the Fields Institute.}

  \makeatletter
\@namedef{subjclassname@2020}{%
	\textup{2020} Mathematics Subject Classification}
\makeatother


\usepackage{amsmath}
\usepackage{amssymb}
\usepackage{amsthm}
\usepackage[all]{xypic}
\usepackage{xy}
\xyoption{curve}
\usepackage[colorlinks=true,linkcolor=blue,citecolor=blue]{hyperref}
\usepackage{mathtools}
\usepackage{xstring}
\usepackage{xparse}

 \usepackage[usenames,dvipsnames]{pstricks}
 \usepackage{epsfig}
 \usepackage{pst-grad} 
 \usepackage{pst-plot} 
 \usepackage[space]{grffile} 
 \usepackage{etoolbox} 
 \makeatletter 
 \patchcmd\Gread@eps{\@inputcheck#1 }{\@inputcheck"#1"\relax}{}{}
 \makeatother

\allowdisplaybreaks

\theoremstyle{plain}

\newtheorem{Thm}{Theorem}[section]
\newtheorem{Cor}[Thm]{Corollary}
\newtheorem{Lemma}[Thm]{Lemma}

\newtheorem{Prop}[Thm]{Proposition}

\theoremstyle{definition}
\newtheorem{Def}[Thm]{Definition}
\newtheorem{Notation}[Thm]{Notation}

\newtheorem{Exl}[Thm]{Example}
\newtheorem{Rmk}[Thm]{Remark}
\newtheorem{Question}[Thm]{Question}

\newcommand{\D}{D}

\newcommand{\Zh}{\mathcal{Z}}

\newcommand{\R}{{\mathbb R}}
\newcommand{\N}{{\mathbb N}}
\newcommand{\Z}{{\mathbb Z}}

\newcommand{\Q}{{\mathbb Q}}

\newcommand{\aut}{\operatorname{Aut}}
\newcommand{\supp}{\operatorname{supp}}

\newcommand{\eps}{\varepsilon}
\numberwithin{equation}{section}

\newcommand{\dr}{\operatorname{dr}}
\newcommand{\id}{\operatorname{id}}

\newcommand{\dimnuc}{\operatorname{dim}_{\operatorname{nuc}}}
\newcommand{\dimnucone}[0]{\dimnuc^{\!+1}}

\newcommand{\LSP}{\operatorname{LSP}}
\newcommand{\mult}{\operatorname{mult}}
\newcommand{\dimltc}{\dim_{\operatorname{LTC}}}
\newcommand{\asdim}{\operatorname{asdim}}
\newcommand{\stab}{\operatorname{stab}}

\newcommand{\short}{\mathcal{F}}

\newcommand{\intervalofintegers}[2]{\left\{ {#1}, \ldots, {#2} \right\}}

\newcommand{\powerset}{\mathcal{P}}

\newcommand{\ball}{}

\newcommand{\compr}{\mathrm{compr}}
\newcommand{\dstab}{d_{\mathrm{stab}}}

\newcommand{\dimone}[0]{\dim^{\!+1}}

\newcommand{\e}{e}

\makeatletter
\def\namedlabel#1#2{\begingroup
	#2%
	\def\@currentlabel{#2}%
	\phantomsection\label{#1}\endgroup
}
\makeatother

\begin{document}
\begin{abstract}
	We establish finite nuclear dimension for crossed product $C^*$-algebras arising from various classes of possibly non-free topological actions, including arbitrary actions of finitely generated virtually nilpotent groups on finite dimensional spaces, certain amenable actions of hyperbolic groups, and certain allosteric actions of wreath products. 
	We obtain these results by introducing a new notion of dimension for topological dynamical systems, called the long thin covering dimension, which involves a suitable version of Rokhlin-type towers with controlled overlaps for possibly non-free actions.
\end{abstract}

\label{newcomm}

\newcommand{\eqasdim}{\operatorname{eqasdim}}
\newcommand{\setstab}{\operatorname{Stab}}
\newcommand{\Stab}{\setstab}
\newcommand{\Rank}[1][]{{\operatorname{Rk}_{#1}}}
\newcommand{\Torsion}{{T}}
\newcommand{\hirsch}{\operatorname{Hirsch}}
\newcommand{\Growth}[1]{\operatorname{Gr}_{#1}}
\newcommand{\Metric}{\rho}
\newcommand{\Staircase}[2]{{\Xi_{{#1},{#2}}}}

\newcommand{\dimc}{\dim_{\operatorname{c}}}

\newcommand{\grpid}{e}

\newcommand{\Rfuncs}[2][]{{\ref{Rfuncs} c_{00}^{#1} (#2 , \R)}} \label{Rfuncs} 
\newcommand{\Rplusfuncs}[2][]{{\ref{Rplusfuncs} c_{00}^{#1}(#2,\R)_{\geq 0} }} \label{Rplusfuncs} 
\newcommand{\ccplus}[1][]{{\ref{ccplus} c_{\operatorname{c},+}^{#1}}} \label{ccplus} 
\newcommand{\Ma}[2][S]{{\ref{Ma} \mu^{\left(#2\right)}_{#1}}} \label{Ma} 
\newcommand{\Di}[1][S]{\ref{Di} \delta_{#1}} \label{Di} 
\newcommand{\Dihi}[2][S]{\ref{Dihi} \Di[#1]^{#2}} \label{Dihi} 
\newcommand{\Dinum}[2][S]{{\Di[#1]^{\left(#2\right)} } } 
\newcommand{\Dihinum}[3][S]{\ref{Dihinum} \Di[#1]^{#2,\left(#3\right)}} \label{Dihinum} 
\newcommand{\Dihisign}[2][S]{\ref{Dihisign} \widetilde{\Di[#1]}^{#2}} \label{Dihisign} 
\newcommand{\Dihisignnum}[3][S]{\ref{Dihisignnum} \widetilde{\Di[#1]}^{#2,\left(#3\right)}} \label{Dihisignnum} 
\newcommand{\Sigmap}[2][S]{\ref{Sigmap} \sigma_{#1}^{#2}} \label{Sigmap} 
\newcommand{\Taumap}[2][S]{\ref{Taumap} \tau_{#1}^{#2}} \label{Taumap} 

\newcommand{\sectionlabel}{} \label{sectionlabel=OCS}\label{sectionlabel=LTC}\label{sectionlabel=BLR}\label{sectionlabel=LSP} 
\newcommand{\Lo}[1][\sectionlabel]{\ref{Lo} 
	{\IfEqCase{#1}{
		{prelim}{L} 	
		{OCS}{L} 
		{LTC}{L} 
		{LSP}{L} 
		{BLR}{L} 
		{dimnuc}{L}
	}[{\color{red} NOT DEFINED!}]}
}\label{Lo}

\newcommand{\tLo}[1][\sectionlabel]{\ref{tLo} 
	{\IfEqCase{#1}{
			{dimnuc}{{\Bo^2}}  
		}[{\color{red} NOT DEFINED!}]}
}\label{tLo}

\newcommand{\Ko}[1][\sectionlabel]{\ref{Ko} 
	{\IfEqCase{#1}{
		{prelim}{K} 	
		{OCS}{K} 
		{LTC}{K} 
		{BLR}{K} 
		{dimnuc}{K} 
	}[{\color{red} NOT DEFINED!}]}
}\label{Ko}

\newcommand{\Er}[1][\sectionlabel]{\ref{Er} 
	{\IfEqCase{#1}{
		{prelim}{\delta} 	
		{OCS}{\delta} 
		{LTC}{\delta} 
		{BLR}{\delta} 
		{dimnuc}{\delta}
	}[{\color{red} NOT DEFINED!}]}
}\label{Er}

\newcommand{\Po}[1][\sectionlabel]{\ref{Po} 
	{\IfEqCase{#1}{
		{OCS}{f} 
		{LTC}{\mu} 
		{BLR}{f} 
		{dimnuc}{\nu} 
	}[{\color{red} NOT DEFINED!}]}
}\label{Po}
\newcommand{\Ponum}[3][\sectionlabel]{{\Po[#1]}^{(#3)}_{#2} }
\newcommand{\tPonum}[3][\sectionlabel]{\widetilde{\Po[#1]}^{(#3)}_{#2} }
\newcommand{\tPo}[1][\sectionlabel]{{\widetilde{\Po[#1]}}}

\newcommand{\Lapo}[1][\sectionlabel]{\ref{Lapo} 
	{\IfEqCase{#1}{
			{OCS}{h} 
			{LTC}{h} 
		}[{\color{red} NOT DEFINED!}]}
}\label{Lapo}
\newcommand{\Lapotil}[1][\sectionlabel]{{\widetilde{\Lapo[#1]}}}

\newcommand{\Bo}[1][\sectionlabel]{\ref{Bo} 
	{\IfEqCase{#1}{
		{prelim}{M} 	
		{OCS}{M}  
		{LTC}{N} 	
		{LSP}{N}  	
		{BLR}{N}	
		{dimnuc}{S} 
	}[{\color{red} NOT DEFINED!}]}
}\label{Bo}

\newcommand{\Aubo}[1][\sectionlabel]{\ref{Aubo} 
{\IfEqCase{#1}{
		{prelim}{S} 	
		{OCS}{S}  
		{LTC}{S} 	
		{LSP}{S}  	
		{BLR}{S}	
		{dimnuc}{T} 
	}[{\color{red} NOT DEFINED!}]}
}\label{Aubo}

\newcommand{\Fi}[1][\sectionlabel]{\ref{Fi} 
	{\IfEqCase{#1}{
			{OCS}{F} 
			{LTC}{F} 
			{LSP}{F}  	 
			{BLR}{F} 
			{dimnuc}{R}
		}[{\color{red} NOT DEFINED!}]}
}\label{Fi}

\newcommand{\Fico}[1][\sectionlabel]{{\ref{Fico} \mathcal{\Fi[#1]}}} \label{Fico}

\newcommand{\Aufi}[1][\sectionlabel]{\ref{Aufi} 
	{\IfEqCase{#1}{
			{OCS}{E} 
			{LTC}{E} 
			{LSP}{E}  	
			{BLR}{E} 
		}[{\color{red} NOT DEFINED!}]}
}\label{Aufi}

\newcommand{\Ne}[1][\sectionlabel]{\ref{Ne} 
	{\IfEqCase{#1}{
		{LTC}{\Lambda} 
		{BLR}{\Lambda} 
		{dimnuc}{\Lambda}
	}[{\color{red} NOT DEFINED!}]}
}\label{Ne}

\newcommand{\La}[1][\sectionlabel]{\ref{La} 
	{\IfEqCase{#1}{
		{BLR}{\Lambda} 
	}[{\color{red} NOT DEFINED!}]}
}\label{La}
\newcommand{\Cose}[1][\sectionlabel]{\ref{Cose} 
	{\IfEqCase{#1}{
		{OCS}{W}  
		{LTC}{U}  
		{BLR}{U}  
		{dimnuc}{U}
	}[{\color{red} NOT DEFINED!}]}
}\label{Cose}
\newcommand{\Coseco}[1][\sectionlabel]{{\ref{Coseco} \mathcal{\Cose[#1]}}} \label{Coseco}
\newcommand{\Cosenum}[2][\sectionlabel]{{\Cose[#1]}^{(#2)} } 
\newcommand{\Coseconum}[2][\sectionlabel]{{\Coseco[#1]}^{(#2)} } 

\newcommand{\Coco}[1][\sectionlabel]{\Coseco[\sectionlabel]}
\newcommand{\Coconum}[2][\sectionlabel]{{\Coseco[#1]}^{(#2)} }

\newcommand{\Ause}[1][\sectionlabel]{\ref{Ause} 
	{\IfEqCase{#1}{
		{OCS}{U}  
		{LTC}{W}  
		{BLR}{W}   
		{LSP}{D} 
		{dimnuc}{W}
	}[{\color{red} NOT DEFINED!}]}
}\label{Ause}
\newcommand{\Auseco}[1][\sectionlabel]{{\ref{Auseco} \mathcal{\Ause[#1]}}} \label{Auseco}
\newcommand{\Auseconum}[2][\sectionlabel]{{\Auseco[#1]}^{(#2)} }

\newcommand{\Thse}[1][\sectionlabel]{\ref{Thse} 
	{\IfEqCase{#1}{
		{OCS}{V}  
		{LTC}{V}  
		{BLR}{V}  
	}[{\color{red} NOT DEFINED!}]}
}\label{Thse}
\newcommand{\Thseco}[1][\sectionlabel]{{\ref{Thseco} \mathcal{\Thse[#1]}}} \label{Thseco}
\newcommand{\Thseconum}[2][\sectionlabel]{{\Thseco[#1]}^{(#2)} }

\newcommand{\Lase}[1][\sectionlabel]{\ref{Lase} 
	{\IfEqCase{#1}{
		{OCS}{Z}  
		{LTC}{Z}  
		{BLR}{Z}  
	}[{\color{red} NOT DEFINED!}]}
}\label{Lase}
\newcommand{\Laseco}[1][\sectionlabel]{{\ref{Laseco} \mathcal{\Lase[#1]}}} \label{Laseco}
\newcommand{\Laseconum}[2][\sectionlabel]{{\Laseco[#1]}^{(#2)} } 
\newcommand{\Lasehat}[1][\sectionlabel]{{{\Lase[#1]}}}
\newcommand{\Lasecohat}[1][\sectionlabel]{{{\Laseco[#1]}}}
\newcommand{\Lasetil}[1][\sectionlabel]{{\widetilde{\Lase[#1]}}}
\newcommand{\Lasecotil}[1][\sectionlabel]{{\widetilde{\Laseco[#1]}}}

\newcommand{\Aulase}[1][\sectionlabel]{\ref{Aulase} 
	{\IfEqCase{#1}{
			{OCS}{Y}  
			{LTC}{Y}  
			{BLR}{Y}  
		}[{\color{red} NOT DEFINED!}]}
}\label{Aulase}
\newcommand{\Aulaseco}[1][\sectionlabel]{{\ref{Aulaseco} \mathcal{\Aulase[#1]}}} \label{Aulaseco}
\newcommand{\Aulaseconum}[2][\sectionlabel]{{\Aulaseco[#1]}^{(#2)} } 
\newcommand{\Aulasehat}[1][\sectionlabel]{{\widehat{\Aulase[#1]}}}
\newcommand{\Aulasecohat}[1][\sectionlabel]{{\widehat{\Aulaseco[#1]}}}

\newcommand{\Binu}[1][\sectionlabel]{\ref{Binu} 
	{\IfEqCase{#1}{
			{OCS}{N}  
			{LTC}{M}  
			{BLR}{M}  
		}[{\color{red} NOT DEFINED!}]}
}\label{Binu}

\newcommand{\Ense}{\ref{Ense} 
	{E}   
}\label{Ense}
\newcommand{\Enseco}{{\ref{Enseco} \mathcal{\Ense}}} \label{Enseco}
\newcommand{\Auense}{\ref{Auense} 
	{F}   
}\label{Auense}

\newcommand{\Dimnu}[1][\sectionlabel]{\ref{Dimnu} 
	{\IfEqCase{#1}{
			{OCS}{m}  
			{LTC}{m}   
			{LSP}{m} 
			{BLR}{m}  
		}[{\color{red} NOT DEFINED!}]}
}\label{Dimnu}

\newcommand{\ErB}{\Er}
\newcommand{\ErT}{\eta'} 
\newcommand{\ErLTC}{\eta}
\newcommand{\ErOrd}{\Er}
\newcommand{\ErPreLTC}{\eta_{0}} 

\newcommand{\Inse}{I} 

\newcommand{\krep}[2][\Ause]{k_{#1}^{(#2)}}
\newcommand{\tkrep}[1][\Ause]{\vec{k}_{#1}}
\newcommand{\vkrep}[1][\Ause]{\vec{u}_{#1}}


\def \OCSmult {prop:orbit-asdim:mult}
\def \OCSfin {prop:orbit-asdim:finitary}
\def \OCSpou {prop:orbit-asdim:pou}
\def \LTCdef {def:ltc-dim}
\def \LTCextra {def:ltc-dim-extra}
\def \LTCrmksep {rmk:ltc-dim-separation}
\def \LTCprop {prop:ltc-dim}
\def \LSPdef {def:LSP}
\def \BLRdef {def:BLR}
\def \BLRlem {lem:BLR-reformu}

\newcommand{\RefcondDefaultValue}[2]{
		{\IfEqCase{#2}{
			{Lo}{\Lo,\Ko}	
			{Bo}{\Bo}
			{Bovar}{\Lo, \Binu}
			{Ca}{\Binu}
			{Co}{\Lo}
			{Coplus}{\Lo, \Binu}
			{Mu}{d}
			{Muplus}{d,\Lo}
			{Eq}{
				\IfEq{#1}{def:BLR}{}{\Lo}
				}
			{Th}{\Thseco}
			{Thplus}{\Thseco, \Bo}
			{Se}{\Lo}
			{Li}{\Lo,\Er}
			{In}{\Lo,\Er}
			{St}{\mathcal{F}}
			{Pa}{\Lo,d,m}
			{Su}{\Lo}
			{Un}{\Ko}
			{Pr}{}
			{Fi}{}
		}}
	}  
	
\newcommand{\Refc}[3][\Topic]{$\left(\textup{\ref{item:#1:#2}}\right)^{\textup{\ref{#1}}}_{#3}$} 
\newcommand{\Refcd}[2][\Topic]{\Refc[#1]{#2}{\RefcondDefaultValue{#1}{#2}}}  


\maketitle

\tableofcontents

\section{Introduction}
\renewcommand{\sectionlabel}{LTC}
Nuclear dimension for $C^*$-algebras, a noncommutative generalization of the 
notion of covering dimension for topological spaces, was introduced by Winter 
and Zacharias 
in \cite{winter-zacharias} (building upon a closely related notion of 
decomposition rank, introduced in \cite{kirchberg-winter}). It has since come 
to play a central role in the structure and classification theory for simple 
nuclear $C^*$-algebras. Indeed, a culmination of work of around four decades by 
many authors brought Elliott's classification program to a nearly final form: 
it is now known that simple separable nuclear $C^*$-algebras which satisfy 
the Universal Coefficient Theorem are classified via the Elliott invariant 
provided that they have finite nuclear dimension (\cite{EGLN,TWW}), whereas 
counterexamples (\cite{rordam-counterexample,toms-counterexample}) show that 
the classification conjecture breaks down if one does not assume finite nuclear 
dimension. 

It is therefore a matter of significant interest to find conditions 
which ensure that $C^*$-algebras arising from various natural constructions 
have, or do not have, finite nuclear dimension. In particular, there has been 
considerable interest in finding conditions on group actions which imply that 
the resulting crossed product $C^*$-algebras have finite nuclear dimension. The first result 
on the matter was due to Toms and Winter (\cite{toms-winter}), who showed that 
if $X$ is an infinite compact 
metric space with finite covering dimension and $\alpha$ is a minimal 
homeomorphism then the crossed product $C(X)\rtimes_{\alpha} \Z$ has finite 
nuclear dimension. The notion of Rokhlin dimension, introduced in \cite{HWZ} 
for actions on $C^*$-algebras (which aren't necessarily commutative),  
provided a different approach for proving the result of Toms and Winter 
above. Those methods were then improved and generalized to cover free (and 
not necessarily minimal) actions of $\Z^n$ and then free actions of finitely 
generated 
virtually nilpotent groups (\cite{szabo,SWZ,Bartels2017Coarse}) as well as free actions of $\R$ and then connected Lie groups of polynomial growth 
(\cite{HSWW16, EnstadFavreRaum2023Free}). A somewhat related groupoid approach
(\cite{GuentnerWillettYu2017Dynamic}) introduces a notion of dynamic 
asymptotic dimension for groupoids, and in particular provides alternative 
proofs to some of the above results. 

In a different approach that goes beyond the assumption that $X$ 
is finite dimensional, Elliott and Niu showed in \cite{Elliott-Niu-mdim-0} that 
if  $X$ is an infinite compact 
metric space and $\alpha$ is a minimal homeomorphism  with mean dimension zero 
then the crossed product has finite nuclear dimension. More precisely, the 
authors showed that the crossed product is $\Zh$-stable, and therefore 
classifiable; this was done using large subalgebra techniques 
(\cite{Phillips-large,ABP}). (The fact that $\Zh$-stability is equivalent to 
finite nuclear dimension in this case was only realized later in 
\cite{CastillejosEvingtonTikuisisWhiteWinter2021Nuclear}.)
For more on connections to mean dimension, see 
\cite{giol-kerr,Niu-mdim,Niu-Rokh,hirshberg-phillips-mcid}. A more recent approach 
(\cite{kerr,kerr-szabo,kerr-naryshkin}), which provides the best known results 
in the free and minimal setting, involves showing that free and minimal actions 
of 
certain amenable groups are almost finite, that is, admit Rokhlin-type towers 
which are large in measure, in a suitable sense. Those towers are used to show 
that the resulting crossed product is tracially $\Zh$-absorbing 
(\cite{hirshberg-orovitz}), therefore $\Zh$-stable and therefore have finite 
nuclear dimension. The advantage over Rokhlin dimension techniques is that here 
one does not need the group to have finite asymptotic dimension, and indeed 
this approach now covers all countable elementary amenable groups as well as groups with subexponential growth.

The approaches discussed above require freeness (while some furthermore require minimality), essentially because they all involve the existence of some 
Rokhlin-type towers. 
For example, the notions of almost finiteness and finite tower dimension (\cite{kerr}) require the existence of \emph{open towers}, each consisting of an open subset (the ``base'') $B$ of the topological space and a finite subset (the ``shape'') $S$ of the acting group such that $\{g \cdot B \colon g \in S \}$ (the collection of the ``levels'') is a disjoint family. 
The existence of  
Rokhlin-type towers is a priori 
stronger than freeness, and a 
significant amount of work goes into showing that they are implied by freeness 
under some additional conditions on the space, the group or the action. 

Now the presence of the freeness and minimality conditions are often justified by the fact that they ensure the resulting crossed products belong to the class of simple $C^*$-algebras, which, for good reasons, has been at the center of attention in the theory of $C^*$-algebras. 
However, 
even if one only cares about simple $C^*$-algebras, one should not restrict themselves to simple crossed product $C^*$-algebras, since, for instance, 
simple $C^*$-algebras may naturally arise as direct limits of non-simple crossed products (see the last paragraph of Subsection~\ref{subsec:ltc-bounds} for an example). 
Moreover, even if one only wishes to focus on dynamical systems which give rise to simple crossed product $C^*$-algebras, 
one 
does not need to require freeness, but only topological freeness, as is shown in 
\cite{archbold-spielberg}.

Hence there are ample motivations to push the results above beyond the free case, though progress in this direction has been quite limited. 
We considered in \cite{Hirshberg-Wu16} arbitrary actions of $\Z$ on locally 
compact Hausdorff spaces with finite covering dimension, and showed that the 
crossed product always has finite nuclear dimension. The bound we found is 
probably not optimal, although usually one mostly cares about the finiteness of the 
nuclear dimension rather than its actual value. (In the case of simple 
$C^*$-algebras, which in this case correspond to free and minimal actions, it 
is now known that the possible values are only $0$,$1$ or 
$\infty$; this is false in the non-simple case.) In \cite{Hirshberg-Wu-flows}, 
we proved analoguous results for flows and for orientable one-dimensional foliations. 
Unfortunately, however, there seems to be no easy way to generalize the method in these papers to even the case of $\mathbb{Z}^2$-actions (see Subsection~\ref{subsec:ideas-main-stabilizers} for a more in-depth discussion). 

We point out in passing that 
while strictly speaking, the notion 
of almost finiteness does not depend on freeness, it does require essential freeness (with regard to any invariant probability measure). It was shown 
recently that there exist examples of minimal, uniquely ergodic, and topologically free actions of 
discrete amenable groups on the Cantor set which are not essentially free and 
thus fail to be almost finite 
\cite{Joseph2023}; such actions are said to be \emph{allosteric}. We provide a 
generalization of the construction in 
\cite{Joseph2023}, as the techniques we develop in this paper can handle some 
of the examples of this type (see the last paragraph of Subsection~\ref{subsec:main-results}). 

\subsection{The main results} \label{subsec:main-results}

The goal of this paper is to extend the results of \cite{Hirshberg-Wu16} from integer actions to a 
much broader class of group actions, e.g., actions by all finitely generated virtually nilpotent groups. 
To this end, we introduce a notion which we call the 
\emph{long thin covering dimension} for a (possibly non-free) action $\alpha$ of a countable 
discrete 
group $G$ on a locally compact Hausdorff space $X$ and we denote it by $\dimltc(\alpha)$. 
This topological-dynamical invariant allows us to establish the following nuclear dimension bound for the crossed product: 
\[
	\dimnuc \left( C_0(X) \rtimes_\alpha G \right) + 1 \leq \left( \dimltc(\alpha) + 1 \right)^2 \left( d_{\operatorname{stab}} + 1 \right) \; ,
\]
where $d_{\operatorname{stab}}$ is
a uniform upper 
bound of the nuclear dimension of the group $C^*$-algebras of all subgroups of the 
stabilizers corresponding to the action $\alpha$. 
In fact we prove something 
more general, as we consider actions on $C_0(X)$-algebras rather than 
actions on $C_0(X)$; 
this is useful in getting applications such as the one for allosteric actions mentioned above. 
The precise statement can be found in Theorem \ref{thm:dimnuc-main}, which involves another dimension we define along side $\dimltc(\alpha)$, 
namely the 
\emph{asymptotic dimension of the coarse orbit space}, as this gives a clearer 
and better 
bound, 
however this latter dimension is bounded above by $\dimltc(\alpha)$ (see Theorem~\ref{thm:dimltc-asdim}). 

By bounding $\dimltc(\alpha)$ under various conditions, we are able to obtain a variety of applications of our main theorem; see Subsection~\ref{subsec:ltc-bounds} for a more detailed description. 
\begin{itemize}
	\item Directly generalizing \cite{Hirshberg-Wu16}, 
	we show that for an arbitrary action of a finitely generated virtually nilpotent group on a locally 
	compact Hausdorff space with finite covering dimension, the associated crossed product has finite nuclear dimension. 
	This provides a key technical tool that will lead to a proof that all (possibly twisted) group $C^*$-algebras of virtually polycyclic groups have finite nuclear dimension (\cite{EckhardtWu}), a result that generalizes \cite{eckhardt-mckenney} and \cite{EGM}\footnote{In fact, here we also need to use the version of our main result for $C_0(X)$-algebras. }. 
	
	\item Going beyond the case of amenable groups, we show that for any simplicial 
	proper cocompact action of a hyperbolic group $G$ on a hyperbolic complex $X$, the induced action of $G$ on the completion $\overline{X}$ or boundary $\partial 
	X$ gives rise to crossed products with finite nuclear dimension. 
	A special case is given by actions 
	of hyperbolic groups on their Gromov boundaries. 
	Various 
	boundary actions were studied in 
	\cite{laca-spielberg}, and generalized recently to the case of arbitrary 
	amenable, topologically free and minimal actions of a large class of 
	non-amenable groups, in \cite{ggkn22}, using a  
	different approach: those papers show that such crossed products are Kirchberg 
	algebras. 
	
	\item As hinted earlier when we review almost finiteness, we show that certain crossed products arising from allosteric actions 
	have 
	finite nuclear dimension. 
	In our approach, in order to bound the nuclear dimension of a simple 
	crossed product, we need to work with a directed system of non-simple ones arising from actions on noncommutative $C^*$-algebras. 
\end{itemize}

\subsection{Rokhlin-type towers for non-free actions} \label{subsec:towers-non-free}

As indicated above, the key tool in establishing our main results is a new topological-dynamical invariant called the long thin covering dimension (or ``LTC dimension'' for short). 
This new dimension can be compared to the existing notions in the literature, such as the Rokhlin dimension and the tower dimension, since, at a conceptual level, it also keeps track of how well a topological space can be covered by a kind of ``Rokhlin-type towers''. 
However, the novelty lies in that the new ``Rokhlin-type towers'' we use are adapted to non-free actions. 

To get an idea of how this works, we consider an integer action $\alpha \colon \mathbb{Z} \curvearrowright X$ on a compact metric space and observe that the structure of an open tower $(B,\intervalofintegers{1}{n})$ (defined above) is encoded in an associated ``labeling function'' $\Lambda \colon \bigcup_{s=1}^{n} \alpha_s (B) \to \intervalofintegers{1}{n} \subseteq \mathbb{Z}$ that maps $\alpha_s(x)$ to $s$, where $x \in B$. 
This ``labeling function'' is continuous (in fact, locally constant) and, as long as $(B,\intervalofintegers{1}{n+1})$ also forms an open tower (this can often be arranged in the context of the tower dimension or almost finiteness), the ``labeling function'' is also ``locally equivariant'' (or more precisely, \emph{$\{1\}$-equivariant}) in the sense that for any $x, y \in \bigcup_{s=1}^{n} \alpha_s (B)$ satisfying $\alpha_1(x) = y$, we have $\Lambda (x) + 1 = \Lambda (y)$. 
It is this ``local equivariance'' that prevents us from having nontrivial 
stabilizers: for example, if $x$ is fixed by the action of $\mathbb{Z}$, then 
clearly it cannot be in the domain of a $\{1\}$-equivariant map to 
$\mathbb{Z}$. We observe at this point that even in the free case, 
aiming to get towers of open sets is an overkill: one really is interested in 
towers consisting of bump functions, that get translated approximately into one 
another via the the action. An open tower as above lets us build associated 
towers of bump functions, by starting out with one supported on the base and 
translating it. However, as one allows for an error when translating one bump 
function roughly into another, there's no reason for the supports to be mapped 
exactly into one another. This fine distinction turns out to be useful in our 
case.

In order to accommodate non-free actions, we need to generalize the notion of 
towers: from the perspective of labeling functions as above, a natural 
approach is to change the target of the labeling functions from the acting 
group $G$ (in this example, $\mathbb{Z}$) to homogeneous spaces $G/H$. We 
furthermore weaken the requirement of being an open tower by asking for 
an open set with a labeling function which is equivariant for points 
where it makes sense (that is, ones that remain in the collection when 
we apply relevant group elements), but we do not require each set 
to be a union of copies of one set that gets translated exactly into the others.  
Those generalized towers are required to cover of $X$ with bounded multiplicity,  
are ``long'' in the sense that for any point $x$ in $X$, a large partial orbit 
around $x$ is contained in a single generalized tower. With this, we arrive at 
an equivalent characterization (see Lemma~\ref{lem:BLR-reformu}) of the 
\emph{equivariant asymptotic dimension}, a dimension concept first developed by 
Bartels, L\"{u}ck and Reich in their work on the Farrell-Jones conjecture 
(\cite{BartelsLueckReich2008Equivariant}) and subsequently studied on its own 
right 
(\cite{Bartels2017Coarse, GuentnerWillettYu2017Dynamic, Sawicki2017equivariant}). 
Note that the equivariant asymptotic dimension depends not only on the action $\alpha$ but also on a predetermined family of admissible stabilizers $H$. 

The approach we actually take in defining our LTC dimension avoids the need to declare a family of admissible stabilizers: by letting the targets of the labeling functions to be ``nearby (patial) orbits'', we allow the ``shapes'' of our ``towers'' to adapt to how ``non-free'' the action is in every local region. 
This design also meshes well with our desire to make the ``towers'' thin, say, in the sense that the levels of the towers (interpreted as the fibers of the labeling functions) have small diameters \textemdash\ this makes it more effective in estimating the nuclear dimension of crossed products (cf. \cite[Definition~4.10 and Proposition~4.11]{kerr}). 
We point out in passing that unlike in the case of the (fine) tower dimension, we do not demand the levels of a ``tower'' to be exactly translates of one another.  
The last ingredient we bake into the definition of the LTC dimension is the requirement that the shapes of our ``towers'' have controlled cardinalities independently of how thin the ``towers'' are \textemdash\ this is needed to control the asymptotic dimension of a natural coarse structure we associate to the action. 

Having briefly discussed the underlying ideas, now we give a somewhat simplified characterization of the LTC dimension that works in the case of a discrete group $G$ acting on a compact metric space $X$: the LTC dimension is no more than $d$ if and only if for any finite subset $\Lo \subseteq G$, there exists a natural number $\Binu$ 
such that for any $r > 0$, 
there exist 
\begin{itemize}
	\item an open cover $\Coseco$ of $X$, and
	\item locally constant functions $\Ne_{\Cose} \colon \Cose \to X$ for $\Cose \in \Coseco$, called \emph{near orbit selection functions}, 
\end{itemize}
satisfying the following conditions:
{
	\begin{enumerate}
		\renewcommand{\labelenumi}{\textup{(\theenumi)}} \renewcommand{\theenumi}{Lo}\item  \label{item:intro:Lo}
		For any $x \in X$, the partial orbit $\bigcup_{g \in \Lo} \alpha_{g} (x)$ is contained in some $\Cose \in \Coseco$.
		\renewcommand{\labelenumi}{\textup{(\theenumi)}} \renewcommand{\theenumi}{Mu}\item  \label{item:intro:Mu} 
		For any $x \in X$, no more than $d+1$ members of $\Coseco$ contain $x$. 
		\renewcommand{\labelenumi}{\textup{(\theenumi)}} \renewcommand{\theenumi}{Eq}\item  \label{item:intro:Eq} 
		Each $\Ne_{\Cose}$ is \emph{$\Lo$-equivariant} in the following sense: for any $x \in \Cose$ and for any $g \in \Lo$, if $\alpha_{g} (x) \in \Cose$ then $\Ne_{\Cose}\left( \alpha_{g} (x) \right) = \alpha_g \left( \Ne_{\Cose}(x) \right)$.  
		\renewcommand{\labelenumi}{\textup{(\theenumi)}} \renewcommand{\theenumi}{Th}\item  \label{item:intro:Th} 
		For any $\Cose \in \Coseco$ and for any $y \in \Ne_{\Cose}(\Cose)$, the fiber $\Ne_{\Cose}^{-1}(y)$ is contained in the $r$-ball around $y$.
		\renewcommand{\labelenumi}{\textup{(\theenumi)}} \renewcommand{\theenumi}{Ca}\item  \label{item:intro:Ca} 
		For any $\Cose \in \Coseco$, we have $\left| \Ne_{\Cose} (\Cose) \right| \leq \Binu$. 
	\end{enumerate}
}
Here the labels stand for ``long'', ``multiplicity'', ``equivariant'', ``thin'', and ``cardinality'', respectively. 

The full-fledged definition is presented in
Definition~\ref{\LTCdef}, while more motivations in relation with the proof of the main theorem will be discussed in Subsection~\ref{subsec:ideas-main-ltc}. 
We conclude this subsection by pointing out that the LTC dimension is invariant under continuous orbit equivalence (see Proposition~\ref{prop:ltc-dim-coe}).

\subsection{Ideas behind the main theorem: ``nearsighted stabilizers''} \label{subsec:ideas-main-stabilizers}

In the following few subsections, 
we provide an informal description of 
the ideas behind the proof of our main theorem, which we hope will help the reader digest the long proof (some familiarity with the Rokhlin dimension theory may be helpful). 

As a starting 
point, we recall the central idea in the proof of the main theorem in 
\cite{Hirshberg-Wu16}. For simplicity of the discussion, let us assume that $X$ 
is a compact metric space. Suppose that $\alpha$ is an automorphism of $C(X)$; 
we will also use $\alpha$ to refer 
to the induced homeomorphism of the Gelfand spectrum $X$. In this 
context, we say $\alpha$ has Rokhlin dimension at most $d$ (in the single tower 
version), if for any $\eps>0$ and for any $n > 0$ (or for arbitrarily large 
$n>0$) there exists a partition of unity $\{\mu_j^{(l)}\}_{j=0,1,\ldots,n-1 ; l = 
0,1,\ldots d}$ such that for any $l$, if $j \neq j'$ then $\| \mu_j^{(l)} 
\mu_{j'}^{(l)} \| < \eps $ and such that for any $j$ and for any $l$, we have 
$\|\alpha ( \mu_j^{(l)} ) - 
\mu_{j+1}^{(l)} \| < \eps$ (where addition is taken modulo $n$). We sometimes 
refer to the $l$ in the superscript as a \emph{color}, so that for any color $l$, the collection $\{\mu_j^{(l)}\}_{j=0,1,\ldots,n-1}$ is considered a Rokhlin tower in the sense that the ``levels'' $\mu_j^{(l)}$ are approximately orthogonal and approximately translates of one another. 

Finite Rokhlin dimension implies freeness \textemdash\ and the two notions are equivalent for spaces with finite covering dimension (\cite{HWZ}) \textemdash\ however one 
can consider a finer, quantitative version that works for non-free actions: for a fixed $\eps>0$ and fixed 
length $n$, in order to obtain  a partition of unity 
$\{\mu_j^{(l)}\}_{j=0,1,\ldots,n-1 ; l = 	0,1,\ldots d}$ with the properties 
above, one does not need to assume that the action is free, but rather that 
 there is 
some number $N$, depending on $\eps>0$, on $n$ and on the dimension 
of 
the space $X$, such that if 
there are no orbits with fewer than $N$ points, then such partitions of unity 
exist. In order to find decomposable approximations for a given finite set of 
the crossed product to within a given tolerance, one can find a suitable 
$\eps>0$ and $n$ needed for the argument based on Rokhlin dimension to work. We 
can then find the corresponding smallest allowed orbit size $N$, and decompose 
the space $X$ into the closed invariant subset $X_{\leq N}$ of points of period 
at most $N$, and the complement $X_{>N}$. Crossed products by periodic 
automorphisms are subhomogeneous, and one can use that fact in order to obtain 
bounds for the nuclear 
dimension of the quotient $C(X_{\leq N}) \rtimes \Z$, essentially combining the dimension of the quotient space $X_{\leq N} / \Z$ (which is Hausdorff) and the nuclear dimension of the group $C^*$-algebras of the stabilizer groups (in this case they are all isomorphic to $\Z$).  
One uses the Rokhlin 
towers to obtain approximations for the long orbit part $C_0(X_{>N})$. Finally, 
we patch those approximations together. A similar idea involving long vs.\ short orbits works for actions 
of 
$\R$, although some of the underlying topological techniques turn out to be 
quite different.

The method described above relies on the fact that for an action of $\Z$, at 
each point, the action is either periodic or free, and we can make a 
distinction between short periods and long periods. Furthermore, we use the 
explicit description of the crossed product corresponding to the short orbit 
part as a subhomogeneous $C^*$-algebra.  
However, this relatively 
hands-on 
approach does not seem to generalize well, even to actions of $\Z^2$, where 
one has to deal with 
a zoo of infinite subgroups with infinite index! 
Intuitively, some orbits may be long in one direction and short 
in another, and there are infinitely many different directions, and thus there is no 
immediate 
analogue of the filtration of points by orbit length which we have in the case 
of actions of $\Z$.  

Instead, we build on the 
idea of limiting ourselves to a finite number of 
subgroups of $G$ that sufficiently resemble actual stabilizers for the purpose of our approximations. 
More precisely, given a finite subset $S$ of the group which we 
care about for the purpose of constructing an approximation, we replace the 
stabilizer group $G_x$ of a point $x$ by 
the subgroup generated by $G_x \cap S$, which may be called the ``nearsighted stabilizer subgroups'' and thought of as the 
stabilizer ``as seen by $S$''. For example, with $G = \Z$ and $S = 
\{-n,-n+1,\ldots,n\}$, if a point $x$ has orbit size greater than $n$, 
then its ``nearsighted stabilizer group'' is trivial; this means that for the 
purposes of the constructions we carry out, the orbit size is long enough so 
that this point behaves as though it is acted upon freely. 
Hence by introducing the ``nearsighted stabilizer groups'', we have found a systematic way to generalize the short orbit - long orbit split 
used for $\mathbb{Z}$-actions in \cite{Hirshberg-Wu16}. 
For more general groups, it 
would mean that for the purpose of the construction, the point 
behaves as though its stabilizer group is smaller than it actually is.

\subsection{Ideas behind the main theorem: LTC dimension recast} \label{subsec:ideas-main-ltc}
To gain some more intuition, let us first run through what the proof of our main theorem roughly looks like if the action in 
question is free, the space $X$ is compact and metrizable, and when convenient,  the group is $\Z$. 
This is a variation 
of the known arguments (e.g., using Rokhlin dimension \cite[Theorem~4.1]{HWZ}, dynamical asymptotic dimension \cite[Theorem~8.6]{GuentnerWillettYu2017Dynamic}, and tower dimension \cite[Theorem~6.2]{kerr}), 
recast in a 
somewhat different form which will be more useful for our purposes.  

Denoting the 
action by $\alpha$, we start 
out with a tolerance $\eps>0$ and a finite set of the crossed product that we want 
to 
approximate. We may assume that the finite set in question involves a finite 
number of elements $F \subseteq C(X)$ and a finite number of elements in the 
group. Consider the forgetful map $C(X) \mapsto l^{\infty}(X)$ gotten by 
discretizing the 
space $X$. This map is equivariant, and therefore induces a homomorphism $C(X) 
\rtimes G \to l^{\infty}(X) \rtimes G$. If the action is free, then $X$ 
decomposes into a disjoint set of orbits, each of which looks like $G$ acted 
upon by itself via left translation, so 
just as $l^{\infty}(G) \rtimes G$ is the uniform Roe algebra\footnote{In fact, it is the maximal uniform Roe algebra that corresponds to the maximal crossed product, but we are not going to pay close attention to the distinction between maximal and reduced completions in this introduction, as it does not play an essential role in the paper. } of the Cayley graph of $G$ (assuming for the ease of our exposition that $G$ is finitely generated), $l^{\infty}(X) \rtimes G$ is the uniform Roe algebra of a disjoint union of copies of the Cayley graph of $G$. 
By \cite[Theorem~8.5]{winter-zacharias}, the nuclear dimension of $l^{\infty}(X) \rtimes G$ is bounded from above by the asymptotic dimension of this disjoint union, which is easily seen to be equal to the asymptotic dimension of (the Cayley graph of) $G$ itself. 
Supposing that $G = \Z$ for illustration, we pick some large $n 
\in \N$ and a small $\delta>0$, depending on the finite collection $F$ in 
$C(X)$ and the finite number of group elements we started with. The constant 
$\delta>0$ is chosen such that the functions in $F$ have small oscillations on any 
set of diameter bounded by $\delta$ (depending on the choices 
above). In the actual proof, this is modified in two ways: first, as we do 
not assume that the space is metrizable, the notion of small diameter is 
replaced by requiring that those sets are subsets of some given cover; second, 
as we work with $C_0(X)$-algebras, thought of as sections of possibly 
non-locally-trivial bundles, we need a suitable substitute for the argument by oscillations.
With this in mind, as in the definition of Rokhlin dimension for $\Z$-actions, we show 
that one can find a partition of unity 
$\{\mu_k^{(l)}\}_{k=-n,-n+1,\ldots,n \, ; \, l=0,1,\ldots d}$ (with $d$ being a 
suitable 
dimension) such that:
\begin{enumerate}
	\renewcommand{\labelenumi}{\textup{(\theenumi)}} \renewcommand{\theenumi}{Or}\item  \label{item:intro:Or} 
	for any $l$, if $k \neq k'$ then $\mu_k^{(l)} \mu_{k'}^{(l)} = 0$; 
	
	\renewcommand{\labelenumi}{\textup{(\theenumi)}} \renewcommand{\theenumi}{In}\item  \label{item:intro:In} 
	$\alpha(\mu_k^{(l)}) \approx \mu_{k+1}^{(l)}$;
	(the approximation is in norm, and for the purpose of the introduction we 
	are not specifying the required tolerance;)
	
	\renewcommand{\labelenumi}{\textup{(\theenumi)}} \renewcommand{\theenumi}{Th'}\item  \label{item:intro:Th'} 
	the diameter of the support of each element $\mu_k^{(l)}$ is less than 
	$\delta$.
	
	\renewcommand{\labelenumi}{\textup{(\theenumi)}} \renewcommand{\theenumi}{Eq'}\item  \label{item:intro:Eq'} 
	we furthermore have a choice of a point $x_l$, such that for 
	$k=-n,-n+1,\ldots,n $, the distance between the point $\alpha^k(x_l)$ is 
	less than $\delta$ away from any point in the support of $\mu_k^{(l)}$.
\end{enumerate}
Here the labels stand for ``orthogonal'', ``invariant'', ``thin'', and ``equivariant''. 
The last two conditions are not part of the definition of finite 
Rokhlin 
dimension, and indeed this version of Rokhlin dimension (which essentially amounts to the 
LTC dimension) is strictly stronger. For example, it implies that 
the space must have finite covering dimension. We are being deliberately 
imprecise here \textemdash\ there are additional requirements (see Proposition~\ref{prop:ltc-dim}), but we omit 
them for the purpose of 
this discussion.

The notation above is intended to be close to what has been used in 
papers involving the Rokhlin dimension, and is therefore repeated for the purpose 
of the introduction, so as to help the reader who may be familiar with Rokhlin 
dimension methods. However, we caution the reader that we use different 
notation in the actual proofs in this paper. To go from what is described in the 
introduction to what 
appears in the actual proof, we may in fact have several disjoint towers 
corresponding to the same color, we will use  $\Cose[LTC]$ to denote an open set 
which is roughly the union of all open supports of 
a tower 
(such $\Cose$'s will form an open cover $\Coseco[LTC]$), 
and the sum of the elements in towers of color $r$ will 
be denoted by $\Po[LTC]^{(r)}$. 
For each such $\Cose[LTC]$, we will have 
a \emph{near orbit selection function} 
$\Ne[LTC]_{\Cose[LTC]} \colon \Cose[LTC] \to X$, so 
that the image of $\Ne[LTC]_{\Cose[LTC]}$
 will be the collection denoted 
$\{\alpha^k(x)\}_{k=-n,\ldots,n}$ above. For now, we maintain this ad-hoc 
notation for the purpose of the introduction.

We illustrate this in the following image. The blobs are intended to indicate 
the open supports of $\mu_k^{(l)}$ (for a given chosen $l$), and $x = x_l$. We 
stress that we do not require that the supports get 
mapped exactly one to the other, so the picture is not that of an open tower or a topological 
Rokhlin tower. The point $x$ gives us a selection of a nearby orbit, which 
behaves like this ``Rokhlin tower''. The point $x$ need not be contained within the 
support, but merely be near it, and it may well be \textemdash\ as illustrated \textemdash\ that some 
translates of $x$ under the relevant group elements will be inside the support 
and some may be outside. This flexibility turns out to be useful when we try to give upper bounds for the resulting dimension concept. 
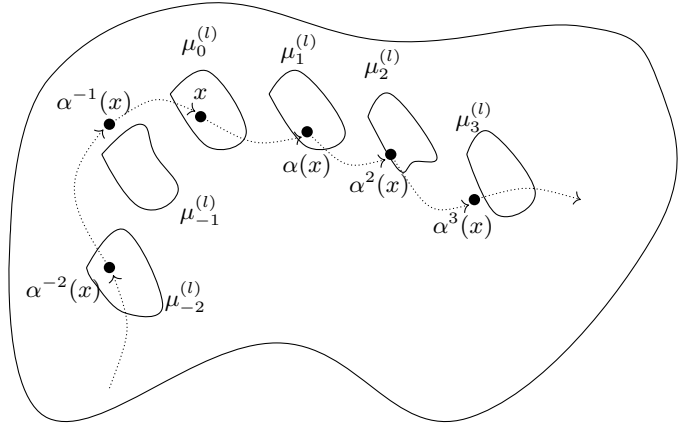
\begin{center}
\begin{tikzpicture}
	\footnotesize
	\draw  plot[smooth, tension=0.7] coordinates {(-3.5,0.5) (-3,2.5) (-1,3.5) 
	(1.5,3) (4,3.2) (5,2.5) (5,0.5) (2.5,-2) (0,-1) (-3,-2) (-3.5,0.5)};
	\draw plot[smooth, tension=.7] coordinates {(-2.5,0) (-2,0.5) (-1.5,-0.5)  
	(-2,-0.6) (-2.4,-0.2) (-2.5,0)};
	\draw plot[smooth, tension=.8] coordinates {(-2.3,1.5) (-1.8,1.9)  
	(-1.6,1.5)  (-1.3,1)  
	(-1.8,0.8) (-2.3,1.5)};
	\draw plot[smooth, tension=.7] coordinates {(-1.4,2.3) (-0.9,2.6) 
	(-0.4,1.8)  (-0.9,1.6) (-1.4,2.3)};
	\draw plot[smooth, tension=.7] coordinates {(-0.1,2.3) (0.3,2.6) (0.9,1.8)  
	(0.4,1.6) (-0.1,2.3)};
	\draw plot[smooth, tension=.7] coordinates {(1.2,2) (1.5,2.3) (2.1,1.5)  
	(1.8,1.4) (1.6,1.3) (1.2,2)};
	\draw plot[smooth, tension=.7] coordinates {(2.5,1.5) (2.8,1.8) (3.4,1)  
	(2.9,0.7) (2.5,1.5)};
	\node at (-2.2,0)[circle,fill,inner sep=1.5pt]{};
	\node at (-2.2,1.9)[circle,fill,inner sep=1.5pt]{};
	\node at (-1,2)[circle,fill,inner sep=1.5pt]{};
	\node at (0.4,1.8)[circle,fill,inner sep=1.5pt]{};
	\node at (1.5,1.5)[circle,fill,inner sep=1.5pt]{};
	\node at (2.6,0.9)[circle,fill,inner sep=1.5pt]{};
	\draw[smooth,densely dotted, ->] (-2.2,-1.6) .. controls (-1.9,-0.8) .. 	
	(-2.15,-0.1);
	\draw[smooth,densely dotted, ->] (-2.2,0) .. controls (-2.8,0.9) and  
	(-2.8,1.2) .. 	(-2.28,1.83);
	\draw[smooth,densely dotted, ->] (-2.2,1.9) .. controls (-1.6,2.3) ..	
	(-1.07,2.07);
	\draw[smooth,densely dotted, ->] (-1,2) .. controls (-0.3,1.6) .. 
	(0.33,1.73);
	\draw[smooth,densely dotted, ->] (0.4,1.8) .. controls (0.9,1.3) .. 
	(1.43,1.43);
	\draw[smooth,densely dotted, ->] (1.5,1.5) .. controls (2,0.7) .. 
	(2.53,0.83);
	\draw[smooth,densely dotted, ->] (2.6,0.9) .. controls (3.3,1.1) .. 
	(4,0.9);
	\node at (-1,2.3){$x$};
	\node at (-2.4,2.2){$\alpha^{-1}(x)$};
	\node at (-2.8,-0.3){$\alpha^{-2}(x)$};
	\node at (0.4,1.35){$\alpha(x)$};
	\node at (1.35,1.1){$\alpha^{2}(x)$};
	\node at (2.45,0.55){$\alpha^{3}(x)$};
	\node at (-1,3){$\mu_0^{(l)}$};
	\node at (0.3,2.9){$\mu_1^{(l)}$};
	\node at (1.4,2.7){$\mu_2^{(l)}$};
	\node at (2.6,2){$\mu_3^{(l)}$};
	\node at (-1,0.75){$\mu_{-1}^{(l)}$};
	\node at (-1.2,-0.4){$\mu_{-2}^{(l)}$};
	\normalsize
\end{tikzpicture}
\end{center}

Let us complete our sketch of the proof in the free case. 
By now we already have a finite collection of ``Rokhlin towers'', and to each we have associated a 
chosen (partial) orbit. By evaluating our given functions just on this finite 
collection of 
orbits, we more or less approximately factor through a finite direct sum of 
copies of the Roe algebra. Thinking of elements of the Roe algebra as infinite 
matrices, each such evaluation 
 sends a function $f$ (for $G = \Z$) to the diagonal matrix
\[
f \mapsto \mathrm{diag}(\ldots, f(\alpha^{-n}(x)), f(\alpha^{-n+1}(x)), 
f(\alpha^{-n+2}(x)), \ldots,  f(\alpha^{n}(x)), \ldots ) .
\]
The canonical unitary $u$ is mapped to a bilateral shift matrix. We 
then cut down to a finite matrix, using a partition of unity at the level of 
the group, which comes from covers of the group that witness its having finite 
asymptotic dimension.
We define completely positive maps from the Roe algebra back to the crossed 
product by taking a finite discrete sequence, thought of as the values of a 
given function on the relevant piece of the orbit, and using it as coefficients 
which ``modulate'' the partition of unity given by the Rokhlin elements 
$\mu_k^{(l)}$.

The partial orbit we chose is in fact a natural way to label the levels in a ``Rokhlin tower'': rather than indexing them by group elements, we could index them by 
the points in the orbit they mirror. This distinction is insignificant in the 
free case, but points to the generalization we need in the non-free case: we 
look for Rokhlin-type towers which are modeled by partial orbits, which may or 
may not be free. 
This is illustrated in the figure below. It shows one orbit 
which is free \textemdash\ or at least long enough to be treated as free \textemdash\ and another 
orbit 
with just three elements: here the ``Rokhlin tower'' consists of just three 
elements, which approximately behave like the nearby orbit. For the purpose of 
illustration, the two towers are drawn as being disjoint; however if the 
colors of those towers are different, they may overlap. 

\begin{center}
\begin{tikzpicture}
	\footnotesize
	\draw  plot[smooth, tension=0.7] coordinates {(-4.5,0.5) (-3,2.5) (-1,3.5) 
		(1.5,3) (4,3.2) (5,2.5) (5.5,0.5) (2.5,-5) (0,-4) (-4.6,-4.8) 
		(-4.5,0.5)};
	\draw plot[smooth, tension=.7] coordinates {(-2.5,0) (-2,0.5) (-1.5,-0.5)  
		(-2,-0.6) (-2.4,-0.2) (-2.5,0)};
	\draw plot[smooth, tension=.8] coordinates {(-2.3,1.5) (-1.8,1.9)  
	(-1.6,1.5) (-1.3,1)  
		(-1.8,0.8) (-2.3,1.5)};
	\draw plot[smooth, tension=.7] coordinates {(-1.4,2.3) (-0.9,2.6) 
		(-0.4,1.8)  (-0.9,1.6) (-1.4,2.3)};
	\draw plot[smooth, tension=.7] coordinates {(-0.1,2.3) (0.3,2.6) (0.9,1.8)  
		(0.4,1.6) (-0.1,2.3)};
	\draw plot[smooth, tension=.7] coordinates {(1.2,2) (1.5,2.3) (2.1,1.5)  
		(1.8,1.4) (1.6,1.3) (1.2,2)};
	\draw plot[smooth, tension=.7] coordinates {(2.5,1.5) (2.8,1.8) (3.4,1)  
		(2.9,0.7) (2.5,1.5)};
	\node at (-2.2,0)[circle,fill,inner sep=1.5pt]{};
	\node at (-2.2,1.9)[circle,fill,inner sep=1.5pt]{};
	\node at (-1,2)[circle,fill,inner sep=1.5pt]{};
	\node at (0.4,1.8)[circle,fill,inner sep=1.5pt]{};
	\node at (1.5,1.5)[circle,fill,inner sep=1.5pt]{};
	\node at (2.6,0.9)[circle,fill,inner sep=1.5pt]{};
	\draw[smooth,densely dotted, ->] (-2.2,-1.6) .. controls (-1.9,-0.8) .. 	
	(-2.15,-0.1);
	\draw[smooth,densely dotted, ->] (-2.2,0) .. controls (-2.8,0.9) and  
	(-2.8,1.2) .. 	(-2.28,1.83);
	\draw[smooth,densely dotted, ->] (-2.2,1.9) .. controls (-1.6,2.3) ..	
	(-1.07,2.07);
	\draw[smooth,densely dotted, ->] (-1,2) .. controls (-0.3,1.6) .. 
	(0.33,1.73);
	\draw[smooth,densely dotted, ->] (0.4,1.8) .. controls (0.9,1.3) .. 
	(1.43,1.43);
	\draw[smooth,densely dotted, ->] (1.5,1.5) .. controls (2,0.7) .. 
	(2.53,0.83);
	\draw[smooth,densely dotted, ->] (2.6,0.9) .. controls (3.3,1.1) .. 
	(4,0.9);
	\node at (-1,2.3){$x_0$};
	\node at (-2.4,2.2){$x_{-1}$};
	\node at (-2.75,-0.2){$x_{-2}$};
	\node at (0.4,1.35){$x_1$};
	\node at (1.4,1.15){$x_2$};
	\node at (2.45,0.55){$x_3$};
	\node at (-1,3){$\mu_{x_0}^{(l)}$};
	\node at (0.3,2.9){$\mu_{x_1}^{(l)}$};
	\node at (1.4,2.7){$\mu_{x_2}^{(l)}$};
	\node at (2.6,2.1){$\mu_{x_3}^{(l)}$};
	\node at (-1,0.75){$\mu_{x_{-1}}^{(l)}$};
	\node at (-1.1,-0.4){$\mu_{x_{-2}}^{(l)}$};
	\draw plot[smooth, tension=.9] coordinates {(-0.5,-2) (0,-2.5) 
	(0.5,-2.3)  
	(0.3,-1.8) (0,-2.2) (-0.2,-1.9) (-0.5,-2)};
	\draw plot[smooth, tension=.9] coordinates {(0.5,-3) (1,-3.5) 
	(1.5,-3.3)  
	(1.3,-2.8) (1,-3) (0.8,-2.9) (0.5,-3)};
	\draw plot[smooth, tension=.9] coordinates {(1.5,-2) (2,-2.5) 
	(2.5,-2.3)  
	(2.3,-1.8) (2,-1.7) (1.7,-1.7) (1.5,-2)};
	\node at (0,-1.9)[circle,fill,inner sep=1.5pt]{};
	\node at (0,-1.6){$y$};
	\node at (2,-2.2)[circle,fill,inner sep=1.5pt]{};
	\node at (2,-1.9){$y'$};
	\node at (1,-3.2)[circle,fill,inner sep=1.5pt]{};
	\node at (1.2,-3.35){$y''$};
	\draw[smooth,densely dotted, ->] (0,-1.9) .. controls (1,-1.8) .. 
	(1.93,-2.13);
	\draw[smooth,densely dotted, ->] (2,-2.2) .. controls (1.4,-2.5) .. 
	(1.07,-3.13);
	\draw[smooth,densely dotted, ->] (1,-3.2) .. controls (0.3,-2.6) .. 
	(0.04,-1.98);
	\node at (-0.5,-2.6){$\mu_{y}^{(l')}$};
	\node at (2.9,-2.3){$\mu_{y'}^{(l')}$};
	\node at (1.5,-3.8){$\mu_{y''}^{(l')}$};
	\normalsize
\end{tikzpicture}
\end{center}

If one adapts the conditions~\eqref{item:intro:Or}, \eqref{item:intro:In}, \eqref{item:intro:Th'}, and~\eqref{item:intro:Eq'} to the non-free setting in this way and inserts condition~\eqref{item:intro:Ca} to take coarse geometry into account (see the last paragraph of Subsection~\ref{subsec:ideas-main-coarse} for why we do this), one precisely arrives at the notion of the LTC dimension described in Subsection~\ref{subsec:towers-non-free}: a sufficiently ``long'' open cover gives rise to a partition of unity subordinate to it and satisfying \eqref{item:intro:In}, and condition~\eqref{item:intro:Mu} may be upgraded to a desired $(d+1)$-coloring on the partition of unity via a barycentric subdivision argument. The details of this connection can be found in Proposition~\ref{prop:ltc-dim}.

\subsection{Ideas behind the main theorem: the coarse orbit space} \label{subsec:ideas-main-coarse}
To complete our adaptation of the proof strategy that works in the free case to the general, possibly non-free setting, we need one more major idea. 
Recall that in the free case, we have used two layers of partitions of unity, one coming from the ``Rokhlin towers'', and the other from the asymptotic dimension of (the Cayley graph of) $G$. 
In Subsection~\ref{subsec:ideas-main-ltc}, we have adapted the former partition of unity to the non-free setting with the help of the LTC dimension. In this subsection, we deal with the latter one. 

A closer look at the free case suggests that this latter kind of partitions of unity ought to be nothing but the one that gives the ``natural'' upper bound (following \cite[Theorem~8.5]{winter-zacharias}) on the nuclear dimension of $\ell^\infty(X) \rtimes G$ by the asymptotic dimension of $G$, using the fact that this crossed product can be identified with the uniform Roe algebra of a disjoint union of copies of the Cayley graph of $G$. 
The reader may use as a guiding example the single-orbit case where $X = G$. 

Now in the general setting, $\ell^\infty(X) \rtimes G$ may no longer be a uniform Roe algebra.  However, in the single-orbit case where $X = G / H$, we may combine the proof of \cite[Theorem~8.5]{winter-zacharias} with Green's imprimitivity theorem (\cite{Green1980structure}) to given a ``natural'' upper bound of the nuclear dimension of $\ell^\infty(G/H) \rtimes G$ by a combination of the asymptotic dimension of the Schreier graph associated to $G/H$ and the nuclear dimension of $C^*(H)$. 
Extending this to the general, multi-orbit case, we can give a ``natural'' upper bound of the nuclear dimension of $\ell^\infty(X) \rtimes G$ by a combination of the asymptotic dimension of a disjoint union of Schreier graphs associated to $G/G_x$ and the supremum of nuclear dimension of $C^*(G_x)$, where $x$ ranges over $X$, a fact whose proof involves partitions of unity arising from said asymptotic dimension. 

In order to give a coordinate-free description of ``the disjoint union of Schreier graphs associated to $G/G_x$'' that also works when $G$ is not finitely-generated, we introduce the notion of \emph{coarse orbit space}, which is just the set $X$ equipped with a suitable coarse structure, where each orbit $G \cdot x$ forms a coarse connected component that is coarsely equivalent to $G / G_x$. 
It is precisely the partitions of unity arising from the asymptotic dimension of the coarse orbit space that we shall use in the proof of our main theorem. 
The role of the coarse orbit space here may be compared with that of box spaces in \cite{SWZ}. 

In the actual proof of our main theorem (Theorem~\ref{thm:dimnuc-main}), we do not work with $\ell^\infty(X) \rtimes G$, since this would involve infinitely many different 
orbit types, which would make it unmanageable for us to construct (uniformly-approximately) order zero maps back into $C(X) \rtimes G$. 
As was explained in Subsection~\ref{subsec:ideas-main-stabilizers}, we form the ``nearsighted stabilizers'' to reduce the complexity of the problem to a finite number of cases. 
This calls for a blow-up construction of sorts, where homogeneous spaces of the form $G / 
G_x$ are 
replaced by $G / \left\langle G_x \cap S \right\rangle$, 
for a suitably large set $S$ 
such that functions in the partition of unity constructed above can be lifted from $G \cdot x \cong G / G_x$ to $G / \left\langle G_x \cap S \right\rangle$ whenever $x$ is in the support of said function.

\subsection{Finite bounds on the long thin covering dimension} \label{subsec:ltc-bounds}

Besides the proof of the main theorem, substantial footwork is devoted to 
showing that the theorem is applicable for large classes of group actions, as mentioned just after we introduced our main theorem earlier. 
The key here is to find finite upper bounds for the LTC dimension. 

We show in Theorem~\ref{thm:lsp-ltc} that actions of finitely generated virtually nilpotent 
groups on finite-dimensional spaces have finite LTC dimension. 
Two main ingredients go into the proof. 
The first is a notion we introduce, called 
\emph{large scale packing constants}, which are geometric-group-theoretic invariants for the acting group $G$ and roughly look like relative versions of the growth condition, baring resemblence to Tempelman's condition (see \cite[Remark~8.5]{kerr-szabo}). The rough idea is to find a uniform bound that controls, for any $r > 0$, 
the number of disjoint $R/2$-balls in a ball of radius $R+r$ in the Cayley graph of $G$ for some large enough $R$ 
(or more generally, the number of $R/2$-balls allowing for controlled overlaps, 
as a higher dimensional version) \textemdash\ though the precise definition (see Definition~\ref{def:LSP}) does not require the use of balls. 
We use a refined estimate (\cite{Breuillard}) of the polynomial growth of finitely generated virtually 
nilpotent groups to obtain finite bounds on their large scale packing constants (see Theorem~\ref{thm:lsp-vnil}). 
The second ingredient is a string of general position lemmas involving two mathematical inductions (see Lemmas~\ref{lem:general-position-action} and~\ref{lem:general-position-combinatorial}) that allow us to efficiently lower the multiplicity of a cover by carefully ``trimming away'' some of the overlaps. 
Roughly speaking, 
the proof of Theorem~\ref{thm:lsp-ltc} (presented at the end of the section) starts with building an open cover of the space by Rokhlin-type towers that is ``thin'' but not necessarily ``long'', and then, wishing to extend these towers to obtain a ``long'' cover but wary of the increased multiplicity brought about by the new overlaps, we apply the general position lemmas to ``trim'' the extended towers to control the multiplicity, using the large scale packing constants as a bookkeeping device. 

In order to go beyond the case of virtually 
nilpotent groups (or even amenable groups, for that matter), 
we show how to use the aforementioned notion of equivariant asymptotic 
dimension to upgrade LTC dimension bounds from a family of subgroups (e.g., the family of all virtually cyclic subgroups) to a larger group we care about (see Theorem~\ref{thm:relative-bound-ltc}). 
The proof of this result requires us to reformulate the equivariant asymptotic dimension in a way more compatible with the LTC dimension (see Lemma~\ref{lem:BLR-reformu}) \textemdash\ in some sense, the equivariant asymptotic dimension amounts to the LTC dimension without the ``thinness'' requirement, a perspective that was already reflected in 
Subsection~\ref{subsec:towers-non-free} 
\textemdash\ and then apply another careful inductive argument (see Lemma~\ref{lem:relative-bound-ltc-inductive}). 
This result can be combined with earlier results of Bartels, L\"{u}ck, and Reich to show that various actions of 
hyperbolic groups have finite LTC dimension as well (see Corollary~\ref{cor:hyperbolic-ltc}).

A byproduct of our result on actions by finitely generated virtually nilpotent 
groups is showing finite nuclear dimension for crossed products arising from 
profinite actions by wreath products of the form $H \wr \mathbb{Z}^d$, where 
$H$ is a finite abelian group (see 
Proposition~\ref{prop:allosteric-finite-nuc-dim}). 
To do this, we use Green's imprimitivity theorem and the Pontryagin duality to write such a crossed product as a direct limit of (non-simple) crossed products by $\mathbb{Z}^d$ of subhomogeneous algebras whose spectra are zero-dimensional; this allows us to derive the result by combining (the $C_0(X)$-algebra version of) our main theorem, our bound on the LTC dimension for topological $\mathbb{Z}^d$-actions, and well-known permanence properties of the nuclear dimension. What makes this result particularly interesting is the fact that some of these profinite actions are allosteric (see Corollary~\ref{cor:allosteric-construction}) and thus not almost finite, a fact we show by extending the method of Joseph (\cite{Joseph2023}).

\subsection{Organization of the paper} \label{subsec:organization}

After a section involving preliminaries, we 
discuss in Section \ref{sec:OCS} the asymptotic dimension of the orbit coarse 
structure, that is, a coarse space consisting of all orbits. This replaces (and 
improves upon) 
the 
role of the asymptotic dimension of the group from the free case. The development of this section may also be viewed as a baby version of what is to come in Section~\ref{sec:LTC}. 
Section 
\ref{sec:simplicial} involves some technical results in topology needed to 
refine various open covers and partitions of unity in the next section. In Section \ref{sec:LTC} we 
introduce the long thin covering dimension, which is our main conceptual 
and technical tool in the paper, and prove various refinements and 
technicalities involving this dimension. Section \ref{sec:misc} contains a few 
simple bounds involving the dimensions above, which don't play an important 
role in 
the proof of our main theorem but may help the reader have a better understanding of the new concepts introduced. In Section \ref{sec:LSP}, we introduce the large 
scale packing constants, the geometric group theoretic invariant mentioned in 
the previous paragraph, and show that 
they are finite for finitely generated virtually nilpotent groups. This notion 
plays an instrumental role in showing that the asymptotic dimension and long 
 thin covering dimension mentioned above are finite for arbitrary actions of 
such groups. The more difficult of the above two bounds, namely the one involving the long 
thin covering dimension, is proved in Section \ref{sec:GP}. Section 
\ref{sec:BLR} develops a more refined notion of equivariant asymptotic 
dimension for an action relative to a family of subgroups. That is needed for 
going beyond the case of virtually nilpotent groups, for instance, to cover the 
case of actions of hyperbolic groups on their Gromov boundary. Section 
\ref{sec:dimnuc} assembles the tools developed in the previous sections in 
order to prove the main theorem concerning a bound for the nuclear dimension of 
the crossed product. In Section \ref{sec:allosteric actions}, we generalize the 
construction of \cite{Joseph2023} to provide examples of allosteric actions on the Cantor set by 
certain amenable groups with finite asymptotic dimension, 
and we show how to apply our techniques to prove they give rise to classifiable crossed products.  In the appendices, we develop some 
side technicalities, 
clarifications and refinements, which may be of independent interest but are 
not strictly necessary to obtain the main results.

\section{Preliminaries}	\label{sec:prelim} 
\renewcommand{\sectionlabel}{LTC}

The following conventions are used throughout the paper.
If $A$ is a $C^*$-algebra and $r \geq 0$, we sometimes denote by $A_{<r}$ (respectively, $A_{\leq r}$) the open (respectively, closed) ball of radius $r$
in $A$, and add the subscript $+$ to restrict to the positive elements therein, as in  $A_{+}$, $A_{+, <r}$, $A_{+, \leq r}$, etc. 
	Suppose $X$ is a set, and $G$ is a group acting on $X$. 
We denote the set of fixed points under the action of $G$ by $X^G$. For $x \in 
X$, we use the notation $G_x$ for the stabilizer subgroup of $x$, that is, the 
subgroup of elements 
of $G$ which fix the point $x$. If $X$ is a topological space, we sometimes use 
the notation $K \Subset X$ to denote (or stress) that $K$ is a compact subset; 
often this is used in the case in which $X$ is discrete, in which case it means 
that $K$ is a finite subset.

Suppose $A$ is a $C^*$-algebra, suppose $G$ is a discrete group and suppose $\alpha \colon G \to \aut(A)$ is an action.  Let $E \colon A \rtimes_{\alpha} G \to A$ be the canonical conditional expectation. Denote by $\{u_g \mid g \in G\}$ the canonical unitaries. For $x \in  A \rtimes_{\alpha} G$ and $g \in G$, by the $g$-th Fourier coefficient of $x$ we mean $E(x u_g^{-1})$. The support of $x$ is $\{g \in G \mid E(x u_g^{-1}) \neq 0\}$.

We shall sometimes use the following notation in the paper.
\begin{Notation} \label{notation:group-action}
	Let $G$ be a discrete group, and let $\alpha$ be an action of $G$ on a $C^*$-algebra $A$. When there is no danger of confusion, we adopt the following notational conventions: 
	\begin{enumerate}
		\item For any $\alpha$-invariant subalgebra $B$ of $A$, the restriction of $\alpha$ on $B$ is still denoted by $\alpha$. 
		\item 
		If $A = C_0 (X)$ for a locally compact Hausdorff space $X$ (in this case, $X$ is isomorphic to $\widehat{A}$, the spectrum of $A$),
		then the induced action on $X$ is also denoted by $\alpha$, so that we have 
		\[
		\alpha_g (f) (x) = f \left( \alpha_{g^{-1}} (x) \right) \quad \text{for any } g \in G, f \in C_0(X) \text{ and } x \in X \; .
		\]
		This choice of notations may be justified by considering each point in 
		$X$ as a ``Dirac delta function'' on $X$. In this case, we also write, 
		for any $g \in G$, for any $x \in X$, for any $F \subseteq G$, for any 
		$ Y \subseteq X$, and for any collection $\mathcal{U}$ of subsets of 
		$X$, 
		\begin{enumerate}
			\item $\displaystyle \alpha_g (Y) = \left\{\alpha_g (y) \colon y \in Y \right\} \vphantom{\bigcup}$, 
			\item $\displaystyle \alpha_g (\mathcal{U}) = \left\{\alpha_g (U) \colon U \in \mathcal{U} \right\} \vphantom{\bigcup}$, 
			\item $\displaystyle \alpha_F (x) = \left\{\alpha_h (x) \colon h \in F \right\} \vphantom{\bigcup}$, 
			\item $\displaystyle \alpha^\cup_F (Y) = \bigcup_{h \in F} \alpha_h (Y) = \left\{\alpha_h (y) \colon h \in F, y \in Y \right\} = \bigcup_{y \in Y} \alpha_F (y)$, 
			\item $\displaystyle \alpha^\cap_F (Y) = \bigcap_{h \in F} \alpha_h (Y) = \left\{ y \in X \colon \alpha_{F^{-1}} (y) \subseteq Y \right\}$, and
			\item $\displaystyle \alpha^\wedge_F (\mathcal{U}) = \bigwedge_{h \in F} \alpha_h (\mathcal{U}) = \left\{ \bigcap_{h \in F} \alpha_h \left( U_h \right) \colon  \left( U_h \right)_{h \in F} \in \mathcal{U}^F \right\}$. 
		\end{enumerate}
		Notice that for any $F \subseteq G$, we have
		\[
			\alpha^\cup_F  \left( \alpha^\cap_{F^{-1}} (Y) \right) ) \subseteq Y \subseteq \alpha^\cap_{F^{-1}} ( \left( \alpha^\cup_F (Y) \right) \quad \text{ for any } Y \subseteq X \; ,
		\]
		and 
		\[
			\alpha^\cap_{F} \left( \bigcup \mathcal{U} \right)  \subseteq \bigcup \alpha^\wedge_{F} \left( \mathcal{U} \right)  \  \text{ for any collection } \mathcal{U} \text{ of subsets of } X \; .
		\]
	\end{enumerate}
\end{Notation}

\begin{Notation} \label{notation:group}
	Let $G$ be a discrete group and let $\Lo$ and $\Bo$ be subsets of $G$. We 
	write 
	\[
	\Lo \Bo = \left\{ g_1 g_2 \colon g_1 \in \Lo , \; g_2 \in \Bo \right\} \; .
	\]
	For $g \in G$, we also write 
	\[
	g \Lo = \{g\} \Lo \quad \text{and} \quad \Lo g = \Lo  \{g\}  \; .
	\]
	For $k = 1, 2, \ldots$, we write 
	\[
	\Lo^k = \left\{ g_1 \cdots g_k \in G \colon g_1 , \ldots, g_k \in \Lo \right\} \; .
	\]
	Also write $\Lo^0 = \{e\}$. 
\end{Notation}

Observe that for any $k \in \N$, we have $\left( \Lo^k \right)^{-1} = \left( \Lo^{-1} \right)^{k}$. If $e \in \Lo$, then we also have 
\[
\{e\} = \Lo^0 \subseteq \Lo^1 \subseteq \Lo^2 \subseteq \ldots \; .
\]

We will need to deal with a variety of dimensions in this paper. For these, it is sometimes convenient to adopt the following shorthand: we write $\dimone (-)$ for $\dim (-) + 1$, and the same goes for $\dimnucone$, $\asdim^{+1}$, $\dimltc^{+1}$, etc.

\subsection{Covering dimension}	\label{sec:prelim:dim}

We use the convention $K \Subset G$ to mean that $K$ is a compact subset of $G$.

\begin{Def}
	Let $X$ be a topological space. Let $\mathcal{U}$ and $\mathcal{U}'$ be two collections of subsets of $X$.
	\begin{enumerate}
		\item The \emph{multiplicity} of $\mathcal{U}$, denoted as $\mult(\mathcal{U})$, is defined by 
		\[
		\mult(\mathcal{U}) = \sup_{x \in X} |\{ U \in \mathcal{U} \colon x \in 
		U \}| .
		\]
		This takes value in $\Z^{\geq 0} \cup \{\infty\}$. Equivalently, 
		$\mult(\mathcal{U})$ is the infimum of natural numbers $d$ satisfying 
		that any $d+1$ different members of $\mathcal{U}$ have empty 
		intersection.
		\item The \emph{covering number} of $\mathcal{U}$, denoted as $\mathrm{Cov}(\mathcal{U})$, is defined to be $\mult(\mathcal{U}) - 1$, which takes value in $\Z^{\geq -1} \cup \{\infty\}$.
		\item When $\mathcal{U}$ is non-empty, a \emph{$d$-coloring} of $\mathcal{U}$ is a $(d+1)$-tuple $(\mathcal{U}^{(0)}, \ldots, \mathcal{U}^{(d)})$ of subcollections of $\mathcal{U}$ such that $\mathcal{U}^{(0)} \cup \ldots \cup \mathcal{U}^{(d)} = \mathcal{U}$ and each $\mathcal{U}^{(l)}$ consists of mutually disjoint members, or equivalently, $\mathrm{Cov}(\mathcal{U}^{(l)}) \leq 0$. In this case, $\mathcal{U}$ is said to be \emph{$d$-colored}.
		\item The \emph{coloring number} of $\mathcal{U}$, denoted as 
		$\mathrm{Col}(\mathcal{U})$, is defined, when $\mathcal{U}$ is 
		non-empty, to be the infimum of natural numbers $d$ such that there are 
		subcollections $\mathcal{U}^{(0)}, \ldots, \mathcal{U}^{(d)} \subseteq 
		\mathcal{U}$ with $\mathcal{U}^{(0)} \cup \ldots \cup \mathcal{U}^{(d)} 
		= \mathcal{U}$ and each $\mathcal{U}^{(l)}$ consisting of mutually 
		disjoint members. By convention, we set $\mathrm{Col}(\varnothing) = 
		-1$. 
		\item A collection $\mathcal{U}$ is said to \emph{refine} (or be a 
		\emph{refinement} of) $\mathcal{U}'$ if for any $U \in \mathcal{U}$, 
		there is $U' \in \mathcal{U}'$ such that $U \subset U'$. 
		\item The \emph{common refinement}, or \emph{join}, of $\mathcal{U}$ 
		and $\mathcal{U}'$, denoted by $\mathcal{U} \wedge \mathcal{U}'$, is 
		the collection $\left\{ U \cap U' \colon U \in \mathcal{U}, \ U' \in 
		\mathcal{U}' \right\}$. Note that $\bigcup \left( \mathcal{U} \wedge 
		\mathcal{U}' \right) = \left( \bigcup \mathcal{U} \right) \cap \left( 
		\bigcup \mathcal{U}' \right)$. 
	\end{enumerate}
\end{Def}

\begin{Def}
	The \emph{covering dimension} of a topological space $X$, denoted by $\dim (X)$, is the infimum of all nonnegative integers $d$ such that for any finite open cover $\Thseco$ of $X$, there is an open cover $\Coseco$ of $X$ that refines $\Thseco$ and has multiplicity at most $d+1$. 
	
	Note that $\dim(\varnothing) = -1$. 
\end{Def}

We record the following two results from dimension theory, for later use.

\begin{Thm}[{\cite[Chapter~3, Theorem~6.4]{Pears75}}]
	\label{thm:Pears75}
	If $M$ is a subspace of a totally normal space\footnote{A topological space 
	is totally normal if any subspace is normal.} $X$, then $\dim (M)
	\leq \dim (X)$.
\end{Thm}

\begin{Prop}[{\cite[Chapter~9, Proposition~2.16]{Pears75}}]
	\label{prop:Pears75}
	If $X$ and $Y$ are weakly paracompact\footnote{A topological space is 
	weakly paracompact if any cover has a refinement which is point-finite, 
	that is, such that any point is contained in finitely many elements of the 
	refinement.} normal Hausdorff spaces and
	$f\colon X \to Y$ is a continuous open surjection such that
	$f^{-1}(y)$ is finite for each point of $Y$, then $\dim (X) = \dim (Y)
	$. 
\end{Prop}

For separable metric spaces, there are refined results from dimension theory, 
which will be useful in our discussion of general positions in 
Section~\ref{sec:GP} (in ways similar to the way in which they were used, for 
example, in 
\cite[Section 3]{Lindenstrauss95} and \cite[Section 3]{szabo}).

\begin{Lemma}[{\cite[4.1.5, 4.1.7, 4.1.9, 4.1.14, 4.1.16]{Engelking}}] \label{lem:Engelking}
	Let $X$ be a separable metric space and let $A, B, B_0, B_1, B_2, \ldots \subseteq X$. 
	{
		\begin{enumerate} 
			\item \label{D1} $A\subset B$ implies $\dim(A)\leq \dim(B)$.
			\item \label{D2} If $\left\{ B_i \right\}_{i\in \N}$ is a countable family of closed sets in $A$ with $\dim(B_i)\leq k$, then $\dim(\bigcup B_i)\leq k$.
			\item \label{D3} Let $\Lase \subset A$ be a zero dimensional 
			subset, let $x\in A$ be a point and let $\Cose$ be an open 
			neighborhood of $x$. Then there exists an open set $\Cose'\subset 
			A$ with $x\in \Cose'\subset \Cose$ such that $\partial \Cose' \cap 
			\Lase =\varnothing$.
			\item \label{D4} If $A\neq\varnothing$ then there exists a zero 
			dimensional $F_\sigma$-set $\Lase \subset A$ such that 
			$\dim(A\setminus \Lase)=\dim(A)-1$.
			\item \label{D5} Any countable union of $F_\sigma$-sets of dimension at most $k$ is an $F_\sigma$-set of dimension at most $k$.
		\end{enumerate} 
	}
\end{Lemma} 

The following is a strengthening of Lemma~\ref{lem:Engelking}\eqref{D3}. 

\begin{Lemma}[cf.~{\cite[Lemma~A.4]{Hirshberg-Wu16}}] \label{Startlemma} 
	Let $X$ be a locally compact metric space. Let $K\subset X$ be compact and 
	let $\Ause\subset X$ be an open neighborhood of $K$. Let $\Lase\subset X$ 
	be a zero dimensional subset. Then there exists a precompact open set 
	$\Cose$ with $K\subset \Cose\subset \overline{\Cose}\subset \Ause$ such 
	that $\partial \Cose\cap \Lase=\varnothing$.
\end{Lemma}

We need a slight modification of the notion of covering dimension which has a 
somewhat closer connection to nuclear dimension of $C^*$-algebras. The 
distinction is relevant only for non-compact, non-second-countable spaces. 

\begin{Def} 
	Let $X$ be a Hausdorff space. Then the \emph{compactly supported covering 
	dimension} of $X$, denoted by $\dimc (X)$, is the infimum of all 
	nonnegative integers $d$ such that for any finite open cover $\Thseco$ of 
	$X$ and any compact subset $\Ko \Subset X$, there is an open cover 
	$\Coseco$ of $\Ko$ in $X$ that refines $\Thseco$ and has multiplicity at 
	most $d+1$. 
\end{Def}

If $X$ is a locally compact Hausdorff space, we denote by $X^+$ the one-point 
compactification. We then have $\dimc(X) = \dim(X^+)$. If $X$ is furthermore 
second countable, then $\dimc(X) = \dim(X)$. We refer the reader to  
\ref{sec:dimc} for details.

\subsection{Asymptotic dimension for coarse spaces}	\label{sec:prelim:asdim}
\renewcommand{\sectionlabel}{OCS}
\ref{sectionlabel=OCS}
We review a few equivalent characterizations of asymptotic dimension. 
We work in general abstract coarse spaces (as opposed to 
just coarse metric spaces). 

\begin{Def}[{\cite[2.1]{HigsonPedersenRoe1997C}, see also \cite[Section~2.1]{Roe2003Lectures}}] \label{def:coarse-structure}
	A \emph{coarse structure} on a set $X$ is a collection $\Enseco$ of subsets of $X \times X$, called the \emph{controlled sets} or \emph{entourages} for the coarse structure, which satisfies: 
	\begin{enumerate}
		\item $\Enseco$ contains the diagonal $\Delta_X$; 
		\item for any $\Ense \in \Enseco$ and $\Auense \subseteq \Ense$, we have $\Auense \in \Enseco$;
		\item for any $\Ense \in \Enseco$, its \emph{inverse} $\Ense^{-1}$, defined as $\left\{ (x,y) \in X \times X \colon (y,x) \in \Ense\right\}$, is in $\Enseco$;
		\item for any $\Ense, \Auense \in \Enseco$, their \emph{product} or \emph{composition} $\Ense \circ \Auense$, defined as $\left\{ (x,y) \in X \times X \colon (x,z) \in \Ense\text{ and } (z,y) \in \Auense \text{ for some } z \in X \right\}$, is in $\Enseco$;
		\item for any $\Ense, \Auense \in \Enseco$, their union $\Ense \cup \Auense$ is in $\Enseco$. 
	\end{enumerate}
	When equipped with a coarse structure $\Enseco$, $X$ (or more precisely $(X, \Enseco)$) is called a \emph{coarse space}. 
	
	The \emph{coarse connected components} of $(X, \Enseco)$ are the equivalence classes associated to the equivalence relation $\sim$ on $X$ defined so that $x \sim y$ if and only $(x,y) \in \Ense$ for some $\Ense \in \Enseco$. 
\end{Def}

\begin{Exl} \label{exl:metric-coarse}
	Let $d \colon X \times X \to [0,\infty]$ be an extended metric\footnote{Here \emph{extended} simply means we allow $d$ to take the value $\infty$.} on $X$. For any $r \geq 0$, write $\Ense_r = \left\{ (x,y) \in X \times X \colon d(x,y) \leq r \right\}$.  Then the collection 
	\[
	\left\{ \Ense\subseteq X \times X \colon \Ense\subseteq \Ense_r \text{ for some } r \geq 0 \right\}
	\] 
	forms a coarse structure on $X$. In this case, the coarse connected components are exactly the equivalence classes associated to the equivalence relation of ``finite distance'' on $(X,d)$.  
\end{Exl}

\begin{Def}
	Let $(X, \Enseco)$ be a coarse space. A set $B \subseteq X$ is said to be 
	\emph{bounded} if $B \times B \in \Enseco$. Let $(Y, \Enseco')$ is another 
	coarse space. A map $f \colon X \to Y$ is said to be a \emph{coarse map} if 
	for any bounded set $B \subseteq Y$ the preimage $f^{-1}(B)$ is bounded and 
	furthermore for any controlled $\Ense \in \Enseco$ we have $(f \times f)(E) 
	\in \Enseco'$. Two maps $f,g \colon X \to Y$ are said to be \emph{close} if 
	$\{ ( f(x),g(x) ) \mid x \in X \} \in \Enseco'$. The spaces  $(X, \Enseco)$ 
	and $(Y, \Enseco')$ are said to be \emph{coarsely equivalent} if there 
	exist coarse maps $f \colon X \to Y$ and $g \colon Y \to X$ such that $f 
	\circ g$ is close to $\id_Y$ and $g \circ f$ is close to $\id_X$. 
\end{Def}

We recall the definition of asymptotic dimension of coarse spaces. For the 
remainder
of the subsection, the pair $(X, \Enseco)$ denotes a fixed coarse space. 

\begin{Def}
	Let $\Ense \in \Enseco$, and let $\Auseco$ be a collection of subsets of $X 
	\times X$. We say that $\Auseco$ is \emph{$\Ense$-separated} if for any to 
	different  $\Ause, \Ause' \in \Auseco$ we have $\Ense\cap (\Ause \times 
	\Ause') = \varnothing$.  
\end{Def}

\begin{Def}[{\cite[Section~9.1]{Roe2003Lectures}}] \label{def:asdim}
	The \emph{asymptotic dimension} of $(X,\Enseco)$, denoted 
	$\asdim(X,\Enseco)$ (or sometimes simply $\asdim(X)$), is the infimum of 
	all nonnegative integers $d$ such that for any controlled set $\Ense$, 
	there is a cover $\Auseco$ of $X$ such that 
	\begin{enumerate}
		\item $\Auseco$ is \emph{$\Enseco$-uniformly bounded} in the sense that $\bigcup_{\Ause \in \Auseco} \Ause \times \Ause \in \Enseco$, and 
		\item $\Auseco$ can be written as a union $\Auseco^{(0)} \cup \ldots 
		\cup \Auseco^{(d)}$ with each $\Auseco^{(l)}$ being 
		\emph{$\Ense$-separated}.
	\end{enumerate}
\end{Def}

The following lemma concerning passing to subspaces follows immediately from 
the 
definitions of a coarse space and its asymptotic dimension. We omit the proof.
\begin{Lemma} \label{lem:asdim-passes-to-subspaces}
	Let $(X,\Enseco)$ be a coarse space. Suppose $Y \subseteq X$ is a subset. 
	Define $\Enseco|_Y = \{E \cap  ( Y \times Y ) \mid E \in \Enseco\}$. Then
	\begin{enumerate}
		\item The pair $(Y,\Enseco|_Y)$ is a coarse space.
		\item $\asdim  (Y,\Enseco|_Y) \leq \asdim (X,\Enseco)$. 
	\end{enumerate} 
\end{Lemma}
We may write $E|_Y = E \cap ( Y \times Y )$ for short in the sequel.

\begin{Def} \label{def:E-connected}
	Let $Y \subseteq X$ and let $\Ense \in \Enseco$. Let $\sim_{\Ense, Y}$ be 
	the equivalence relation on $Y$ generated by the requirement that for any 
	$x, y \in Y$, we have $x \sim_{\Ense, Y} y$ if $(x,y) \in \Ense$. The 
	equivalence classes of $\sim_{\Ense, Y}$ are called the 
	\emph{$\Ense$-connected components} of $Y$. Any two points in the same 
	$\Ense$-connected component are \emph{$\Ense$-connected in $Y$}. 
	We say $Y$ is \emph{$\Ense$-connected} if it forms a single $\Ense$-connected component. 
\end{Def}

\begin{Rmk} \label{rmk:E-connected}
	In Definition~\ref{def:E-connected}, for any $x, y \in Y$, we have $x 
	\sim_{\Ense, Y} y$ if and only if there are $y_0, \ldots, y_n \in Y$ such 
	that $y_0 = x$, $y_n = y$, and $(y_{i-1}, y_i) \in \Ense \cup \Ense^{-1}$ 
	for $i = 1, \ldots, n$. 
	
	Notice that the $\Ense$-connected components of $Y$ are $\Ense$-separated. 
\end{Rmk}

We record the following characterization of the asymptotic dimension. We assume 
that it is well-known to experts, however we include a proof for completeness.

\begin{Lemma} \label{lem:asdim-connected-components}
	Let $d$ be a nonnegative number. Then $\asdim(X,\Enseco) \leq d$ if and 
	only if for any controlled set $\Ense$ there exists a cover $\left\{ 
	\Ause^{(0)}, \ldots, \Ause^{(d)} \right\}$ of $X$ such that the collection 
	of all $\Ense$-connected components of $\Ause^{(0)}, \ldots, \Ause^{(d)}$ 
	is $\Enseco$-uniformly bounded. 
\end{Lemma}

\begin{proof}
	To prove the ``if'' direction, we fix a controlled set $\Ense \in \Enseco$ and obtain a cover $\left\{ \Ause^{(0)}, \ldots, \Ause^{(d)} \right\}$ of $X$ as in the statement of the lemma. For $l = 0, \ldots, d$, we define $\Auseconum{l}$ to be the collection of all $\Ense$-connected components of $\Ause^{(l)}$, which is $\Ense$-separated by Remark~\ref{rmk:E-connected}, whence the cover $\Auseco = \Auseco^{(0)} \cup \ldots \cup \Auseco^{(d)}$ satisfies the conditions in Definition~\ref{def:asdim}. 
	
	For the other direction, we fix $\Ense \in \Enseco$ and obtain a cover 
	$\Auseco = \Auseco^{(0)} \cup \ldots \cup \Auseco^{(d)}$ as in 
	Definition~\ref{def:asdim}. For $l = 0, \ldots, d$, we define $\Ause^{(l)} 
	= \bigcup \Auseconum{l}$.  Because each 
	$\Ense$-connected component of $\Ause^{(l)}$ is contained in some $\Ause 
	\in \Auseconum{l}$, the sets $\Auseconum{l}$ for $l=0,1,\ldots d$  satisfy the statement of the lemma.
\end{proof}

As in the case of covering dimension, we may replace the requirement of 
decomposing $\Auseco$ into $(d+1)$ disjoint families by a formally weaker requirement on the 
multiplicity of $\Auseco$. When doing so, it is convenient to replace the 
separatedness condition by a ``wideness'' condition (though this is not 
absolutely necessary: one may also talk about $\Ense$-multiplicity instead of 
multiplicity, in a way similar to Definition~\ref{def:multiplicity-action}). 

\begin{Lemma}[{see \cite[Theorem~9]{Grave2006Asymptotic}}] \label{lem:asdim-multiplicity}
	Let $d$ be a nonnegative number. Then $\asdim(X,\Enseco) \leq d$ if and only if for any controlled set $\Ense$, there is a cover $\Auseco$ of $X$ satisfying 
	\begin{enumerate}
		\item \label{item:lem:asdim-multiplicity:Bo} $\bigcup_{\Ause \in \Auseco} \Ause \times \Ause \in \Enseco$, 
		\item \label{item:lem:asdim-multiplicity:Mu} the multiplicity of $\Auseco$ is at most $d+1$,	and
		\item \label{item:lem:asdim-multiplicity:Lo} for any $x \in X$, there is $\Ause \in \Auseco$ that contains $\{ y \in X \colon (y,x) \in \Ense\}$. 
	\end{enumerate}
\end{Lemma}

The change from the separatedness condition in Definition~\ref{def:asdim} to the ``wideness'' condition in Lemma~\ref{lem:asdim-multiplicity}\eqref{item:lem:asdim-multiplicity:Lo} is not absolutely necessary: one may instead strengthen the multiplicity condition in Lemma~\ref{lem:asdim-multiplicity}\eqref{item:lem:asdim-multiplicity:Mu} to obtain the following characterization:  

\begin{Def} \label{def:multiplicity-wide}
	Let $\Auseco$ be a collection of subsets of $X$ and let $\Ense \subseteq X \times X$. 
	The \emph{$\Ense$-multiplicity} of $\Auseco$, denoted by $\mult_{\Ense} 
	(\Auseco)$, is the supremum of the cardinalities of finite subcollections 
	$\Fi$ of $\Auseco$ such that 
	there is a subset $\Bo = \left\{ x_{\Ause} \colon \Ause \in \Fi \right\}$ of $X$ with $\Bo \times \Bo \subseteq \Ense \cup \Delta_X \cup \Ense^{-1}$ and $x_{\Ause} \in \Ause$ for any $\Ause \in \Fi$. 
\end{Def}
Note that if $\Ense = \Delta_X$, then the $\Ense$-multiplicity is simply the 
multiplicity.

We relate this notion to the ordinary multiplicity by use of the following constructions: 

\begin{Def} \label{def:neighborhood-core}
	Let $\Ause \subseteq X$ and let $\Ense \subseteq X \times X$. 
	Then the \emph{$\Ense$-neighborhood} of $\Ause$, denoted by $\Ense_\cup (\Ause)$, is the set 
	\[
		\left\{ x \in X \colon (x,y) \in \Ense \text{ for some } y \in \Ause \right\} \; . 
	\]
	When $\Ause = \{x\}$ for some $x \in X$, we also write $\Ense (x)$ for  $\Ense_\cup (\{x\})$.
	The \emph{$\Ense$-core} of $\Ause$, denoted by $\Ense_\cap (\Ause)$, is the set 
	\[
		\left\{ y \in X \colon \Ense (y) \subseteq \Ause \right\} 
	\]
\end{Def}

Observe that $\Ense_\cup (\Ause) = \bigcup_{y \in \Ause} \Ense (y)$ and
\[
	\Ense_\cup \left( \Ense_\cap (\Ause) \right) \subseteq \Ause \subseteq \Ense_\cap \left( \Ense_\cup (\Ause) \right) \; .
\]

\begin{Rmk} \label{rmk:multiplicity-wide-shrink-enlarge}
	Let $\Ense, \Ense' \subseteq X \times X$ and let $\Auseco$ be a collection 
	of subsets of $X$. If $ \Ense \subseteq \Ense'$ then we clearly have 
	\[
		\mult (\Auseco) = \mult_{\Delta_X} (\Auseco) \leq \mult_{\Ense} 
		(\Auseco) \leq \mult_{\Ense'} (\Auseco) \; .
	\]
	Moreover, observe that when $\Ense^{-1} = \Ense \supset \Delta_X$, 
	the shrunken collection $\Ense_\cap  \left( \Auseco \right) := \left\{ 
	\Ense_\cap  \left( \Ause \right) \colon \Ause \in \Auseco \right\}$ 
	satisfies 
	\[
	\mult_{\Ense} \left( \Ense_\cap  \left( \Auseco \right) \right) \leq \mult (\Auseco)
	\] 
	and the enlarged collection $\Ense_\cup  \left( \Auseco \right) := \left\{ 
	\Ense_\cup  \left( \Ause \right) \colon \Ause \in \Auseco \right\}$ 
	satisfies 
	\[
	\mult \left( \Ense_\cup  \left( \Auseco \right) \right) \leq \mult_{\Ense \circ \Ense} (\Auseco) \; .
	\]
\end{Rmk}

\begin{Lemma} \label{lem:asdim-multiplicity-wide}
	Let $d$ be a nonnegative number. Then $\asdim(X,\Enseco) \leq d$ if and only if for any controlled set $\Ense$, there is a cover $\Auseco$ of $X$ satisfying 
	\begin{enumerate}
		\item \label{item:lem:asdim-multiplicity-wide:Bo} $\bigcup_{\Ause \in \Auseco} \Ause \times \Ause \in \Enseco$,and
		\item \label{item:lem:asdim-multiplicity-wide:Muplus} the $\Ense$-multiplicity of $\Auseco$ is at most $d+1$. 
	\end{enumerate}
\end{Lemma}

\begin{proof}
	Given $\Ense \in \Enseco$ with $\Ense^{-1} = \Ense \supset \Delta_X$, we use Remark~\ref{rmk:multiplicity-wide-shrink-enlarge} to deduce the following:  
	\begin{enumerate}
		\item If a cover $\Auseco$ of $X$ satisfies 
		\eqref{item:lem:asdim-multiplicity:Bo} and 
		\eqref{item:lem:asdim-multiplicity:Mu} in 
		Lemma~\ref{lem:asdim-multiplicity}, then the shrunken collection 
		$\Ense_\cap  \left( \Auseco \right)$ is again a cover and satisfies 
		\eqref{item:lem:asdim-multiplicity-wide:Bo} and 
		\eqref{item:lem:asdim-multiplicity-wide:Muplus} in 
		Lemma~\ref{lem:asdim-multiplicity-wide}. 
		\item If a cover $\Auseco$ of $X$ satisfies 
		\eqref{item:lem:asdim-multiplicity-wide:Bo} and 
		\eqref{item:lem:asdim-multiplicity-wide:Muplus} in 
		Lemma~\ref{lem:asdim-multiplicity-wide} with $\Ense \circ \Ense$ in 
		place of $\Ense$, then the enlarged cover $\Ense_\cup  \left( \Auseco 
		\right)$ satisfies \eqref{item:lem:asdim-multiplicity:Bo} and 
		\eqref{item:lem:asdim-multiplicity:Mu} in 
		Lemma~\ref{lem:asdim-multiplicity}. 
	\end{enumerate}
	It follows that the condition in Lemma~\ref{lem:asdim-multiplicity-wide} is equivalent to that in Lemma~\ref{lem:asdim-multiplicity}, since in the statements of both lemmas, clearly we may assume $\Ense^{-1} = \Ense \supset \Delta_X$. This proves Lemma~\ref{lem:asdim-multiplicity-wide}. 
\end{proof}

We conclude this subsection with the following technical 
characterization of asymptotic dimension. In order to bound the asymptotic 
dimension, one does not have to construct 
a cover of the entire space at once; instead, it suffices to construct covers 
of all finite subspaces, one at a time, as long as those covers satisfy some 
uniform constraints. This fact may well be known to experts (see 
\cite[Lemma~2.7]{DelabieTointon2017}), but did find in the literature the 
exact statements we need. As we only need this fact in a relatively minor 
way, we state this fact here and provide a proof in 
\ref{sec:asdim-local}.

\begin{Lemma}[{see Corollary~\ref{cor:asdim-finitary}}] \label{lem:asdim-finitary}
	Let $d$ be a nonnegative number. Then $\asdim(X,\Enseco) \leq d$ if and only if for any controlled set $\Ense$, there is a controlled set $\Auense$ such that for any finite subset $Y \subseteq X$, there is a cover $\Auseco$ of $Y$ such that 
	\begin{enumerate}
		\item \label{lem:asdim-finitary:Bo} $\bigcup_{\Ause \in \Auseco} \Ause \times \Ause \subseteq \Auense$, and
		\item \label{lem:asdim-finitary:Muplus} the $\Ense$-multiplicity of $\Auseco$ is at most $d+1$. 
	\end{enumerate}
\end{Lemma}

\subsection{$C_0(X)$-algebras}
\label{subsection:prelim:C_0(X)}
We recall some facts and fix notation concerning $C_0(X)$-algebras (introduced 
by 
Kasparov, \cite{Kasparov}). A $C_0(X)$-algebra consists of a $C^*$-algebra $A$ 
and a locally compact
Hausdorff space  $X$, 
along with a homomorphism $\iota \colon C_0(X) 
\to Z(M(A))$ which is non-degenerate in the sense that for any $a \in A$ and 
for any $\eps>0$ there exists a positive contraction $f \in C_0(X)$ such that 
$\|\iota(f) a - a\| < \eps$. 
We usually suppress the notation for $\iota$, and write $f \cdot a$ in place of 
$\iota(f)a$, as we typically consider 
the case in which $C_0(X)$ is a subalgebra of the center of the multiplier 
algebra of $A$.

If $Y \subset X$ is closed then $C_0(X\smallsetminus Y)A$ is an 
ideal of $A$. (We recall that by the Cohen factorization, $C_0(X\smallsetminus 
Y)A$ is automatically closed, and therefore there is no need to take its 
closure; 
see for example \cite[Proposition 2.33]{raeburn-williams}.)
We denote by $A_Y$ or $A|_{Y}$ the quotient 
$A/C_0(X\smallsetminus Y)A$. If $x \in X$, we write $A_x$ in place of 
$A_{\{x\}}$. For $a \in A$, we denote by $a_Y$ and $a_x$ the image of $a$ in 
$A_Y$ or $A_x$ respectively. For $a \in A$, the map $x \mapsto 
\|a_x\|$ need not be continuous, however it is always upper semicontinuous. (To 
see that this map is upper semicontinuous, one can express it as the infimum of 
a set of continuous maps, as follows: for any $f \in C_0(X)$, the map $x 
\mapsto \|a - f\cdot a + f(x)a\|$ is continuous, and the map $x \mapsto 
\|a_x\|$ the infimum over all $f \in C_0(X)$ of the previous expression.)

A $C_0(X)$-automorphism $\alpha$ of a $C_0(X)$-algebra $A$ is an automorphism 
of $A$ as a $C^*$-algebra which commutes with the action of $C_0(X)$, that is, 
for any $f \in C_0(X)$ and for any $a \in A$, we have $\alpha (f \cdot a) = f 
\cdot \alpha(a)$. 
Suppose $A$ is a $C_0(X)$-algebra, $G$ is a locally compact Hausdorff group and 
$\alpha \colon G \to \aut(A)$ is a point-norm continuous action via 
$C_0(X)$-automorphisms. For any $x \in X$, the ideal $C_0(X 
\smallsetminus \{x\})\cdot A$ is invariant for the action, and therefore 
we have an induced action $\alpha^x 
\colon  G \to \aut(A_x)$. We therefore have an exact sequence
\[
0 \to C_0(X\smallsetminus\{x\})\cdot A \rtimes_{\alpha} G \to 
A \rtimes_{\alpha} G \to A_x \rtimes_{\alpha^x} G
\]
and 
consequently the crossed product $A \rtimes_{\alpha} G$ 
inherits a structure of a 
$C_0(X)$-algebra, with fibers $\left ( A \rtimes_{\alpha} G \right )_x \cong 
A_x \rtimes_{\alpha^x} G$. 

We note in passing that in case one wishes to consider reduced rather than full 
crossed products, then the question of whether we have an exact sequence
\[
0 \to C_0(X\smallsetminus\{x\})\cdot A \rtimes_{\alpha,r} G \to 
A \rtimes_{\alpha,r} G \to A_x \rtimes_{\alpha^x,r} G
\]
turns out to be delicate. We refer the reader to 
\cite{kirchberg-wassermann-exact-groups-cts-bundles} for more on this. When $G$ 
is an exact group, this sequence is always exact. In this paper, it so happens 
that 
the crossed products we will consider are ones where the full and reduced 
crossed products coincide, although we do not need this fact in any of the 
proofs; we shall therefore focus on full crossed products, with no need to 
further discuss those technicalities.

We also record the following fact, whose proof is similar to \cite[Lemma~1.4]{Hirshberg-Wu16} and is left to the reader. 
\begin{Lemma}\label{lem:quasicentral-approximate-unit}
	There is a quasicentral approximate unit for $C_0(X\smallsetminus\{x\}) + C_0(X\smallsetminus\{x\}) \cdot A \rtimes_\alpha G \subset M(A \rtimes_\alpha G)$ which is contained in $C_c(X\smallsetminus\{x\})_{+, \leq 1}$.
\end{Lemma}

\subsection{Nuclear dimension}	\label{sec:prelim:dimnuc}

We recall the definition of the nuclear dimension of a $C^*$-algebra and a number of its basic properties. 

\begin{Def}\label{def:dimnuc}
	A completely positive map $\varphi \colon A \to B$ is \emph{order zero} if $\varphi(a) \varphi(a') = 0$ for any $a, a' \in A_+$ with $a a' = 0$. When $A$ is unital, an equivalent characterization is that $\varphi(1) \varphi(a a') = \varphi(a) \varphi(a') = \varphi(a a') \varphi(1)$ for any $a, a' \in A$.  
	
	If a completely positive map $\varphi \colon A \to B$ can be written as a sum $\sum_{l = 0}^{d} \varphi^{(l)}$ of $(d+1)$ completely positive contractive order zero maps $\varphi^{(0)}, \ldots, \varphi^{(d)}$, then we say $\varphi$ is \emph{$(d+1)$-decomposable}. 
	
	The \emph{nuclear dimension} of a $C^*$-algebra $A$ is the infimum of all the natural numbers $d$ such that for any finite subset $F \subseteq A$ and any $\varepsilon > 0$ there are a finite dimensional $C^*$-algebra $B$, a completely positive contractive map $\kappa \colon A \to B$, and a completely positive $(d+1)$-decomposable map $\gamma \colon A \to B$, such that $\| \gamma \circ \kappa (a) - a \| < \varepsilon$ for any $a \in F$.  
\end{Def}

We list here some permanence properties for nuclear dimension, taken from \cite{winter-zacharias}. 
\begin{Prop}\label{prop:dimnuc-basic}
	\
	\begin{enumerate}
	\item For any $C^*$-algebras $A$ and $B$ we have $\dimnuc (A \oplus B) = \max\{ \dimnuc(A) , \dimnuc(B)\}$ (\cite[Prosition 2.3(i)]{winter-zacharias}).
	\item If $B$ is a hereditary subalgebra of $A$ then $\dimnuc(B) \leq \dimnuc(A)$ (\cite[Proposition 2.5]{winter-zacharias}). 
	\item For any $C^*$-algebra $A$ and any $n \in \N$ we have $\dimnuc(M_n(A)) = \dimnuc(A)$ (\cite[Corollary 2.8(i)]{winter-zacharias}).
	\item For any $C^*$-algebra $A$ we have $\dimnuc(A) = \dimnuc (A^+)$ (\cite[Remark 2.11]{winter-zacharias}). 
		\end{enumerate}
\end{Prop}

We record a few technical lemmas concerning nuclear dimension and order zero 
maps, which will be used in the paper.
The following is a standard fact about lifting decomposable maps.
\begin{Lemma}[{\cite[Lemma~1.1]{Hirshberg-Wu16}}]\label{lem:lifting-order-zero}
	Let $A$ be a $C^*$-algebra, let $J \lhd A$ be an ideal, let $F$ be a finite dimensional $C^*$-algebra, let $d$ be a non-negative integer and let $\varphi \colon F \to A/J$ be a $(d+1)$-decomposable map. Let $\pi \colon A \to A/J$ be the quotient map. Then there exists a $(d+1)$-decomposable map $\tilde{\varphi} \colon F \to A$ such that $\varphi = \pi \circ \tilde{\varphi}$. 
\end{Lemma}

The following lemma is an invariant version of \cite[Proposition 2.6]{winter-zacharias}. 

\begin{Lemma}[{\cite[Lemma~1.3]{Hirshberg-Wu16}}] \label{lem:separable-dimnuc}
	Let $G$ be a locally compact Hausdorff and second countable group, and let 
	$A$ be a $C^*$-algebra equipped with a point-norm continuous action of $G$. 
	Then any countable subset $S \subset A$ is contained in a $G$-invariant 
	separable $C^*$-subalgebra $B \subset A$ with $\dimnuc(B) \leq \dimnuc(A)$. 
	In particular, $A$ can be written as a direct limit of separable 
	$G$-invariant $C^*$-algebras with nuclear dimension no more than 
	$\dimnuc(A)$. 
\end{Lemma}

The following is standard fact concerning perturbations of order zero maps, 
which is a straightforward consequence of semiprojectivity of cones over finite 
dimensional $C^*$-algebras. See, for example, \cite[1.2.3]{winter-covering-II}.

\begin{Lemma}
	\label{lem:perturb-order-zero}
	Let $A$ be a finite dimensional $C^*$-algebra. For any $\eps>0$ there 
	exists $\delta>0$ such that whenever $B$ is a 
	$C^*$-algebra, 
	$\varphi \colon A \to B$ is a $*$-linear contraction 
	satisfying
	\[
		\|\varphi(x)\varphi(y) - \varphi(xy)\varphi(1)\| < \delta \quad \text{ for any } x,y \in A_{\leq 1} \; ,
	\]
	then there exists an order zero contraction $\tilde{\varphi} \colon A \to 
	B$ such that $\|\varphi - \tilde{\varphi}\| < \eps$.
\end{Lemma}

We clarify the connection between the nuclear dimension of a commutative $C^*$-algebra and the dimension of its spectrum (extending results from \cite{winter-zacharias} to the non-separable case in Appendix~\ref{sec:dimc}). In particular, we show in Theorem~ \ref{thm:dimnuc-dim-general} that for any locally compact Hausdorff space $X$ we have $\dimnuc (C_0(X)) = \dim(X^+)$ (with $X^+$ denoting the one-point compactification). 
The following lemma is essentially the same as \cite[Lemma 3.1]{carrion}; see 
also \cite[Lemma 3.3]{Hirshberg-Wu16}. The 
difference is that the version here is for nuclear dimension rather than 
decomposition rank, and we do not assume that the algebras in question are 
separable. The proof goes through almost verbatim, and thus we do not repeat it.
\begin{Lemma}
	\label{lem:dimnuc-C(X)-algebras}
	Let $X$ be a locally compact Hausdorff space, and let $A$ be a 
	$C_0(X)$-algebra, then 
	\[
	\dimnucone(A) \leq \dimone(X^+) \cdot \sup_{x \in X} \dimnucone(A_x)
	\]
\end{Lemma}
Consequently, we obtain the 
following immediate corollary from Lemma \ref{lem:dimnuc-C(X)-algebras}.
\begin{Lemma}
	\label{lem:dimnuc-C(X)-crossed-products}
	Let $A$ is a $C_0(X)$-algebra, let $G$ be a locally compact Hausdorff group 
	and let
	$\alpha \colon G \to \aut(A)$ be a point-norm continuous action via 
	$C_0(X)$-automorphisms. Then
		\[
	\dimnucone(A \rtimes_{\alpha} G ) \leq \dimone(X^+) \cdot \sup_{x \in X} 
	\dimnucone(A_x \rtimes_{\alpha^x} G)
	\]
\end{Lemma}

\subsection{Virtually nilpotent and polycyclic groups}	\label{sec:prelim:groups}

	An (extended) metric space is said to be \emph{proper} if any closed and 
	bounded set is compact. For discrete extended metric spaces, this simply 
	means that any bounded set is finite. For groups, we are particularly 
	interested in left or right invariant metrics. For finitely generated 
	groups, the word length function $\ell$ associated to a finite set of 
	generators gives rise to a right-invariant proper metric given by $d(g,h) = 
	\ell(gh^{-1})$ (or a left-invariant proper metric given by $d(g,h) = 
	\ell(g^{-1}h)$). 
	Another way to construct a proper right-invariant metric for a countable discrete group is by enumerating the elements of $G$ by $e = g_0, g_1, g_2, \ldots$, choosing an unbounded sequence of positive real numbers $r_1 , r_2 , \ldots$, and taking the largest metric such that $d(g, g_n g) \leq r_n$ for all $g \in G$ and $n \in \mathbb{Z}^+$.

\begin{Rmk}\label{rmk:proper-right-invariant-coarse-equiv}
	One can show that any two proper right-invariant metrics are coarsely equivalent via the identity map. 
\end{Rmk}

\begin{Def} \label{def:nilpotent}
	Let $G$ be a discrete group. The \emph{descending central series} of $G$ is a series 
	\[
	G_1 \trianglerighteq G_2 \trianglerighteq G_3 \trianglerighteq \ldots
	\]
	of subgroups of $G$ defined inductively so that $G_1 = G$ and 
	\[
	G_{n+1} = [G_{n}, G] = \left\langle \left\{ g_1^{-1} g_2^{-1} g_1 g_2 \colon g_1 \in G_{n}, g_2 \in G \right\}\right\rangle \quad \text{for } n = 1, 2, \ldots \; .
	\]
	We say $G$ is \emph{nilpotent} if $G_n = 1$ for some integer $n$. In this 
	case, we write $\Rank[n](G)$ for the rank of the abelian group $G_{n} / 
	G_{n+1}$ (that is, the dimension of the rational vector space  $\Q 
	\otimes_{\Z} 
	( G_{n} / G_{n+1} )$).
	
	We say $G$ is \emph{virtually nilpotent} if it has a finite-index subgroup 
	$H$ that is nilpotent, in which case we write $\Rank[n](G) = \Rank[n](H)$ 
	for $n = 1,2,\ldots$; this does not depend on the choice of $H$. 
\end{Def}

\begin{Def} \label{def:polycyclic}
	We say a discrete group $G$ is \emph{polycyclic} if there is a finite series 
	\[
	G = G_1 \trianglerighteq G_2 \trianglerighteq G_3 \trianglerighteq \ldots \trianglerighteq G_{m-1} \trianglerighteq G_{m} = 1
	\]
	of subgroups of $G$, each being normal in its predecessor, such that $G_{n} / G_{n+1}$ is a finitely generated abelian group for $n \in \{1,2,\ldots, m-1 \}$. In this case, we define $\hirsch(G)$, the \emph{Hirsch length} of $G$, to be the sum of the ranks of $G_{n} / G_{n+1}$ as $n$ ranges from $1$ to $m-1$. It is known that $\hirsch(G)$ does not depend on the choice of this series. 
	
	We say a discrete group $G$ is \emph{virtually polycyclic} if it has a 
	finite-index subgroup $H$ that is polycyclic, in which case we write 
	$\hirsch(G) = \hirsch(H)$; this does not depend on the choice of $H$. 
\end{Def}

It is clear that any finitely generated virtually nilpotent group $G$ is virtually polycyclic, in which case we also have
\[
\hirsch(G) = \sum_{n=1}^{\infty} \Rank[n](G) \; .
\]

\begin{Def} \label{def:growth}
	Let $G$ be a finitely generated group and let $S$ be a generating set. For 
	$n = 0, 1, 2, \ldots$, we denote by $B_S (1,n)$ the ball of radius $n$ 
	around the identity with respect ot the word-length metric defined by $S$. 
	The \emph{growth} of $G$ with regard to $S$ is the function $\Growth{G,S} \colon \N \to \N$ given by $\Growth{G,S}(n) = \left| B_S (e,n) \right|$ for $n \in \N$. 
\end{Def}

If $S'$ is another generating set of $G$, then there is a positive integer $C$ 
such that for any $n \in \N$, we have 
\[
\Growth{G,S'} (n) \leq C \Growth{G,S} (C n) + C \quad \text{and} \quad \Growth{G,S} (n) \leq C \Growth{G,S'} (C n) + C \; .
\]
This defines an equivalence relation among increasing functions from $\N$ to itself. Modulo this equivalence relation, the growth of $G$ is uniquely defined. 

By a theorem of Bass (\cite{Bass1972degree}), the growth of a virtually nilpotent group is equivalent to a polynomial 
(\cite{Bass1972degree}) of degree
\[
d(G) = \sum_{n=1}^{\infty} n \Rank[n](G) \, .
\]
The converse also holds by a theorem of Gromov 
(\cite{Gromov1981Groups}). 
Bass' result was improved in \cite[Theorem 1.1]{Breuillard}, which states that if $G$ is a finitely generated virtually nilpotent group and $S$ is a finite symmetric generating set then there exists a positive constant $C$ depending on $S$,  such that
\[
\frac{ \Growth{G,S} (n)}{n^{d(G)}} \xrightarrow{n \to \infty} C 
\, .
\]
We refer the reader to \cite{Breuillard-LeDonne} for more refined estimates on the rate of growth of balls in nilpotent groups, though those further refinements are not needed in this paper.

\section{The orbit coarse structure and its asymptotic dimension} \label{sec:OCS} 
\renewcommand{\sectionlabel}{OCS}
\ref{sectionlabel=OCS}
Let $\alpha$ be an action of a discrete group $G$ on a locally compact 
Hausdorff space $X$. In this section, we will equip $X$ with a coarse structure 
which we call the \emph{orbit coarse structure} associated to $\alpha$. We are 
interested in the asymptotic dimension of this coarse structure for two 
reasons: (1) it will play an auxiliary role in our main theorem, and (2) its 
study serves as a springboard for understanding the more complicated 
\emph{long thin covering dimension}, discussed in Section~\ref{sec:LTC}. 

We first define the construction in terms of abstract coarse structures. Later 
we will explain how to realize this in terms of metric spaces, under suitable 
countability conditions. 
\begin{Def} \label{def:OCS}
	For any compact $\Ko \Subset X$ and for any finite subset $\Lo \Subset G$, 
	denote
	\[
	\Ense_{\Ko,\Lo} = \left\{ \left( \alpha_g (x) , x \right) \in X \times X 
	\colon g \in \Lo , x \in \Ko  \text{ and } \alpha_g (x)  \in \Ko \right\} 
	. \]
	 The \emph{orbit coarse structure} $\Enseco_\alpha$ associated to $\alpha$ 
	 is given by 
	\[
		\left\{ \Ense\subseteq \Ense_{\Ko,\Lo} \cup \Delta_X \colon \Ko \Subset 
		X, \Lo \Subset G \right\} .
	\]
\end{Def}
Notice that for any $\Ko, \Ko' \Subset X$ and for any $\Lo, \Lo' \Subset G$ we 
have
\[
\Ense_{\Ko,\Lo} ^{-1} = \Ense_{\Ko, \Lo^{-1}} \, , \; \Ense_{\Ko,\Lo} \circ 
\Ense_{\Ko',\Lo'} \subseteq \Ense_{\Ko \cup \Ko', \Lo \cdot \Lo'}  \text{ and } 
\Ense_{\Ko,\Lo} \cup \Ense_{\Ko',\Lo'} \subseteq \Ense_{\Ko \cup \Ko', \Lo \cup 
\Lo'} .
\]
From this it follows that $\Enseco_\alpha$ is indeed a coarse structure.
	Note that the coarse connected components of $\Enseco_\alpha$ are 
	exactly the orbits of the action $\alpha$. 

The reason for restricting to compact sets in Definition~\ref{def:OCS} (as 
opposed to using $\Ense_{X, \Lo}$ to generate a coarse structure) is partly 
explained in Remark~\ref{rmk:OCS-why-compact}. 

\begin{Exl}
	To illustrate the effect of restricting to compact sets, consider the case 
	of a discrete group $G$ acting on itself by left translation. 
	A subset $\Ense \subseteq G \times G$ is a controlled set under the orbit 
	coarse structure if and only if $\Ense \setminus \Delta_G$ is finite.
	This gives the sparsest coarse structure on $G$ with a single coarse 
	connected component. 
	
	Unless $G$ is locally finite, this orbit coarse structure is very different 
	from the coarse structure of $G$ as a metric space induced by a proper 
	right-invariant metric (e.g., a word-length metric when $G$ is finitely 
	generated). 
	Indeed, it is easy to see that the asymptotic dimension of this orbit coarse structure is $0$ (one can also deduce it as a special case of Proposition \ref{prop:orbit-asdim-proper} further in the paper), regardless of what the group is. 
\end{Exl}

Under suitable countability conditions, the orbit coarse structure $\Enseco_\alpha$ can be induced from a metric in the sense of Example~\ref{exl:metric-coarse}. Observe that 
\begin{enumerate}
	\item when $X$ is $\sigma$-compact, there is a proper continuous map $X \to [0, \infty)$; 
	\item when $G$ is countable, there is a proper right-invariant metric on $G$. 
\end{enumerate}

\begin{Prop}
	Let $f \colon X \to [0, \infty)$ be a proper continuous map and let $\rho$ 
	be a proper right-invariant metric on $G$. Then the orbit coarse structure 
	$\Enseco_\alpha$ is induced, in the sense of 
	Example~\ref{exl:metric-coarse}, from the extended metric $d_{f, \rho}$ on 
	$X$ given, for $x , y \in X$, by 
	\[
		d_{f, \rho} (x , y) = \inf \left\{ \rho(g, e) \left( 1 + \max \left\{ f 
		(x) , f(y) \right\} \right) \colon g \in G \text{ such that } x = 
		\alpha_g (y) \right\}  .
	\]
	When $x$ and $y$ are not on the same orbit, we use the convention $\inf 
	\varnothing = \infty$. 
\end{Prop}

\begin{proof}
	It is routine to verify that $d_{f, \rho}$ is indeed an extended metric. 	
	To show that $\Enseco_\alpha$ coincides with $\Enseco_{d_{f, \rho}}$ in the sense of Example~\ref{exl:metric-coarse}, it suffices to show that 
	\begin{enumerate}
		\item for any $\Ko \Subset X$ and $\Lo \Subset G$, there exists $r_0 
		\geq 0$ such that $\Ense_{\Ko,\Lo} \subseteq \Ense_{r_0}$, and 
		\item for any $r \geq 0$, there exist subsets $\Ko_0 \Subset X$ and 
		$\Lo_0 
		\Subset G$ such that $\Ense_r \subseteq \Ense_{\Ko_0,\Lo_0} \cup 
		\Delta_X$. 
	\end{enumerate}
	We observe that these statements hold with the following choices 
	\begin{align*}
		r_0 &= \max_{g \in \Lo} \rho (g, e) \left( 1+ \max_{x \in \Ko} f(x) \right) \\
		\Lo_0 &= \left\{ g \in G  \colon \rho(g, e) \leq r \right\} \\
		\Ko_0 &= \left\{ x \in X \colon 1 + f(x) \leq \frac{r}{\inf_{g \in G \setminus \{ e \}} \rho (g ,e)} \right\} 
	\end{align*}
	noting that $\inf_{g \in G \setminus \{ e \}} \rho (g ,e) > 0$. 
\end{proof}

\begin{Rmk}
	When $X$ is compact, the constant zero function on $X$ is proper. Thus if $\rho$ is a proper right-invariant metric on $G$, then the orbit coarse structure $\Enseco_\alpha$ is induced, in the sense of Example~\ref{exl:metric-coarse}, from the extended metric $d_{0, \rho}$ on $X$ given by 
	\[
	d_{0, \rho} (x , y) = \inf \left\{ \rho(g, e) \colon g \in G \text{ such that } x = \alpha_g (y) \right\}
	\]
	for any $x , y \in X$, where we use the convention $\inf \varnothing = 
	\infty$. In this case, each coarse connected component of $X$ is an orbit 
	$\left\{ \alpha_g (x) \colon g \in G \right\}$ and is isometric to $G / 
	\stab(x)$ with the quotient metric induced from $\rho$. 
\end{Rmk}

\begin{Rmk} \label{rmk:uniformly-locally-finite}
	The coarse structure $\Enseco_{\alpha}$ is \emph{uniformly locally finite}, 
	in the sense  that for any $\Ense \in \Enseco_{\alpha}$, 
	there is a natural number $\Binu$ such that for any $x \in X$, the 
	cardinality of $\Ense(x)$ (Definition~\ref{def:neighborhood-core}) is 
	bounded by $\Binu$. 
	Indeed, it suffices to show this for $\Ense_{\Ko,\Lo}$ for some arbitrary 
	$\Ko \Subset X$ and $\Lo \Subset G$, but this case is clear by choosing 
	$\Binu = |\Lo|$. 
\end{Rmk}
 For the remainder of this section, we shall apply the various 
 characterizations of 
 asymptotic dimension discussed in Section~\ref{sec:prelim:asdim} to 
 $\Enseco_\alpha$ and explain how they are interpreted in dynamical terms (see 
 Proposition~\ref{prop:orbit-asdim}). Some of the  
 concepts we introduce and observations we make will play a role in 
 Section~\ref{sec:LTC}. 
We begin with a dynamical analogue of Definition~\ref{def:E-connected}. 

\begin{Def} \label{def:Lo-connected}
	Let $Y \subseteq X$ and let $\Lo \Subset G$. Let $\sim_{\Lo, Y}$ be the  
	equivalence relation on $Y$ given by completing the relation given by 
	declaring $x \sim_{\Lo, Y} y$ if $y \in \alpha_{\Lo} (x)$ for $x, y \in Y$, 
	we have . The equivalence classes of $\sim_{\Lo, Y}$ are called the 
	\emph{$\Lo$-connected components} of $Y$. Any two points in the same 
	$\Lo$-connected component are \emph{$\Lo$-connected in $Y$}. 
	We say $Y$ is \emph{$\Lo$-connected} if it forms a single $\Lo$-connected component. When it is important to emphasize the action $\alpha$, we write ``$\sim_{\alpha, \Lo, Y}$'' and ``$(\alpha, \Lo)$-connected'' instead. 
\end{Def}

\begin{Rmk} \label{rmk:Lo-connected}
	In Definition~\ref{def:Lo-connected}, it is clear that for any $x, y \in Y$, we have $x \sim_{\Lo, Y} y$ if and only if there are $y_0, \ldots, y_n \in Y$ such that $y_0 = x$, $y_n = y$, and $y_{i-1} \in \alpha_{\Lo \cup \Lo^{-1}} \left( y_{i} \right)$ for $i = 1, \ldots, n$. 
	
\end{Rmk}
We illustrate the equivalence relation in the picture below. For the purpose of 
the picture, consider $G = \Z$ and $L$ to be a singleton consisting of a 
generator. The points in the illustration are in the same orbit, however the 
ones indicated by filled-in dots are not equivalent to the other points in the 
orbit which are in $Y$, as getting from one to the next involves going through 
points which lie outside of $Y$.
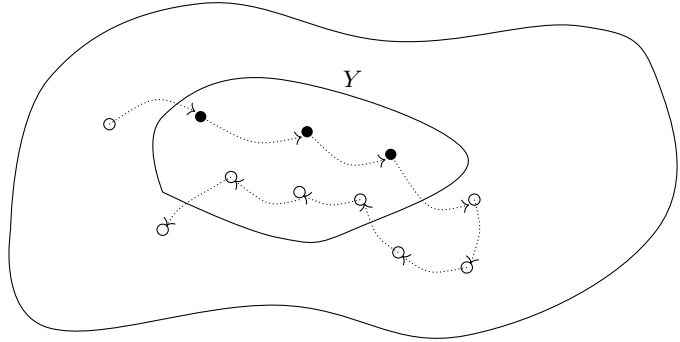
\begin{center}
	\begin{tikzpicture}
		\footnotesize
		\draw  plot[smooth, tension=0.7] coordinates {(-3.5,0.5) (-3,2.5) 
		(-1,3.5) 
			(1.5,3) (4,3.2) (5,2.5) (5,0.5) (2.5,-0.9) (0,-0.5) (-3,-0.8) 
			(-3.5,0.5)};
		\draw plot[smooth, tension=.7] coordinates {(-1.5,1) (-1.5,2) (0,2.5) 
		(2.5,1.5)  
			(1,0.5) (0,0.4) (-1.5,1)};
		\node at (-2.2,1.9)[circle,draw=black,inner sep=1.5pt]{};
		\node at (-1,2)[circle,fill,inner sep=1.5pt]{};
		\node at (0.4,1.8)[circle,fill,inner sep=1.5pt]{};
		\node at (1.5,1.5)[circle,fill,inner sep=1.5pt]{};
		\node at (2.6,0.9)[circle,draw=black,inner sep=1.5pt]{};
		\node at (2.5,0)[circle,draw=black,inner sep=1.5pt]{};
		\node at (1.6,0.2)[circle,draw=black,inner sep=1.5pt]{};
		\node at (1.1,0.9)[circle,draw=black,inner sep=1.5pt]{};
		\node at (0.3,1)[circle,draw=black,inner sep=1.5pt]{};
		\node at (-0.6,1.2)[circle,draw=black,inner sep=1.5pt]{};
		\node at (-1.5,0.5)[circle,draw=black,inner sep=1.5pt]{};
		\draw[smooth,densely dotted, ->] (-2.2,1.9) .. controls (-1.6,2.3) ..	
		(-1.07,2.07);
		\draw[smooth,densely dotted, ->] (-1,2) .. controls (-0.3,1.6) .. 
		(0.33,1.73);
		\draw[smooth,densely dotted, ->] (0.4,1.8) .. controls (0.9,1.3) .. 
		(1.43,1.43);
		\draw[smooth,densely dotted, ->] (1.5,1.5) .. controls (2,0.7) .. 
		(2.53,0.83);
		\draw[smooth,densely dotted, ->] (2.6,0.9) .. controls (2.7,0.5) .. 
		(2.55,0.05);
		\draw[smooth,densely dotted, ->] (2.5,0) .. controls (2,-0.1) .. 
		(1.65,0.15);
		\draw[smooth,densely dotted, ->] (1.6,0.2) .. controls (1.3,0.4) .. 
		(1.15,0.85);
		\draw[smooth,densely dotted, ->] (1.1,0.9) .. controls (0.7,0.8) .. 
		(0.35,0.95);
		\draw[smooth,densely dotted, ->] (0.3,0.9) .. controls (-0.1,0.8) .. 
		(-0.55,1.15);
		\draw[smooth,densely dotted, ->] (-0.6,1.2) .. controls (-1.3,0.8) .. 
		(-1.45,0.55);
		\node at (1,2.5){$Y$};
		\normalsize
	\end{tikzpicture}
\end{center}

The following is a quantitative refinement of Definition \ref{def:Lo-connected}.

\begin{Def} \label{def:Lo-bounded}
	Let $\Lo \Subset G$ and let $\Binu$ be a natural number. A subset $Y 
	\subseteq X$ is \emph{$\Binu$-boundedly $\Lo$-connected} or 
	\emph{$\left(\Lo, \Binu\right)$-bounded} if for any $x, y \in Y$, there 
	exist a natural number $n \leq \Binu$ and points $y_0, y_1, \ldots, y_n \in 
	Y$ 
	such that $y_0 = x$, $y_n = y$, and $y_{i-1} \in \alpha_{\Lo \cup \Lo^{-1}} 
	\left( y_{i} \right)$ for $i = 1, \ldots, n$. 
	
	When $\Binu = 1$, we also write \emph{$\Lo$-bounded} instead of 
	``{$\left(\Lo, 1 \right)$-bounded}''. 
	
	We also say $Y \subseteq X$ has \emph{$\Binu$-bounded $\Lo$-connected 
	components} if every $\Lo$-connected component of $Y$ is $\Binu$-boundedly 
	$\Lo$-connected. 
\end{Def}

\begin{Rmk} \label{rmk:Lo-bounded}
	Note that an $\left(\Lo, \Binu\right)$-bounded set is $\left(\Lo \cup 
	\Lo^{-1} \right)^{\Binu}$-bounded. 
	
	Moreover, if $Y \subseteq X$ is finite and $\Lo$-connected, then it is 
	$\left(\Lo, |Y|\right)$-bounded (and thus also $\left(\Lo \cup \Lo^{-1} 
	\right)^{|Y|}$-bounded).  
	Indeed, this follows from the explicit characterization of $x \sim_{\Lo, Y} y$ in Remark~\ref{rmk:Lo-connected} and the fact that the sequence $y_0, \ldots, y_n$ there can always be chosen to consist of distinct elements (and thus to have length at most $|Y|$), since whenever $y_i = y_j$ for $0 \leq i < j \leq n$, we can contract the sequence by discarding everything between $i$ and $j$. 
\end{Rmk}

The following is a dynamical analogue of 
Definition~\ref{def:multiplicity-wide}, which 
strengthens the notion of multiplicity in the presence of a group action.  
\begin{Def} \label{def:multiplicity-action} 
	Let $\Lo$ be a subset of $G$. We say a subset $\Aulase$ of $X$ is \emph{$(\alpha,\Lo)$-close} if for any $x, y \in \Aulase$, there is $g \in \Lo \cup \{e\} \cup \Lo^{-1}$ such that $x = \alpha_g (y)$. 
	
	Let $\Auseco$ be a collection of subsets of $X$. 
	Then its \emph{$(\alpha,\Lo)$-multiplicity}, denoted by $\mult_{\alpha,\Lo} (\Auseco)$, is defined as the supremum of the cardinalities of finite subcollections $\Fi$ of $\Auseco$ such that 
	there exists an $\left(\alpha, \Lo\right)$-close subset $\left\{ x_{\Ause} 
	\colon \Ause \in \Fi \right\}$ of $X$ with $x_{\Ause} \in \Ause$ for any 
	$\Ause \in \Fi$. 
\end{Def}

\begin{Rmk} \label{rmk:multiplicity-action-shrink-enlarge}
	Let $\Lo \subseteq \Lo'$ be subsets of $G$ and let $\Auseco$ be a 
	collection of 
	subsets of $X$. It is immediate from the definition that we have
	\[
	\mult (\Auseco) = \mult_{\alpha,\{e\}} (\Auseco) \leq \mult_{\alpha,\Lo} 
	(\Auseco) \leq \mult_{\alpha,\Lo'} (\Auseco) .
	\]
	Moreover, using Notation~\ref{notation:group-action}, we observe that if $\Lo^{-1} = \Lo \ni e$, then the shrunken collection $\alpha^{\cap}_{\Lo} \left( \Auseco \right) := \left\{ \alpha^{\cap}_{\Lo} \left( \Ause \right) \colon \Ause \in \Auseco \right\}$ satisfies 
	\[
	\mult_{\alpha,\Lo} \left( \alpha^{\cap}_{\Lo} \left( \Auseco \right) \right) \leq \mult (\Auseco)
	\] 
	and the enlarged collection $\alpha^{\cup}_{\Lo} \left( \Auseco \right) := \left\{ \alpha^{\cup}_{\Lo} \left( \Ause \right) \colon \Ause \in \Auseco \right\}$ satisfies 
	\[
	\mult \left( \alpha^{\cup}_{\Lo} \left( \Auseco \right) \right) \leq \mult_{\alpha,\Lo^{2}} (\Auseco) \; ,
	\]
	and more generally, 
	\[
	\mult_{\alpha,\Lo'} \left( \alpha^{\cup}_{\Lo} \left( \Auseco \right) \right) \leq \mult_{\alpha,\Lo\Lo'\Lo} (\Auseco) \; .
	\]
\end{Rmk}

In the remainder of this section, we construct partitions of unity 
which are almost invariant, that is, satisfy a suitable Lipschitz condition with 
respect to the group action. In order to have such partitions of unity, their 
supports need to be large enough in the direction of the group action, in the 
sense appearing in 
Lemma~\ref{lem:asdim-multiplicity}\eqref{item:lem:asdim-multiplicity:Lo}. We 
 now make this precise.

\begin{Def} \label{def:Lipschitz-alpha}
	Let $\Lo$ be a subset of $G$ and let $\Er \geq 0$. A function $\Po$ from $X$ into a metric space $(Y,\Metric)$ is \emph{$(\Lo, \Er)$-invariant} if
	\[
		\Metric \left( \Po (x) , \alpha_{g} (\Po) (x)  \right) \leq \Er
	\]
	for any $g \in \Lo$ and $x \in X$. (Recall that $\alpha_{g} (\Po) (x) = \Po 
	\left( \alpha_{g^{-1}} (x)  \right) $.) 
\end{Def}
Equivalently, consider the Schreier graph associated with the action of $G$, 
that is, we view $X$ as a non-directed graph where there's a vertex between two 
distinct points $x,y$ if there's an element $g \in L \cup L^{-1}$ such that $x 
= \alpha(y)$. This graph structure induces an extended metric on $X$ given by path length along 
the graph, where the distance between two points in different connected 
components is taken to be infinite. A function $f$ is $(L,\delta)$-invariant if 
and only if it is Lipschitz with constant $\delta$ when the domain is endowed 
with this extended metric.

This notion and some of its basic properties (discussed below) will appear in 
Propositions~\ref{prop:orbit-asdim} and~\ref{prop:ltc-dim}. 

\begin{Rmk} \label{rmk:Lipschitz-alpha-composition}
	Let $\Lo$ be a subset of $G$ and let $\Binu, \Er \geq 0$. If a function 
	$\Po$ from $X$ into a metric space $(Y,\Metric_Y)$ is {$(\Lo, 
	\Er)$-invariant} and a function $\Lapo$ from $Y$ into another metric space 
	$(Z, \Metric_Z)$ is Lipschitz with constant $\Binu$, then the composition 
	$\Lapo \circ \Po$ is $(\Lo, \Binu \Er)$-invariant. 
\end{Rmk}

The following simple lemma follows immediately from well known basic properties 
of Lipschitz functions. 
\begin{Lemma} \label{lem:Lipschitz-alpha-arithmetics}
	Let $\Lo$ be a subset of $G$ and let ${\Binu}_1, {\Binu}_2, {\Er}_1, {\Er}_2 \geq 0$. Let $\Po_1$ and $\Po_2$ be functions from $X$ into $\R$ (with the Euclidean metric) such that $\Po_i$ is $(\Lo, {\Er}_i)$-invariant for $i = 1,2$. Then
	\begin{enumerate}
		\item \label{item:lem:Lipschitz-alpha-arithmetics:addition} $\Po_1 + \Po_2$ and $\Po_1 - \Po_2$ are $(\Lo, {\Er}_1 + {\Er}_2)$-invariant; 
		\item \label{item:lem:Lipschitz-alpha-arithmetics:multiplication} if $\left| \Po_i (x) \right| \leq {\Binu}_i$ for $i = 1,2$ and any $x \in X$, then $\Po_1 \cdot \Po_2$ is $(\Lo, {\Er}_1 {\Binu}_2 + {\Binu}_1 {\Er}_2)$-invariant; 
		\item \label{item:lem:Lipschitz-alpha-arithmetics:division} if $\left| \Po_1 (x) \right| \leq {\Binu}_1$ and $\left| \Po_2 (x) \right| \geq {\Binu}_2 > 0$ for any $x \in X$, then $\frac{\Po_1}{\Po_2}$ is $\left(\Lo, \frac{{\Er}_1 {\Binu}_2 + {\Binu}_1 {\Er}_2} {({\Binu}_2)^2} \right)$-invariant; 
		\item \label{item:lem:Lipschitz-alpha-arithmetics:max} $\max \left\{ \Po_1, \Po_2 \right\}$ and $\min \left\{ \Po_1, \Po_2 \right\}$ are $(\Lo, \max \left\{ {\Er}_1 , {\Er}_2) \right\}$-invariant; 
	\end{enumerate}
\end{Lemma}

The following construction starts with a given bounded real function on $X$ and 
constructs a $(\Lo, \Er)$-invariant function, while enlarging its support in a 
controlled way. We will use it to construct partitions of unity which are 
furthermore $(\Lo, \Er)$-invariant; thus we only care about functions 
whose range is the unit interval $[0,1]$. 

\begin{Lemma} \label{lem:Lipschitz-alpha-staircase}
	Let $\Lo$ be a subset of $G$ such that $\Lo^{-1} = \Lo$ and $e \in \Lo$. 
	Let $\Binu$ be a positive integer. Let $\Po \colon X \to [0,1]$ be a 
	function. Define a function (following 
	Notations~\ref{notation:group-action} and~\ref{notation:group})
	\[
		\Staircase{\Lo}{\Binu} (\Po) \colon X \to \R 
	\]
	by
	\[	
		\Staircase{\Lo}{\Binu} (\Po) ( x ) = \max_{k \in \{ 0, \ldots, \Binu 
		\}} \max_{g \in \Lo^k} \left( \alpha_{g} (\Po)  (x)   - \frac{k}{\Binu} 
		\right) \; . 
	\]
	 Then the following statements hold: 
	\begin{enumerate}
		\item \label{item:lem:Lipschitz-alpha-staircase:bounds} For any $x \in X $, we have 
		\[
			0 \leq \Po (x) \leq \Staircase{\Lo}{\Binu} (\Po) (x) \leq 1 \; .
		\]
		\item \label{item:lem:Lipschitz-alpha-staircase:continuous} If $\Po$ is continuous, then so is $\Staircase{\Lo}{\Binu} (\Po)$.  
		\item \label{item:lem:Lipschitz-alpha-staircase:support} We have 
		\[
			\left( \Staircase{\Lo}{\Binu} (\Po)  \right)^{-1} \left((0,1]\right) \subseteq \alpha^\cup_{\Lo^{\Binu}} \left( \Po^{-1} \left((0,1]\right) \right) \; .
		\]
		\item \label{item:lem:Lipschitz-alpha-staircase:Lipschitz} The function $\Staircase{\Lo}{\Binu} (\Po)$ is $\left(\Lo, \frac{1}{\Binu} \right)$-invariant. 
	\end{enumerate}
\end{Lemma}

In the case in which the orbit coarse structure $\Enseco_{\alpha}$ can be 
metrized, this amounts to designing a decay function using the metric. 

\begin{proof}
	The first two statements are immediate from the definition, as $\Lo^0 = 
	\{e\}$. The third statement follows from the observation that for any $x 
	\in X$ with $\Staircase{\Lo}{\Binu} (\Po) (x) > 0$, we have $\alpha_{g} 
	(\Po)  (x) > 0$ for some $x \in \Lo^{\Binu}$. For the last statement, we 
	first observe that the first statement implies that we can replace the set 
	$\{0, \ldots, \Binu \}$ by $\{0, 1, \ldots \}$ in the definition of 
	$\Staircase{\Lo}{\Binu} (\Po)$ without changing its value, as the term $ 
	\alpha_{g} (\Po)  (x)   - \frac{k}{\Binu} $ is negative whenever $k > 
	\Binu$. Thus,	for any $g \in \Lo$ and for $k \in \{0, 1, \ldots \}$, the 
	containments $g \cdot \Lo^k \subseteq \Lo^{k+1}$ and $g^{-1} \cdot \Lo^k 
	\subseteq \Lo^{k+1}$ imply that for any $x \in X$, we have 
	\begin{align*}
	\max_{h \in \Lo^k} \left( \Po \left( \alpha_{h^{-1}}  \left( \alpha_{g^{-1}} (x)  \right) \right) - \frac{k}{\Binu} \right) = &\ \max_{h \in \Lo^k} \left( \Po \left( \alpha_{(g h)^{-1}} (x) \right) - \frac{k}{\Binu} \right)  \\
	\leq &\ \max_{t \in \Lo^{k+1}} \left( \Po \left( \alpha_{t^{-1}} (x) \right) - \frac{k+1}{\Binu} \right) + \frac{1}{\Binu} \\
	\leq &\ \Staircase{\Lo}{\Binu} (\Po)  (x) + \frac{1}{\Binu} .
	\end{align*}
	Thus, by taking maximum over $k$, we obtain
	\[
	\Staircase{\Lo}{\Binu} (\Po)  (  \alpha_{g^{-1}} (x) )  \leq \Staircase{\Lo}{\Binu} (\Po)  (x) + \frac{1}{\Binu} .
	\]
	Because $L = L^{-1}$, replacing $g$ by $g^{-1}$ and $x$ by 
	$\alpha_{g^{-1}}(x)$ gives us the inequality in the other direction. 
	Therefore we obtain  
	\[
	\left| \Staircase{\Lo}{\Binu} (\Po) (x) - \Staircase{\Lo}{\Binu} (\Po) \left( \alpha_{g^{-1}} (x)  \right) \right| \leq \frac{1}{\Binu}\] 
	as desired.  
\end{proof}

\begin{Rmk} \label{rmk:Lipschitz-alpha-staircase-depend}
	Note that the value $\Staircase{\Lo}{\Binu} (\Po) (x)$, defined in Lemma~\ref{lem:Lipschitz-alpha-staircase} above, depends only on the values of $f$ in the set $\alpha_{\Lo^{\Binu}} (x)$. 
\end{Rmk}

We now list a number of equivalent characterizations of the asymptotic 
dimension of the orbit coarse structure. Those are used later 
in the paper, and also provide some motivation for similar characterizations 
of 
the long thin covering dimension in Section~\ref{sec:LTC}. 
The labels (Lo), (Mu), (Bo), (Mu$^+$), (Ca), (Se), (Su), (In), and (Un) in the 
statement stand for, respectively, ``long'', ``multiplicity'', ``bounded'', ``multiplicity (strengthened)'', ``cardinality'', ``separated'', ``support'', ``invariant'', and ``partition of unity''. 

{
	
\begin{Prop} \label{prop:orbit-asdim}
	Let $\alpha \colon G \curvearrowright X$ be an action of a discrete group on a locally compact Hausdorff space. Let $d$ be a natural number. Then the following are equivalent: 
	{
	\newcounter{thmenumi}
	\renewcommand{\thethmenumi}{\theThm(\arabic{thmenumi})}
	\begin{enumerate}
		\item \label{item:prop:orbit-asdim:orig} \refstepcounter{thmenumi} \label{prop:orbit-asdim:orig} 
		We have $\asdim(X, \Enseco_\alpha) \leq d$. 
		
		\item \label{item:prop:orbit-asdim:mult} \refstepcounter{thmenumi}  \label{prop:orbit-asdim:mult} 
		For any finite subset $\Lo \Subset G$ and for any compact subset $\Ko \Subset X$, there exist 
		a finite subset $\Bo \Subset G$ 
		and a collection $\Coseco$ of subsets of $X$ satisfying the following conditions:
		{
			\renewcommand{\theenumi}{}	
			\begin{enumerate}
				\renewcommand{\labelenumii}{\textup{(\theenumii)}}
				\renewcommand{\theenumii}{Lo}\item \label{item:prop:orbit-asdim:mult:Lo}
				For any $x \in \Ko$, there exists $\Cose \in \Coseco$ such that $\alpha_{\Lo} (x) \subseteq \Cose$.	
				
				\renewcommand{\labelenumii}{\textup{(\theenumii)}} \renewcommand{\theenumii}{Mu}\item \label{item:prop:orbit-asdim:mult:Mu}
				The multiplicity of $\Coseco$ is at most $d+1$.

				\renewcommand{\labelenumii}{\textup{(\theenumii)}} \renewcommand{\theenumii}{Bo}\item \label{item:prop:orbit-asdim:mult:Bo}
				Any $\Cose \in \Coseco$ is $\Bo$-bounded (as in Definition~\ref{def:Lo-bounded}). 
			\end{enumerate}
		}
		
		\item \label{item:prop:orbit-asdim:finitary} \refstepcounter{thmenumi}  \label{prop:orbit-asdim:finitary} 
		For any finite subset $\Lo \Subset G$ and for any compact subset $\Ko \Subset X$, there exists 
		a natural number $\Binu$ 
		such that for any finite subset $Y$ of $\Ko$, there exists
		a cover $\Coseco$ of $Y$ satisfying the following conditions:
		{
			\renewcommand{\theenumi}{}	
			\begin{enumerate}
				
				\renewcommand{\labelenumii}{\textup{(\theenumii)}} \renewcommand{\theenumii}{Mu$^+$}\item \label{item:prop:orbit-asdim:finitary:Muplus}
				The $\left(\alpha, \Lo\right)$-multiplicity (as in Definition~\ref{def:multiplicity-action}) of $\Coseco$ is at most $d+1$.
				
				\renewcommand{\labelenumii}{\textup{(\theenumii)}} \renewcommand{\theenumii}{Ca}\item \label{item:prop:orbit-asdim:finitary:Ca}
				For any $\Cose \in \Coseco$, we have $|\Cose| \leq \Binu$. 
				
			\end{enumerate}
		}
		
		\item \label{item:prop:orbit-asdim:pou} \refstepcounter{thmenumi} \label{prop:orbit-asdim:pou}
		For any finite subset $\Lo \Subset G$, for any compact subset $\Ko 
		\Subset X$, and for any $\Er >0$, there exist
		\begin{itemize}
			\item a finite subset $\Bo \Subset G$,
			\item collections $\Coseconum{0} , \ldots , \Coseconum{d}$ of disjoint finite subsets of $X$, together with $\Coseco = \Coseconum{0} \cup \ldots \cup \Coseconum{d}$, and
			\item finitely supported functions $\Ponum{\Cose}{l} \colon X \to 
			[0,1]$ for $l = 0, 1, \ldots, d$ and $\Cose \in \Coseconum{l}$ with 
			$\Ponum{\Cose}{l}(x) = 0 $ for any $x \not \in \Cose$,
		\end{itemize}
		satisfying the following conditions:
		{
			\renewcommand{\theenumi}{}	
			\begin{enumerate}
				\renewcommand{\labelenumii}{\textup{(\theenumii)}}
				\renewcommand{\theenumii}{Lo}\item \label{item:prop:orbit-asdim:pou:Lo}
				For any $x \in \Ko$, there exists $\Cose \in \Coseco$ such that $\alpha_{\Lo} (x) \subseteq \Cose$.

				\renewcommand{\labelenumii}{\textup{(\theenumii)}} \renewcommand{\theenumii}{Bo}\item \label{item:prop:orbit-asdim:pou:Bo}
				Any $\Cose \in \Coseco$ is $\Bo$-bounded (see Definition~\ref{def:Lo-bounded}).

				\renewcommand{\labelenumii}{\textup{(\theenumii)}} \renewcommand{\theenumii}{Su}\item \label{item:prop:orbit-asdim:pou:Su}
				For any $l \in \intervalofintegers{0}{d}$, for any $\Cose 
				\in 
				\Coseconum{l}$, and for any $x \in X$, 
				we have either $\Ponum{\Cose}{l}(x) = 0 $ or $\alpha_{\Lo}(x) 
				\subseteq \Cose$.  
				
				\renewcommand{\labelenumii}{\textup{(\theenumii)}} \renewcommand{\theenumii}{In}\item \label{item:prop:orbit-asdim:pou:Li}
				For any $l \in \intervalofintegers{0}{d}$ and any $\Cose \in 
				\Coseconum{l}$, the function $\Ponum{\Cose}{l}$ is \emph{$(\Lo, 
				\Er)$-invariant} (that is, $\left| \Ponum{\Cose}{l} (x) - 
				\Ponum{\Cose}{l} \left( \alpha_{g^{-1}} (x)  \right) \right| 
				\leq \Er$ for any $g \in \Lo$ and $x \in X$). 
				
				\renewcommand{\labelenumii}{\textup{(\theenumii)}} \renewcommand{\theenumii}{Un}\item \label{item:prop:orbit-asdim:pou:Un}
				For any $x \in \Ko$, we have $\displaystyle \sum_{l = 0}^{d} \sum_{\Cose \in \Coseconum{l}} \Ponum{\Cose}{l} (x) = 1$ while for any other $x \in X$, we have $\displaystyle \sum_{l = 0}^{d} \sum_{\Cose \in \Coseconum{l}} \Ponum{\Cose}{l} (x) \leq 1$. 
			\end{enumerate}
		}
	\end{enumerate}
	}
\end{Prop}

We will also write in the sequel
\def \Topic {\OCSmult}
\Refcd{Lo}, \Refcd{Mu}, \Refcd{Bo}, 
\def \Topic {\OCSfin}
\Refcd{Muplus}, \Refcd{Ca}, 
\def \Topic {\OCSpou}
\Refcd{Lo}, \Refcd{Bo},  \Refcd{Su}, \Refcd{Li}, and \Refcd{Un} 
to specify the parameters used in these conditions and the fact that they come from Proposition~\ref{prop:orbit-asdim}.

Informally speaking, the characterizations above unpack the definition of 
$\asdim(X, \Enseco_\alpha)$. 
The differences between them are: 
the second and third characterizations are analogues of 
Lemmas~\ref{lem:asdim-multiplicity} and~\ref{lem:asdim-finitary}, respectively, 
 capturing the idea of dimension in terms of multiplicity.
The fourth 
characterization uses decompositions into disjoint families (as in 
Definition~\ref{def:asdim}); the characterization also includes partitions of 
unity subordinate to the covers as part of the definition, which are needed 
later. In the latter sections, the second and third 
characterizations will be needed to obtain bounds on $\asdim(X, 
\Enseco_\alpha)$, while the fourth one will be used to in Section 
\ref{sec:dimnuc}, in order to obtain bounds on nuclear dimension. 

\begin{proof}
	We will show \eqref{item:prop:orbit-asdim:pou} $\Rightarrow$ \eqref{item:prop:orbit-asdim:mult}  $\Rightarrow$ \eqref{item:prop:orbit-asdim:finitary}  $\Rightarrow$ \eqref{item:prop:orbit-asdim:orig}  $\Rightarrow$ \eqref{item:prop:orbit-asdim:pou}. 
		
	The implication \eqref{item:prop:orbit-asdim:pou} $\Rightarrow$ \eqref{item:prop:orbit-asdim:mult} is obvious with the same variables (and an arbitrary choice of $\Er$), after observing that 
	\Refcd[\OCSmult]{Mu} 
	follows from the disjointness of the collections $\Coseconum{0} , \ldots , \Coseconum{d}$ in \eqref{item:prop:orbit-asdim:pou}. 
	
	The implication \eqref{item:prop:orbit-asdim:mult} $\Rightarrow$ 
	\eqref{item:prop:orbit-asdim:finitary} is also straightforward. Indeed, 
	given $\Lo \Subset G$ (which, we assume without loss of generality, is 
	symmetric and contains $e$) and $\Ko \Subset X$, if we can find $\Bo 
	\Subset G$ and $\Coseco$ satisfying \Refcd[\OCSmult]{Lo}, 
	\Refcd[\OCSmult]{Mu}, and \Refcd[\OCSmult]{Bo}, then by 
	Remark~\ref{rmk:multiplicity-action-shrink-enlarge}, the shrunken 
	collection $\alpha^\cap_{\Lo} \left( \Coseco \right)$ satisfies 
	\Refcd[\OCSfin]{Muplus}. It follows from \Refcd[\OCSmult]{Lo} for $\Coseco$ 
	that $\alpha^\cap_{\Lo} \left( \Coseco \right)$ covers $\Ko$, and it 
	follows from \Refcd[\OCSmult]{Bo} for $\Coseco$ and 
	Remark~\ref{rmk:uniformly-locally-finite} that there exists $\Binu$ such 
	that $\Coseco$ satisfies \Refcd[\OCSfin]{Ca}, and therefore also 
	$\alpha^\cap_{\Lo} \left( \Coseco 
	\right)$ satisfies \Refcd[\OCSfin]{Ca}. Now given any finite subset $Y 
	\subseteq \Ko$, we can simply restrict $\alpha^\cap_{\Lo} \left( \Coseco 
	\right)$ 
	to $Y$ to obtain a cover of $Y$ satisfying \Refcd[\OCSfin]{Muplus} and 
	\Refcd[\OCSfin]{Ca}.

	To show \eqref{item:prop:orbit-asdim:finitary} implies 
	\eqref{item:prop:orbit-asdim:orig}, it suffices to verify the conditions in 
	Lemma~\ref{lem:asdim-finitary} with regard to an arbitrary controlled set 
	$\Ense \in \Enseco_\alpha$. By Definition~\ref{def:OCS}, there exist $\Ko 
	\Subset X$ and $\Lo 
	\Subset G$ such that $\Ense \subseteq \Ense_{\Ko, \Lo} \cup \Delta_X$. 
	Without loss of generality, we may assume that $\Ense = \Ense_{\Ko, \Lo} 
	\cup 
	\Delta_X$ and $e \in \Lo = \Lo^{-1}$. By 
	assumption, there exists a natural number $\Binu$ 
	such that for any finite subset $Y \subseteq \Ko$, there exists a cover 
	$\Coseco$ of $Y$ satisfying \Refcd[\OCSfin]{Muplus} and 
	\Refcd[\OCSfin]{Ca}. 
	Set $\Auense = \Ense_{\Ko, \Lo^{\Binu}} \cup \Delta_X$. Given a  
	finite subset $Y \subseteq X$, we fix a cover $\Coseco$ of $Y \cap \Ko$ 
	satisfying \Refcd[\OCSfin]{Muplus} and \Refcd[\OCSfin]{Ca}. 
	Let $\Auseco_1$ be the collection of $\Lo$-connected components of 
	$\Coseco$. The collection $\Auseco_1$ clearly satisfies 
	\Refcd[\OCSfin]{Ca}, and also satisfies \Refcd[\OCSfin]{Muplus}, because 
	any 
	$(\alpha, \Lo)$-close set cannot contain two different $\Lo$-connected 
	components of a single member of $\Coseco$. Define $\Auseco_2 = \left\{ \{ 
	x \} \colon x \in Y \setminus \Ko \right\}$ and $\Auseco = \Auseco_1 \cup 
	\Auseco_2$. It is clear that $\Auseco$ covers $Y$ and $\mult_{\Ense} 
	(\Auseco) = \mult_{\Ense} (\Auseco_1) = \mult_{\alpha, \Lo} (\Auseco_1) 
	\leq d+1$. This verifies condition~\eqref{lem:asdim-finitary:Muplus} of 
	Lemma~\ref{lem:asdim-finitary}. It is also clear that $\bigcup_{\Ause \in 
	\Auseco_2} \Ause \times \Ause = \Delta_{Y \setminus \Ko} \subseteq 
	\Auense$. On the other hand, for any $\Ause \in \Auseco_1$, since $\Ause$ 
	is $\Lo$-connected and $|\Ause| \leq \Binu$, by 
	Remark~\ref{rmk:Lo-bounded}, 
	it is $\Lo^{\Binu}$-bounded, 
	whence $\Ause \times \Ause \subseteq \Auense$. This shows $\bigcup_{\Ause \in \Auseco} \Ause \times \Ause \subseteq \Auense$, which verifies condition~\eqref{lem:asdim-finitary:Bo} of Lemma~\ref{lem:asdim-finitary}. 
	
	It remains to show \eqref{item:prop:orbit-asdim:orig} implies 
	\eqref{item:prop:orbit-asdim:pou}. We assume  $\asdim(X, \Enseco_\alpha) 
	\leq d$.  Given a finite subset $\Lo \Subset G$, a compact subset $\Ko 
	\Subset X$ and $\Er >0$, we need to find a finite subset $\Bo \subseteq G$, 
	collections of finite subsets $\Coseconum{0}, \ldots, 
	\Coseconum{d}$, and functions $\Ponum{\Cose}{l}$ for $l=0,1,\ldots,d$ that 
	satisfy the conditions in 
	\eqref{item:prop:orbit-asdim:pou}. To this end, by replacing $\Lo$ by $\Lo 
	\cup \{e\} \cup \Lo^{-1}$, we may assume without loss of generality that 
	$\Lo = \Lo^{-1}$ and $e \in \Lo$. Set 
	\[
		\Binu = \left\lceil \frac{d+2}{\Er} \right\rceil
	\]
	We apply Definitions~\ref{def:asdim} and~\ref{def:OCS} to obtain a cover $\Auseco$ of $X$ with a decomposition $\Auseco = \Auseco^{(0)} \cup \ldots \cup \Auseco^{(d)}$, a compact set $\Ko' \Subset X$ and a finite set $\Bo' \Subset G$ such that 
	\begin{enumerate}
		\item $\bigcup_{\Ause \in \Auseco} \Ause \times \Ause \subseteq \Ense_{\Ko', \Bo'} \cup \Delta_X$, and 
		\item each $\Auseco^{(l)}$ is {$\Ense_{\Ko, \Lo^{2 \Binu + 2}}$-separated} 
		(that is, $\Ense_{\Ko, \Lo^{2 \Binu + 2}} \cap (\Ause \times \Ause') = 
		\varnothing$ for any two different $\Ause, \Ause' \in \Auseco^{(l)}$). 
	\end{enumerate}
	For each $\Ause \in \Auseco$, we define the set
	\[
		\Cose_{\Ause} = \alpha^\cup_{\Lo^{\Binu + 1}} \left( \Ause \cap \Ko \right) \subseteq X \; .
	\]
	We then define $\Bo = \Lo^{\Binu + 1} \left( \Bo' \cup \{e\} \right) \Lo^{\Binu + 1}$, 
	\[
		\Coseconum{l} = \left\{ \Cose_{\Ause} \colon \Ause \in \Auseconum{l} \right\} \quad \text{for } l = 1, \ldots, d \, , \quad \text{and} \quad \Coseco = \Coseconum{0} \cup \ldots \cup \Coseconum{d} \; .
	\]
	Since $\Auseco$ covers $X$, it is clear that $\Coseco$ satisfies 
	condition~\eqref{item:prop:orbit-asdim:pou:Lo}. 
	To show that $\Bo$ and 
	$\Coseco$ satisfy condition~\eqref{item:prop:orbit-asdim:pou:Bo}, we 
	observe that for any $\Ause \in \Auseco$ and for any $x , y \in 
	\Cose_{\Ause}$, 
	there exist $x', y' \in \Ause \cap \Ko$ such that $x \in 
	\alpha_{\Lo^{\Binu + 1}} 
	(x')$ and $y \in \alpha_{\Lo^{\Binu + 1}} (y')$. Because $\Ause \times \Ause 
	\subseteq \Ense_{\Ko', \Bo'} \cup \Delta_X$, we have $x' \in \alpha_{\Bo' 
	\cup \{e\}} (y')$, whence $x \in \alpha_{\Bo} (y)$. 
	
	Next, we show that for any $l \in \intervalofintegers{0}{d}$, the 
	collection $\Coseconum{l}$ is disjoint. Indeed, for any two $\Ause, 
	\Ause' \in \Auseconum{l}$, since $\Ense_{\Ko, \Lo^{2 \Binu + 2}} \cap (\Ause 
	\times \Ause') = \varnothing$, it follows that for any $x \in \Ause \cap 
	\Ko$ and $y \in \Ause' \cap \Ko$, we have $x \notin \alpha_{\Lo^{2\Binu + 2}} 
	(y)$ and thus $\alpha_{\Lo^{\Binu + 1}} (x) \cap \alpha_{\Lo^{\Binu + 1}} (y) = 
	\varnothing$, whence $\Cose_{\Ause} \cap \Cose_{\Ause'} = \varnothing$, as 
	desired. 
	
	Now, for each $\Ause \in \Auseco$, we apply Lemma~\ref{lem:Lipschitz-alpha-staircase} to the characteristic function $\chi_{\Ause \cap \Ko}$ and obtain
	\[
		\Lapo_{\Ause} = \Staircase{\Lo}{\Binu} \left( \chi_{\Ause \cap \Ko} 
		\right) \colon X \to [0,1] \; .
	\]
	The function $\Lapo_{\Ause}$ is $\left(\Lo, \frac{1}{\Binu} 
	\right)$-invariant and satisfies $0 \leq \chi_{\Ause \cap \Ko} (x) \leq 
	\Lapo_{\Ause} (x) \leq 1$ for any $x \in X $ 
	and 
	\[
		\left( \Lapo_{\Ause}  \right)^{-1} \left((0,1]\right) \subseteq \alpha^\cup_{\Lo^{\Binu}} \left( \Ause \cap \Ko \right) \; .
	\]
	Hence we have
	\[
		\sum_{\Ause \in \Auseco} \Lapo_{\Ause} (x) \geq \sum_{\Ause \in \Auseco} \chi_{\Ause \cap \Ko} (x)  \geq 1 \qquad \text{for any } x \in \Ko \; .
	\]
	For any $l \in \intervalofintegers{0}{d}$ and $\Ause \in 
	\Auseconum{l}$, we define $\Po^{(l)}_{\Cose_{\Ause}}  \colon X \to [0, 1]$ 
	by
	\[	
		\Po^{(l)}_{\Cose_{\Ause}} ( x ) =  \frac{\Lapo_{\Ause} (x)}{ \max 
		\left\{ 1, \displaystyle \sum_{k = 0}^{d} \sum_{\Ause' \in 
		\Auseconum{k}} \Lapo_{\Ause'} (x) \right\} } \; . 
	\]
	It is clear that these functions satisfy 
	condition~\eqref{item:prop:orbit-asdim:pou:Un}. To show 
	condition~\eqref{item:prop:orbit-asdim:pou:Su}, we observe that for any 
	$x \in X$, if $\Po^{(l)}_{\Cose_{\Ause}} ( x ) \not= 0$, that is, $\Lapo_{\Ause}(x) \not= 0$, we must have $x \in \alpha^\cup_{\Lo^{\Binu}} \left( \Ause \cap \Ko \right)$ and thus 
	\[
		\alpha_{\Lo} (x) \subseteq  \alpha^\cup_{\Lo^{\Binu + 1}} \left( \Ause \cap \Ko \right) = \Cose_{\Ause} \; ,
	\]
	as desired. 
	To verify condition~\eqref{item:prop:orbit-asdim:pou:Li}, we apply 
	Lemma~\ref{lem:Lipschitz-alpha-arithmetics}\eqref{item:lem:Lipschitz-alpha-arithmetics:division}
	 to the definition of $\Po^{(l)}_{\Cose_{\Ause}}$. Note that for 
	 $k=0,1,\ldots,d$, because 
	  $\left\{ \alpha^\cup_{\Lo^{\Binu}} \left( \Ause' \cap \Ko \right) \colon 
	  \Ause' \in \Auseconum{k} \right\}$ is disjoint,
	  we have 
	 $\sum_{\Ause' \in 
	 	\Auseconum{k}} \Lapo_{\Ause'} (x)  \leq 1$. Thus  
 	$\sum_{k = 0}^{d} \sum_{\Ause' \in 
 		\Auseconum{k}} \Lapo_{\Ause'} (x) \leq d+1$ for all $x \in X$. 
	 The 
	numerator in the definition of $\Po^{(l)}_{\Cose_{\Ause}}$ is bounded from 
	above by $1$ and is $\left(\Lo, \frac{1}{\Binu} 
	\right)$-invariant, while the denominator is bounded from below by $1$ and 
	is $\left(\Lo, \frac{d+1}{\Binu} \right)$-invariant by 
	Lemma~\ref{lem:Lipschitz-alpha-arithmetics}\eqref{item:lem:Lipschitz-alpha-arithmetics:addition}
	 and~\eqref{item:lem:Lipschitz-alpha-arithmetics:max}. Therefore we have 
	 verified all the conditions in \eqref{item:prop:orbit-asdim:pou}. 
\end{proof}

}


\section{Barycentric subdivisions via a difference operator} \label{sec:simplicial} 
\renewcommand{\sectionlabel}{LTC}
\ref{sectionlabel=LTC}
In this section, we derive some auxiliary results that allow us to manipulate 
and modify open covers. These results are based on techniques inspired by 
barycentric subdivisions of simplicial complexes, though our results are 
formulated in terms of function spaces and proved using a difference operator. 
This approach is more suitable for our purposes, as we need to modify 
partitions of unity as well as open covers.

Throughout this section, we let $S$ be a set.

\begin{Def} \label{def:Rplus}
	We let $\Rfuncs{S}$ be the collection of all finitely supported real-valued functions on $S$, equipped with the $l^{\infty}$ norm. Let $\Rplusfuncs{S}$ be the subspace consisting finitely supported nonnegative-valued functions on $S$. 
	
	For any $\xi \in \Rfuncs{S}$, we write $\supp (\xi)$ for its support, that 
	is, the set $\{ s \in S \colon \xi(s) \not= 0 \}$. 
\end{Def}

We may view $\Rplusfuncs{S}$ as the cone over the topological 
realization of the simplex over the vertex set $S$, by identifying each point 
in the simplex as a convex combination of $S$.

\begin{Def}\label{def:relative-pou}
	Let $ X $ be a topological space and let $\Ko \subset X $ a subset. 
	Let $ \Coseco$ be a finite collection of open sets in $X$ such that $\Ko \subset \bigcup \Coseco$.  
	A \emph{partition of unity for $\Ko \subset X $ subordinate to $ \Coseco $} is a collection of continuous functions $ \left\{ f_{\Cose} \colon  X \to [0, 1] \right\}_{\Cose \in \Coseco} $ such that
	\begin{enumerate} 
		\item for each $\Cose \in \Coseco$, the support of $f_{\Cose}$ is contained in $ \Cose$;
		\item we have $\displaystyle \sum_{\Cose \in \Coseco} f_{\Cose} (x) \leq 1 $ for all $x \in X$ and equality holds when $x \in \Ko$.
	\end{enumerate} 
	Equivalently, a partition of unity for $\Ko \subset X $ subordinate to $ \Coseco $ is a continuous map $f \colon X \to \Rplusfuncs{\Coseco}$ such that 
	\begin{enumerate} 
		\item for each $\Cose \in \Coseco$, we have $\overline{ f^{-1} \left( \left\{ \xi \in \Rplusfuncs{S} \colon \xi(\Cose) > 0 \right\} \right) } \subset \Cose$;
		\item we have $\displaystyle \sum_{\Cose \in \Coseco} f (x) (\Cose) \leq 1 $ for all $x \in X$ and equality holds when $x \in \Ko$.
	\end{enumerate} 
\end{Def}

\begin{Lemma}\label{lem:relative-pou}
	Let $ X $ be a locally compact Hausdorff space and let $\Ko \subset X $  be 
	a compact subset. Let $ \Coseco = \{ \Cose_i \}_{i\in I} $ be a finite 
	collection of open sets in $X$ such that $\Ko \subset \bigcup \Coseco$. 
	Then there exists a partition of unity $ \left\{ f_{\Cose} \colon  X \to 
	[0, 1] \right\}_{\Cose \in \Coseco} $ for $\Ko \subset X $ subordinate to $ 
	\Coseco $ such that each $f_{\Cose}$ is compactly supported. 
\end{Lemma}

\begin{proof}
	We first observe that we may assume without loss of generality that 
	$\Coseco$ consists of \emph{precompact} open sets. Indeed, we may use the 
	fact that $X$ is locally compact and Hausdorff to obtain a precompact open 
	set $\Thse$ that contains $\Ko$ and replace the collection $\Coseco$ by $\{ 
	\Cose \cap \Thse \colon \Cose \in \Coseco \}$. The rest of the proof then 
	follows directly from the existence of partitions of unity subordinate to 
	open covers of compact spaces, by applying a one-point compactification to 
	$X$ (see, e.g., \cite[Lemma~8.15]{HSWW16}). 
\end{proof}

The following simple fact may be proved directly or as a corollary of Lemma~\ref{lem:relative-pou}. We leave the proof to the reader. 

\begin{Lemma}\label{lem:shrunken-cover}
	Let $ X $ be a locally compact Hausdorff space and let $\Ko \subset X $ be 
	a compact subset. Let $ \Coseco = \{ \Cose_i \}_{i\in I} $ be a finite 
	collection of open sets in $X$ such that $\Ko \subset \bigcup \Coseco$. 
	Then there exist precompact open sets $\Lase_{i}$ with 
	$\overline{\Lase_{i}} \subseteq \Cose_{i}$ for $i \in I$, such that $\Ko 
	\subset \bigcup_{i \in I} \Lase_{i}$.  
	\qed
\end{Lemma}

\begin{Def} \label{def:simplicial-sorting}
	For any nonnegative integer $l$, set $\Ma[S]{l} \colon \Rplusfuncs{S} 
	\to [0,\infty)$ to be
	\[
		\Ma[S]{l} ( \xi ) = \inf\left\{ t \in [0,\infty) \colon \left| 
		\xi^{-1}\left( (t, \infty) \right) \right| < l+1 \right\} \; . 
	\]
\end{Def}

As described in the next lemma, for any positive integer $l$, $\Ma[S]{l-1} (\xi)$ is the $l$-th greatest value of the tuple $\left( \xi(s) \right)_{s \in S}$ (counting repetitions) when $l \leq |S|$, and $0$ when $l > |S|$.  The proof is straightforward and we leave it to the reader.

\begin{Lemma} \label{lem:simplicial-sorting} 
	Let $\xi \in \Rplusfuncs{S}$. Denote  $m = | \supp (\xi) |$. Fix an 
	enumeration $\supp (\xi) = \left\{ s^{(1)}, s^{(2)}, \ldots, s^{(m)} 
	\right\}$ such that $\xi(s^{(1)}) \geq \xi(s^{(2)}) \geq \ldots \geq 
	\xi(s^{(m)}) $. Then 
	\[
		\quad \Ma[S]{l-1} (\xi) = 
		\begin{cases}
			\xi(s^{(l)})  & \mid  0 < l \leq m \\
			0 & \mid  l > m 
		\end{cases}
		\; .
	\]
\end{Lemma}

Lemma~\ref{lem:simplicial-sorting} allows us to introduce a difference operator 
$\Di[S]$ that will play a central role in this section. 

\begin{Def} \label{def:simplicial-diff-op}
	For any nonnegative integer $l$, define the function 
	\[
		\Dinum[S]{l} = \Ma[S]{l} - \Ma[S]{l+1}  \colon \Rplusfuncs{S} \to [0, \infty)  
	\]
	and define
	\[
		\Di[S] \colon \Rplusfuncs{S} \to \Rplusfuncs{\N} 
	\]	
	by	
	\[
	 \Di[S] ( \xi ) = \left( \Dinum[S]{l} (\xi) \right)_{l \in \N} \; .
	\]
	By definition, we have $\Dinum[S]{l} (\xi) = \Di[S] (\xi) (l)$ for any $\xi \in \Rplusfuncs{S}$. 
\end{Def}

\begin{Lemma} \label{lem:simplicial-diff-op}
	For any nonnegative integer $l$, the following statements hold: 
	\begin{enumerate}
		\item \label{lem:simplicial-diff-op::Lipschitz} The function $\Ma[S]{l}$ is $1$-Lipschitz, while the function $\Dinum{l}$ is $2$-Lipschitz. In particular, both functions are continuous. 
		\item \label{lem:simplicial-diff-op::support} For any $\xi \in 
		\Rplusfuncs{S}$ and for any $\Er \geq 0$, 
		we have 
		\[
			\left\{ s \in S \colon \Di(\xi)(s) > \Er \right\} \subseteq \intervalofintegers{0}{\left| \left\{ s \in S \colon \xi(s) > \Er \right\}  \right| - 1} 
		\; .
		\]
		\item \label{lem:simplicial-diff-op::mu-delta} We have 
		\[
			\Ma[S]{l} (\xi) = \sum_{k=l}^{\infty} \Dinum{k} (\xi)  = \sum_{k=l}^{|\supp(\xi)|-1} \Dinum{k} (\xi) \quad \text{for any } \xi \in \Rplusfuncs{S} \; .
		\]
		\item \label{lem:simplicial-diff-op::sum} We have 
		\[
			\sum_{s \in S} \xi(s)  = \sum_{l=1}^{\infty} \Ma[S]{l} (\xi) = \sum_{l=0}^{|\supp(\xi)|-1} \Ma[S]{l} (\xi)  \quad \text{for any } \xi \in \Rplusfuncs{S} \; .
		\]
		\item \label{lem:simplicial-diff-op::norm} For any $\xi \in \Rplusfuncs{S}$, we have 
		\[
			\left\| \Di(\xi) \right\| \leq \left\| \xi \right\| \leq \left| \supp \left( \Di(\xi) \right) \right| \cdot \left\| \Di(\xi) \right\| \; .
		\] 
	\end{enumerate}
\end{Lemma}	

\begin{proof}
	\begin{enumerate}
		\item [\eqref{lem:simplicial-diff-op::Lipschitz}] It suffices to show 
		the first statement. To this end, we fix arbitrary $\xi_1, \xi_2 \in 
		\Rplusfuncs{S}$ and write $d = \left\| \xi_1 - \xi_2 \right\|$. Note 
		that 
		\[
		\xi_1^{-1}\left( (t+d, \infty) \right) \subseteq \xi_2^{-1}\left( (t, 
		\infty) \right) \quad \text{for any } t \in [0, \infty) \; .
		\]
		This implies that $\Ma[S]{l}(\xi_1) \leq \Ma[S]{l}(\xi_2) + d$. 
		Similarly, we have $\Ma[S]{l}(\xi_2) \leq \Ma[S]{l}(\xi_1) + d$, which 
		shows $\left| \Ma[S]{l}(\xi_1) - \Ma[S]{l}(\xi_2) \right| \leq d$, as 
		desired.
		\item [\eqref{lem:simplicial-diff-op::support}] This follows from 
		Lemma~\ref{lem:simplicial-sorting} and the fact $\Dinum{l} (\xi) \leq 
		\Ma[S]{l} (\xi)$ for any $\xi \in \Rplusfuncs{S}$ and for any 
		nonnegative integer $l$. 
		\item [\eqref{lem:simplicial-diff-op::mu-delta}] This follows from the definition of $\Dinum{l}$. 
		\item [\eqref{lem:simplicial-diff-op::sum}] This follows from Lemma~\ref{lem:simplicial-sorting}. 
		\item [\eqref{lem:simplicial-diff-op::norm}] This follows from \eqref{lem:simplicial-diff-op::mu-delta} and Lemma~\ref{lem:simplicial-sorting}. 
	\end{enumerate}
\end{proof}

\begin{Def} \label{def:simplicial-barycentric}
	We associate the following sets to a given set $S$: 
	\begin{enumerate}[label=(\roman*)]
		\item For any nonnegative integer $l$, let $\powerset^{(l)}(S) = \left\{ F \subseteq S \colon |F| = l \right\}$. 
		\item For any nonnegative integer $l$ and for any $\varepsilon \geq 0$, 
		define the set
		\[
			\Rplusfuncs[{(l, \varepsilon)}]{S} = \left\{ \xi \in \Rplusfuncs{S} \colon \Dinum{l} (\xi) > \varepsilon \right\} \; .
		\]
		\item For any nonempty finite subset $F \Subset S$ and any $\varepsilon \geq 0$, define the set 
		\[
			\Rplusfuncs[F, \varepsilon]{S} = \left\{ \xi \in \Rplusfuncs{S}_+ \colon \inf \xi(F) >  \varepsilon + \sup \left( \xi(S \setminus F) \cup \{0\} \right) \right\} \; .
		\]
	\end{enumerate}
\end{Def}

\begin{Lemma} \label{lem:simplicial-barycentric}
	The following statements hold: 
	\begin{enumerate}
		\item \label{lem:simplicial-barycentric::nbhd} For any nonempty finite 
		subset $F \Subset S$ and for any $\varepsilon, r \geq 0$, the set 
		$\Rplusfuncs[F, \varepsilon]{S} $ contains the $r$-neighborhood of 
		$\Rplusfuncs[F, \varepsilon + 2r]{S} $, that is, for any $\xi \in 
		\Rplusfuncs[F, \varepsilon + 2r]{S}$ and for any $\eta \in 
		\Rplusfuncs{S}$, if 
		$\|\xi - \eta \| \leq r$, then $\eta \in \Rplusfuncs[F, 
		\varepsilon]{S}$. 
		\item \label{lem:simplicial-barycentric::open} For any nonempty finite 
		subset $F \Subset S$ and for any $\varepsilon \geq 0$, the set 
		$\Rplusfuncs[F, \varepsilon]{S}$ is open. 
		\item \label{lem:simplicial-barycentric::partition} For any nonnegative 
		integer $l$ and for any $\varepsilon \geq 0$, we have a partition
		\[
		\Rplusfuncs[{(l, \varepsilon)}]{S} = \bigsqcup_{F \in \powerset^{(l+1)}(S)} \Rplusfuncs[F, \varepsilon]{S} \; .
		\]
		\item \label{lem:simplicial-barycentric::cover} For any $\xi \in 
		\Rplusfuncs{S} \setminus \{0\}$ and for any $\varepsilon \in \left[ 0, 
		\frac{\|\xi\|}{|\supp(\xi)|} \right)$, there exists $l \in \left\{ 0, 
		\ldots, |\supp(\xi)|-1 \right\}$ such that $\xi \in \Rplusfuncs[{(l, 
		\varepsilon)}]{S}$. In particular, we have 
		\[
		\Rplusfuncs{S} \setminus \{0\} = \bigcup_{l=0}^\infty \Rplusfuncs[{(l, 0)}]{S} \; .
		\]
		\item \label{lem:simplicial-barycentric::clumping} For any collection $\left\{ F_i \right\}_{i \in I}$ of nonempty finite subsets of $S$, if 
		\[
		\bigcap_{i \in I} \Rplusfuncs[F_i, 0]{S} \not= \varnothing \; ,
		\]
		then there exists $s \in S$ such that 
		\[
		\bigcup_{i \in I} \Rplusfuncs[F_i, 0]{S} \subseteq \left\{ \xi \in \Rplusfuncs{S} \colon \xi(s) > 0 \right\} \; .
		\]
	\end{enumerate}
\end{Lemma}	

\begin{proof}
	\begin{enumerate}
		\item [\eqref{lem:simplicial-barycentric::nbhd}] This follows from the fact that for any nonempty finite subset $F \Subset S$ and any $\xi_1, \xi_2 \in \Rplusfuncs{S}$, both $| \inf \xi_1(F) - \inf \xi_2(F)  |$ and $| \sup \xi_1(S \setminus F) - \sup \xi_2(S \setminus F)  |$ are bounded by $\left\| \xi_1 - \xi_2 \right\|$.  
		\item [\eqref{lem:simplicial-barycentric::open}]  For a subset $Y$ of $\Rplusfuncs{S}$, the open set $N_r(Y)$ is defined as the union of the open $r$-balls centered at $\xi$, for $\xi \in Y$.  This item then follows from \eqref{lem:simplicial-barycentric::nbhd} and the observation that
		when
		\[
		\Rplusfuncs[F, \varepsilon]{S} = \bigcup_{r>0} \Rplusfuncs[F, \varepsilon + 2r]{S} = \bigcup_{r>0} N_r \left( \Rplusfuncs[F, \varepsilon + 2r]{S} \right) \; ,
		.
		\]
		\item [\eqref{lem:simplicial-barycentric::partition}] Given any $\xi 
		\in \Rplusfuncs[{(l, \varepsilon)}]{S}$, it follows from the 
		definitions and Lemma~\ref{lem:simplicial-sorting} that there exists a 
		unique $F \in \powerset^{(l+1)}(S)$ such that $\xi \in  \Rplusfuncs[F, 
		\varepsilon]{S}$, namely 
		\[
		F = \xi^{-1} \left( \left(\Ma[S]{l} (\xi), \infty  \right) \right) = \xi^{-1} \left( \left[\Ma[S]{l-1} (\xi), \infty  \right) \right) \; .
		\]
		\item [\eqref{lem:simplicial-barycentric::cover}] Given any $\xi \in 
		\Rplusfuncs{S} \setminus \{0\}$ and any $\varepsilon \in \left[ 0, 
		\frac{\|\xi\|}{|\supp(\xi)|} \right)$, by 
		Lemma~\ref{lem:simplicial-sorting}, we have $\Ma[S]{0} (\xi) = \|\xi\| 
		> 0$; thus, by 
		Lemma~\ref{lem:simplicial-diff-op}\eqref{lem:simplicial-diff-op::mu-delta}
		 and the pigeonhole principle, there exists $l \in \left\{ 0, \ldots, 
		|\supp(\xi)|-1 \right\}$ such that $\Dinum{l} (\xi) \geq 
		\frac{\|\xi\|}{|\supp(\xi)|} > \varepsilon$, that is, $\xi \in 
		\Rplusfuncs[{(l,\varepsilon)}]{S}$. 
		\item [\eqref{lem:simplicial-barycentric::clumping}] By our assumption, 
		there exists some $\xi_0 \in \bigcap_{i \in I} \Rplusfuncs[F_i,0]{S}$. 
		Fix $s \in S$ such that $\xi_0(s) = \| \xi_0 \|$. Then for any $i 
		\in I$, we have $s \in F_i$ by definition, which implies that for any 
		$\xi \in \Rplusfuncs[F_i,0]{S}$, we have $\xi (s) \geq \inf \xi(F_i,0) 
		> 0$. This shows $\bigcup_{i \in I} \Rplusfuncs[F_i,0]{S} \subseteq 
		\left\{ \xi \in \Rplusfuncs{S} \colon \xi(s) > 0 \right\}$. 
	\end{enumerate}
\end{proof}

We end the section with two purely topological lemmas that follow from the above simplicial technique. 
More applications will come in Section~\ref{sec:LTC}. 

\begin{Lemma}\label{lem:cover-clumping}
	Let $ X $ be a locally compact Hausdorff space and let $\Ko \subseteq X $ be a compact subset. Let $\Thseco$ be a finite collection of open sets in $X$ such that $\Ko \subseteq \bigcup \Thseco$. 
	Then there exists a finite collection $\Auseco$  of open sets in $X$ such 
	that $\Ko \subseteq \bigcup \Auseco$ and, for any subcollection $\Auseco' 
	\subset \Auseco$, if $\bigcap \Auseco' \not= \varnothing$, then there 
	exists $\Thse \in \Thseco$ such that $\bigcup \Auseco' \subseteq \Thse$. 
\end{Lemma}

\begin{proof}
	Given $X$, $\Ko$ and $\Thseco$ as stated, we apply Lemma~\ref{lem:relative-pou} to obtain a partition of unity $f \colon X \to \Rplusfuncs{\Thseco}$ for $\Ko \subseteq X $ subordinate to $ \Thseco $. 	Notice that $f(x) \not= 0$ for any $x \in \Ko$. 
	Applying Definition~\ref{def:simplicial-barycentric}, we define 
	\[
	\Auseco = \left\{ f^{-1} \left( \Rplusfuncs[{F,0}]{\Thseco} \right) \colon  F \subseteq \Thseco \text{ and } F \not= \varnothing \right\} \; . 
	\]
	By 
	Lemma~\ref{lem:simplicial-barycentric}\eqref{lem:simplicial-barycentric::open},
	 \eqref{lem:simplicial-barycentric::cover}, 
	and~\eqref{lem:simplicial-barycentric::partition}, we see that $\Auseco$ is 
	a finite open collection such that $\Ko \subseteq \bigcup \Auseco$. Now, 
	given any subcollection $\Auseco' \subseteq \Auseco$ satisfying $\bigcap 
	\Auseco' \not= \varnothing$, we have a finite
	finite collection $\{F_i \colon i \in I\}$ of nonempty subsets of 
	$\Thseco$ such that $\Auseco' = \left\{ f^{-1} \left( 
	\Rplusfuncs[{F_i,0}]{\Thseco} \right) \colon  i \in I \right\}$. Thus, 
	$\bigcap_{i \in I} \Rplusfuncs[F_i,0]{\Thseco} \not= 
	\varnothing$. It follows from 
	Lemma~\ref{lem:simplicial-barycentric}\eqref{lem:simplicial-barycentric::clumping}
	 that there exists $\Thse \in \Thseco$ such that $\bigcup_{i \in I} 
	\Rplusfuncs[F_i,0]{\Thseco} \subseteq \left\{ \xi \in \Rplusfuncs{\Thseco} 
	\colon \xi(\Thse) > 0 \right\}$. Therefore we have
	\[
	\bigcup \Auseco' \subseteq f^{-1} \left( \bigcup_{i \in I} \Rplusfuncs[F_i,0]{\Thseco} \right) \subseteq f^{-1} \left( \left\{ \xi \in \Rplusfuncs{\Thseco} \colon \xi(\Thse) > 0 \right\} \right) \subseteq \Thse 
	\]
	as desired. 
\end{proof}

Note that in the above lemma, when $X$ is a metrizable, one can write an 
alternative proof that makes use of the Lebesgue number of a finite open cover. 
Furthermore, because the $C^*$-subalgebra of $C_0(X)$ generated by a partition 
of unity subordinate to $\Thseco$ is separable, and thus has a metrizable 
spectrum, the general case follows from the metrizable case. Nevertheless, we 
have opted to give the above proof as an application of the techniques 
developed before. 

\begin{Lemma} \label{lem:thicken-multiplicity}
	Let $X$ be a locally compact Hausdorff space. Let $\Ko_1, \ldots, \Ko_n$ be 
	compact subsets in $X$.
	Then there exist open neighborhoods $\Cose_1 , \ldots, \Cose_n$ of $\Ko_1, 
	\ldots, \Ko_n$, respectively, such that 
	\[
		\mult \left( \left\{  \Cose_i \colon i = 1, \ldots, n \right\} \right) = \mult \left( \left\{  \Ko_i \colon i = 1, \ldots, n \right\} \right) \; .
	\]
\end{Lemma}

\begin{proof}
	It is clear that any open neighborhoods $\Cose_1 , \ldots, \Cose_n$ of 
	$\Ko_1, \ldots, \Ko_n$ satisfy $\mult \left( \left\{  \Cose_i \colon i = 1, 
	\ldots, n \right\} \right) \geq \mult \left( \left\{  \Ko_i \colon i = 1, 
	\ldots, n \right\} \right)$. It remains to show that we can arrange the 
	opposite direction. Set
	$\Dimnu = \mult \left( \left\{  \Ko_i \colon i = 1, \ldots, n \right\} \right)$. 
	
	For any $\Fi \subseteq \intervalofintegers{1}{n}$, we define the open set
	\[
	\Ause_{\Fi} = X \smallsetminus \left( \bigcup_{ j \in 
	\intervalofintegers{1}{n} \smallsetminus \Fi } \Ko_j \right) \; .
	\]
	We claim that 
	\[
	X = \bigcup_{\Fi \subseteq \intervalofintegers{1}{n}, |\Fi| \leq \Dimnu} \Ause_{\Fi} \; .
	\]
	Indeed, for any $x \in X$, we can take $\Fi = \left\{ j \in 
	\intervalofintegers{1}{n} \colon x \in \Ko_j \right\}$ and see that $|\Fi| 
	\leq \Dimnu$ and $x \in \Ause_{\Fi}$. 
	
	We then apply Lemma~\ref{lem:cover-clumping} with the compact set $\bigcup_{i = 1}^{n} \Ko_i$ and the open cover $\left\{ \Ause_{\Fi} \colon \Fi \subseteq \intervalofintegers{1}{n}, |\Fi| \leq \Dimnu \right\}$ in place of $\Ko$ and $\Thseco$ to obtain a finite collection $\Thseco$ of open sets in $X$ such that $\bigcup_{i = 1}^{n} \Ko_i \subseteq \bigcup \Thseco$ and, for any subcollection $\Thseco' \subset \Thseco$, if $\bigcap \Thseco' \not= \varnothing$, then there is $\Fi \subseteq \intervalofintegers{1}{n}$ such that $|\Fi| \leq \Dimnu$ and $\bigcup \Thseco' \subseteq \Ause_{\Fi}$. 
	
	Now, for any $i \in \intervalofintegers{1}{n}$, we define the open set
	\[
		\Cose_i = \bigcup \left\{  \Thse \in \Thseco \colon 
		 \Thse   \cap \Ko_i
		\not= \varnothing \right\}   \; .
	\]
	It remains to show $\operatorname{mult} \left( \left\{ \Cose_i \colon i = 
	1, \ldots, n \right\} \right) \leq \Dimnu$. To this end, it suffices to 
	show that for any $x \in X$, the set 
	\[
		\Fi_x = \left\{ i \in \intervalofintegers{1}{n} \colon x \in \Cose_i 
		\right\} \; 
	\] 
	has cardinality no more than $\Dimnu$. To this end, for each $i \in \Fi_x$ 
	we choose a member $\Thse_{x,i} \in \Thseco$ such that $x \in \Thse_{x,i}$ 
	and $\Thse_{x,i} \cap \Ko_i	\not= \varnothing$. By our choice of 
	$\Thseco$, since $x \in \bigcap_{i \in \Fi_x}  \Thse_{x,i}$, there exists 
	$\Fi 
	\subseteq \intervalofintegers{1}{n}$ such that $|\Fi| \leq \Dimnu$ and 
	$\bigcup_{i \in \Fi_x}  \Thse_{x,i}  \subseteq \Ause_{\Fi}$. It follows 
	from the last condition that $\Ause_{\Fi} \cap \Ko_i \not= \varnothing$ 
	for any $i \in \Fi_x$, which implies, by our definition of $\Ause_{\Fi}$, 
	that $\Fi_x \subseteq \Fi$. Hence we conclude $|\Fi_x| \leq \Dimnu$ as 
	desired. 
\end{proof}


\section{The long thin covering dimension} \label{sec:LTC} 
\renewcommand{\sectionlabel}{LTC}
\ref{sectionlabel=LTC}
In this section, we introduce one of our key concepts, the \emph{long thin 
covering dimension} for a topological dynamical system. Throughout the section, 
we let $G$ be a discrete group, and let $\alpha$ be an action of $G$ on a 
locally compact Hausdorff space $X$; we sometimes denote an action of $G$ on 
$X$ by homeomorphisms as $\alpha \colon G \curvearrowright X$.

\def \Topic {\LTCdef}

\begin{Def}
	\label{\LTCdef}
	Let $G$ be a discrete group, and let $X$ be a locally compact Hausdorff 
	space.
	The \emph{long thin covering dimension} of an action $\alpha \colon G \curvearrowright X$, denoted by $\dimltc(\alpha)$, is the infimum of all natural numbers $d$ satisfying: 
	
	For any finite subset $\Lo \Subset G$ and for any compact subset $\Ko 
	\Subset X$, there exists a natural number $\Binu$ 
	such that for any finite open cover $\Thseco$ of $X$, 
	there exist 
	\begin{itemize}
		\item a collection $\Coseco$ of open sets in $X$, and
		\item locally constant functions $\Ne_{\Cose} \colon \Cose \to X$ for $\Cose \in \Coseco$, called \emph{near orbit selection functions}, 
	\end{itemize}
	satisfying the following conditions:
	{
		\begin{enumerate}
			\renewcommand{\labelenumi}{\textup{(\theenumi)}} \renewcommand{\theenumi}{Lo}\item \label{item:\LTCdef:Lo}
			For any $x \in \Ko$, there exists $\Cose \in \Coseco$ such that $\alpha_{\Lo} (x) \subseteq \Cose$.
			\renewcommand{\labelenumi}{\textup{(\theenumi)}} \renewcommand{\theenumi}{Mu}\item \label{item:\LTCdef:Mu}
			We have $\mult(\Coseco) \leq d+1$. 
			\renewcommand{\labelenumi}{\textup{(\theenumi)}} \renewcommand{\theenumi}{Eq}\item \label{item:\LTCdef:Eq}
			Each $\Ne_{\Cose}$ is \emph{$\Lo$-equivariant} in the following sense: for any $x \in \Cose$ and for any $g \in \Lo$, if $\alpha_{g} (x) \in \Cose$ then $\Ne_{\Cose}\left( \alpha_{g} (x) \right) = \alpha_g \left( \Ne_{\Cose}(x) \right)$.  
			\renewcommand{\labelenumi}{\textup{(\theenumi)}} \renewcommand{\theenumi}{Th}\item \label{item:\LTCdef:Th}
			For any $\Cose \in \Coseco$ and for any $y \in \Ne_{\Cose}(\Cose)$, there exists $\Thse \in \Thseco$ such that $\Ne_{\Cose}^{-1}(y) \cup \{y\} \subseteq \Thse$.
			\renewcommand{\labelenumi}{\textup{(\theenumi)}} \renewcommand{\theenumi}{Ca}\item \label{item:\LTCdef:Ca}
			For any $\Cose \in \Coseco$, we have $\left| \Ne_{\Cose} (\Cose) \right| \leq \Binu$. 
		\end{enumerate}
	}
\end{Def}

Here the labels stand for ``long'', ``multiplicity'', 
``equivariant'', ``thin'' and ``cardinality''. 
We will also write \Refcd{Lo}, \Refcd{Mu}, \Refcd{Eq}, \Refcd{Th}, and \Refcd{Ca} 
to specify the parameters used in these conditions and the fact that they come from Definition~\ref{\Topic}. 

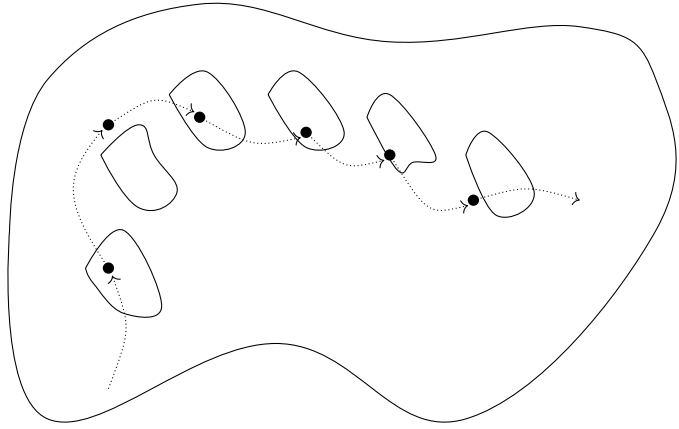
\begin{center}
	\begin{tikzpicture}
		\footnotesize
		\draw  plot[smooth, tension=0.7] coordinates {(-3.5,0.5) (-3,2.5) (-1,3.5) 
			(1.5,3) (4,3.2) (5,2.5) (5,0.5) (2.5,-2) (0,-1) (-3,-2) (-3.5,0.5)};
		\draw plot[smooth, tension=.7] coordinates {(-2.5,0) (-2,0.5) (-1.5,-0.5)  
			(-2,-0.6) (-2.4,-0.2) (-2.5,0)};
		\draw plot[smooth, tension=.8] coordinates {(-2.3,1.5) (-1.8,1.9)  
			(-1.6,1.5)  (-1.3,1)  
			(-1.8,0.8) (-2.3,1.5)};
		\draw plot[smooth, tension=.7] coordinates {(-1.4,2.3) (-0.9,2.6) 
			(-0.4,1.8)  (-0.9,1.6) (-1.4,2.3)};
		\draw plot[smooth, tension=.7] coordinates {(-0.1,2.3) (0.3,2.6) (0.9,1.8)  
			(0.4,1.6) (-0.1,2.3)};
		\draw plot[smooth, tension=.7] coordinates {(1.2,2) (1.5,2.3) (2.1,1.5)  
			(1.8,1.4) (1.6,1.3) (1.2,2)};
		\draw plot[smooth, tension=.7] coordinates {(2.5,1.5) (2.8,1.8) (3.4,1)  
			(2.9,0.7) (2.5,1.5)};
		\node at (-2.2,0)[circle,fill,inner sep=1.5pt]{};
		\node at (-2.2,1.9)[circle,fill,inner sep=1.5pt]{};
		\node at (-1,2)[circle,fill,inner sep=1.5pt]{};
		\node at (0.4,1.8)[circle,fill,inner sep=1.5pt]{};
		\node at (1.5,1.5)[circle,fill,inner sep=1.5pt]{};
		\node at (2.6,0.9)[circle,fill,inner sep=1.5pt]{};
		\draw[smooth,densely dotted, ->] (-2.2,-1.6) .. controls (-1.9,-0.8) .. 	
		(-2.15,-0.1);
		\draw[smooth,densely dotted, ->] (-2.2,0) .. controls (-2.8,0.9) and  
		(-2.8,1.2) .. 	(-2.28,1.83);
		\draw[smooth,densely dotted, ->] (-2.2,1.9) .. controls (-1.6,2.3) ..	
		(-1.07,2.07);
		\draw[smooth,densely dotted, ->] (-1,2) .. controls (-0.3,1.6) .. 
		(0.33,1.73);
		\draw[smooth,densely dotted, ->] (0.4,1.8) .. controls (0.9,1.3) .. 
		(1.43,1.43);
		\draw[smooth,densely dotted, ->] (1.5,1.5) .. controls (2,0.7) .. 
		(2.53,0.83);
		\draw[smooth,densely dotted, ->] (2.6,0.9) .. controls (3.3,1.1) .. 
		(4,0.9);
		\normalsize
	\end{tikzpicture}
\end{center}

The illustration above shows a picture of such a set $\Cose$, drawn as a union 
of a few disjoint blobs, with the black dots being the images of map 
$\Ne_{\Cose}$, associating to each component of $\Cose$ a nearby point along an 
orbit. 
Intuitively, each open set $\Cose$ in the definition should be thought of as a 
union of disjoint smaller open sets. The blobs are the preimages of the 
near orbit selection function $\Ne_{\Cose}$. 
Those aren't quite Rokhlin towers, in the sense that we do not require that 
each blob gets sent exactly to another via a group element; rather, points 
which do get sent from one blob to another do so in way consistent with the 
labels, and for any point $x$, there will be some set $\Cose$ which contains a 
long 
enough partial
orbit of $x$, so that this partial equivariance condition is meaningful.

We will show in Sections~\ref{sec:GP} and~\ref{sec:BLR} by topological 
dynamical techniques that the long thin covering dimension is finite for a 
large class of actions and then apply this notion in Section~\ref{sec:dimnuc} 
to estimate the nuclear dimension of the associated crossed product 
$C^*$-algebras. 

In the rest of the section, we study various properties and equivalent 
characterizations of the long thin covering dimension, which are needed in the 
later sections. 

\begin{Rmk} \label{rmk:ltc-dim-compact-space}
	When applying Definition~\ref{\LTCdef} to the case in which $X$ is compact, 
	it suffices to check the definition for $\Ko = X$. 
\end{Rmk}

We now discuss permanence properties with regard to subgroups and quotient groups.

\begin{Rmk} \label{rmk:ltc-dim-subgroup}
	It is clear from Definition~\ref{\LTCdef} that for any subgroup $H$ of $G$, 
	if we denote by $\alpha_{|_H}$ the restriction of $\alpha$ to $H$, then we 
	have 
	\[
	\dimltc\left( \alpha_{|_H} \right) \leq \dimltc(\alpha) \; .
	\]
	Moreover, if $G$ is the union of a directed family $\left\{ G_i \right\}_{i \in I}$ of subgroups, then we have 
	\[
	\dimltc(\alpha) = \sup_{i \in I} \dimltc\left( \alpha_{|_{G_i}} \right) \; .
	\]
	For the claim concerning the union of a directed family of groups, notice 
	that we consider each time a finite subset $L \subseteq G$, so in 
	particular it is contained in one of the groups $G_i$ in the directed 
	system.
\end{Rmk}

\begin{Rmk} \label{rmk:ltc-dim-quotient}
	It is also clear that if the action $\alpha \colon G \curvearrowright X$ factors through an action $\beta \colon H \curvearrowright X$ by a quotient of $G$, then we have 
	\[
	\dimltc\left( \alpha \right) = \dimltc(\beta) \; .
	\]
\end{Rmk}

\begin{Rmk} \label{rmk:ltc-dim-trivial-group}
	When we apply Definition~\ref{\LTCdef} to the trivial action $\id_X \colon 
	G \curvearrowright X$, it is clear from Definition~\ref{def:dimc} that 
	$\dimltc(\id_X) = \dimc(X)$, since condition~\eqref{item:\LTCdef:Eq} 
	becomes vacuous, while condition~\eqref{item:\LTCdef:Ca} can be trivially 
	arranged with $\Binu = 1$ by replacing each $\Cose$ by the fibers of the 
	map $\Ne_{\Cose}$. In particular, this equation applies to the case in 
	which $G$ is trivial. Combined with Remark~\ref{rmk:ltc-dim-subgroup}, 
	Theorem~\ref{thm:dimnuc-dim-general}, this shows 
	\[
		\dimnuc \left( C_0(X) \right) = \dim(X^+) = \dimc(X) \leq \dimltc(\alpha)
	\]
	for any action of any discrete group $G$. 
\end{Rmk}

Next we discuss various equivalent formulations of Definition~\ref{\LTCdef} , which are needed for the later sections.  
To this end, we need to introduce a few additional conditions. 

\def \Topic {\LTCextra}
\begin{Def}
	\label{\LTCextra}	
	Given any subsets $\Lo, \Bo \subseteq G$, any subset $\Ko \subseteq X$, any natural number $\Binu$, any collection $\Thseco$ of open subsets of $X$, consider the following conditions for a collection $\Coseco$ of open sets in $X$ and locally constant functions $\Ne_{\Cose} \colon \Cose \to X$ for $\Cose \in \Coseco$: 
	{
		\begin{enumerate}
			\renewcommand{\labelenumi}{\textup{(\theenumi)}} \renewcommand{\theenumi}{Mu$^+$}\item \label{item:\LTCextra:Muplus}
			We have $\mult_{\alpha,\Lo} (\Coseco) \leq d+1$ (see Definition~\ref{def:multiplicity-action}). 
			
			\renewcommand{\labelenumi}{\textup{(\theenumi)}} \renewcommand{\theenumi}{Co}\item \label{item:\LTCextra:Co}
			For any $\Cose \in \Coseco$, the set $\Ne_{\Cose}(\Cose)$ is {$\Lo$-connected} (see Definition~\ref{def:Lo-connected}).

			\renewcommand{\labelenumi}{\textup{(\theenumi)}} \renewcommand{\theenumi}{Co$^+$}\item \label{item:\LTCextra:Coplus}
			For any $\Cose \in \Coseco$, the set $\Ne_{\Cose}(\Cose)$ is {$\left(\Lo, \Binu\right)$-bounded} (see Definition~\ref{def:Lo-bounded}). 
			
			\renewcommand{\labelenumi}{\textup{(\theenumi)}} \renewcommand{\theenumi}{Bo}\item \label{item:\LTCextra:Bo}
			For any $\Cose \in \Coseco$, 
			the set $\Ne_{\Cose} (\Cose)$ is $\Bo$-bounded (see Definition~\ref{def:Lo-bounded}), i.e., 
			for any $x, y \in \Cose$, we have $\Ne_{\Cose}(x) \in \alpha_{\Bo} \left( \Ne_{\Cose}(y) \right)$. 
		\end{enumerate}
	}
	We also define, for any $\Cose \in \Coseco$, the notion of \emph{$\left( 
	\Ne_{\Cose}, \Lo \right)$-connected components} of $\Cose$ as the inverse 
	images under $\Ne_{\Cose}$ of the $\Lo$-connected components of 
	$\Ne_{\Cose} (\Cose)$ in the sense of 
	Definition~\ref{def:Lo-connected}. 	 
\end{Def}

Here the labels stand for ``multiplicity (strengthened)'', 
``connected'', ``connected (strengthened)'' and ``bounded''.
We will also write \Refcd{Muplus}, \Refcd{Co}, \Refcd{Coplus}, and \Refcd{Bo} 
to specify the parameters used in these conditions and the fact that they come from Definition~\ref{\Topic}.

\def \Topic {\LTCdef}

The next remark lists monotonicity properties of the conditions in 
Definitions~\ref{\LTCdef} and \ref{\LTCextra}, which follow immediately from 
the definitions. Recall that given two collections $\Thseco, \Thseco'$ of 
subsets of $X$, we say $\Thseco'$ refines $\Thseco$ if for any $\Thse' \in 
\Thseco'$, there exists $\Thse \in \Thseco$ such that $\Thse' \subseteq \Thse$. 
Condition~\Refcd[\LTCdef]{Th} makes sense when $\Thseco$ is an arbitrary 
collection of subsets of $X$, rather than just a finite open cover. 

\begin{Rmk} \label{rmk:ltc-dim-monotonous}
	In the context of Definition~\ref{\LTCdef}, given 
	\begin{itemize}
		\item natural numbers $d \leq d'$ and $\Binu \leq \Binu'$, 
		\item finite subsets $\Lo \subseteq \Lo' \Subset G$ and $\Bo \subseteq \Bo' \Subset G$, 
		\item compact subsets $\Ko \subseteq \Ko' \Subset X$, 
		\item finite collections $\Thseco, \Thseco'$ of open sets in $X$ with $\Thseco'$ refining $\Thseco$, 
		\item collections $\Coseco, \Coseco'$ of open sets in $X$ with $\Coseco'$ refining $\Coseco$ and,  
		\item locally constant functions $\Ne_{\Cose} \colon \Cose \to X$ for 
		$\Cose \in \Coseco$ and $\Ne'_{\Cose'} \colon \Cose' \to X$ for $\Cose' 
		\in \Coseco'$ satisfying that for any $\Cose' \in \Coseco'$, there 
		exists $\Cose \in \Coseco$ such that $\Cose' \subseteq \Cose$ and 
		$\Ne'_{\Cose'} = \Ne_{\Cose} |_{\Cose'}$,
	\end{itemize}
	we have the following: 
	\begin{enumerate}
		\item \label{rmk:ltc-dim-monotonous::parameters} If $\Coseco$ and $\left( \Ne_{\Cose} \right)_{\Cose \in \Coseco}$ satisfy \Refc{Lo}{\Lo',\Ko'} (respectively, \Refcd{Mu}, \Refc{Eq}{\Lo'}, \Refc{Th}{\Thseco'}, \Refcd{Ca}, \Refc[\LTCextra]{Muplus}{\Lo', d}, \Refcd[\LTCextra]{Co}, \Refcd[\LTCextra]{Coplus}, and \Refcd[\LTCextra]{Bo}), then they also satisfy \Refcd{Lo} (respectively, \Refc{Mu}{d'}, \Refcd{Eq}, \Refcd{Th}, \Refc{Ca}{\Binu'}, \Refc[\LTCextra]{Muplus}{\Lo, d'}, \Refc[\LTCextra]{Co}{\Lo'}, \Refc[\LTCextra]{Coplus}{\Lo', \Binu'}, and \Refc[\LTCextra]{Bo}{\Bo'}). 
		\item \label{rmk:ltc-dim-monotonous::Lo} If $\Coseco'$ satisfies \Refcd{Lo}, then $\Coseco$ also satisfies \Refcd{Lo}. 
		\item \label{rmk:ltc-dim-monotonous::Eq-Th-Bo} If $\Coseco$ and and $\left( \Ne_{\Cose} \right)_{\Cose \in \Coseco}$ satisfy \Refcd{Eq} (respectively, \Refcd{Th}, \Refcd{Ca} and \Refcd[\LTCextra]{Bo}), then $\Coseco'$ and $\left( \Ne'_{\Cose} \right)_{\Cose \in \Coseco'}$ also satisfy \Refcd{Eq} (respectively, \Refcd{Th}, \Refcd{Ca} and \Refcd[\LTCextra]{Bo}). 
		\item \label{rmk:ltc-dim-monotonous::Mu} If $\Coseco$ satisfies \Refcd{Mu} (respectively, \Refcd[\LTCextra]{Muplus}) and there is an \emph{injection} $\iota \colon \Coseco' \to \Coseco$ such that for any $\Cose' \in \Coseco'$, we have $\Cose' \subseteq \iota\left(\Cose'\right)$, then $\Coseco'$ also satisfies \Refcd{Mu} (respectively, \Refcd[\LTCextra]{Muplus}). 
		\item \label{rmk:ltc-dim-monotonous::Coplus} If $\Coseco$ satisfies \Refcd[\LTCextra]{Co} (respectively, \Refcd[\LTCextra]{Coplus}) and $\Coseco'$ is a subcollection of $\Coseco$, then $\Coseco'$ also satisfies \Refcd[\LTCextra]{Co} (respectively, \Refcd[\LTCextra]{Coplus}). 
	\end{enumerate}
\end{Rmk}
	
The following simple proposition lists a number of equivalent formulations of 
the long thin covering dimension. They are all easy to prove, and 
none are of any independent interest. However, it will be easier in various 
technical proofs to use as some of those equivalent formulations 
instead of the original definition, and therefore we record all of them in one 
proposition, to serve as a convenient reference.  
\def \Topic {\LTCdef}
\begin{Prop}\label{prop:ltc-dim-new}
	Let $\alpha \colon G \curvearrowright X$ be an action of a discrete group on a locally compact Hausdorff space. Let $d$ be a natural number. Then the following are equivalent: 
	{
		\begin{enumerate}
			
			\item \label{item:prop:ltc-dim-new:orig} \refstepcounter{thmenumi} \label{prop:ltc-dim-new:orig} 
			We have $\dimltc(\alpha) \leq d$, that is, 
			for any finite $\Lo 
			\Subset G$ and for any compact $\Ko \Subset X$, 
			there exists a natural number $\Binu$ such that 
			for any finite open cover $\Thseco$ of $X$, there exist $\Coseco$ and $\left( \Ne_{\Cose} \right)_{\Cose \in \Coseco}$ as in Definition~\ref{\LTCdef} 
			satisfying  
			\begin{center}
				\Refcd[\LTCdef]{Lo}, \Refcd[\LTCdef]{Mu}, \Refcd[\LTCdef]{Eq}, \Refcd[\LTCdef]{Th}, and \Refcd[\LTCdef]{Ca}. 
			\end{center}

			\item \label{item:prop:ltc-dim-new:Muplus} \refstepcounter{thmenumi}  \label{prop:ltc-dim-new:Muplus} 
			For any finite $\Lo \Subset G$ and for any compact $\Ko \Subset X$, 
			there exists a natural number $\Binu$ such that 
			for any finite open cover $\Thseco$ of $X$, there exist $\Coseco$ and $\left( \Ne_{\Cose} \right)_{\Cose \in \Coseco}$ as in Definition~\ref{\LTCdef} 
			satisfying  
			\begin{center}
				\Refcd[\LTCdef]{Lo}, \Refcd[\LTCextra]{Muplus}, \Refcd[\LTCdef]{Eq}, \Refcd[\LTCdef]{Th}, and \Refcd[\LTCdef]{Ca}. 
			\end{center}

			\item \label{item:prop:ltc-dim-new:Lominus} \refstepcounter{thmenumi}  \label{prop:ltc-dim-new:Lominus} 
			For any finite $\Lo \Subset G$ and for any compact $\Ko \Subset X$, 
			there exists a natural number $\Binu$ such that 
			for any finite open cover $\Thseco$ of $X$, there exist $\Coseco$ and $\left( \Ne_{\Cose} \right)_{\Cose \in \Coseco}$ as in Definition~\ref{\LTCdef} 
			satisfying  
			\begin{center}
				\Refc[\LTCdef]{Lo}{\{\grpid\},\Ko}, 
				\Refcd[\LTCextra]{Muplus}, \Refcd[\LTCdef]{Eq}, \Refcd[\LTCdef]{Th}, and  \Refcd[\LTCdef]{Ca}. 
			\end{center}
			Note that \Refc[\LTCdef]{Lo}{\{\grpid\},\Ko} is equivalent to requiring $\Coseco$ covers $\Ko$. 
			
			\item \label{item:prop:ltc-dim-new:connected} \refstepcounter{thmenumi}  \label{prop:ltc-dim-new:connected} 
			For any finite $\Lo \Subset G$ and for any compact $\Ko \Subset X$, 
			there exists a natural number $\Binu$ such that 
			for any finite open cover $\Thseco$ of $X$, there exist $\Coseco$ and $\left( \Ne_{\Cose} \right)_{\Cose \in \Coseco}$ as in Definition~\ref{\LTCdef} 
			satisfying  
			\begin{center}
				\Refcd[\LTCdef]{Lo}, \Refcd[\LTCdef]{Mu}, \Refcd[\LTCdef]{Eq}, \Refcd[\LTCdef]{Th}, \Refcd[\LTCdef]{Ca}, and \Refcd[\LTCextra]{Co}. 
			\end{center}

			\item \label{item:prop:ltc-dim-new:connected-plus} \refstepcounter{thmenumi}  \label{prop:ltc-dim-new:connected-plus} 
			For any finite $\Lo \Subset G$ and for any compact $\Ko \Subset X$, 
			there exists a natural number $\Binu$ such that 
			for any finite open cover $\Thseco$ of $X$, there exist $\Coseco$ and $\left( \Ne_{\Cose} \right)_{\Cose \in \Coseco}$ as in Definition~\ref{\LTCdef} 
			satisfying  
			\begin{center}
				\Refcd[\LTCdef]{Lo}, \Refcd[\LTCdef]{Mu}, \Refcd[\LTCdef]{Eq}, \Refcd[\LTCdef]{Th}, and \Refcd[\LTCextra]{Coplus}. 
			\end{center}

			\item \label{item:prop:ltc-dim-new:bounded} \refstepcounter{thmenumi}  \label{prop:ltc-dim-new:bounded} 
			For any finite $\Lo \Subset G$ and for any compact $\Ko \Subset X$, 
			there exists a finite subset $\Bo$ in the subgroup $\langle \Lo \rangle$ generated by $\Lo$ such that 
			for any finite open cover $\Thseco$ of $X$, there exist $\Coseco$ and $\left( \Ne_{\Cose} \right)_{\Cose \in \Coseco}$ as in Definition~\ref{\LTCdef} 
			satisfying  
			\begin{center}
				\Refcd[\LTCdef]{Lo}, \Refcd[\LTCdef]{Mu}, \Refcd[\LTCdef]{Eq}, \Refcd[\LTCdef]{Th}, and \Refcd[\LTCextra]{Bo}.
			\end{center}

			\item \label{item:prop:ltc-dim-new:combo-pre} \refstepcounter{thmenumi}  \label{prop:ltc-dim-new:combo-pre} 
			For any finite $\Lo \Subset G$ and for any compact $\Ko \Subset X$, 
			there exists a natural number $\Binu$ such that 
			for any finite open cover $\Thseco$ of $X$, there exist $\Coseco$ and $\left( \Ne_{\Cose} \right)_{\Cose \in \Coseco}$ as in Definition~\ref{\LTCdef} 
			satisfying  
			\begin{center}
				\Refc[\LTCdef]{Lo}{\{\grpid\}, \Ko}, \Refcd[\LTCextra]{Muplus}, \Refcd[\LTCdef]{Eq}, \Refcd[\LTCdef]{Th}, \Refcd[\LTCextra]{Coplus} and $\Coseco$ is finite. 
			\end{center}

			\item \label{item:prop:ltc-dim-new:cover-Ko} \refstepcounter{thmenumi} \label{prop:ltc-dim-new:cover-Ko} 
			For any finite $\Lo \Subset G$ and for any compact $\Ko \Subset X$, 
			there exists a natural number $\Binu$ such that 
			for any compact subset $\Ko' \Subset \Ko$ and 
			any collection $\Thseco$ of open sets in $X$ covering $\Ko'$, 
			there exist $\Coseco$ and $\left( \Ne_{\Cose} \right)_{\Cose \in \Coseco}$ as in Definition~\ref{\LTCdef} 
			satisfying  
			\begin{center}
				\Refc[\LTCdef]{Lo}{\{\grpid\}, \Ko'}, \Refcd[\LTCextra]{Muplus}, \Refcd[\LTCdef]{Eq}, \Refcd[\LTCdef]{Th}, \Refcd[\LTCextra]{Coplus} and $\Coseco$ is finite.  
			\end{center}
			
		\end{enumerate}
	}
\end{Prop}

\begin{proof}
	We show 
	\eqref{item:prop:ltc-dim-new:orig} $\Rightarrow$ 
	\eqref{item:prop:ltc-dim-new:Muplus} $\Rightarrow$ 
	\eqref{item:prop:ltc-dim-new:Lominus}  $\Rightarrow$ 
	\eqref{item:prop:ltc-dim-new:orig} $\Rightarrow$ 
	\eqref{item:prop:ltc-dim-new:connected} $\Rightarrow$ 
	\eqref{item:prop:ltc-dim-new:connected-plus}  $\Rightarrow$ 
	\eqref{item:prop:ltc-dim-new:bounded}  $\Rightarrow$ 
	\eqref{item:prop:ltc-dim-new:orig}   
	and 
	\eqref{item:prop:ltc-dim-new:Lominus}  $\Rightarrow$ 
	\eqref{item:prop:ltc-dim-new:combo-pre}  $\Rightarrow$ 
	\eqref{item:prop:ltc-dim-new:cover-Ko} $\Rightarrow$ 
	\eqref{item:prop:ltc-dim-new:Lominus}.  
	Before starting, we observe that upon replacing $\Lo$ by $\Lo \cup \Lo^{-1} \cup \{\grpid\}$, we may assume without loss of generality that $\grpid \in \Lo = \Lo^{-1}$ in each of the equivalent conditions above. Throughout the proof, we will make this assumption for our convenience. 
	
	To show 
	\eqref{item:prop:ltc-dim-new:orig} $\Rightarrow$ 
	\eqref{item:prop:ltc-dim-new:Muplus}, 
	we fix arbitrary $\Lo$, $\Ko$ and $\Thseco$ as in 
	\eqref{item:prop:ltc-dim-new:Muplus} and apply 
	\eqref{item:prop:ltc-dim-new:orig}, with $\Lo$ replaced by $\Lo^2$, to 
	obtain $\Binu$, which depends only on $\Lo$ and $\Ko$, as well as $\Coseco$ 
	and $\left( \Ne_{\Cose} \right)_{\Cose \in \Coseco}$, that satisfy 
	\begin{center}
		\Refc[\LTCdef]{Lo}{\Lo^2, \Ko}, \Refcd[\LTCdef]{Mu}, \Refc[\LTCdef]{Eq}{\Lo^2}, \Refcd[\LTCdef]{Th}, and \Refcd[\LTCdef]{Ca}. 
	\end{center}
	Then, using Notation~\ref{notation:group-action}, Remark~\ref{rmk:multiplicity-action-shrink-enlarge}, Remark~\ref{rmk:ltc-dim-monotonous}, and the fact that
	\[
	\alpha^{\cap}_{\Lo} \left( \Cose \right) = \left\{ x \in X \colon \alpha_{\Lo} (x) \subseteq \Cose \right\} \subseteq \Cose  \quad \text{for any } \Cose\in \Coseco \; ,
	\]
	we observe that the shrunken collection  $\Coseco' := \left\{ 
	\alpha^{\cap}_{\Lo}\left(\Cose\right) \colon \Cose \in \Coseco 
	\right\}$  	
	and the restricted maps $\Ne_{\Cose} |_{\alpha^{\cap}_{\Lo} \left( \Cose \right)} \colon \alpha^{\cap}_{\Lo} \left( \Cose \right) \to X$, for $\Cose \in \Coseco$, satisfy \Refcd{Lo}, \Refcd[\LTCextra]{Muplus}, \Refc[\LTCdef]{Eq}{\Lo^2} (and thus \Refcd[\LTCdef]{Eq}), \Refcd{Th}, and \Refcd{Ca}. 
	
	The implication \eqref{item:prop:ltc-dim-new:Muplus} $\Rightarrow$ 
	\eqref{item:prop:ltc-dim-new:Lominus}
	is clear in view of Remark~\ref{rmk:ltc-dim-monotonous}. 
	
	To show 
	\eqref{item:prop:ltc-dim-new:Lominus}  $\Rightarrow$ 
	\eqref{item:prop:ltc-dim-new:orig}, 
	we fix arbitrary $\Lo$, $\Ko$ and $\Thseco$ as in 
	\eqref{item:prop:ltc-dim-new:orig} and obtain a finite collection $\Auseco$ satisfying the conclusion of Lemma~\ref{lem:cover-clumping} relative to $\Ko$ and $\Thseco$. 
	We then apply	
	\eqref{item:prop:ltc-dim-new:Lominus}, with $\Lo$ replaced by $\Lo^3$ and 
	$\Thseco$ replaced by $\alpha^{\wedge}_{\Lo} \left( \Auseco \right)$ (see 
	Notation~\ref{notation:group-action}), to obtain $\Binu$, which depends 
	only on $\Lo$ and $\Ko$, as well as $\Coseco$ and $\left( \Ne_{\Cose} 
	\right)_{\Cose \in \Coseco}$, that satisfy 
	\begin{center}
		\Refc{Lo}{\{e\},\Ko}, \Refc[\LTCextra]{Muplus}{d, \Lo^{3}}, \Refc{Eq}{\Lo^{3}}, \Refc{Th}{\alpha^{\wedge}_{\Lo}\left(\Auseco\right)}, and \Refcd{Ca}. 
	\end{center}
	Then the enlarged collection $\Coseco' := \left\{ 
	\alpha^{\cup}_{\Lo}\left(\Cose\right) \colon \Cose \in \Coseco \right\}$ 
	supports functions 
	\[
		\Ne'_{\Cose} \colon \alpha^{\cup}_{\Lo} \left( \Cose \right) \to X
	\]
	given by
	\[
		\Ne'_{\Cose} ( \alpha_g (x) ) =  \alpha_g \left( \Ne_{\Cose} (x) \right) 
	\]
	for all  $g \in \Lo$ and  for all $ x \in \Cose$. These functions are 
	well-defined because of \Refc{Eq}{\Lo^{3}}, and it follows from the definitions that $\Coseco'$ and $\left( \Ne'_{\Cose} \right)_{\Cose \in \Coseco}$ satisfy 
	\Refcd{Lo}, \Refcd{Mu}, \Refcd{Eq}, and \Refc{Ca}{\Binu 
	\left|\Lo\right|}. 
	Moreover, observe that since $\Ne'_{\Cose} (\Cose') = \alpha^{\cup}_{\Lo} \left( \Ne_{\Cose} (\Cose) \right)$, i.e., for any $y \in \Ne'_{\Cose} (\Cose')$, we may write $y = \alpha_g (y')$ for some $g \in \Lo$ and $y' \in \Ne_{\Cose} (\Cose)$, and thus 
	\begin{align*}
		\left( \Ne'_{\Cose} \right)^{-1} (y) \cup \{y\} = \bigcup_{g \in \Lo \colon  \alpha_{g^{-1 }} (y) \in \Ne_{\Cose} (\Cose) } \alpha_{g} \left( \left( \Ne_{\Cose} \right)^{-1} \left( \alpha_{g^{-1 }} (y) \right) \cup \left\{ \alpha_{g^{-1 }} (y) \right\} \right) \; ,
	\end{align*}
	whence since $\Coseco$ and $\left( \Ne_{\Cose} 
	\right)_{\Cose \in \Coseco}$ satisfy  \Refc{Th}{\alpha^{\wedge}_{\Lo}\left(\Auseco\right)}, every term in the above union is contained in some $\Ause \in \Auseco$, and thus the union is contained in some $\Thse \in \Thseco$ by our choice of $\Auseco$ via Lemma~\ref{lem:cover-clumping}. This verifies \Refcd{Th} for $\Coseco'$ and $\left( \Ne'_{\Cose} \right)_{\Cose \in \Coseco}$. 
	Hence \eqref{item:prop:ltc-dim-new:orig} is verified 
	with $\Binu$ replaced by $\Binu \left|\Lo\right|$.
	
	To show 
	\eqref{item:prop:ltc-dim-new:orig} $\Rightarrow$ 
	\eqref{item:prop:ltc-dim-new:connected}, 
	we fix arbitrary $\Lo$, $\Ko$ and $\Thseco$ as in 
	\eqref{item:prop:ltc-dim-new:connected} and apply 
	\eqref{item:prop:ltc-dim-new:orig} to obtain $\Binu$, which depends only on 
	$\Lo$ and $\Ko$, as well as $\Coseco$ and $\left( \Ne_{\Cose} 
	\right)_{\Cose \in \Coseco}$, that satisfy 
		\Refcd[\LTCdef]{Lo}, \Refcd[\LTCdef]{Mu}, \Refcd[\LTCdef]{Eq}, \Refcd[\LTCdef]{Th}, and \Refcd[\LTCdef]{Ca}. 
	Then we replace each $\Cose \in \Coseco$ by its $\left( \Ne_{\Cose}, \Lo \right)$-connected components in the sense of Definition~\ref{\LTCextra}, 
	and replace $\Ne_{\Cose}$ by the corresponding restrictions to these subsets. 
	In view of Remark~\ref{rmk:ltc-dim-monotonous}, it is clear that the new $\Coseco$ and the new $\left( \Ne_{\Cose} \right)_{\Cose \in \Coseco}$ still satisfy \Refcd[\LTCdef]{Mu}, \Refcd[\LTCdef]{Eq}, \Refcd[\LTCdef]{Th}, and \Refcd[\LTCdef]{Ca}, as well as the additional \Refcd[\LTCextra]{Co}. To see the new $\Coseco$ also satisfies \Refcd[\LTCdef]{Lo}, it suffices to observe that for any $x \in \Ko$, if $\alpha_{\Lo}(x) \subseteq \Cose$ for some $\Cose$ in the original $\Coseco$, then by \Refcd[\LTCdef]{Eq}, $\alpha_{\Lo}(x)$ is contained in a single $\left( \Ne_{\Cose}, \Lo \right)$-connected component 
	of $\Cose$. 
	
	To show 
	\eqref{item:prop:ltc-dim-new:connected} $\Rightarrow$ 
	\eqref{item:prop:ltc-dim-new:connected-plus}, 
	we simply observe that \Refcd[\LTCdef]{Ca} and~\Refcd[\LTCextra]{Co} together imply \Refcd[\LTCextra]{Coplus}. 
	
	To show 
	\eqref{item:prop:ltc-dim-new:connected-plus} $\Rightarrow$ 
	\eqref{item:prop:ltc-dim-new:bounded}, 
	we simply observe that \Refcd[\LTCextra]{Coplus} implies \Refc[\LTCextra]{Bo}{\Lo^{\Binu}}. 
	
	To show 
	\eqref{item:prop:ltc-dim-new:bounded} $\Rightarrow$ 
	\eqref{item:prop:ltc-dim-new:orig}, 
	we simply observe that \Refc[\LTCextra]{Bo}{\Bo} implies \Refc[\LTCdef]{Ca}{\left|\Bo\right|}. 
	
	The proof of 
	\eqref{item:prop:ltc-dim-new:Lominus}  $\Rightarrow$ 
	\eqref{item:prop:ltc-dim-new:combo-pre} 
	is analogous to the combination of the proofs of \eqref{item:prop:ltc-dim-new:orig} $\Rightarrow$ 
	\eqref{item:prop:ltc-dim-new:connected} $\Rightarrow$ 
	\eqref{item:prop:ltc-dim-new:connected-plus}, 
	that is, replacing each $\Cose \in \Coseco$ by its $\left( \Ne_{\Cose}, \Lo \right)$-connected components in the sense of Definition~\ref{\LTCextra} 
	and replacing $\Ne_{\Cose}$ by the corresponding restrictions to these subsets. 
	The only deviations from \eqref{item:prop:ltc-dim-new:orig} $\Rightarrow$ 
	\eqref{item:prop:ltc-dim-new:connected} $\Rightarrow$ 
	\eqref{item:prop:ltc-dim-new:connected-plus} are the following: 
	\begin{enumerate*}
		\item the preservation of \Refc[\LTCdef]{Lo}{\{\grpid\}, \Ko}\ is straightforward to see; 
		\item the preservation of \Refcd[\LTCextra]{Muplus} follows from the observation that any $(\alpha,\Lo)$-close subset of $X$ cannot have nonempty intersections with two or more $\left( \Ne_{\Cose}, \Lo \right)$-connected components of a single $\Cose$ in the original $\Coseco$; 
		\item the additional requirement that $\Coseco$ is finite is easily arranged by picking a finite subcollection using the compactness of $\Ko$ and applying Remark~\ref{rmk:ltc-dim-monotonous} to verify that the subcollection also satisfy the desired conditions. 
	\end{enumerate*} 
	
	To show 
	\eqref{item:prop:ltc-dim-new:combo-pre} $\Rightarrow$ 
	\eqref{item:prop:ltc-dim-new:cover-Ko}, 
	we fix arbitrary $\Lo$ and $\Ko$ as in \eqref{item:prop:ltc-dim-new:cover-Ko} and apply \eqref{item:prop:ltc-dim-new:combo-pre} to obtain $\Binu$ that satisfy the requirement in \eqref{item:prop:ltc-dim-new:combo-pre}. 
	Next we fix arbitrary $\Ko'$ and $\Thseco$ as in \eqref{item:prop:ltc-dim-new:cover-Ko}. 	
	We apply the compactness of $\Ko'$ to pick a finite subcollection $\Thseco'$ of $\Thseco$ that covers $\Ko'$ and then define a finite open cover $\Thseco'' = \Thseco' \cup \{ X \setminus \Ko' \}$ of $X$. 
	Thus by our assumption on $\Binu$, there exist $\Coseco$ and $\left( \Ne_{\Cose} \right)_{\Cose \in \Coseco}$  that satisfy 
	\begin{center}
		\Refc[\LTCdef]{Lo}{\{\grpid\}, \Ko}, \Refcd[\LTCextra]{Muplus}, \Refcd[\LTCdef]{Eq}, \Refc[\LTCdef]{Th}{\Thseco''}, \Refcd[\LTCextra]{Coplus} and $\Coseco$ is finite. 
	\end{center}
	We then modify $\Coseco$ and $\left( \Ne_{\Cose} \right)_{\Cose \in \Coseco}$ to form $\Coseco'$ and $\left( \Ne'_{\Cose'} \right)_{\Cose' \in \Coseco'}$ as follows: we replace each $\Cose \in \Coseco$ by 
	\[
	\Cose' = \left\{ x \in \Cose \colon \text{there exists } \Thse \in \Thseco' \text{ such that } \Ne_{\Cose}^{-1} \left( \Ne_{\Cose} (x) \right) \cup \left\{ \Ne_{\Cose} (x) \right\} \subseteq \Thse \right\}  \;,
	\]
	let $\Coseco' = \{ \Cose' \colon \Cose \in \Coseco \}$, and let $\Ne'_{\Cose'}$ be the restriction of $\Ne_{\Cose}$ to $\Cose'$ for any $\Cose \in \Coseco$. 
	It is clear that $\Coseco'$ satisfies \Refcd{Th}. Since $\Coseco$ satisfies \Refc{Th}{\Thseco''}, 
	it follows that for each $\Cose \in \Coseco$, we have $\Cose \setminus \Cose' \subseteq X \setminus \Ko'$. Since $\Coseco$ satisfies \Refc[\LTCdef]{Lo}{\{\grpid\},\Ko}, it then follows that $\Coseco'$ satisfies \Refc[\LTCdef]{Lo}{\{\grpid\},\Ko'}. 
	The fact that  $\Coseco'$ and $\left( \Ne'_{\Cose'} \right)_{\Cose' \in \Coseco'}$ satisfy \Refcd[\LTCextra]{Muplus}, \Refcd[\LTCdef]{Eq}, and \Refcd[\LTCextra]{Coplus} follows from Remark~\ref{rmk:ltc-dim-monotonous}.

	To show 
	\eqref{item:prop:ltc-dim-new:cover-Ko} $\Rightarrow$ 
	\eqref{item:prop:ltc-dim-new:Lominus}, 
	we simply fix $\Ko'$ to be equal to $\Ko$ and observe that \Refc[\LTCextra]{Bo}{\Bo} implies \Refc[\LTCdef]{Ca}{\left|\Bo\right|}. 
\end{proof}

The following somewhat long and technical proposition is the key technical tool from this section which will be used in Section \ref{sec:dimnuc} in order to prove the main theorem. This characterization involves constructing almost invariant partitions of unity, which are needed in order to construct decomposable approximations for the crossed product $C^*$-algebras.

{

\def \Topic {\LTCprop}
\begin{Prop}\label{prop:ltc-dim}
	Let $\alpha \colon G \curvearrowright X$ be an action of a discrete group on a locally compact Hausdorff space. Let $d$ be a natural number. Then we have $\dimltc(\alpha) \leq d$ if and only if the following holds:

	For any finite subset $\Lo \Subset G$, for any compact subset $\Ko \Subset X$, and for any $\Er>0$, there is a finite subset $\Bo$ in the subgroup $\langle \Lo \rangle$ generated by $\Lo$, such that for any open cover $\Thseco$ of $\alpha^\cup_{\Lo} (\Ko)$, there exist
		\begin{itemize}
			\item collections $\Coseconum{0} , \ldots , \Coseconum{d}$ of disjoint open subsets of $X$, together with $\Coseco = \Coseconum{0} \cup \ldots \cup \Coseconum{d}$, 
			\item locally constant functions 
			$\Ne^{(l)} \colon \bigcup\Coseconum{l} \to X$ for $l \in \intervalofintegers{0}{d}$, 
			and
			\item  continuous functions $\Ponum{\Cose}{l} \colon X \to [0,1]$ for $l \in \intervalofintegers{0}{d}$,
		\end{itemize}
		satisfying the following conditions:
		{
		\begin{enumerate}
			\renewcommand{\labelenumi}{\textup{(\theenumi)}}
			\renewcommand{\theenumi}{Lo}\item \label{item:prop:ltc-dim:Lo}
				For any $x \in \Ko$, there exists $\Cose \in \Coseco$ such that $\alpha_{\Lo} (x) \subseteq \Cose$.

			\renewcommand{\labelenumi}{\textup{(\theenumi)}} \renewcommand{\theenumi}{Eq}\item \label{item:prop:ltc-dim:Eq}
				For any $l \in \intervalofintegers{0}{d}$,  $\Ne^{(l)}$ is {$\Lo$-equivariant} in the following sense: for any $x \in \bigcup\Coseconum{l}$ and for any $g \in \Lo$, if $\alpha_{g} (x) \in \bigcup\Coseconum{l}$ then $\Ne^{(l)}\left( \alpha_{g} (x) \right) = \alpha_g \left( \Ne^{(l)}(x) \right)$.  
				
			\renewcommand{\labelenumi}{\textup{(\theenumi)}} \renewcommand{\theenumi}{Th}\item \label{item:prop:ltc-dim:Th}
				For any 
				$y \in \bigcup_{l = 0}^d  \Ne^{(l)} \left(\bigcup \Coseconum{l} \right)$, there exists $\Thse \in \Thseco$ such that $\{y\} \cup \left( \bigcup_{l = 0}^d  \left( \Ne^{(l)} \right)^{-1}(y) \right) \subseteq \Thse$.

			\renewcommand{\labelenumi}{\textup{(\theenumi)}} \renewcommand{\theenumi}{Bo}\item \label{item:prop:ltc-dim:Bo}
				For any $l \in \intervalofintegers{0}{d}$, for any $\Cose \in \Coseconum{l}$ and for any $x, y \in \Cose$, we have $\Ne^{(l)}(x) \in \alpha_{\Bo} \left( \Ne^{(l)}(y) \right)$. 
				
			\renewcommand{\labelenumi}{\textup{(\theenumi)}} \renewcommand{\theenumi}{Pr}\item \label{item:prop:ltc-dim:Pr}
				Every $\Cose \in \Coseco$ is precompact. 

			\renewcommand{\labelenumi}{\textup{(\theenumi)}} \renewcommand{\theenumi}{Fi}\item \label{item:prop:ltc-dim:Fi}
				The collection $\Coseco$ is finite. 

			\renewcommand{\labelenumi}{\textup{(\theenumi)}} \renewcommand{\theenumi}{Su}\item \label{item:prop:ltc-dim:Su}
				For any $l \in \intervalofintegers{0}{d}$ and for any $x$ in  
				the support of $\Ponum{\Cose}{l}$, 
				we have $\alpha_{\Lo} (x) \subseteq \Cose$ for some $\Cose \in \Coseco^{(l)}$. 
				
			\renewcommand{\labelenumi}{\textup{(\theenumi)}} \renewcommand{\theenumi}{In}\item \label{item:prop:ltc-dim:In}
				For any $l \in \intervalofintegers{0}{d}$, 
				the function $\Ponum{\Cose}{l}$ is \emph{$(\Lo, \Er)$-invariant} in the sense of Definition~\ref{def:Lipschitz-alpha}. 
				
			\renewcommand{\labelenumi}{\textup{(\theenumi)}} \renewcommand{\theenumi}{Un}\item \label{item:prop:ltc-dim:Un}
				For any $x \in \Ko$, we have $\sum_{l = 0}^{d}  \Ponum{\Cose}{l} (x) = 1$ while for any other $x \in X$, we have $\sum_{l = 0}^{d} \Ponum{\Cose}{l} (x) \leq 1$. 
		\end{enumerate}
		}
\end{Prop}

Here the labels stand for ``long'', 
``equivariant'', ``thin'', ``bounded'', ``precompact'', ``finite'', ``support'', ``invariant'', and ``partition of unity''. 
We will also write \Refcd{Lo}, 
\Refcd{Eq}, \Refcd{Th}, 
\Refcd{Bo}, \Refcd{Pr}, \Refcd{Fi}, \Refcd{Su}, \Refcd{In}, and \Refcd{Un} 
to specify the parameters used in these conditions and the fact that they come from Proposition~\ref{\Topic}.

The proof below 
shares some common ideas with the (easier) proof of Proposition~\ref{prop:orbit-asdim}.

\begin{proof}
	The ``if'' direction is immediate since, upon defining, for any $\Cose \in \Coseco$, $\Ne_{\Cose} := \Ne^{(l)}|_{\Cose}$ for the smallest $l$ such that $\Cose \in \Coseconum{l}$, we see that \Refcd{Lo}, \Refcd{Eq}, and \Refcd{Th}, respectively, directly imply \Refcd[\LTCdef]{Lo}, \Refcd[\LTCdef]{Eq}, and \Refcd[\LTCdef]{Th}, respectively, while the fact that $\Coseco$ can be decomposed into $(d+1)$ disjoint families clearly implies that it satisfies \Refcd[\LTCdef]{Mu}, and \Refcd[\LTCextra]{Bo} clearly implies \Refc[\LTCdef]{Ca}{|\Bo|}. 
	
	The proof of the ``only if'' direction will be divided into steps: 
	\begin{enumerate}[itemindent=*,leftmargin=1em,label=\textit{Step~(\arabic*).},ref=\arabic*]
	
		\item \label{item:prop:ltc-dim:proof:prep-Ko-Lo} 
		Following the statement of the proposition, we fix a finite subset $\Lo \Subset G$, a compact subset $\Ko \Subset X$, 
		and $\Er>0$. 
		By replacing $\Lo$ by $\Lo \cup \{e\} \cup \Lo^{-1}$, we may assume, without loss of generality, that $\Lo$ contains the unit $e$ of $G$ and $\Lo = \Lo ^{-1}$.

		\item \label{item:prop:ltc-dim:proof:prep-M} 
		Set
		\[
		\Binu := \left\lceil \frac{4(d+1) \left( 2 (d+1)^2 + 1 \right) }{\min\{ \Er, 1 \}} \right\rceil  \; .
		\]
		This guarantees that 
		\[
		\frac{4(d+1) \left( 2d^2+4d+3 \right) }{\Binu}  \leq \Er
		\]
		and 
		\[
		\Binu > 10 (d+1)^2  \; .
		\]
		
		\item \label{item:prop:ltc-dim:proof:ltc-Ca} 
		Since $\dimltc(\alpha) \leq d$, we may fix a natural number $\Binu'$ that satisfies the requirement of Proposition~\ref{prop:ltc-dim-new}\eqref{item:prop:ltc-dim-new:cover-Ko} (in place of $\Binu$ there).

		\item \label{item:prop:ltc-dim:proof:prep-Bo}
		Define 
		\[
			\Bo := L^{3 \Binu (\Binu'+2)} \Subset \langle \Lo \rangle \; .
		\] 
		In the rest of the proof, we show that $\Bo$ satisfies the requirement in the statement of the proposition. 
		
		\item \label{item:prop:ltc-dim:proof:prep-Thseco} 
		To do this, we further follow the statement of the proposition and fix an open cover $\Thseco$ of $\alpha^\cup_{\Lo} (\Ko)$. 
		By the compactness of $\alpha^\cup_{\Lo} (\Ko)$, there is a finite subcollection $\Thseco'$ in $\Thseco$ that covers $\alpha^\cup_{\Lo} (\Ko)$. 
		
		\item \label{item:prop:ltc-dim:proof:prep-W} We apply Lemma~\ref{lem:cover-clumping} to $\Thseco'$ to obtain a finite collection $\Auseco$  of open sets in $X$ such that $\alpha^\cup_{\Lo} (\Ko) \subset \bigcup \Auseco$ and, for any subcollection $\Auseco' \subseteq \Auseco$, if $\bigcap \Auseco' \not= \varnothing$, then there exists $\Thse \in \Thseco'$ such that $\bigcup \Auseco' \subseteq \Thse$. 
		By the observation at the end of Notation~\ref{notation:group-action}, we have 
		\[
			\Ko \subseteq \bigcup \alpha^\cap_{\Lo^{\Binu}} \left( \alpha^\cup_{\Lo^{\Binu}} (\Ko) \right) \subseteq  \alpha^\cap_{\Lo^{\Binu}} \left( \bigcup \Auseco \right) \subseteq \bigcup \alpha^\wedge_{\Lo^{\Binu}}\left(\Auseco\right) \; .
		\]
		
		\item \label{item:prop:ltc-dim:proof:ltc-Coseco} 
		By Step~\eqref{item:prop:ltc-dim:proof:ltc-Ca} and following Proposition~\ref{prop:ltc-dim-new}\eqref{item:prop:ltc-dim-new:cover-Ko} (with $\Ko' := \Ko$), there exist 
		a finite collection $\Laseco$ of open sets in $X$ and locally constant functions $\Ne_{\Lase} \colon \Lase \to X$ for $\Lase \in \Laseco$ satisfying  
		\Refc[\LTCdef]{Lo}{\{e\}, \Ko} (i.e., $\Ko \subseteq \bigcup \Laseco$), \Refc[\LTCextra]{Muplus}{d, \Lo^{3\Binu}}, \Refc[\LTCdef]{Eq}{\Lo^{3\Binu}}, \Refc[\LTCdef]{Th}{\alpha^\wedge_{\Lo^{\Binu}}\left(\Auseco\right)}, and~\Refc[\LTCextra]{Coplus}{\Lo^{3\Binu}, \Binu'}. 
		Note that 
		\Refc[\LTCextra]{Coplus}{\Lo^{3\Binu}, \Binu'} implies \Refc[\LTCextra]{Bo}{\Lo^{3 \Binu \Binu'}}.

		\item \label{item:prop:ltc-dim:proof:extend} As in the proof of \eqref{item:prop:ltc-dim-new:Lominus}  $\Rightarrow$ 
		\eqref{item:prop:ltc-dim-new:orig} in Proposition~\ref{prop:ltc-dim-new}, if we set
		\[
		 \Lasetil = \alpha^\cup_{\Lo^{\Binu}} \left( \Lase \right)  \text{ for any } \Lase \in \Laseco \, ,
		\]
		and define
		\[
			\Lasecotil = \left\{ \Lasetil \colon \Lase \in \Laseco \right\}
			\quad \, ,
		\]
		then using \Refc[\LTCdef]{Eq}{\Lo^{3\Binu}}, we may extend each $\Ne_{\Lase}$ to a well-defined function (with the same notation)
		\[
			\Ne_{\Lase} \colon \Lasetil \to X \, 
		\]
		given by
		\[	
			 \Ne_{\Lase} ( \alpha_g (x) ) = \alpha_g \left( \Ne_{\Lase} (x) \right) \quad \text{for any } g \in \Lo^{\Binu} \text{ and any } x \in \Lase \; ,
		\]
		such that together $\Lasecotil$ and $\left( \Ne_{\Lase} \right)_{\Lase \in \Laseco}$ satisfy \Refc[\LTCdef]{Lo}{\Lo^{\Binu}, \Ko}, \Refc[\LTCdef]{Mu}{d}, \Refc[\LTCdef]{Eq}{\Lo^{\Binu}}, \Refc[\LTCdef]{Th}{\Auseco}, and~\Refc[\LTCextra]{Bo}{ \Lo^{\Binu} \Lo^{3 \Binu \Binu'} \Lo^{\Binu}}.

		\item \label{item:prop:ltc-dim:proof:pou} Since $\Ko \subseteq \bigcup \Laseco$ by Step~\eqref{item:prop:ltc-dim:proof:ltc-Coseco}, we may apply Lemma~\ref{lem:relative-pou} to obtain a partition of unity
		\[
			\left\{ \Lapo_{\Lase} \colon  X \to [0,1] \right\}_{\Lase \in \Laseco} 
		\]
		for $\Ko \subseteq X $ subordinate to $ \Lasecohat $ and with compact supports, that is, for each $\Lase \in \Laseco$, the support of $\Lapo_{\Lase}$ is compact and contained in $ \Lasehat $, and 
		$\sum_{\Lase \in \Laseco} \Lapo_{\Lase} (x) \leq 1 $ for all $x \in X$ with equality holding when $x \in \Ko$.
		
		\item \label{item:prop:ltc-dim:proof:flatten} \label{item:prop:ltc-dim:proof:h-invariant} For each $\Lase \in \Laseco$, we apply Lemma~\ref{lem:Lipschitz-alpha-staircase} to ``flatten'' $\Lapo_{\Lase}$ and obtain a continuous function
		\[
			\Lapotil_{\Lase} = \Staircase{\Lo}{\Binu} \left( \Lapo_{\Lase} \right) \colon X \to [0,1] \; ,
		\]
		which is $\left(\Lo, \frac{1}{\Binu}\right)$-invariant and satisfies $0 \leq \Lapo_{\Lase} (x) \leq \Lapotil_{\Lase} (x) \leq 1$ for any $x \in X $, Thus,
		\[
			\sum_{\Lase \in \Laseco} \Lapotil_{\Lase} (x) \geq \sum_{\Lase \in \Laseco} \Lapo_{\Lase} (x)  = 1 \qquad \text{for any } x \in \Ko \; .
		\]
		In addition, it follows from Lemma~\ref{lem:Lipschitz-alpha-staircase}\eqref{item:lem:Lipschitz-alpha-staircase:support} that 
		\[
			\supp \left( \Lapotil_{\Lase} \right) \subseteq \alpha^\cup_{\Lo^{\Binu}}  \left(\supp \left( \Lapo_{\Lase} \right)\right) \subseteq \alpha^\cup_{\Lo^{\Binu}}  \left(\Lasehat\right) = \Lasetil \; ,
		\]
		which also shows $\supp \left( \Lapotil_{\Lase} \right)$ is compact.

		\item \label{item:prop:ltc-dim:proof:h-norm} Since $\Lasecotil$ satisfies \Refc[\LTCdef]{Mu}{d} in Step~\eqref{item:prop:ltc-dim:proof:extend}, there are no more than $(d+1)$ nonzero terms in the sum $\sum_{\Lase \in \Laseco} \Lapotil_{\Lase} (x)$ for each $x \in X$. It then follows from the pigeonhole principle that for any $x \in \Ko$, there exists $\Lase_x \in \Laseco$ such that $\Lapotil_{\Lase_x} (x) \geq \frac{1}{d+1}$. 
		
		\item \label{item:prop:ltc-dim:proof:nerve} Using the notations of Definition~\ref{def:Rplus}, we have a continuous map 
		\[
			{\Lapotil} \colon X \to \Rplusfuncs{\Laseco} 
		\]
		given by
		\[	
			{\Lapotil}( x )(\Lase) = {\Lapotil}_{\Lase} (x) \, . 
		\]
		It follows from Steps~\eqref{item:prop:ltc-dim:proof:flatten} and~\eqref{item:prop:ltc-dim:proof:h-norm} that 
		${\Lapotil}$ is $\left(\Lo, \frac{1}{\Binu}\right)$-invariant with regard to the $\ell^\infty$-metric on $\Rplusfuncs{\Laseco}$, that
		\[
		\left\| {\Lapotil} (x)  \right\| \leq 1  \quad \text{and} \quad  \left| \supp \left( {\Lapotil} (x)  \right) \right| \leq d+1 \quad \text{for any } x \in X \; ,
		\]
		and that
		\[
		\left\| {\Lapotil} (x)  \right\| \geq \frac{1}{d+1}  \quad \text{and in particular} \quad {\Lapotil} (x) \not= 0  \quad \text{for any } x \in \Ko \; .
		\]
		
		\item \label{item:prop:ltc-dim:proof:U-F} Using the notations of Definition~\ref{def:simplicial-barycentric}, we define, for any nonempty subcollection $F$ of $\Laseco$ and any $\varepsilon \geq 0$,  
		\[
			\Cose^{F,\varepsilon} := \Lapotil^{-1} \left( \Rplusfuncs[{F,\varepsilon}]{\Laseco} \right)
			 = \{ x \in X \mid \inf_{\Lase \in F} \Lapotil_{\Lase}(x) > \varepsilon +  \sup_{\Lase \in \Laseco \smallsetminus F} \Lapotil_{\Lase}(x)  \}
			 \, .
		\]
		(If $\Laseco \smallsetminus F = \varnothing$ then we take the supremum to be $0$.)
		This is an open subset of $X$ by Lemma~\ref{lem:simplicial-barycentric}\eqref{lem:simplicial-barycentric::open}. Observe that $\Cose^{F,\varepsilon}$ is decreasing in $\varepsilon$. Since $\Lapotil$ is $\left(\Lo, \frac{1}{\Binu}\right)$-invariant, it follows from Lemma~\ref{lem:simplicial-barycentric}\eqref{lem:simplicial-barycentric::nbhd} that 
		\[
			\alpha^\cup_{\Lo} \left(\Cose^{F,\varepsilon + \frac{2}{\Binu}}\right) \subseteq \Cose^{F,\varepsilon}  \qquad \text{for any }\varepsilon \geq 0 \; .
		\]
		Furthermore, it follows from Lemma~\ref{lem:simplicial-barycentric}\eqref{lem:simplicial-barycentric::nbhd} and~\eqref{lem:simplicial-barycentric::partition} that 
		\[
			\Cose^{F,\varepsilon} \cap \Cose^{F',\varepsilon'} = \varnothing
		\]
		for any $\varepsilon, \varepsilon' \geq 0$ and any nonempty $F, F' \subseteq \Laseco$ that are different but have the same cardinality. 
		
		\item \label{item:prop:ltc-dim:proof:F-too-large} 
		For any $F \subseteq \Laseco$ with $|F| > d+1$ we claim that
		 $\Cose^{F,\varepsilon} = \varnothing$ for any $\varepsilon \geq 0$. 
		Indeed, for any $x \in X$, 
		by Step~\eqref{item:prop:ltc-dim:proof:nerve} we have $\left| \supp \left( {\Lapotil} (x) \right) \right| \leq d+1$. By Lemma~\ref{lem:simplicial-sorting}, this implies $\Ma[\Laseco]{|F|} \circ \Lapotil (x) = 0$,  and thus also $\Dinum[\Laseco]{|F|} \circ \Lapotil (x) = 0$ by Definition~\ref{def:simplicial-diff-op}. Therefore $\Lapotil (x) \not\in \Rplusfuncs[{\left( |F|, \varepsilon \right)}]{\Laseco}$ by Definition~\ref{def:simplicial-barycentric}, which implies $x \not\in \Cose^{F,\varepsilon}$ by Lemma~\ref{lem:simplicial-barycentric}\eqref{lem:simplicial-barycentric::partition}. 
		
		\item \label{item:prop:ltc-dim:proof:U-l} Using the notations of Definition~\ref{def:simplicial-barycentric}, we define, for any nonnegative integer $l$, 
		\[
			\Coseco^{(l)} := \left\{ \Cose^{F,\frac{2}{\Binu}} \colon F \in \powerset^{(l+1)}(\Laseco) \right\} \, .
		\]
		This is a disjoint family of open subsets of $X$. 
		Let 
		\[
			\Coseco := \Coseconum{0} \cup \Coseconum{1} \cup \ldots \cup \Coseconum{d} \; . 
		\]
		Since $\Coseco^{(l)} = \{ \varnothing \}$ for any integer $l > d$ by Step~\eqref{item:prop:ltc-dim:proof:F-too-large}, we have 
		\[
			\Coseco = \bigcup_{l=0}^{\infty} \Coseconum{l} = \left\{ \Cose^{F,\frac{2}{\Binu}} \colon \text{nonempty } F \subseteq \Laseco \right\} \; .
		\]
		
		\item \label{item:prop:ltc-dim:proof:F-choose} Choose, for any $\Cose \in \Coseco$, a nonempty $F_{\Cose} \subseteq \Laseco$ with $\Cose = \Cose^{F_{\Cose},\frac{2}{\Binu}}$ and an element $\Lase_{\Cose} \in F_{\Cose}$. 
		It follows from the definition
		in Step~\eqref{item:prop:ltc-dim:proof:U-F} and Definition~\ref{def:simplicial-barycentric} that for 
		any $\varepsilon \geq 0$, we have ${\Lapotil}_{\Lase_{\Cose}} (x) > 0$ for any $x \in \Cose^{F_{\Cose},\varepsilon}$, whence $\Cose^{F_{\Cose},\varepsilon} \subseteq \widetilde{\Lase_{\Cose}}$ by Step~\eqref{item:prop:ltc-dim:proof:flatten}. We may now define
		\[
		\Ne_{U}  \colon \Cose = \Cose^{F_{\Cose},\frac{2}{\Binu}} \to X
		\]
		to be the restriction of $\Ne_{{\Lase_{\Cose}}} \colon \widetilde{\Lase_{\Cose}} \to X$ (see Step~\eqref{item:prop:ltc-dim:proof:extend}) to $\Cose^{F_{\Cose},\frac{2}{\Binu}}$. 
		We also define 
		\[
			\Ne^{(l)} = \bigcup_{\Cose \in \Coseconum{l}} \Ne_{\Cose} \colon \bigcup \Coseconum{l} \to X \; ,
		\]
		which is well-defined thanks to the disjointness of $\Coseconum{l}$.

		\item \label{item:prop:ltc-dim:proof:Lo} We prove condition \Refcd[\LTCprop]{Lo}: Given any $x \in \Ko$, since we observed $\left\| {\Lapotil} (x)  \right\| \geq \frac{1}{d+1}$ and $\left| \supp \left( {\Lapotil} (x) \right) \right| \leq d+1$ in Step~\eqref{item:prop:ltc-dim:proof:nerve}, and defined $\Binu$ so that $\frac{4}{\Binu} <  \frac{1}{(d+1)^2}$ in Step~\eqref{item:prop:ltc-dim:proof:prep-M}, it follows from Lemma~\ref{lem:simplicial-barycentric}\eqref{lem:simplicial-barycentric::cover} and ~\eqref{lem:simplicial-barycentric::partition} that there exists 
		$F \subseteq \Laseco$ such that ${\Lapotil} (x)  \in \Rplusfuncs[F, \frac{4}{\Binu} ]{\Laseco}$. 
		By Lemma~\ref{lem:simplicial-barycentric}\eqref{lem:simplicial-barycentric::nbhd} and~\eqref{lem:simplicial-barycentric::partition} together with the fact that $\Lapotil$ is $\left(\Lo, \frac{1}{\Binu}\right)$-invariant, this implies that 
		\[
			{\Lapotil} (\alpha_{\Lo} (x))  \subseteq \Rplusfuncs[{F, \frac{2}{\Binu} }]{\Laseco} \; ,
		\]
		that is, $\alpha_{\Lo} (x) \subseteq \Cose^{F, \frac{2}{\Binu}} \in \Coseco$.

		\item \label{item:prop:ltc-dim:proof:Eq} 
		We prove condition~\Refcd[\LTCprop]{Eq}: Given any $x \in \bigcup\Coseconum{l}$ and any $g \in \Lo$ with $\alpha_{g} (x) \in \bigcup\Coseconum{l}$, we fix $\Cose \in \Coseconum{l}$ such that $x \in \Cose$. 
		By Steps~\eqref{item:prop:ltc-dim:proof:F-choose} and~\eqref{item:prop:ltc-dim:proof:U-F}, we have $x \in \Cose = \Cose^{F_{\Cose},\varepsilon} \subseteq \widetilde{\Lase_{\Cose}}$ and
		\[
		\alpha_{g} (x) \in \alpha^\cup_{\Lo} \left(\Cose^{F_{\Cose},\frac{2}{\Binu}}\right) \subseteq \Cose^{F_{\Cose},0} \subseteq \left(\bigcup\Coseconum{l}\right)  \setminus \left( \bigcup_{F \in \powerset^{(l+1)}(\Laseco) \setminus \left\{ F_{\Cose} \right\}} \Cose^{F,\frac{2}{\Binu}}  \right)  \subseteq \Cose  \; . 
		\]
		Hence by Step~\eqref{item:prop:ltc-dim:proof:F-choose} and condition~\Refc[\LTCdef]{Eq}{\Lo^{\Binu}} in Step~\eqref{item:prop:ltc-dim:proof:extend}, we have \begin{align*}
		\Ne^{(l)} \left(\alpha_{g} (x)\right) & = \Ne_{\Cose} \left(\alpha_{g} (x)\right)  = \Ne_{\Lase_{\Cose}} \left(\alpha_{g} (x)\right) 
		\\
		& =  \alpha_{g} \left( \Ne_{\Lase_{\Cose}} (x) \right) =  \alpha_{g} \left( \Ne_{\Cose} (x) \right)  =  \alpha_{g} \left( \Ne^{(l)} (x) \right)
		\, ,
		\end{align*}
		 as claimed.

		\item \label{item:prop:ltc-dim:proof:Th} We prove condition~\Refcd[\LTCprop]{Th}.
		Given 
		any $y \in \bigcup_{l = 0}^d  \Ne^{(l)} \left( \bigcup \Coseconum{l} \right)$, by Steps~\eqref{item:prop:ltc-dim:proof:F-choose} and~\eqref{item:prop:ltc-dim:proof:extend} as well as \Refc[\LTCdef]{Th}{\alpha^\wedge_{\Lo^{\Binu}}\left(\Auseco\right)} in Step~\eqref{item:prop:ltc-dim:proof:ltc-Coseco} (see also the proof of \eqref{item:prop:ltc-dim-new:Lominus}  $\Rightarrow$ 
		\eqref{item:prop:ltc-dim-new:orig} in Proposition~\ref{prop:ltc-dim-new}), for any 
		$g \in \Lo^{\Binu}$, for any $l \in \intervalofintegers{0}{d}$, and for any $\Cose \in \Coseconum{l}$ such that $ y \in \Ne_{\Cose} ( \Cose ) $ we fix a set $\Ause_{\Cose, g} \in \Auseco$ satisfying 
		\[
		\alpha_{g} \left( \left( \Ne_{\Lase_{\Cose}} \right)^{-1} \left( \alpha_{g^{-1 }} (y) \right) \cup \left\{ \alpha_{g^{-1 }} (y) \right\} \right)  \subseteq \Ause_{\Cose, g}
		\, ,
		\]
		and then 
		we have 
		\begin{align*}
			&\ \left( \Ne^{(l)} \right)^{-1}(y) \cup \{y\}  \\
			=&\ \bigcup_{ \{ \Cose \in \Coseconum{l} \colon y \in \Ne_{\Cose} \left( \Cose \right) \} } \left( \left(\Ne_{\Cose}\right)^{-1}(y)  \cup \{y\}  \right) \\
			=&\ \bigcup_{ \{ \Cose \in \Coseconum{l} \colon y \in \Ne_{\Cose} \left( \Cose \right) \} } \left( \left(\Ne_{\Lase_{\Cose}}\right)^{-1}(y)  \cup \{y\}  \right) \\
			=&\ \bigcup_{ \left \{ \underset {y \in \Ne_{\Cose} \left( \Cose \right)} {\Cose \in \Coseconum{l} \colon } \right \} } 
			\left( \bigcup_{\underset { \alpha_{g^{-1 }} (y) \in \Ne_{\Lase_{\Cose}} (\Lase_{\Cose}) } {g \in \Lo^{\Binu} \colon }} \alpha_{g} \left( \left( \Ne_{\Lase_{\Cose}} \right)^{-1} \left( \alpha_{g^{-1 }} (y) \right) \cup \left\{ \alpha_{g^{-1 }} (y) \right\} \right) \right) \\
			\subseteq&\ \bigcup_{ \left \{ \underset {y \in \Ne_{\Cose} \left( \Cose \right)} {\Cose \in \Coseconum{l} \colon} \right \} } 
			\left( \bigcup_{\underset { \alpha_{g^{-1 }} (y) \in \Ne_{\Lase_{\Cose}} (\Lase_{\Cose}) } {g \in \Lo^{\Binu} \colon }} \Ause_{\Cose, g} \right) \; .
		\end{align*}
		Note that for any $g \in \Lo^{\Binu}$ we have $y \in \Ause_{\Cose, g}$. It then follows from our choice of $\Auseco$ that $\{y\} \cup \left( \bigcup_{l = 0}^d  \left( \Ne^{(l)} \right)^{-1}(y) \right) \subseteq \Thse$ for some $\Thse \in \Thseco$, as desired.

		\item \label{item:prop:ltc-dim:proof:Bo} Condition \Refc[\LTCextra]{Bo}{\Bo} follows from 
		our construction 
		in Step~\eqref{item:prop:ltc-dim:proof:F-choose}, 
		condition~\Refc[\LTCextra]{Bo}{ \Lo^{\Binu} \Lo^{3 \Binu \Binu'} \Lo^{\Binu}} for $\widetilde{\Laseco}$ and $\left( \Ne_{{\Lase}} \right)_{\Lase \in \Laseco}$ in Step~\eqref{item:prop:ltc-dim:proof:extend}, and the definition of $\Bo$ in Step~\eqref{item:prop:ltc-dim:proof:prep-M}. 
		
		\item \label{item:prop:ltc-dim:proof:Pr} We prove condition~\Refcd[\LTCprop]{Pr}. Since we showed in Step~\eqref{item:prop:ltc-dim:proof:flatten} that  $\Lapotil_{\Lase}$ is compactly supported for any $\Lase \in \Laseco$, it follows that $\Lapotil^{-1} \left( \Rplusfuncs{\Laseco} \setminus \{0\} \right)$ is precompact. For any $l \in \intervalofintegers{0}{d}$, using Steps~\eqref{item:prop:ltc-dim:proof:U-F} and~\eqref{item:prop:ltc-dim:proof:U-l} as well as Lemma~\ref{lem:simplicial-barycentric}\eqref{lem:simplicial-barycentric::partition} and~\eqref{lem:simplicial-barycentric::cover}, we have
		\[
			\bigcup \Coseco = \bigcup_{l=0}^{\infty} \bigcup_{F \in \powerset^{(l+1)}(\Laseco)} \Lapotil^{-1} \left( \Rplusfuncs[{F,\frac{2}{\Binu}}]{\Laseco} \right) = \Lapotil^{-1} \left( \Rplusfuncs{\Laseco} \setminus \{0\} \right) \; .
		\]
		It follows that each $\Cose \in \Coseco$ is precompact. 
		
		\item \label{item:prop:ltc-dim:proof:Fi} 
		Condition~\Refcd[\LTCprop]{Fi} follows directly from the definition of $\Coseco$ in Step~\eqref{item:prop:ltc-dim:proof:U-l} and the finiteness of $\Laseco$ in Step~\eqref{item:prop:ltc-dim:proof:ltc-Coseco}. 
		
		\item \label{item:prop:ltc-dim:proof:f} Using the notations of Definition~\ref{def:simplicial-diff-op},  we define, for any nonnegative integer $l$, continuous functions
		\[
			\Lapotil^{(l)} \colon X \to [0, \infty) \, \text{ given by } \, \Lapotil^{(l)} ( x ) = \max \left\{ 0, \Dinum[\Laseco]{l+1} \circ \Lapotil (x) - \frac{5}{\Binu} \right\}
		\]
		and 
		\[	
			\Po^{(l)}  \colon X \to [0, 1] \, \text{ given by } \, \Po^{(l)} ( x ) =  \frac{\Lapotil^{(l)} (x)}{ \max \left\{ \frac{1}{2(d+1)}, \displaystyle \sum_{k=0}^{d} \Lapotil^{(k)} (x) \right\} } \; . 
		\]
		 
		\item \label{item:prop:ltc-dim:proof:Su} We prove condition~\Refcd[\LTCprop]{Su}. 
		For any $l \in \intervalofintegers{0}{d}$, it follows from 
		the continuity of the maps $\Dinum[\Laseco]{l+1}$ and $\Lapotil$ as well as our constructions (see the detailed references below) that 
		\begin{align*}
			&\ \supp \left( \Po^{(l)} \right) \\
			{=} &\ \supp \left( \Lapotil^{(l)} \right) && \text{\footnotesize (by Step~\eqref{item:prop:ltc-dim:proof:f})} \\
			= &\ \overline{ \left\{ x \in X \colon \Dinum[\Laseco]{l+1} \circ \Lapotil (x) > \frac{5}{\Binu}  \right\} } \\
			\subseteq  &\ \left\{ x \in X \colon \Dinum[\Laseco]{l+1} \circ \Lapotil (x) \geq \frac{5}{\Binu}  \right\} \\
			\subseteq  &\ \left\{ x \in X \colon \Dinum[\Laseco]{l+1} \circ \Lapotil (x) > \frac{4}{\Binu}  \right\} \\
			=  &\ \left\{ x \in X \colon \Lapotil (x) \in \left( \Rplusfuncs[{\left( l,\frac{4}{\Binu} \right)}]{\Laseco} \right)  \right\} && \text{\footnotesize (by Definition~\ref{def:simplicial-barycentric})} \\
			=  &\ \bigsqcup_{F \in \powerset^{(l+1)}(S)} \left\{ x \in X \colon \Lapotil (x) \in \left( \Rplusfuncs[{F,\frac{4}{\Binu}}]{\Laseco} \right)  \right\} && \text{\footnotesize (by Lemma~\ref{lem:simplicial-barycentric}\eqref{lem:simplicial-barycentric::partition})} \\
			=  &\ \bigsqcup_{F \in \powerset^{(l+1)}(S)} \Cose^{F, \frac{4}{\Binu}} && \text{\footnotesize (by Step~\eqref{item:prop:ltc-dim:proof:U-F})} \; . 
		\end{align*} 
		It follows that for any $x \in \supp \left( \Po^{(l)} \right)$, there is $F \in \powerset^{(l+1)}(S)$ such that $x \in \Cose^{F, \frac{4}{\Binu}}$ and thus by Step~\eqref{item:prop:ltc-dim:proof:U-F}, we have 
		\[
			\alpha_{\Lo} (x) \subseteq \alpha^\cup_{\Lo} \left( \Cose^{F, \frac{4}{\Binu}} \right) \subseteq \Cose^{F, \frac{2}{\Binu}} \in \Coseco^{(l)} \; .
		\]
		as desired. 
		
		\item \label{item:prop:ltc-dim:proof:In} We prove condition~\Refcd[\LTCprop]{In}. By Step~\eqref{item:prop:ltc-dim:proof:nerve},  $\Lapotil$ is $\left(\Lo, \frac{1}{\Binu}\right)$-invariant, while by Lemma~\ref{lem:simplicial-diff-op}\eqref{lem:simplicial-diff-op::Lipschitz}, $\Dinum[\Laseco]{l+1}$ is $2$-Lipschitz for $l \in \N$. It follows from Remark~\ref{rmk:Lipschitz-alpha-composition} that $\Lapotil^{(l)}$ is $\left(\Lo, \frac{2}{\Binu}\right)$-invariant for $l \in \N$, and thus the function $\max \left\{ \frac{1}{2(d+1)}, \displaystyle \sum_{k=0}^{d} \Lapotil^{(k)} \right\}$ is $\left(\Lo, \frac{2(d+1)}{\Binu}\right)$-invariant by Lemma~\ref{lem:Lipschitz-alpha-arithmetics}\eqref{item:lem:Lipschitz-alpha-arithmetics:addition} and~\eqref{item:lem:Lipschitz-alpha-arithmetics:max}. Since 
		\[
		\left| \Lapotil^{(l)} (x) \right| \leq \left| \Dinum[\Laseco]{l+1} \circ  \Lapotil (x) \right| \leq \left| \Ma[\Laseco]{l+1} \circ  \Lapotil (x) \right|  \leq \left\| \Lapotil (x) \right\| \leq 1
		\]
		 for any $x \in X$ by Lemma~\ref{lem:simplicial-sorting} and Step~\eqref{item:prop:ltc-dim:proof:nerve}, it follows from Lemma~\ref{lem:Lipschitz-alpha-arithmetics}\eqref{item:lem:Lipschitz-alpha-arithmetics:division} that $\Po^{(l)}$ is $\left(\Lo, \Er \right)$-invariant, because 
		\[
			(2(d+1))^2 \cdot \left( \frac{2(d+1)}{\Binu} + \frac{1}{2(d+1)} \cdot \frac{2}{\Binu} \right)=  \frac{4(d+1) \left( 2 (d+1)^2 + 1 \right) }{\Binu}  \leq \Er
		\]
		by Step~\eqref{item:prop:ltc-dim:proof:prep-M}. 
		
		\item \label{item:prop:ltc-dim:proof:Un} We prove condition~\Refcd[\LTCprop]{Un}. It follows from the definition in Step~\eqref{item:prop:ltc-dim:proof:f} that $\Po^{(0)} (x) + \ldots + \Po^{(d)} (x) \leq 1$ for any $x \in X$. On the other hand, for any $x \in \Ko$, since we showed $\left| \supp \left( {\Lapotil} (x)  \right) \right| \leq d+1$ and $\left\| {\Lapotil} (x) \right\| \geq \frac{1}{d+1}$ in Step~\eqref{item:prop:ltc-dim:proof:nerve}, thus by Lemma~\eqref{lem:simplicial-sorting} and Lemma~\ref{lem:simplicial-diff-op}\eqref{lem:simplicial-diff-op::mu-delta}, we have
		\begin{align*}
			\sum_{k=0}^{d} \Lapotil^{(k)} (x) \geq &\  \sum_{k=0}^{d} \left( \Dinum[\Laseco]{l+1} \circ \Lapotil (x) - \frac{5}{\Binu}  \right) \\
			=&\ \Ma[\Laseco]{1} \circ \Lapotil (x) - \frac{5 (d+1)}{\Binu} \\
			=&\ \left\| {\Lapotil} (x) \right\|  - \frac{5 (d+1)}{\Binu} \\
			\geq&\  \frac{1}{d+1} - \frac{5 (d+1)}{\Binu} \\
			>&\  \frac{1}{2(d+1)} 
		\end{align*}
		by Step~\eqref{item:prop:ltc-dim:proof:prep-M}, whence 
		\[
			\sum_{l=0}^{d} \Po^{(l)} (x) = \frac{\displaystyle \sum_{l=0}^{d} \Lapotil^{(l)} (x) }{ \max \left\{ \frac{1}{2(d+1)}, \displaystyle \sum_{k=0}^{d} \Lapotil^{(k)} (x) \right\} } = \frac{\displaystyle \sum_{l=0}^{d} \Lapotil^{(l)} (x) }{ \displaystyle \sum_{k=0}^{d} \Lapotil^{(k)} (x) } = 1 \; .
		\]
	\end{enumerate}
	Therefore, $\Coseco$, $\Ne^{(l)}$ and $\Po^{(l)}$, for $l \in \{0, \ldots, d\}$, as defined in Steps~\eqref{item:prop:ltc-dim:proof:U-l} and~\eqref{item:prop:ltc-dim:proof:f}, satisfy all the desired conditions, following Steps~\eqref{item:prop:ltc-dim:proof:Lo}, 
	\eqref{item:prop:ltc-dim:proof:Eq}, \eqref{item:prop:ltc-dim:proof:Th}, \eqref{item:prop:ltc-dim:proof:Bo}, \eqref{item:prop:ltc-dim:proof:Pr}, \eqref{item:prop:ltc-dim:proof:Fi}, \eqref{item:prop:ltc-dim:proof:Su}, \eqref{item:prop:ltc-dim:proof:In}, and~\eqref{item:prop:ltc-dim:proof:Un}. 
\end{proof}

}

We conclude the section by showing that the long thin covering dimension is an invariant of continuous orbit equivalence, which we define presently. This fact is not used in the paper, but is recorded as it may be of independent interest for future work. Recall from \cite{Li2018} that two topological dynamical systems $\alpha \colon G \curvearrowright X$ and $\beta \colon H \curvearrowright Y$ are said to be \emph{continuously orbit equivalent} if there exist a homeomorphism $\varphi \colon X \to Y$ and continuous maps $a \colon G \times X \to H$ and $b \colon H \times Y \to G$ such that $\varphi (\alpha_g (x)) = \beta_{a(g,x)} (\varphi (x))$ for any $g \in G$ and $x \in X$ and $\varphi^{-1} (\beta_h (y)) = \alpha_{b(h,y)} \left(\varphi^{-1} (y) \right)$ for any $h \in H$ and $y \in Y$. 

\begin{Prop} \label{prop:ltc-dim-coe}
	If two topological dynamical systems $\alpha \colon G \curvearrowright X$ and $\beta \colon H \curvearrowright Y$ are continuously orbit equivalent, then $\dimltc(\alpha) = \dimltc (\beta)$. 
\end{Prop}

\begin{proof}
	Let $(\varphi, a, b)$ be a continuous orbit equivalence between $(X, G, \alpha)$ and $(Y, H, \beta)$. 
	By pulling back $\beta$ via $\varphi$ to an action on $X$, we may assume without loss of generality that $X = Y$ and $\varphi$ is the identity map. 
	By symmetry, it suffices to show, for any integer $d$, that $\dimltc(\beta) \leq d$ implies $\dimltc (\alpha) \leq d$. To this end, in view of 
	Proposition~\ref{prop:ltc-dim-new}\eqref{item:prop:ltc-dim-new:Lominus}, 
	it suffices to fix a finite subset $\Lo \Subset G$, a compact subset $\Ko \Subset X$ and a collection $\Thseco$ of open subsets of $X$ covering $\Ko$ and find a natural number $\Binu$ that depends only on $\Lo$ and $\Ko$, as well as a collection $\Coseco$ of open sets in $X$, and locally constant functions $\Ne_{\Cose} \colon \Cose \to X$ for $\Cose \in \Coseco$ satisfying \Refcd[\LTCextra]{Muplus}, \Refcd[\LTCdef]{Eq}, \Refcd[\LTCdef]{Th}, \Refcd{Ca} with regard to $\alpha$, and such that $\Ko \subseteq \bigcup \Coseco$. 
	
	To do this, we define the finite subset $\widehat{\Lo} = a (\Lo, \Ko)$ in $H$. 
	For any $g \in G$ and $h \in H$, let $\Ause_{g,h}$ be the open set $\left\{ x \in X \colon a(g, x) = h \right\}$. Observe that for any $g \in \Lo$, the collection $\Auseco_g$ given by $\left\{ \Ause_{g,h} \colon h \in \widehat{\Lo} \right\}$ covers $\Ko$. We may thus define an open cover of $\Ko$ in $X$ by
	\[
	\Auseco := \bigwedge_{g \in \Lo} \Auseco_g = \left\{ \bigcap_{g \in \Lo} \Ause_{g,h_g} \colon (h_g)_{g \in \Lo} \in \widehat{\Lo} ^{\Lo}  \right\}  \; .
	\]
	Since we assumed $\dimltc(\beta) \leq d$, there is a natural number $\Binu$ that depends only on $\widehat{\Lo}$ and $\Ko$ (and thus ultimately only on $\Lo$ and $\Ko$), as well as an open cover $\Coseco$ of $\Ko$ in $X$, and locally constant functions $\Ne_{\Cose} \colon \Cose \to X$ for $\Cose \in \Coseco$ satisfying \Refc[\LTCextra]{Muplus}{d, \widehat{\Lo}}, \Refc[\LTCdef]{Eq}{\widehat{\Lo}}, \Refc[\LTCdef]{Th}{\Thseco \wedge \Auseco}, and~\Refcd{Ca} with regard to $\beta$. 
	It remains to verify they also satisfy \Refcd[\LTCextra]{Muplus}, \Refcd[\LTCdef]{Eq}, \Refcd[\LTCdef]{Th}, and~\Refcd{Ca} with regard to $\alpha$.  
	\begin{enumerate}
		\item The conditions \Refcd{Th} and \Refcd{Ca} are automatically satisfied since they do not explicitly involve the actions. 
		\item Since $\alpha_{\Lo} (x) = \beta_{a(\Lo, x)} (x) \subseteq \beta_{\widehat{\Lo}} (x)$ for any $x \in \Ko$, any $(\alpha, \Lo)$-close set in the sense of Definition~\ref{def:multiplicity-action} is also $(\beta, \widehat{\Lo})$-close. This implies that $\mult_{\alpha,\Lo} (\Coseco) \leq \mult_{\beta,\widehat{\Lo}} (\Coseco) \leq d+1$, thus verifying \Refcd[\LTCextra]{Muplus} with regard to $\alpha$. 
		\item To see each $\Ne_{\Cose}$ is $(\alpha, \Lo)$-equivariant, we observe that for any $x \in \Cose$ and any $g \in \Lo$ satisfying $\alpha_g (x) \in \Cose$, by \Refc{Th}{\Thseco \wedge \Auseco}, there is $(h_{g'})_{g' \in \Lo} \in (\widehat{\Lo})^{\Lo}$ such that $\left\{ \Ne_{\Cose}(x) \right\}\cup \{x\} \subseteq \bigcap_{g' \in \Lo} \Ause_{g',h_{g'}} \subseteq \Ause_{g,h_{g}}$, whence $ \beta_{h_g} (x) = \alpha_{g} (x) \in \Cose$ and thus by \Refc[\LTCdef]{Eq}{\widehat{\Lo}} with regard to $\beta$, we have
		\[
			\Ne_{\Cose}\left( \alpha_{g} (x) \right) =  \Ne_{\Cose}\left( \beta_{h_g} (x) \right) =  \beta_{h_g} \left( \Ne_{\Cose}(x) \right) =  \alpha_g \left( \Ne_{\Cose}(x) \right) \; .
		\] 
		This verifies \Refcd[\LTCdef]{Eq} with regard to $\alpha$.  
	\end{enumerate}
	This shows $\dimltc (\alpha) \leq d$, as desired. 
\end{proof}

This result suggests that at least for topologically free actions, the long thin covering dimension is a groupoid invariant. This will be discussed in future work. 

\begin{Rmk}
A close examination of the proof of Theorem~\ref{thm:dimnuc-main} reveals that 
the proof still works if one were to remove 
condition~\eqref{item:\LTCdef:Ca} from the definition of the LTC dimension.
We found it nevertheless more convenient to include it, as it somewhat 
streamlines the presentation and does not restrict the applicability of our 
results. In particular, its inclusion guarantees that the LTC dimension 
dominates the asymptotic dimension of the coarse orbit space, and it is needed 
for the permanence property to hold in Theorem~\ref{thm:relative-bound-ltc}.
\end{Rmk}

\section{A few simple bounds} 
\label{sec:misc} 
\renewcommand{\sectionlabel}{LTC}
\ref{sectionlabel=LTC}
In this section, we establish some relatively easy inequalities involving 
$\asdim(X, \Enseco_\alpha)$ and $\dimltc(\alpha)$. These inequalities only play 
marginal (and dispensable) roles in our main results, though they provide a 
better understanding of the relation and behavior of those quantities. A 
reader who wishes to take the shortest path to the main result may skip this 
section. Throughout the section, we let $G$ be a discrete group, and let 
$\alpha$ be an action of $G$ on a locally compact Hausdorff space $X$. 

\begin{Thm} \label{thm:dimltc-asdim}
	We have $\asdim(X, \Enseco_\alpha) \leq \dimltc(\alpha)$. 
\end{Thm}

\begin{proof}
	Let $d$ be a natural number such that $\dimltc(\alpha) \leq d$. It suffices to verify condition~\ref{item:prop:orbit-asdim:finitary} of Proposition~\ref{prop:orbit-asdim}. 
	To this end, given $\Lo \Subset G$ and $\Ko \Subset X$, we choose a natural 
	number $\Binu$ that satisfies the requirement in 
	Proposition~\ref{prop:ltc-dim-new}(\ref{item:prop:ltc-dim-new:Lominus}). 
	Now, for any finite subset $Y$ of $\Ko$, since $X$ is Hausdorff, there 
	exists a finite open cover $\Thseco = \left\{ \Thse_y \colon y \in Y 
	\right\}$ of $Y$ in $X$ such that $y \in \Thse_y$ and $\Thse_y \cap 
	\Thse_{y'} = \varnothing$ for any different $y, y' \in Y$. By our 
	assumption on $\Binu$ above, we may choose an open cover $\Coseco$ of $Y$ 
	in $X$ and locally constant functions $\Ne_{\Cose} \colon \Cose \to X$ for 
	$\Cose \in \Coseco$ satisfying \Refcd[\LTCextra]{Muplus}, 
	\Refcd[\LTCdef]{Eq}, \Refcd[\LTCdef]{Th}, and \Refcd[\LTCdef]{Ca}. Define 
	$\Auseco$ to be the restriction of $\Coseco$ to $Y$, i.e., $\Auseco = 
	\left\{ \Cose \cap Y \colon \Cose \in \Coseco \right\}$, which is a cover 
	of $Y$. Then we have $\mult_{\alpha, \Lo} (\Auseco) \leq \mult_{\alpha, 
	\Lo} (\Coseco) \leq d+1$, verifying \Refcd[\OCSfin]{Muplus}. It remains to 
	verify $\Auseco$ satisfies \Refcd[\OCSfin]{Ca}. To do that, we observe that 
	by \Refcd[\LTCdef]{Th} and the fact that $\Thseco$ consist of disjoint sets 
	containing one point in $Y$ each, the maps $\Ne_{\Cose}$ are injective when 
	restricted to $\Cose \cap Y$, whence for any $\Cose \in \Coseco$, we have 
	$\left| \Cose \cap Y \right| = \left| \Ne_{\Cose} \left(\Cose \cap Y\right) 
	\right| \leq \left| \Ne_{\Cose} \left(\Cose\right) \right| \leq \Binu$ by 
	\Refcd[\LTCdef]{Ca}. 
	
\end{proof}

Next we give a lower bound of $\asdim(X, \Enseco_\alpha)$ in terms of the 
asymptotic dimension of the group $G$, under the assumption that $X$ is compact 
and the action has a free orbit. 
Recall that any discrete group $G$ comes with a canonical coarse structure generated by entourages of the form 
\[
	\left\{ (g, g') \in G \times G \colon g' g^{-1} \in \Lo \right\}
\]
for all finite subsets $\Lo \Subset G$. The asymptotic dimension of $G$, 
denoted by $\asdim(G)$, refers to the asymptotic dimension of this canonical 
coarse structure. 

\begin{Prop} \label{prop:orbit-asdim-compact-free}
	Assume $X$ is compact and $\alpha$ has a free orbit, that is, there exists 
	$x \in X$ such that the stabilizer group $G_x$ is trivial.
	Then
	\[
	\asdim \left( X , \Enseco_\alpha \right) \geq \asdim(G) \; .
	\]
	If the action $\alpha$ is free then we have equality in the inequality.
\end{Prop}

\begin{proof}
	Recall (from after Definition \ref{def:OCS}) that the coarse connected components of $\Enseco_\alpha$ are exactly the orbits of $\alpha$. 
	Compactness of $X$ implies that $\Enseco_\alpha$ is generated by the controlled sets $\Ense_{X, \Lo}$ for finite subsets $\Lo$ of $G$ (see Definition~\ref{def:OCS}). Hence for any $x \in X$ with $G_x = \{e\}$, if we identify $G$ with $\alpha_{G} (x)$ using the orbit map $g \mapsto \alpha_g (x)$, then the canonical coarse structure on $G$ coincides with the restriction of $\Enseco_\alpha$ to $\alpha_{G} (x)$. 
	Hence the first claim follows from the immediate fact that the asymptotic 
	dimension of a coarse component is bounded from above by that of the entire 
	space. The second claim follows from the observation that if every orbit is 
	free, then the coarse space $(X, \Enseco_\alpha)$ is the disjoint union of 
	copies of $G$ with the canonical coarse structure, and thus the asymptotic 
	dimension of $(X, \Enseco_\alpha)$ agrees with that of one of its 
	components, i.e., $\asdim (G)$. 
\end{proof}

\begin{Rmk}
	The assumption of a free orbit in Proposition~\ref{prop:orbit-asdim-compact-free} is there to exclude degenerate cases such as the trivial action, where clearly we have $\asdim \left( X , \Enseco_\alpha \right) = 0$ regardless of the group. 
	
	There is of course a lot of room between having a free orbit and being a trivial action. With a bit more effort, one may relax the requirement of the existence of a trivial stabilizer to that of a sequence of stabilizers that are normal and have trivial intersection (e.g., by following the strategy of \cite[Corollary~3.14]{SWZ}). 
\end{Rmk}

\renewcommand{\sectionlabel}{OCS}
\ref{sectionlabel=OCS}
The assumption of compactness of $X$ is also crucial in 
Proposition~\ref{prop:orbit-asdim-compact-free}, as illustrated by the 
following case: Recall that the action $\alpha$ is \emph{proper} if for any 
compact subset $\Ko \subseteq X$, the set $\left\{ g \in G \colon \alpha_g 
(\Ko) \cap \Ko \not= \varnothing \right\}$ is finite (or equivalently, for any 
precompact subsets $Y, Z \subseteq X$, the set $\left\{ g \in G \colon \alpha_g 
(Y) \cap Z \not= \varnothing \right\}$ is finite). In this case, the quotient 
space $X / G$ is again locally compact and Hausdorff.

\begin{Prop} \label{prop:orbit-asdim-proper}
	If $\alpha$ is proper,  
	then $\asdim \left( X , \Enseco_\alpha \right) = 0$. 
\end{Prop}

\begin{proof}
	It suffices to verify condition~\eqref{item:prop:orbit-asdim:mult} of Lemma~\ref{prop:orbit-asdim} for $d = 0$. To this end, given a finite subset $\Lo \Subset G$ and a compact subset $\Ko \Subset X$, we define
	\[
	\Coseco = \left\{ \alpha^\cup_{\Lo} (\Ko) \cap \alpha_{G}(x) \subseteq X \colon x \in X \right\}
	\]
	and 
	\[
	\Bo = \left\{ g \in G \colon \alpha_g \left( \alpha^\cup_{\Lo} (\Ko) \right) \cap \alpha^\cup_{\Lo} (\Ko) \not= \varnothing \right\}
	\]
	Since $\alpha^\cup_{\Lo} (\Ko)$ is compact and $\alpha$ is proper, the set 
	$\Bo$ is finite. It is clear that 
	$\Coseco$ and $\Bo$ satisfy \eqref{item:prop:orbit-asdim:mult:Lo}, 
	\eqref{item:prop:orbit-asdim:mult:Bo} 
	and~\eqref{item:prop:orbit-asdim:mult:Mu} of 
	Lemma~\ref{prop:orbit-asdim}\eqref{item:prop:orbit-asdim:mult}. 
\end{proof}

\begin{Rmk} \label{rmk:OCS-why-compact}
	Proposition~\ref{prop:orbit-asdim-proper} may be seen as the motivation to involve compact sets in the definition of the coarse orbit structure (Definition~\ref{def:OCS}). Without them, the proof of Proposition~\ref{prop:orbit-asdim-proper} falls apart. 
\end{Rmk}

\renewcommand{\sectionlabel}{LTC}
\ref{sectionlabel=LTC}
In fact, for proper actions, we can also determine $\dimltc(\alpha)$. Let us write $\pi \colon X \to X / G$ for the quotient map and recall that proper actions by discrete groups have the following slice property: for any $x \in X$, 
there exists a precompact neighborhood $\Ause$ of $x$ such that $\Ause$ is 
invariant under the stabilizer group $G_x \leq G$ and $\alpha_g (\Ause) \cap 
\Ause = \varnothing$ for any $g \in G \setminus G_x$.
(To see this, first choose an open neighborhood $U$ of $x$ such that  $\alpha_g 
(U) \cap U = \varnothing$ for any $g \in G \setminus G_x$; then set $\Ause = 
\alpha^\cap_{G_x} ( U )$.)
\begin{Rmk}\label{rmk:proper-dim-quotient}
	If the action of $G$ on $X$ is proper, then we have $\dim((X/G)^+) \leq 
	\dim(X^+)$. To see this, we observe that by the above slice property, for 
	any compact subset $\Ko \Subset X / G$, there exists a compact subset $\Ko' 
	\Subset X$ such that $\pi(\Ko') \supset \Ko$. Since $\pi|_{\Ko'} \colon 
	\Ko' \to \pi(\Ko')$ is a finite-to-one open surjection, by \cite[Ch.~9, 
	2.16]{Pears75}, we have $\dim(\Ko') = \dim(\pi(\Ko'))$. It then follows 
	from Lemma~\ref{lem:dimc-compactification} that $\dim((X/G)^+) = \sup_{\Ko 
	\Subset X/G} \dim(\Ko) \leq \sup_{\Ko' \Subset X} \dim(\Ko) = \dim(X^+)$, 
	as desired. 
\end{Rmk}

\begin{Prop} \label{prop:ltc-dim-proper}
	If $\alpha$ is proper, then 
	\[
		\dimltc(\alpha) = \dim\left((X/G)^+\right) = \dim\left(X^+\right) \; .
	\]
\end{Prop}

\begin{proof}
	By Remark~\ref{rmk:ltc-dim-trivial-group}, we have $\dimltc(\alpha) \geq \dim(X^+)$. It remains to show $\dimltc(\alpha) \leq \dim((X/G)^+) \leq \dim(X^+)$. To prove $\dimltc(\alpha) \leq \dim((X/G)^+)$, let $d$ be a nonnegative integer such that $\dim((X/G)^+) \leq d$. We want to show that $\dimltc(\alpha) \leq d$. Fix a compact subset $K \subseteq X$.
	Because $X$ is locally compact, there exists an open precompact subset 
	$\Aulase$ that contains $\Ko$. 
	Define
	\[
		\Bo := \left\{ g \in G \colon \alpha_g \left( \Aulase \right) \cap \left( \Aulase \right) \not= \varnothing \right\} \; .
	\]
	The set $\Bo$ is finite,  symmetric and contains $e$. Fix a given finite 
	open cover $\Thseco$ of $X$. Define $\Thseco_1 := 
	\alpha^\wedge_{\Bo}(\Thseco \cup \{ \Aulase, X \setminus \Ko \} )$ (see 
	Notation~\ref{notation:group-action}).  
	For any $x \in X$, we choose a precompact neighborhood $\Ause_x$ of $x$ which is invariant with regard to the stabilizer group $G_x \leq G$ and satisfies $\alpha_g \left( \Ause_x \right) \cap \Ause_x = \varnothing$ for any $g \in G \setminus G_x$. 
	These conditions guarantee the existence of a locally constant $G$-equivariant map $\Ne_{x} \colon \alpha^\cup_{G} \left(\Ause_x\right) \to X$ sending $\alpha_{g} \left(\Ause_x\right)$ to $\alpha_{g} (x)$ for any $g \in G$. 
	Furthermore, by replacing $\Ause_x$ with $\Ause_x \cap \alpha^\cap_{G_x}  \left( \bigcap \left\{ \Lase \in \Thseco_1 
	\colon x \in \Lase \right\} \right)$ if necessary, we may additionally 
	assume that for any $g \in \Bo$ and for any $\Thse \in \Thseco$, we have 
	\begin{enumerate}
		\item \label{prop:ltc-dim-proper:proof::Thse} $\alpha_g \left( \Ause_x \right) \subseteq \Thse$ whenever $\alpha_g (x) \in \Thse$, 
		\item \label{prop:ltc-dim-proper:proof::Aulase} $\alpha_g \left( \Ause_x \right) \subseteq \Aulase$ whenever $\alpha_g (x) \in \Aulase$, and 
		\item \label{prop:ltc-dim-proper:proof::Ko} $\alpha_g \left( \Ause_x \right) \cap \Ko = \varnothing$ whenever $\alpha_g (x) \notin \Ko$. 
	\end{enumerate}
	
	By compactness, there exists a finite subset $\Fi \subseteq \Ko$ such that 
	$\left\{ \Ause_{x} \colon x \in \Fi \right\}$ covers $\Ko$. 
	Since $\dim((X/G)^+) \leq d$, by Lemma~\ref{lem:dimc-compactification}, 
	there exists a finite open cover $\Laseco$ of $\pi(\Ko)$ in $X / G$ that 
	refines $\left\{ \pi\left(\Ause_{x}\right) \colon x \in \Fi \right\} \cup 
	\{ (X / G) \setminus \pi(\Ko) \}$ and satisfies $\mult(\Laseco) \leq d + 
	1$. We may also assume without loss of generality that each open set in 
	$\Laseco$ intersects $\pi(\Ko)$ nontrivially, and thus $\Laseco$ refines 
	$\left\{ \pi\left(\Ause_{x}\right) \colon x \in \Fi \right\}$. For any 
	$\Lase \in \Laseco$, choose $x_{\Lase} \in \Fi$ such that $\Lase \subseteq 
	\pi\left(\Ause_{x_{\Lase}}\right)$. Write $\Ause_{\Lase}$ in place of 
	$\Ause_{x_{\Lase}}$, define 
	\[
	\Bo_{\Lase} := \left\{ g \in G \colon \alpha_g 
	\left( x_{\Lase} \right) \in \Aulase \right\}
	\; \text { and } \; \Cose_{\Lase} := \pi^{-1}(\Lase) \cap 
	\alpha^\cup_{\Bo_{\Lase}} \left( \Ause_{\Lase} \right) \subseteq X
	\, ,
	\]
	and let $\Ne_{\Lase} \colon \Cose_{\Lase} \to X$ be the restriction of the 
	locally constant map $\Ne_{x_{\Lase}}$. Define $\Coseco := \left\{ 
	\Cose_{\Lase} \colon \Lase \in \Laseco \right\}$. Set $\Binu = | \Bo |$. We 
	claim that $\Coseco$ 
	and $\left( \Ne_{\Lase} \right)_{\Lase \in \Laseco}$ satisfy 
	\Refc[\LTCdef]{Lo}{\{e\}, \Ko}, \Refcd[\LTCextra]{Muplus}, 
	\Refcd[\LTCdef]{Eq}, \Refcd[\LTCdef]{Th}, and \Refcd[\LTCdef]{Ca}, 
	whence $\dimltc(\alpha) \leq d$ by 
	 Proposition~\ref{prop:ltc-dim-new}(\ref{item:prop:ltc-dim-new:Lominus}). 
	 To 
	prove the claim, we first observe that $\Ne_{\Lase} \left( \Cose_{\Lase} 
	\right) \subseteq \alpha^\cup_{\Bo_{\Lase}} \left( x_{\Lase} \right) 
	\subseteq \Aulase$ for any $\Lase \in \Laseco$. 
	\begin{itemize}
		\item To show \Refc[\LTCdef]{Lo}{\{e\}, \Ko}, that is, that $\Coseco$ 
		covers $\Ko$, we first observe that $\pi^{-1} (\Laseco)$ covers $\Ko$. 
		Hence, for any $x \in \Ko$, there exists $\Lase \in \Laseco$ such that 
		$x 
		\in \pi^{-1} (\Lase) \subseteq \alpha^\cup_G \left( \Ause_{\Lase} 
		\right)$, since $\Lase \subseteq \pi \left( \Ause_{\Lase} \right)$. It 
		follows that there exists $g \in G$ such that $x \in \alpha_g \left( 
		\Ause_{\Lase} \right)$. In view of the definition of $\Cose_{\Lase}$, 
		it suffices to show that $g \in \Bo_{\Lase}$. Indeed, if $g \notin 
		\Bo_{\Lase}$, that is, $\alpha_g \left( x_{\Lase} 
		\right) \notin \Aulase$, then it follows from 
		\eqref{prop:ltc-dim-proper:proof::Ko} that $\alpha_g \left( 
		\Ause_{\Lase} \right) \cap \Ko = \varnothing$, which contradicts the 
		assumptions that $x \in \Ko$ and $x \in \alpha_g \left( \Ause_{\Lase} 
		\right)$. 
		
		\item To show \Refcd[\LTCextra]{Muplus} for $\Coseco$, we see that $\mult_{\alpha, \Lo}\left( \Coseco \right) \leq \mult_{\alpha, \Lo}\left( \pi^{-1} (\Laseco) \right)$ since $\Cose_{\Lase} \subseteq \pi^{-1}(\Lase)$ for each $\Lase \in \Laseco$, while $\mult_{\alpha, \Lo}\left( \pi^{-1} (\Laseco) \right) = \mult \left( \pi^{-1} (\Laseco) \right) = \mult \left( \Laseco \right)$ since $\pi^{-1} (\Laseco)$ is a cover consisting of $G$-invariant sets, whence $\mult_{\alpha, \Lo}\left( \Coseco \right) \leq d+1$. 
		
		\item To show \Refcd[\LTCdef]{Eq}, it suffices to see that the $G$-equivariance of $\Ne_{x}$ for each $x \in X$ implies that $\Ne_{\Lase}$ is $\Lo$-equivariant for any $\Lase \in \Laseco$. 
		
		\item It is immediate from the definition that $\left | \Ne_{\Lase} 
		\left( \Cose_{\Lase} \right) \right | \leq | \Bo | = \Binu$, that is, 
		we have  \Refcd[\LTCdef]{Ca}.

		\item To show \Refcd[\LTCdef]{Th}, we observe that for any $\Lase \in 
		\Laseco$ and for any $y \in \Ne_{\Lase} \left( \Cose_{\Lase} \right) 
		\subseteq \alpha^\cup_{\Bo_{\Lase}} \left( x_{\Lase} \right)$, we have 
		$y = \alpha_{g} \left( x_{\Lase} \right) \in \Thse$ for some $g \in 
		\Bo_{\Lase}$ and $\Thse \in \Thseco$. Thus, we have  
		\[
		\Ne_{\Lase}^{-1} (y) \cup \{y\} \subseteq \Ne_{x_{\Lase}}^{-1} \left( 
		\alpha_{g} \left( x_{\Lase} \right) \right) \cup \left\{ \alpha_{g} 
		\left( x_{\Lase} \right) \right\} = \alpha_{g} \left( \Ause_{\Lase} 
		\right)
		\, ,
		\]
		 and  $ \alpha_{g} \left( \Ause_{\Lase} \right ) \subseteq \Thse$ by 
		 \eqref{prop:ltc-dim-proper:proof::Aulase} above. 
	\end{itemize}
	This proves $\dimltc(\alpha) \leq \dim((X/G)^+)$. The fact that  
	$\dim((X/G)^+) \leq \dim(X^+)$ is shown in 
	Remark~\ref{rmk:proper-dim-quotient}.
\end{proof}

\begin{Cor} \label{cor:ltc-finite-group}
	If $G$ is a finite group, then we have $\asdim \left( X , \Enseco_\alpha \right) = 0$ and $\dimltc(\alpha) = \dim\left((X/G)^+\right) = \dim\left(X^+\right)$. 
\end{Cor}

\begin{proof}
	This follows from Propositions~\ref{prop:orbit-asdim-proper} 
	and~\ref{prop:ltc-dim-proper} together with the trivial fact that any 
	action by a finite group is proper. 
\end{proof}

\section{The large scale packing constants}\label{sec:LSP}
\renewcommand{\sectionlabel}{LSP}
\ref{sectionlabel=LSP}
In this section, we introduce the notion of \emph{large scale packing 
constants} of a discrete group. Those are group theoretic invariants. 
The main theorems of this section are Theorem \ref{thm:lsp-asdim}, in which 
bound the asymptotic dimension of the coarse orbit spaces in terms of the 
zeroth large scale packing constant, and Theorem \ref{thm:lsp-vnil}, in which 
we show that finitely generated virtually nilpotent groups have finite large 
scale packing constants. The large scale packing constants are used again in 
Section~\ref{sec:GP}, in order 
to obtain bounds on the long thin covering dimension. 

\begin{Def}	\label{def:LSP}
	Let $G$ a discrete group and let $d$ be a nonnegative integer. 
	The \emph{$d$-dimensional large scale packing constant} of $G$, denoted $\LSP_d(G)$, is the infimum of all nonnegative integers $m$ satisfying the following:
	
	For any finite subset $\Lo \Subset G$, there exists a symmetric finite 
	subset $\Bo \Subset G$ that contains $e$ (i.e., $e \in \Bo^{-1} = \Bo$), 
	which satisfies the following condition: 
	{
		\begin{enumerate}
			\renewcommand{\labelenumi}{\textup{(\theenumi)}}
			\renewcommand{\theenumi}{Pa}\item \label{item:def:LSP:Pa}
			For any $n>m$ and for any $g_1,\ldots g_n \in \Bo \Lo$, there 
			exists a 
			subset $\Fi \subseteq \intervalofintegers{1}{n}$ such that $|\Fi| > 
			d+1$ and 
			$g_i^{} g_j^{-1} \in \Bo$ for any $i,j \in \Fi$. 
		\end{enumerate}
	}
\end{Def}

\begin{Rmk}	\label{rmk:LSP-reformulate}
	We may reformulate condition~\eqref{item:def:LSP:Pa} as follows: 
	
	For any $g_1,\ldots g_n \in \Bo \Lo$, if there is no subset $\Fi \subseteq \intervalofintegers{1}{n}$ such that $|\Fi| > d+1$ and 
	$g_i^{} g_j^{-1} \in \Bo$ for any $i,j \in \Fi$, 
	then $n \leq m$. 
\end{Rmk}

\begin{Rmk}	\label{rmk:LSP}
	The following inequalities are immediate:
	\[
		0 \leq \LSP_0(G) \leq  \LSP_1(G) \leq \LSP_2(G) \leq \ldots 
	\]
	and for all $d$ we have
	\[
		\LSP_d(G) \geq  d+1 
	\]
	 The latter inequality becomes an equality for any $d$ when $G$ is a finite 
	 group, since we may take $\Bo = G$ in Definition~\ref{def:LSP}. 
\end{Rmk}

In the above, the label \eqref{item:def:LSP:Pa} stands for ``packing''. 
Intuitively speaking, for a finitely generated group $G$ and for $d=0$, if we 
take both $\Lo$ and $\Bo$ to be balls on a Cayley graph of $G$ centered at the 
identity, say of radii $r$ and $R$, then the requirement in 
Definition~\ref{def:LSP} amounts to finding the maximal number of elements we 
can pick from a ball of radius $R+r$ such that the distance between any two 
elements are more than $R$ (or equivalently, the smallest number which 
guarantees that any list of elements in the a ball of radius $R+r$ of that 
given length would 
have at least two elements whose distance is at most $R$; for larger $d$, the 
requirement would be that there are at least $d+1$ elements such that 
the distance between any two of them is at most $R$). This is essentially 
equivalent, for $d=0$, to finding the maximal number of $\frac{R}{2}$-balls one 
can pack into a ball of radius $\left( \frac{3}{2}R+r \right)$ without 
overlapping.

The next theorem shows that the asymptotic dimension of the coarse orbit space 
is bounded by $\LSP_0(G)$. The need for $\LSP_d(G)$ for higher values of $d$ 
will arise 
in Section~\ref{sec:GP}, when we obtain a bound for $\dimltc(\alpha)$.

\begin{Thm} \label{thm:lsp-asdim}
	Let $\alpha$ be an action of a countable discrete group $G$ on a locally 
	compact space $X$. Then we have 
	\[
	\asdim \left(X, \Enseco_\alpha \right) \leq \LSP_0(G) - 1 \; .
	\]
\end{Thm}

\begin{proof}
	Let $\Dimnu$ be a natural number such that $\LSP_0(G) \leq \Dimnu + 1$. We 
	show that condition~\eqref{item:prop:orbit-asdim:mult} of 
	Proposition~\ref{prop:orbit-asdim} is satisfied. 
	To this end, fix an arbitrary finite subset $\Lo \Subset G$ and an 
	arbitrary compact subset $\Ko[OCS] \Subset X$. By replacing $\Lo$ with $\Lo 
	\cup \Lo^{-1}$ if necessary, we may assume without loss of generality that 
	$\Lo = \Lo^{-1}$. Then we use the assumption $\LSP_0(G) \leq \Dimnu + 1$ to 
	obtain a finite symmetric subset $\Bo \Subset G$ containing $e$ and 
	satisfying condition~\eqref{item:def:LSP:Pa} of Definition~\ref{def:LSP}. 
	
	We say a subset $\Ause$ of $\Ko[OCS]$ is \emph{$\Bo$-discrete} if for any 
	$x \in \Ause$, we have $\alpha_{\Bo} (x) \cap \Ause \subseteq \{ x \}$. By 
	Zorn's lemma, there is a maximal set $\Ause$ in the collection of all 
	{$\Bo$-discrete} subsets of $\Ko[OCS]$. Notice that for any $x \in 
	\Ko[OCS] \setminus \Ause$, we have $\alpha_{\Bo} (x) \cap \Ause \not= 
	\varnothing$, since otherwise $\Ause \cup \{ x \}$ would be a larger 
	$\Bo$-discrete set.
	
	We define $\Bo[OCS] = (\Lo \Bo)^{-1} (\Lo \Bo)$ and $\Coseco[OCS] = \left\{ 
	\alpha_{\Lo \Bo} (y) \colon y \in \Ause \right\}$. We claim that they 
	satisfy 
	conditions~\eqref{item:prop:orbit-asdim:mult:Lo}, 
	\eqref{item:prop:orbit-asdim:mult:Bo} 
	and~\eqref{item:prop:orbit-asdim:mult:Mu} of 
	Proposition~\ref{prop:orbit-asdim}\eqref{item:prop:orbit-asdim:mult}. 
	Indeed, condition~\eqref{item:prop:orbit-asdim:mult:Bo} is immediate from 
	our construction. To prove condition~\eqref{item:prop:orbit-asdim:mult:Lo}, 
	for any $x \in \Ko[OCS]$, we know that $\alpha_{\Bo} (x) 
	\cap \Ause \not= \varnothing$, which implies that $x \in \alpha_{\Bo} (y)$ 
	for some $y \in \Ause$, and thus $\alpha_{\Lo} (x) \subseteq \alpha_{\Lo 
		\Bo} (y)$. Finally, to prove 
	condition~\eqref{item:prop:orbit-asdim:mult:Mu}, we show, for any $x \in 
	X$, the set $\Aufi = \left\{ y \in \Ause \colon x \in \alpha_{\Lo \Bo} (y) 
	\right\}$ has cardinality no more than $\Dimnu + 1$. To this end, we 
	choose, for any $y \in \Aufi$, a group element $g_y \in \Bo \Lo = (\Lo 
	\Bo)^{-1}$ such that $y = \alpha_{g_y} (x)$. Note that any such  
	selection 
	$y \mapsto g_y$ is injective. Thus it suffices to show the set ${\Aufi}' = 
	\left\{ g_y \in \Bo \Lo \colon y \in \Aufi \right\}$ has cardinality no 
	more than $\Dimnu + 1$. Suppose this were not true. Then by 
	condition~\eqref{item:def:LSP:Pa} of Definition~\ref{def:LSP}, 
	there would be distinct $y, z \in D$ such that $g_y, g_z \in \Aufi'$ and 
	$g_y^{} g_z^{-1} \in \Bo$; this would imply $y 
	= \alpha_{g_y^{}} \left( \alpha_{g_z^{-1}} (z) \right) \in \alpha_{\Bo} 
	(z)$, which cannot happen because $\Ause$ is $\Bo$-discrete.
\end{proof}

 The next theorem shows that the large scale packing constants of finitely 
 generated virtually nilpotent groups are finite. Recall from the discussion after Definition~\ref{def:growth} that if $G$ is a finitely generated virtually nilpotent group and $S$ is a finite symmetric generating set, then the growth function associated with $S$ is equivalent to a polynomial of degree $d(G)$, which depends on the group but not on the set $S$.
\begin{Thm} \label{thm:lsp-vnil}
	Let $G$ be finitely generated virtually nilpotent group $G$. Set $k=d(G)$. Then for any non-negative integer $d$ we have
	\[
	\LSP_d(G) \leq 3^k (d+1) \, .
	\]
\end{Thm}

\begin{proof}
	Let us fix a symmetric generating set $S$ for $G$. 
	Recall that by \cite[Theorem 1.1]{Breuillard}, there exists a positive constant $C$ (depending on $S$) such that
	\[
	\frac{| B_S (e,n) |}{n^k} \xrightarrow[n \to \infty]{} C 
	\, .
	\]
	Given any finite subset $\Lo \Subset G$, choose a positive integer $r$ such that $\Lo \subseteq B_S(e,r)$ (see Definition~\ref{def:growth}). 
	Notice that, taking limits over even numbers $R$, we have
	\[
	\frac{\left | B_S (e,\frac{3}{2}R + r)  \right |}{ \left | B_S (e,\frac{1}{2}R) \right | } \xrightarrow[R \to \infty]{} 3^k
	\, .
	\]
	We may thus fix a positive even integer $R$ such that 
	\[
		 3^k + \frac{1}{d+1} > \frac{\left | B_S (e,\frac{3}{2}R + r)  \right |}{ \left | B_S (e,\frac{1}{2}R) \right | } \, .
	\]
	We rewrite the last equation as 
	\[
	\left(   3^k (d+1)  + 1 \right) \Growth{G,S} \left( \frac{R}{2} \right)  
	 > (d+1) \cdot \Growth{G,S} \left( \frac{3}{2}R+r \right) \, .
	\]
	Set $\Bo = B_S(e,R)$. Note that $\Bo \Lo \subseteq B_S(e,R + r)$. 
	Furthermore, $\Bo$ contains $e$ and is symmetric. To prove 
	condition~(\ref{item:def:LSP:Pa}) of Definition~\ref{def:LSP}, observe that 
	for any integer $n > 3^k (d+1)$ and for any $g_1,\ldots g_n \in 
	\Bo 
	\Lo$, we have 
	\[
	B_S\left(g_i , \frac{R}{2}\right) \subseteq B_S \left(e, 
	\frac{3}{2}R + r \right)
	\, .
	\]
	 Hence if we write $\chi_{B_S\left(g_i , 
	\frac{R}{2}\right)}$ for the characteristic function of $B_S\left(g_i , 
	\frac{R}{2}\right)$, then we have
	\begin{align*}
		\sum_{g \in B_S \left(e, \frac{3}{2}R + r \right)} \sum_{i = 1}^{n} \chi_{B_S\left(g_i , \frac{R}{2}\right)} (g) &= \sum_{i = 1}^{n} \sum_{g \in B_S \left(e, \frac{3}{2}R + r \right)} \chi_{B_S\left(g_i , \frac{R}{2}\right)} (g) \\
		& = n \cdot \Growth{G,S} \left( \frac{R}{2} \right) 	\\
		& \geq \left(   3^k (d+1)  + 1 \right) \Growth{G,S} \left( \frac{R}{2} \right)  \\
		& > (d+1) \cdot \Growth{G,S} \left( \frac{3}{2}R+r \right) \; .
	\end{align*}
	Thus, there exists $g_0 \in B_S \left(e, \frac{3}{2}R + r \right)$ such 
	that $\displaystyle \sum_{i = 1}^{n} \chi_{B_S\left(g_i , 
	\frac{R}{2}\right)} (g_0) > d+1$. Therefore, the set $F$ given by $F = 
	\left\{ i 
	\in \intervalofintegers{1}{n} \colon g_0 \in B_S \left(g_i , 
	\frac{R}{2}\right) \right\}$ has cardinality greater than $d+1$. Finally, 
	we see that for any $i, j \in F$, the distances from $g_0$ to $g_i$ and 
	$g_j$ are both no more than $\frac{R}{2}$, whence the distance between 
	$g_i$ and $g_j$ is no more than $R$, i.e., $g_i^{} g_j^{-1} \in \Bo$. 
\end{proof}

\begin{Question} \label{ques:lsp-vnil} We showed that finitely generated 
	virtually nilpotent groups have finite large scale packing constants. Does 
	the converse hold? 
\end{Question}

\section{General positions and bounding $\dimltc$}\label{sec:GP}
\renewcommand{\sectionlabel}{LTC}
\ref{sectionlabel=LTC}
This section is devoted to proving the following result, which gives a bound of the long thin covering dimension of an action $\alpha \colon G \curvearrowright X$ in terms of the covering dimension of $X$ and the large scale packing constants of $G$.

\begin{Thm} \label{thm:lsp-ltc}
	Let $\alpha$ be an action of a discrete group $G$ on a locally compact Hausdorff space $X$.  
	Assume $\dim (X^+) \leq d$ for some $d \in \N$. 
	Then we have 
	\[
	\dimltc (\alpha) \leq \LSP_{d}(G) - 1 \; .
	\]
\end{Thm}

To prove this result, we need a few lemmas which provide general position 
techniques. We will get to the proof of Theorem \ref{thm:lsp-ltc} at the end of 
this section.

\begin{Lemma} \label{lem:thin-neighborhood}
	Let $\alpha$ be an action of a discrete group $G$ on a Hausdorff space $X$. Let $\Lo$ be a finite subset of $G$ and let $x \in X$. Then there is an open neighborhood $\Ause$ of $x$ such that for any $g, g' \in \Lo$, we have 
	\[
		\alpha_g \left(\Ause\right) \cap \alpha_{g'} \left(\Ause\right) \not= \varnothing \quad \text{if and only if} \quad \alpha_g \left( x \right) = \alpha_{g'} \left( x \right)
	\]
\end{Lemma}

\begin{proof}
	Since $X$ is Hausdorff, there is a disjoint family $\left\{ \Ause_y \colon y \in \alpha^\cup_{\Lo} (x) \right\}$ of open subsets such that $y \in \Ause_y$ for any $y \in \alpha^\cup_{\Lo} (x)$. It is clear that the open set $\Ause$ defined by 
	\[
		\Ause = \bigcap_{g \in \Lo} \alpha_{g^{-1}} \left( \Ause_{\alpha_g(x)} \right)
	\]
	satisfies the requirement, since we have $\alpha_g \left( \Ause \right) \subseteq \Ause_{\alpha_g (x)}$ for any $g \in \Lo$. 
\end{proof}

\begin{Def}[{\cite[Section 3]{Kulesza95}} and {\cite[3.4]{Lindenstrauss95}}] \label{def:general-position}
	Let $X$ be a locally compact separable metric space of finite covering dimension $d$. A collection $\Aulaseco$ of subsets of $X$ is in \emph{general position} if for all finite subcollections $\Fico \Subset \Aulaseco$ we have
	\[
	\dim \left(\bigcap \Fico \right) \leq \max \left\{ -1, d - |\Fico| \right\} \; .
	\]
\end{Def}

\begin{Rmk} \label{rmk:general-position-multiplicity}
	In Definition~\ref{def:general-position}, we have $\mult\left( \Aulaseco \right) \leq \dim (X)$. 
\end{Rmk}

\begin{Lemma} \label{lem:general-position}
	Let $X$ be a locally compact separable metric space of finite covering dimension $d$. 
	Let $\Aulaseco$ be a countable family of closed subsets in $X$ that is in general position. 
	Let $\Ause$ be an open subset in $X$ and let $\Ko$ be a compact subset of 
	$\Ause$. Then there exists a precompact open subset $\Cose$ in $X$ such 
	that 
	\begin{enumerate}
		\item $\Ko \subseteq \Cose \subseteq \overline{\Cose} \subseteq \Ause$,  and
		\item $\Aulaseco \cup \left\{ \partial \Cose \right\}$ is in general position. 
	\end{enumerate}
\end{Lemma}

\begin{proof}
	By assumption, we have $\dim \left(\bigcap \Fico \right) \leq \max \left\{ -1, d - |\Fico| \right\}$ for any finite subset $\Fico \Subset \Aulaseco$. In particular, we have $\bigcap \Fico = \varnothing$ if $|\Fico| \geq d+1$. It follows that the set $\Thseco = \left\{ \Fico \Subset \Aulaseco \colon \bigcap \Fico \not= \varnothing \right\}$ is countable. 
	By Lemma~\ref{lem:Engelking}\eqref{D4}, for any finite family $\Fico \in \Thseco$, there is a zero-dimensional subset $\Lase_{\Fico} \subset {\bigcap \Fico}$ that is $F_\sigma$ in ${\bigcap \Fico}$, and thus also in $X$, and such that
	\[
		\dim\left({\left( \bigcap \Fico \right)} \setminus \Lase_{\Fico} \right) \leq \max \left\{ -1, d - |\Fico| - 1 \right\} \; .
	\]
	Now let $\Lase = \bigcup_{\Fico \in \Thseco} \Lase_{\Fico}$. By 
	Lemma~\ref{lem:Engelking}\eqref{D5}, the set $\Lase$ is also an 
	$F_\sigma$-set  of dimension at most zero. By Lemma~\ref{Startlemma}, there 
	exists a precompact open subset $\Cose$ in $X$  such that  
	\[
	K\subset 
	\Cose\subset \overline{\Cose}\subset \Ause
	\; \text { and } \partial \Cose\cap 
	\Lase=\varnothing
	\, .
	\]
	To see that $\Aulaseco \cup \left\{ \partial \Cose \right\}$ is in general 
	position, it suffices to show that for any $\Fico \in \Thseco$, we have 
	$\dim \left( \partial \Cose \cap {\left( \bigcap \Fico \right)} \right) 
	\leq \max \left\{ -1, d - (|\Fico| + 1) \right\}$; this follows from the 
	equality $\partial \Cose\cap \Lase=\emptyset$, since we have 
	\[
		\partial \Cose \cap {\left( \bigcap \Fico \right)} = \partial \Cose \cap \left( {\left( \bigcap \Fico \right)} \setminus \Lase \right) \subseteq \partial \Cose \cap \left( {\left( \bigcap \Fico \right)} \setminus \Lase_{\Fico} \right) \subseteq {\left( \bigcap \Fico \right)} \setminus \Lase_{\Fico}
	\]
	and thus, using  Lemma~\ref{lem:Engelking}\eqref{D1}, we have
	\[
		\dim \left( \partial \Cose \cap {\left( \bigcap \Fico \right)} \right) 
		\leq \dim \left( {\left( \bigcap \Fico \right)} \setminus \Lase_{\Fico} \right)  \leq \max \left\{ -1, d - (|\Fico| + 1) \right\}
	\, .
	\]
\end{proof}

Next we give a strengthening of Lemma~\ref{lem:general-position} in the presence of a group action.

\begin{Lemma} \label{lem:general-position-action-verbose}
	Let $\alpha$ be an action of a discrete group $G$ on a locally compact 
	separable metric space $X$ of finite covering dimension $d$. Let $\Lo 
	\subseteq G$ be a finite subset. 
	Let $\Aulaseco$ be a countable family of closed subsets in $X$ such that 
	\[
		\left\{  \alpha^\cup_{\Lo^{-1} \Lo}  \left( \Aulase  \right) \colon \Aulase \in \Aulaseco \right\}
	\]
	is in general position. Let $\Ause$ be an open set in $X$ and let $\Ko$ be 
	a compact subset of $\Ause$. Then there exists a precompact open subset 
	$\Cose$ in $X$ such that 
	\begin{enumerate}
		\item \label{lem:general-position-action-verbose::squeeze} $\Ko \subseteq \Cose \subseteq \overline{\Cose} \subseteq \Ause$, and
		\item \label{lem:general-position-action-verbose::general-position} $\left\{  \alpha^\cup_{\Lo}  \left( \partial \Cose  \right) ,  \alpha^\cup_{\Lo}  \left( \Aulase  \right) \colon \Aulase \in \Aulaseco \right\}$ is in general position. 
	\end{enumerate}
\end{Lemma}

\begin{proof}
	Without loss of generality, we may assume $\Lo \not= \varnothing$. 
	Applying Lemma~\ref{lem:general-position} with $\left\{  
	\alpha^\cup_{\Lo^{-1} \Lo}  \left( \Aulase  \right) \colon \Aulase \in 
	\Aulaseco \right\}$ in place of $\Aulase$, we obtain a precompact open 
	subset $\Cose$ in $X$ satisfying 
	condition~\eqref{lem:general-position-action-verbose::squeeze} and such 
	that $\left\{  \alpha^\cup_{\Lo^{-1} \Lo}  \left( \Aulase  \right) \colon 
	\Aulase \in \Aulaseco \right\} \cup \left\{ \partial \Cose \right\}$ is in 
	general position. Using basic properties of covering dimension listed 
	in Lemma~\ref{lem:Engelking}\eqref{D1} and~\eqref{D2}, 
	and the fact that covering dimension is invariant under homeomorphisms, we 
	see that for any finite subset $\Fico \Subset \Aulaseco$ and for any 
	element $g_0 \in \Lo$, we have 
	\begin{align*}
		\dim \left(\bigcap_{\Aulase \in \Fico} \alpha^\cup_{\Lo}  \left( \Aulase  \right) \right)  &=  \dim \left( \alpha_{g_0^{-1}} \left(\bigcap_{\Aulase \in \Fico} \alpha^\cup_{\Lo}  \left( \Aulase  \right) \right) \right) = \dim \left( \bigcap_{\Aulase \in \Fico} \alpha^\cup_{g_0^{-1} \Lo}  \left( \Aulase  \right) \right) \\
		&\leq \dim \left( \bigcap_{\Aulase \in \Fico} \alpha^\cup_{\Lo^{-1} \Lo}  \left( \Aulase  \right) \right) \leq \max \left\{ -1, d - |\Fico| \right\} 
	\end{align*}
	and 
	\begin{align*}
		\dim \left( \alpha^\cup_{\Lo}  \left( \partial \Cose  \right)  \cap \left(\bigcap_{\Aulase \in \Fico} \alpha^\cup_{\Lo}  \left( \Aulase  \right) \right) \right) &=  \dim \left( \bigcup_{g \in \Lo} \alpha_{g} \left( \partial \Cose  \cap \left(\bigcap_{\Aulase \in \Fico} \alpha^\cup_{g^{-1} \Lo}  \left( \Aulase  \right) \right) \right)  \right) \\
		&\leq \dim \left( \bigcup_{g \in \Lo} \alpha_{g} \left( \partial \Cose  \cap \left(\bigcap_{\Aulase \in \Fico} \alpha^\cup_{\Lo^{-1} \Lo}  \left( \Aulase  \right) \right) \right)  \right) \\
		&\leq \max_{g \in \Lo} \left( \dim \alpha_{g} \left( \partial \Cose  
		\cap \left(\bigcap_{\Aulase \in \Fico} \alpha^\cup_{\Lo^{-1} \Lo}  
		\left( \Aulase  \right) \right) \right) \right) \\
		&=  \dim \left( \partial \Cose  \cap 
		\left(\bigcap_{\Aulase \in \Fico} \alpha^\cup_{\Lo^{-1} \Lo}  \left( 
		\Aulase  \right) \right) \right)  \\
		&\leq \max \left\{ -1, d - (|\Fico| + 1) \right\} \; ,
	\end{align*}	
	which proves condition~\eqref{lem:general-position-action-verbose::general-position}. 
\end{proof}

The following lemma follows from the previous lemma by induction. 

\begin{Lemma} \label{lem:general-position-action}
	Let $\alpha$ be an action of a discrete group $G$ on a locally compact separable metric space $X$ of finite covering dimension $d$. Let $n$ be a nonnegative integer, let $\Ause_1 , \ldots, \Ause_n$ be open subsets in $X$, and let $\Ko_1, \ldots, \Ko_n$ be compact subsets in $X$ satisfying $\Ko_i \subseteq \Ause_i$ for any $i \in \intervalofintegers{1}{n}$. Let $\Lo$ be a finite subset of $G$. Then there are precompact open subsets $\Cose_1 , \ldots, \Cose_n$ in $X$ 
	such that 
	\begin{enumerate}
		\item \label{lem:general-position-action::squeeze} $\Ko_i  \subseteq \Cose_i \subseteq \overline{\Cose_i} \subseteq \Ause_i$ for any $i \in \intervalofintegers{1}{n}$, and 
		\item \label{lem:general-position-action::general-position} $\left\{  \alpha^\cup_{\Lo}  \left( \partial \Cose_i  \right) \colon i = 1, \ldots, n \right\}$ is in general position. 
	\end{enumerate}
\end{Lemma}

\begin{proof}
	We apply induction on $n$, starting with $n = 0$, in which case the statement is vacuously true. 
	
	Assuming the statement holds for some nonnegative integer $n$, we prove it also holds for $n+1$. To this end, we fix a finite subset $\Lo$ of $G$, open sets $\Ause_1 , \ldots, \Ause_{n+1}$ in $X$ and compact sets $\Ko_1, \ldots, \Ko_{n+1}$ in $X$ satisfying $\Ko_i \subseteq \Ause_i$ for any $i \in \intervalofintegers{1}{n+1}$. Let $\Lo' = \Lo^{-1} \Lo \Subset G$. By our inductive assumption, there are precompact open subsets $\Cose_1 , \ldots, \Cose_n$ in $X$ such that 
	\begin{enumerate}
		\item $\Ko_i  \subseteq \Cose_i \subseteq \overline{\Cose_i} \subseteq \Ause_i$ for any $i \in \intervalofintegers{1}{n}$, and 
		\item $\left\{  \alpha^\cup_{\Lo'}  \left( \partial \Cose_i  \right) \colon i = 1, \ldots, n \right\}$ is in general position. 
	\end{enumerate}
	Then we apply Lemma~\ref{lem:general-position-action-verbose} with $\left\{   \partial \Cose_i   \colon i = 1, \ldots, n \right\}$, $\Ause_{n+1}$ and $\Ko_{n+1}$ in place of $\Aulaseco$, $\Ause$ and $\Ko$,
	and obtain a precompact open subset $\Cose_{n+1}$ in $X$ such that $\Ko_{n+1}  \subseteq \Cose_{n+1} \subseteq \overline{\Cose_{n+1}} \subseteq \Ause_{n+1}$ and $\left\{  \alpha^\cup_{\Lo}  \left( \partial \Cose_i  \right) \colon i = 1, \ldots, n+1 \right\}$ is in general position, as desired. 
\end{proof}

In fact, the separability and metrizability conditions can be removed in the 
above lemma, at the cost of relaxing 
condition~(\ref{lem:general-position-action::general-position}) in the 
conclusion. 

\begin{Lemma} \label{lem:general-position-action-general}	
	Let $\alpha$ be an action of a discrete group $G$ on a locally compact Hausdorff space $X$. 
		Let $n$, $\Ause_1 , \ldots, \Ause_n$, $\Ko_1, \ldots, \Ko_n$ and $\Lo$ 
		be as in Lemma~\ref {lem:general-position-action}. 
	Then there are precompact open subsets $\Cose_1 , \ldots, \Cose_n$ in $X$ 
	such that 
	\begin{enumerate}
		\item \label{lem:general-position-action-general::squeeze} $\Ko_i  \subseteq \Cose_i \subseteq \overline{\Cose_i} \subseteq \Ause_i$ for any $i \in \intervalofintegers{1}{n}$, and 
		\item \label{lem:general-position-action-general::multiplicity} 
		$\operatorname{mult} \left( \left\{ \alpha^\cup_{\Lo}  \left( \partial 
		\Cose_i \right) \colon i = 1, \ldots, n \right\} \right) \leq 
		\dim(X^+)$. 
	\end{enumerate}
\end{Lemma}

\begin{proof}
	If $\dim(X^+) = \infty$ then there is nothing to prove, so we assume it is 
	finite. 	Set $d = \dim (X^+)$. 
	Let $H$ be the subgroup of $G$ generated by $\Lo$. Then $H$ is countable. 
	For $i \in \intervalofintegers{1}{n}$, we apply the Tietze extension 
	theorem to obtain continuous functions $\Po_i \colon X \to [0,1]$ such that 
	$\Po_i \in C_0(X)$, $\Po_i^{-1} (\{1\}) = \Ko_i$ and $\Po_i^{-1} (\{0\}) = 
	X \setminus \Ause_i$. We then apply Lemma~\ref{lem:separable-dimnuc} and 
	Theorem~\ref{thm:dimnuc-dim-general} to obtain an $H$-invariant separable 
	$C^*$-subalgebra $B \subset C_0(X)$ with $\dimnuc(B) \leq \dimnuc(C_0(X)) = 
	\dim(X^+) \leq d$ and $\Po_i \in B$ for any $i \in 
	\intervalofintegers{1}{n}$. Let $Y$ be the spectrum of the commutative 
	$C^*$-algebra $B$. We have $\dim(Y) \leq d$. For $i \in 
	\intervalofintegers{1}{n}$, let $\widetilde{\Po}_i \colon Y \to [0,1]$ be 
	the induced functions on $Y$.  Define an open subset 
	$\widetilde{\Ause}_i = \widetilde{\Po}_i^{-1} ((0,1])$ in $Y$ and, inside 
	it, a compact subset $\widetilde{\Ko}_i = \widetilde{\Po}_i^{-1} (\{1\})$. 
	By Lemma~\ref{lem:general-position-action}, there are precompact open 
	subsets $\Thse_1 , \ldots, \Thse_n$ in $Y$ 
	such that 
	\begin{enumerate}
		\item $\Ko_i  \subseteq \Thse_i \subseteq \overline{\Thse_i} \subseteq \Ause_i$ for any $i \in \intervalofintegers{1}{n}$, and 
		\item $\left\{  \alpha^\cup_{\Lo}  \left( \partial \Thse_i  \right) \colon i = 1, \ldots, n \right\}$ is in general position. 
	\end{enumerate}
	In particular, we have $\operatorname{mult} \left( \left\{ 
	\alpha^\cup_{\Lo}  \left( \partial \Thse_i \right) \colon i = 1, \ldots, n 
	\right\} \right) \leq \dim(Y)$. Let $J$ be the ideal in $C_0(X)$ generated 
	by $B$, which we can identify as the algebra of functions vanishing off an 
	open subset $X_0 \subseteq X$. 
	Let $\pi \colon X_0 \to Y$ be the quotient map induced from the inclusion 
	$B \subseteq C_0(X_0)$. Observe that $\partial \left( \pi^{-1} \left( 
	\Thse_i \right) \right) \subseteq \pi^{-1} \left( \partial \Thse_i \right)$ 
	for any $i \in \intervalofintegers{1}{n}$. It is clear that the open 
	subsets $\Cose_i = \pi^{-1} \left( \Thse_i \right)$ satisfy the desired 
	condition. 
\end{proof}

The following lemma shows that we can thicken the boundaries of the open sets 
we obtained in Lemma~\ref{lem:general-position-action-general} without 
increasing the multiplicity. 

\begin{Lemma} \label{lem:general-position-action-thick}
	Let $\alpha$ be an action of a discrete group $G$ on a locally compact Hausdorff space $X$. 
	Let $n$, $\Ause_1 , \ldots, \Ause_n$, $\Ko_1, \ldots, \Ko_n$ and $\Lo$ be as in Lemmas~\ref {lem:general-position-action} and~\ref {lem:general-position-action-general}. 
	Then there are open sets $\Ause'_1 , \ldots, \Ause'_n$ and compact sets $\Ko'_1, \ldots, \Ko'_n$ in $X$ such that 
	\begin{enumerate}
		\item \label{lem:general-position-action-thick::squeeze} 
		$\Ko_i  \subseteq \Ko'_i \subseteq \Ause'_i \subseteq  \Ause_i$ 
		for any $i \in \intervalofintegers{1}{n}$, and 
		\item \label{lem:general-position-action-thick::multiplicity} $\operatorname{mult} \left( \left\{ \alpha^\cup_{\Lo}  \left( \Ause'_i \setminus \Ko'_i \right) \colon i = 1, \ldots, n \right\} \right) \leq \dim(X)$. 
	\end{enumerate}
\end{Lemma}

\begin{proof}
	Without loss of generality, we may assume $\dim(X) < \infty$. 
	Given $n$, $\Ause_1 , \ldots, \Ause_n$, $\Ko_1, \ldots, \Ko_n$ and $\Lo$ as above, we apply Lemma~\ref{lem:general-position-action-general} to obtain precompact open subsets $\Cose_1 , \ldots, \Cose_n$ in $X$ 
	such that $\Ko_i  \subseteq \Cose_i \subseteq \overline{\Cose_i} \subseteq \Ause_i$ for any $i \in \intervalofintegers{1}{n}$ and $\operatorname{mult} \left( \left\{ \alpha^\cup_{\Lo}  \left( \partial \Cose_i  \right) \colon i = 1, \ldots, n \right\} \right) \leq \dim(X)$. 
	By Lemma~\ref{lem:thicken-multiplicity}, there are open neighborhoods $\Thse_i$ of $\alpha^\cup_{\Lo}  \left( \partial \Cose_i  \right)$, for $i = 1, \ldots, n$, such that $\mult \left( \left\{  \Thse_i \colon i = 1, \ldots, n \right\} \right) \leq \dim(X)$. For any $i \in \intervalofintegers{1}{n}$, we define the open set 
	\[
		\Lase_i = \left( \Ause_i \setminus \Ko_i \right) \cap \left( \bigcap_{g 
		\in \Lo} \alpha_{g^{-1}} \left( \Thse_i \right) \right)  \; .
	\]
	Note that $\Lase_i \supseteq \partial \Cose_i$. Define
	\[
		\Ause'_i = \Cose_i \cup \Lase_i \quad \text{and} \quad \Ko'_i = \Cose_i 
		\setminus \Lase_i = \overline{\Cose_i} \setminus \Lase_i \; .
	\]
	Observe that $\Ause'_i$ is open, $\Ko'_i$ is compact, $\Ause'_i \setminus 
	\Ko'_i = \Lase_i$, and we have  $\Ko_i  \subseteq \Ko'_i 
	\subseteq \Ause'_i \subseteq  \Ause_i$. Moreover, because 
	$\alpha^\cup_{\Lo}  \left( \Ause'_i \setminus \Ko'_i \right) = 
	\alpha^\cup_{\Lo}  \left( \Lase_i \right) \subseteq \Thse_i$, for any $i 
	\in \intervalofintegers{1}{n}$ we have 
	\[
		\mult \left( \left\{ \alpha^\cup_{\Lo}  \left( \Ause'_i \setminus \Ko'_i \right) \colon i = 1, \ldots, n \right\} \right) \leq \mult \left( \left\{  \Thse_i \colon i = 1, \ldots, n \right\} \right) \leq \dim(X) \; ,
	\]
	as desired.
\end{proof}

The next lemma provides a version of the previous lemmas which furthermore 
allows for control 
over the $(\alpha,\Bo)$-multiplicity (see 
Definition~\ref{def:multiplicity-action}). 

\begin{Lemma} \label{lem:general-position-combinatorial}
	Let $\alpha$ be an action of a discrete group $G$ on a locally compact Hausdorff space $X$. 
	Let $n$, $\Ause_1 , \ldots, \Ause_n$, $\Ko_1, \ldots, \Ko_n$  be as in Lemmas~\ref {lem:general-position-action}, \ref {lem:general-position-action-general} and~\ref {lem:general-position-action-thick}. 
	Let $\Bo$ be a symmetric finite subset of $G$ that contains $e$. 
	Then there are open sets $\Ause'_1 , \ldots, \Ause'_n$ and compact sets $\Ko'_1, \ldots, \Ko'_n$ in $X$ such that 
	\begin{enumerate}
		\item \label{lem:general-position-combinatorial::squeeze}
		$\Ko'_i \subseteq \Ause'_i \subseteq \Ause_i$
		for any $i \in \intervalofintegers{1}{n}$, 
		\item \label{lem:general-position-combinatorial::cover} 
		$\displaystyle \bigcup_{j=1}^{i} \Ko_j \subseteq \Ko'_i \cup \left( \bigcup_{j=1}^{i-1} \alpha^\cup_{\Bo} \left(\Ko'_j\right) \right)$
		for any $i \in \intervalofintegers{1}{n}$, 
		\item \label{lem:general-position-combinatorial::band} 
		$\operatorname{mult} \left( \left\{ \alpha^\cup_{\Bo}  \left( \Ause'_i 
		\setminus \Ko'_i \right) \colon i = 1, \ldots, n \right\} \right) \leq 
		\dim(X^+)$, 
		and 
		\item \label{lem:general-position-combinatorial::multiplicity} 
		$\mult_{\alpha,\Bo} \left( \left\{ \Ause'_i  \colon i = 1, \ldots, n 
		\right\} \right) \leq \dimone (X^+)$. 
	\end{enumerate}
\end{Lemma}

\begin{proof}
	We prove the claim by induction on $n$. The case where $n=0$ is vacuously 
	true. 
	
	Now assuming the statement holds for some nonnegative integer $n$, we prove it also holds for $n+1$. To this end, we fix $\Ause_1 , \ldots, \Ause_{n+1}$, $\Ko_1, \ldots, \Ko_{n+1}$, and $\Bo$ as in the statement, from which we are going to produce $\Ause'''_1 , \ldots, \Ause'''_{n+1}$ and $\Ko'''_1, \ldots, \Ko'''_{n+1}$ that satisfy \eqref{lem:general-position-combinatorial::squeeze}-\eqref{lem:general-position-combinatorial::multiplicity} in place of $\Ause'_1 , \ldots, \Ause'_{n}$ and $\Ko'_1, \ldots, \Ko'_{n}$. 
	
	Applying our inductive assumption to the first $n$ pairs of sets, we obtain open sets $\Ause'_1 , \ldots, \Ause'_n$ and compact sets $\Ko'_1, \ldots, \Ko'_n$ in $X$ satisfying \eqref{lem:general-position-combinatorial::squeeze}-\eqref{lem:general-position-combinatorial::multiplicity}. 
	Using the fact that $X$ is locally compact and Hausdorff, for each $i \in 
	\intervalofintegers{1}{n}$ we choose a precompact open subset $\Thse_i$ 
	such that $\Ko'_i \subseteq \Thse_i \subseteq \overline{\Thse_i} \subseteq 
	\Ause'_i$. 
	We now define 
	\[
		\Ause''_i = \Ause'_i   \quad \text{and} \quad \Ko''_i = 
		\overline{\Thse_i} \quad \text{for } i \in \intervalofintegers{1}{n} \; 
		.
	\]
	Those sets clearly also satisfy  
	\eqref{lem:general-position-combinatorial::squeeze}-\eqref{lem:general-position-combinatorial::multiplicity}
	 in place of $\Ause'_1 , \ldots, \Ause'_{n}$ and $\Ko'_1, \ldots, 
	\Ko'_{n}$. We  define
	\begin{equation*} 
		\Ause''_{n+1} = \Ause_{n+1} \setminus \left( \bigcup_{j=1}^{n} 
		\alpha^\cup_{\Bo} \left( \Ko'_j \right) \right)  \quad \text{and} \quad 
		\Ko''_{n+1} = \Ko_{n+1} \setminus \left( \bigcup_{j=1}^{n} 
		\alpha^\cup_{\Bo} \left( \Thse_j \right) \right) \; .
	\end{equation*}
	Notice that $\Ko''_{n+1} \subseteq \Ause''_{n+1} \subseteq \Ause_{n+1}$ and 
	\[
		\Ko_{n+1} \subseteq \Ko''_{n+1} \cup \left( \bigcup_{j=1}^{n} \alpha^\cup_{\Bo} \left( \Thse_j \right) \right) 
		\subseteq \Ko''_{n+1} \cup \left( \bigcup_{j=1}^{n} \alpha^\cup_{\Bo} 
		\left( \Ko''_j \right) \right) \; .
	\]
	Thus, by our inductive assumption, we have
	\begin{align*}
		\bigcup_{j=1}^{n+1} \Ko_j &= \Ko_{n+1} \cup \left( \bigcup_{j=1}^{n} \Ko_j \right) \\
		&\subseteq \left( \Ko''_{n+1} \cup \left( \bigcup_{j=1}^{n} \alpha^\cup_{\Bo} \left( \Ko''_j \right) \right) \right) \cup  \left( \Ko''_n \cup \left( \bigcup_{j=1}^{n-1} \alpha^\cup_{\Bo} \left(\Ko''_j\right) \right) \right) \\
		&\subseteq \Ko''_{n+1} \cup \left( \bigcup_{j=1}^{n} \alpha^\cup_{\Bo} \left(\Ko''_j\right) \right)  \; .
	\end{align*}
	This shows $\Ause''_1 , \ldots, \Ause''_{n+1}$ and $\Ko''_1, \ldots, 
	\Ko''_{n+1}$  satisfy \eqref{lem:general-position-combinatorial::squeeze} 
	and \eqref{lem:general-position-combinatorial::cover} in place of $\Ause'_1 
	, \ldots, \Ause'_{n}$ and $\Ko'_1, \ldots, \Ko'_{n}$. 
	
	We now apply Lemma~\ref{lem:general-position-action-thick} to obtain open 
	sets $\Ause'''_1 , \ldots, \Ause'''_{n+1}$ and compact sets $\Ko'''_1, 
	\ldots, \Ko'''_{n+1}$ in $X$ such that 
	\begin{itemize}
		\item 
		$\Ko''_i  \subseteq \Ko'''_i \subseteq \Ause'''_i \subseteq  \Ause''_i$ 
		for any $i \in \intervalofintegers{1}{n+1}$, and 
		\item  
		$\operatorname{mult} \left( \left\{ \alpha^\cup_{\Bo}  \left( 
		\Ause'''_i \setminus \Ko'''_i \right) \colon i = 1, \ldots, n \right\} 
		\right) \leq \dim(X^+)$. 
	\end{itemize}
	It follows 
	that $\Ause'''_1 , \ldots, \Ause'''_{n+1}$ and $\Ko'''_1, \ldots, \Ko'''_{n+1}$ satisfy \eqref{lem:general-position-combinatorial::squeeze}-\eqref{lem:general-position-combinatorial::band} in place of $\Ause'_1 , \ldots, \Ause'_{n}$ and $\Ko'_1, \ldots, \Ko'_{n}$. 
	It remains to show that 
	\[
		\mult_{\alpha,\Bo} \left( \left\{ \Ause'''_i  \colon i = 1, \ldots, n+1 
		\right\} \right) \leq \dimone (X^+) \; . 
	\]
	To this end, it suffices to show that for any subset $\Fi \subseteq 
	\intervalofintegers{1}{n+1}$, and any tuple $\left(x_i\right) \in X^{\Fi}$ 
	satisfying $x_i \in \Ause'''_i$ for any $i \in \Fi$, if $\left\{ x_i \colon 
	i \in \Fi \right\}$ is $(\alpha, \Bo)$-close, then $|\Fi| \leq \dimone 
	(X^+)$. 
	The case where $n+1 \not \in \Fi$ is covered by our inductive assumption, 
	so we may assume $n+1 \in F$. 
	Since $x_{n+1} \in \Ause'''_{n+1} \subseteq \Ause''_{n+1}$ by our assumption on $\left(x_i\right)$ and the choice of $\Ause'''_{n+1}$, 
	it follows from the construction of $\Ause''_{n+1}$ that
	$x_{n+1} \notin \alpha^\cup_{\Bo} \left( \Ko'_j \right)$ for any $j \in 
	\intervalofintegers{1}{n}$.
	Since $\left\{ x_i \colon i \in \Fi \right\}$ 
	is $(\alpha, \Bo)$-close, it follows that for any $i \in \Fi \setminus 
	\{n+1\}$, we have $x_i \notin \Ko'_i$, whence $x_{n+1} \in \alpha_{\Bo} 
	\left( x_i \right) \subseteq \alpha^\cup_{\Bo}  \left( \Ause'_i \setminus 
	\Ko'_i \right)$. Since by the inductive assumption, we have 
	$\operatorname{mult} \left( \left\{ \alpha^\cup_{\Bo}  \left( \Ause'_i 
	\setminus \Ko'_i \right) \colon i = 1, \ldots, n \right\} \right) \leq 
	\dim(X)$, it follows that $\left| \Fi \setminus \{n+1\} \right| \leq 
	\dim(X)$, whence $|\Fi| \leq \dimone (X)$ as desired. 
\end{proof}

We are finally ready to give a proof of Theorem~\ref{thm:lsp-ltc}. The 
reader may find it helpful to first focus on the case where $\dim(X) = 0$ in 
order to better understand the combinatorics of the proof.

\begin{proof}[Proof of Theorem~\ref{thm:lsp-ltc}]
	We want to show 
	\[
	\dimltc (\alpha) \leq \LSP_{d}(G) - 1 \; .
	\]
	Let $\Dimnu \geq \LSP_d(G) - 1$ be a natural number. We need to show 
 	 that for any finite subset $\Lo \Subset G$ and for any compact subset $\Ko 
	 \Subset X$ there exists a natural number $\Binu$ such that and any finite 
	 collection $\Thseco$ of open sets in $X$ 
	 covering $\alpha^\cup_{\Lo} (\Ko)$, there exist 
	 a  collection $\Coseco$ of open sets in $X$ and locally constant functions 
	 $\Ne_{\Cose} \colon \Cose \to X$ for $\Cose \in \Coseco$ satisfying the 
	 conditions \Refc{Lo}{\Lo,\Ko}, 
	 \Refc{Mu}{\Dimnu}, 
	 \Refc{Eq}{\Lo}, 
	 \Refc{Th}{\Thseco} 
	 and 
	 \Refc{Ca}{\Binu}.   
	 Furthermore, by 
	 replacing $\Lo$ with $\Lo \cup \Lo^{-1}$ if necessary, we may assume 
	 without loss of generality that $\Lo = \Lo^{-1}$. 
	
	Following Definition~\ref{def:LSP}, we choose a finite symmetric subset 
	$\Bo[LSP] \Subset G$ containing $e$ such that 
	for any positive integer $n$ and $g_1,\ldots g_n \in \Bo[LSP]\Lo$, 
	if there is no subset $\Fi \subseteq \intervalofintegers{1}{n}$ such that $|\Fi| > d+1$ and 
	$g_i^{} g_j^{-1} \in \Bo[LSP]$ for any $i,j \in \Fi$, 
	then $n \leq m$. Note that $\Bo[LSP]$ depends only on $\Lo$. We define 
	$\Binu = |\Lo \Bo[LSP]|$.
	
	By Lemma~\ref{lem:thin-neighborhood}, for any $x \in \Ko$, there is an open neighborhood $\Ause_x$ of $x$ such that for any $g, g' \in \Lo\Lo\Bo[LSP]$, we have $\alpha_g \left(\Ause_x\right) \cap \alpha_{g'} \left(\Ause_x\right) \not= \varnothing$ if and only if $\alpha_g \left( x \right) = \alpha_{g'} \left( x \right)$. 
	Therefore, we can define a locally constant function $\Ne_x \colon 
	\alpha^\cup_{\Lo\Lo\Bo[LSP]} \left( \Ause_x \right) \to X$ taking 
	$\alpha_{g} \left( \Ause_x \right)$ to $\alpha_{g} \left( x \right)$ for 
	any $g \in \Lo\Lo\Bo[LSP]$. 
	{
		We claim that the following holds:
		\begin{enumerate}
			\renewcommand{\labelenumi}{\textup{(\theenumi)}} \renewcommand{\theenumi}{Pre-Eq}\item \label{thm:lsp-ltc::pre-Eq}
			The restriction of $\Ne_x$ to $\alpha^\cup_{\Lo\Bo[LSP]} \left( \Ause_x \right)$ is $\Lo$-equivariant in the sense of \eqref{item:def:ltc-dim:Eq} in Definition~\ref{def:ltc-dim}, that is, for any $y \in \alpha^\cup_{\Lo\Bo[LSP]} \left( \Ause_x \right)$ and for any $g \in \Lo$, if $\alpha_{g} (y) \in \alpha^\cup_{\Lo\Bo[LSP]} \left( \Ause_x \right)$ then $\Ne_{x}\left( \alpha_{g} (y) \right) = g \cdot \Ne_{x}(y)$. 
		\end{enumerate}
		Indeed, if we let $h, h' \in \Lo\Bo[LSP]$ be such that $y \in \alpha_{h} \left(\Ause_x\right)$ and $\alpha_{g} (y) \in \alpha_{h'} \left(\Ause_x\right)$, then since $g h \in \Lo\Lo\Bo[LSP]$ and $\alpha_{g} (y) \in \alpha_{g h} \left(\Ause_x\right)$, it follows by our construction that $\alpha_{h'} (x) = \alpha_{g h} (x)$, whence $\Ne_{x}\left( \alpha_{g} (y) \right) = \alpha_{h'} (x) = \alpha_{g h} (x) = g \cdot \Ne_{x}(y)$. 
	}
	
	By replacing $\Ause_x$ with the open set 
	\[
		\Ause_x \cap \left( \bigcap_{g \in \Lo\Lo\Bo[LSP]} \left\{ \alpha_{g^{-1}} \left( \Thse \right) \colon \Thse \in \Thseco , \alpha_g \left( x \right) \in \Thse \right\} \right) \; , 
	\]
	we may further assume the following condition without loss of generality: 
	{
		\begin{enumerate}
			\renewcommand{\labelenumi}{\textup{(\theenumi)}} \renewcommand{\theenumi}{Pre-Th}\item \label{thm:lsp-ltc::pre-Th}
			For any $\Thse \in \Thseco$ and any $g \in \Lo\Lo\Bo[LSP]$, if $\alpha_g \left( x \right) \in \Thse$, then $\alpha_g \left( \Ause_x \right) \subseteq \Thse$. 
		\end{enumerate}
	}
	Because $\Ko$ is compact, we can choose a finite set $x_1, \ldots, x_n \in 
	X$ such that $\Ko 
	\subseteq \bigcup_{i=1}^n \Ause_{x_i}$. We denote $\Ause_i$ for 
	$\Ause_{x_i}$ to lighten notation. It follows from 
	Lemma~\ref{lem:relative-pou}, by taking the supports of a partition of 
	unity, 
	that there are compact subsets $\Ko_1, \ldots, \Ko_n$ such that $\Ko 
	\subseteq \bigcup_{i=1}^n \Ko_i$ and $\Ko_i \subseteq \Ause_i$ for any $i 
	\in \intervalofintegers{1}{n}$. 
	
	By Lemma~\ref{lem:general-position-combinatorial}, there are open sets $\Ause'_1 , \ldots, \Ause'_n$ and compact sets $\Ko'_1, \ldots, \Ko'_n$ 
	in $X$ satisfying \eqref{lem:general-position-combinatorial::squeeze}-\eqref{lem:general-position-combinatorial::multiplicity} of Lemma~\ref{lem:general-position-combinatorial}. In particular, we have 
	\begin{enumerate}
		\item \label{thm:lsp-ltc::trimmed}
		$\Ause'_i \subseteq \Ause_i$
		for any $i \in \intervalofintegers{1}{n}$, 
		\item \label{thm:lsp-ltc::cover-translates}
		$\displaystyle \Ko \subseteq \bigcup_{j=1}^{n} \alpha^\cup_{\Bo[LSP]} \left(\Ause'_j\right)$, 
		and 
		\item \label{thm:lsp-ltc::mult}
		$\mult_{\alpha,\Bo[LSP]} \left( \left\{ \Ause'_i  \colon i = 1, \ldots, 
		n \right\} \right) \leq \dimone (X^+)$. 
	\end{enumerate}
	Now, for any $i \in \intervalofintegers{1}{n}$, we define  $\Cose_i = 
	\alpha^\cup_{\Lo\Bo[LSP]} \left(\Ause'_i\right)$ and we let $\Ne_{\Cose_i} 
	\colon \Cose_i \to X$ be the restriction of the map $\Ne_{x_i} \colon 
	\alpha^\cup_{\Lo\Lo\Bo[LSP]} \left( \Ause_{i} \right) \to X$. Set 
	\[
	\Coseco 
	= \left\{ \Cose_i \colon i \in \intervalofintegers{1}{n} \right\}
	\; \text{ and } \; 
	\Bo 
	= \left(\Lo\Bo[LSP]\right) \left(\Lo\Bo[LSP]\right)^{-1}
	\, .
	 \]
	
	We claim that they satisfy 
	 \Refc{Lo}{\Lo,\Ko}, 
	\Refc{Mu}{\Dimnu}, 
	\Refc{Eq}{\Lo}, 
	\Refc{Th}{\Thseco} 
	and 
	\Refc{Ca}{\Binu}.   
	Indeed, 
	\Refc{Eq}{\Lo}
	and 
	\Refc{Th}{\Thseco}
	follow 
	from \eqref{thm:lsp-ltc::pre-Eq} and \eqref{thm:lsp-ltc::pre-Th} combined 
	with condition~\eqref{thm:lsp-ltc::trimmed} above. Condition
	\Refc{Eq}{\Lo}
	follows from 
	condition~\eqref{thm:lsp-ltc::cover-translates} above. Condition
	\Refc{Ca}{\Binu}    
	follows from the fact that the image of 
	$\Ne_{\Cose_i}$ is $\alpha_{\Lo\Bo[LSP]}(x_i)$ for any $i \in \{1, \ldots, 
	n\}$. 
	
	To prove 
	\Refc{Mu}{\Dimnu}, 
	that is, $\mult(\Coseco) \leq \Dimnu + 1$, it suffices to show that for any 
	$x \in X$, the set $\Inse_x = \left\{ i \in \intervalofintegers{1}{n} 
	\colon x \in \Cose_i \right\}$ has cardinality no more than $\Dimnu + 1$. 
	To this end, by our construction of $\Cose_i$, we may choose, for each $i 
	\in \Inse_x$, an element $g_i \in (\Lo\Bo[LSP])^{-1} = \Bo[LSP]\Lo$ such 
	that $\alpha_{g_i} (x) \in \Ause'_i$. 
	Observe that for any subset $\Fi \subseteq \Inse_x$ satisfying $g_i^{} 
	g_j^{-1} \in \Bo[LSP]$ for any $i,j \in \Fi$, the points $\alpha_{g_i} (x) 
	\in \Ause'_i$ for $i \in \Fi$ form an $(\alpha, \Bo[LSP])$-close subset in 
	the sense of Definition~\ref{def:multiplicity-action}; therefore, by 
	condition~\eqref{thm:lsp-ltc::mult} above, we have $|\Fi| \leq 
	\dimone(X^+)$. It then follows from our choice of $\Bo[LSP]$ that the 
	cardinality of $\Inse_x$ is no more than $\Dimnu + 1$, as claimed. 
\end{proof}


\section{The equivariant asymptotic dimension and relative bounds} \label{sec:BLR}
\renewcommand{\sectionlabel}{BLR}
\ref{sectionlabel=BLR}
In this section, we study the notion of equivariant asymptotic dimension and 
its relations to the dimensions we introduced in the previous sections. This 
will be needed only when we extend our results on nuclear dimension beyond the 
virtually nilpotent case (and cover, for example, various actions of hyperbolic 
groups). For actions of discrete groups on compact Hausdorff 
spaces, the notion of equivariant asymptotic dimension grew out of the work of 
Bartels, L\"{u}ck and Reich \cite{BartelsLueckReich2008Equivariant} on the 
Farrell-Jones conjecture and was subsequently studied on its own right in 
\cite{Bartels2017Coarse, GuentnerWillettYu2017Dynamic, Sawicki2017equivariant} 
(under the terms finite $\mathcal{F}$-amenability and $d$-BLR in the first two 
papers). We introduce here a natural generalization to the case of locally 
compact spaces\footnote{For non-compact spaces, our definition 
differs from that of finite $\mathcal{F}$-amenability in 
\cite[Definition~0.1]{Bartels2017Coarse}. For this reason, we chose to use 
different terminology.}.  

\def \Topic {\BLRdef}
\begin{Def} \label{def:BLR}
	Let $\alpha$ be an action of a discrete group $G$ on a locally compact 
	Hausdorff space $X$. Let $d$ be a nonnegative integer and let $\mathcal{F}$ 
	be a family of subgroups of $G$ closed under conjugation and under taking 
	subgroups. Let $\Delta$ be the diagonal action of $G$ on $G \times X$, 
	i.e., $\Delta_g(g', x) = \left( g' g^{-1}, \alpha_g (x) \right)$. The 
	\emph{equivariant asymptotic dimension of $\alpha$ relative to 
	$\mathcal{F}$}, denoted by $\eqasdim(\alpha, \mathcal{F})$, is the infimum 
	of all nonnegative integers $d$ satisfying:
	
	For any finite subset $\Lo \Subset G$ and for any compact subset $\Ko 
	\Subset X$, there exists a collection $\Auseco$ of open sets in $G \times 
	X$ 
	satisfying the following conditions: 
	{
	\begin{enumerate}
		\renewcommand{\labelenumi}{\textup{(\theenumi)}} \renewcommand{\theenumi}{Lo}\item \label{item:def:BLR:Lo}
		For any $x \in \Ko$, there exists $\Ause \in \Auseco$ such that $\Lo \times  \{x\} \subseteq \Ause$.
		
		\renewcommand{\labelenumi}{\textup{(\theenumi)}} \renewcommand{\theenumi}{Mu}\item \label{item:def:BLR:Mu}
		We have $\mult(\Auseco) \leq d+1$. 
		
		\renewcommand{\labelenumi}{\textup{(\theenumi)}} \renewcommand{\theenumi}{Eq}\item \label{item:def:BLR:Eq}
		For any $\Ause \in \Auseco$ and for any $g \in G$, the translate 
		$\Delta_g (\Ause)$ is also in $\Auseco$ and is either equal to $\Ause$ 
		or disjoint from $\Ause$. 
		
		\renewcommand{\labelenumi}{\textup{(\theenumi)}} \renewcommand{\theenumi}{St}\item \label{item:def:BLR:St}
		For any $\Ause \in \Auseco$, its \emph{set stabilizer} $\setstab_\Delta(\Ause)$, defined by 
		\[
			\setstab_\Delta(\Ause) = \left\{ g \in G \colon \Delta_g (\Ause) = \Ause \right\} \, ,
		\]
		is an element of $\mathcal{F}$. 		
	\end{enumerate}
	}
\end{Def}

The labels (Lo), (Mu), (Eq), and (St) stand for, respectively, ``long'', 
``multiplicity'', ``equivariant'', and ``stabilizer''. 
We will also use the notation 
\Refcd{Lo}, \Refcd{Mu}, \Refcd{Eq}, and \Refcd{St} 
to specify the parameters used in these conditions and the fact that they come from Definition~\ref{\Topic}.

\begin{Lemma} \label{lemma:BLR-cofinite}
	In Definition~\ref{def:BLR}, one can furthermore assume that we have the 
	following cofiniteness conditions:  
	\begin{itemize}
		\item $\Auseco$ is cofinite with regard to the $G$-action, that is, 
		there exists a finite collection $\left\{ \Ause_i \colon i \in I 
		\right\}$ such that $\Auseco = \left\{ \Delta_g \left( \Ause_i \right) 
		\colon i \in I , g \in G \right\}$; 
		\item for any $\Ause \in \Auseco$, the projection of $\Ause$ on the first coordinate, i.e., the set 
		\[
		\Aufi_{\Ause} := \left\{g \in G \colon \Ause \cap \left( \{g\} \times X \right) \not= \varnothing \right\}
		\]
		is cofinite with regard to $\setstab_\Delta \left(\Ause \right)$, that is, there is a finite set $\Fi_{\Ause} \Subset G$ such that $\Aufi_{\Ause}  \subseteq \Fi_{\Ause}  \cdot \setstab_\Delta \left(\Ause \right)$. 
	\end{itemize}
	\end{Lemma}
	\begin{proof}
	Let  $\Lo \Subset G$ be a given finite subset and let  $\Ko \Subset 
	X$ be compact. We may assume without loss of generality that $\grpid \in 
	\Lo$. 
	We apply Definition~\ref{def:BLR} to obtain a collection $\Auseco$ that 
	satisfies \Refcd{Lo}, \Refcd{Mu}, 
	\Refcd{Eq}, and \Refcd{St}. We need to modify $\Auseco$ to 
	obtain a new collection $\Auseco'$ that satisfies the conditions in 
	Definition~\ref{def:BLR} together with the above cofiniteness conditions. 
	
	To this end, let $\left\{ \Ause_i \right\}_{i \in I}$ be a set of representatives of the $G$-action on $\Auseco$. For each $i \in I$, decompose $\Ause_i$ into $\bigsqcup_{g \in G} \{g\} \times \Ause_{i,g}$ where $\Ause_{i,g} := \left\{ x \in X \colon (g, x) \in \Ause_i \right\}$ is an open subset of $X$ for any $g \in G$. 
	For any $i \in I$ and $g \in G$, consider the open set 
	\[
	\Lase_{i,g} = \bigcap_{h \in \Lo } \Ause_{i, h g} = \left\{ x \in X \colon \Lo g \times \{ x \} \subseteq \Ause_i \right\} \subseteq \Ause_{i, g} \; .
	\]
	It follows from \Refcd{Lo} that for any $x \in \Ko$, there exists	$i \in 
	I$ and $g \in G$ such that $\Lo \times  \{x\} \subseteq \Delta_g \left( 
	\Ause_i \right)$, or equivalently, 
	\[
	\Lo g \times  \left\{ \alpha_{g^{-1}} (x) \right\} = \Delta_{g^{-1}} \left( \Lo  \times  \{x\} \right) \subseteq \Ause_i \; ,
	\]
	or equivalently, $\alpha_{g^{-1}} (x) \in \Lase_{i,g}$. Hence we have
	\[
	\Ko \subseteq \bigcup_{i \in I} \bigcup_{g \in G} \alpha_{g} \left( \Lase_{i, g} \right) \; .
	\]
	Since $\Ko$ is compact, there exist finite subsets $I_0 \subseteq I$ and 
	$\Bo_0 \subseteq G$ such that 
	\[
	\Ko \subseteq \bigcup_{i \in I_0} \bigcup_{g \in \Bo_0} \alpha_{g} \left( \Lase_{i, g} \right) \; .
	\]
	For any $i \in I_0$ and for any $g \in G$, we define 
	\[
	\Lase_{i,g}' = 
	\begin{cases}
	\Lase_{i,g} \, , & g \in \Bo_0 \cdot \setstab_\Delta \left(\Ause_i \right) \\
	\varnothing	\, , & g \notin \Bo_0 \cdot \setstab_\Delta \left(\Ause_i \right)
	\end{cases}
	\quad \text{and} \quad 
	\Ause_{i,g}' = \bigcup_{h \in \Lo} \Lase_{i, h^{-1} g}'
	\; .
	\]
	Observe that 
	\[
	\Ause_{i,g}'  \subseteq  \bigcup_{h \in \Lo} \bigcap_{h' \in \Lo} \Ause_{i, h' h^{-1} g}	\subseteq \Ause_{i, g} \quad \text{for any } g \in G
	\]
	and $\Ause_{i,g}' = \varnothing$ if $g \notin \Lo \Bo_0$. For any $i \in 
	I_0$, set
	\[
	\Ause_{i}' = \bigsqcup_{g \in G} \{g\} \times \Ause_{i,g}'  \subseteq 
	\Ause_i \; .
	\]
	Notice that:
	\begin{itemize}
		\item if $g \in \setstab_\Delta \left(\Ause_i \right)$, then $\Delta_g \left( \Ause_{i}' \right) = \Ause_{i}'$ by our construction, and 
		\item if $g \notin \setstab_\Delta \left(\Ause_i \right)$, then $\Delta_g \left( \Ause_{i}' \right) \cap \Ause_{i}' \subseteq \Delta_g \left( \Ause_{i} \right) \cap \Ause_{i} = \varnothing$. 
	\end{itemize} 
	It follows that the collection $\Auseco' := \left\{ \Delta_g \left( \Ause_i' \right) \colon i \in I_0 \right\}$ 
	satisfies  \Refcd{Lo}, \Refcd{Mu}, \Refcd{Eq}, and \Refcd{St} together with the desired cofiniteness conditions. 
\end{proof}

The following equivalent characterization of equivariant asymptotic dimension 
is reminiscent of the definition of $\dimltc(\alpha)$ 
(Definition~\ref{def:ltc-dim} and the various equivalent formulations discussed 
in that section).

\def \Topic {\BLRlem}
\begin{Lemma} \label{lem:BLR-reformu} 	
	Let $d$ and $\mathcal{F}$ be as in Definition~\ref{def:BLR}. 
	Then 
	$\eqasdim(\alpha, \mathcal{F}) \leq d$ if and only if for any finite subset 
	$\Lo \Subset G$ and for any compact subset $\Ko \Subset X$, there exist 
	\begin{itemize}
		\item a collection $\Coseco$ of open sets in $X$, 
		\item subgroups $G_{\Cose} \in \mathcal{F}$ for $\Cose \in \Coseco$, and
		\item locally constant functions $\La_{\Cose} \colon \Cose \to G / G_{\Cose}$ for $\Cose \in \Coseco$, called \emph{labeling functions}, 
	\end{itemize}
	satisfying the following conditions: 
	{
		\begin{enumerate}
			\renewcommand{\labelenumi}{\textup{(\theenumi)}} \renewcommand{\theenumi}{Lo}\item \label{item:lem:BLR-reformu:Lo}
			For any $x \in \Ko$, there exists $\Cose \in \Coseco$ such that $\alpha_{\Lo} (x) \subseteq \Cose$.
			
			\renewcommand{\labelenumi}{\textup{(\theenumi)}} \renewcommand{\theenumi}{Mu}\item \label{item:lem:BLR-reformu:Mu}
			We have $\mult(\Coseco) \leq d+1$. 
			
			\renewcommand{\labelenumi}{\textup{(\theenumi)}} \renewcommand{\theenumi}{Eq}\item \label{item:lem:BLR-reformu:Eq}
			Each $\La_{\Cose}$ is \emph{$\Lo$-equivariant} in the following sense: for any $x \in \Cose$ and for any $g \in \Lo$, if $\alpha_{g} (x) \in \Cose$ then $\La_{\Cose} \left(\alpha_{g} (x) \right) = g \cdot \La_{\Cose}(x)$. 
			
		\end{enumerate}
	}
	Moreover, this statement still holds if the following condition is added to the above list: 
	There exists a finite subset $\Bo \Subset G$ such that 	
	{
		\begin{enumerate}
			\renewcommand{\labelenumi}{\textup{(\theenumi)}} \renewcommand{\theenumi}{Bo}\item \label{item:lem:BLR-reformu:Bo}
			for any $\Cose \in \Coseco$ and for any $x, y \in \Cose$, we have $\Ne_{\Cose}(x) \in \Bo \cdot  \Ne_{\Cose}(y) $. 	
		\end{enumerate}
	}
\end{Lemma}

Here the labels (Lo), (Mu), (Eq), and (Bo) stand for, respectively, ``long'', ``multiplicity'', ``equivariant'', and ``bounded''. 
We will also write 
\Refcd{Lo}, \Refcd{Mu}, \Refcd{Eq}, and \Refcd{Bo} 
to specify the parameters used in these conditions and the fact that they come from Definition~\ref{\Topic}.

{
	\newcommand{\Ausehat}{{\widehat{\Ause}}}
\begin{proof}
	It suffices to prove the ``if'' direction without condition~\eqref{item:lem:BLR-reformu:Bo} and then prove the ``only if'' direction with condition~\eqref{item:lem:BLR-reformu:Bo} added. 
	
	To prove the ``if'' direction, we fix an arbitrary finite subset $\Lo 
	\Subset G$ and an arbitrary compact subset $\Ko \Subset X$. By our 
	assumption, there exist a collection $\Coseco$ of open sets in $X$ as well 
	as subgroups $G_{\Cose} \in \mathcal{F}$ and locally constant functions 
	$\La_{\Cose} \colon \Cose \to G / G_{\Cose}$ for $\Cose \in \Coseco$ 
	satisfying the conditions~\eqref{item:lem:BLR-reformu:Lo}, 
	\eqref{item:lem:BLR-reformu:Mu}, and~\eqref{item:lem:BLR-reformu:Eq} above. 
	For each $\Cose \in \Coseco$, define
	\[
		\Ause_{\Cose} = \left\{ (g, x) \in G \times X \colon \alpha_g (x) \in  
		{\La_{\Cose}}^{-1} \left( g \cdot G_{\Cose} \right) \right\} \; .
	\]
	Notice that $\Ause_{\Cose}$ is an open subset in $G \times X$, as 
	$\La_{\Cose}$ is locally constant. Now define $\Auseco = \left\{ \Delta_g 
	\left( \Ause_{\Cose} \right) \colon \Cose \in \Coseco , g \in G \right\}$. 
	Observe that for any $\Cose \in \Coseco$ and for any $g \in G$, we have
	\[
		\Delta_g \left( \Ause_{\Cose} \right) = \left\{ (g', x) \in G \times X 
		\colon \alpha_{g'} (x) \in  {\La_{\Cose}}^{-1} \left( g' g \cdot 
		G_{\Cose} \right) \right\} \; .
	\]
	Therefore:
	\begin{itemize}
		\item if $g \in G_{\Cose}$, then $\Delta_g \left( \Ause_{\Cose} \right) = \Ause_{\Cose}$, and 
		\item if $g \notin G_{\Cose}$, then $\Delta_g \left( \Ause_{\Cose} \right) \cap \Ause_{\Cose} = \varnothing$ as $g' \cdot G_{\Cose} \not= g'g \cdot G_{\Cose}$ for any $g' \in G$. 
	\end{itemize}
	It follows that $\Auseco$ satisfies the conditions~\eqref{item:def:BLR:Eq} 
	and \eqref{item:def:BLR:St} of Definition~\ref{def:BLR}. For any $x \in K$, 
	by our assumption, there exists $\Cose \in \Coseco$ such that $\alpha_{\Lo} 
	(x) \subseteq \Cose$ and $\La_{\Cose} \left( \alpha_{g} (x) \right) = g 
	\cdot \La_{\Cose} (x) $ for any $g \in \Lo$. This implies that $\Lo 
	\times  \{x\} \subseteq \Delta_{g_0} \left( \Ause_{\Cose} \right)$ for any 
	$g_0 \in G$ with $g_0 \cdot G_{\Cose} = \La_{\Cose} (x)$, which shows 
	condition~\eqref{item:def:BLR:Lo} of Definition~\ref{def:BLR}. Lastly, for 
	any $(g, x) \in G \times X$, we have 
	\begin{align*}
		&\ \left| \left\{ \Delta_{g'} \left( \Ause_{\Cose} \right) \colon \Cose \in \Coseco, g' \in G,  (g,x) \in \Delta_{g'} \left( \Ause_{\Cose} \right) \right\} \right| \\
		=&\ \left| \left\{ \Cose \in \Coseco \colon \text{there exists } g' \in 
		G \text{ such that } \alpha_{g} (x) \in  {\La_{\Cose}}^{-1} \left( g g' 
		\cdot G_{\Cose} \right)  \right\} \right| \\
		=&\ \left| \left\{ \Cose \in \Coseco \colon \alpha_{g} (x) \in \Cose \right\} \right| \\
		\leq &\ \mult\left( \Coseco \right) \leq d+1 \, .
	\end{align*}
	This establishes condition~\eqref{item:def:BLR:Mu} of 
	Definition~\ref{def:BLR}. Thus, $\eqasdim(\alpha, \mathcal{F}) \leq d$. 
	
	To prove the ``only if'' direction with 
	condition~\eqref{item:lem:BLR-reformu:Bo} added, we again fix an arbitrary 
	finite subset $\Lo \Subset G$ and an arbitrary compact subset $\Ko \Subset 
	X$. By our assumption, there exists a collection $\Auseco$ of open sets in 
	$G 
	\times X$ satisfying the conditions~\Refc[\BLRdef]{Lo}{(\Lo \cup \{\grpid\} 
	) \cdot \Lo,\Ko},
	\Refc[\BLRdef]{Mu}{d},
	\Refc[\BLRdef]{Eq}{(\Lo \cup \{\grpid\} ) \cdot \Lo},
	and~\Refc[\BLRdef]{St}{\mathcal{F}}.
	By Lemma~\ref{lemma:BLR-cofinite}, 
	we may assume that there exists a finite collection $\left\{ \Ause_i 
	\right\}_{i \in 
	I}$ of representatives of the $G$-action on $\Auseco$, and for each $i \in 
	I$, we have open subsets $\Ause_{i,g} \subseteq X$ for any $g \in G$ and  a 
	finite subset $\Fi_i \Subset G$ which give rise to a decomposition
	\[
		\Ause_i = \bigsqcup_{g \in G} \{g\} \times \Ause_{i,g} = \bigsqcup_{g 
		\in \Fi_i \cdot \setstab_\Delta \left(\Ause_i \right)} \{g\} \times 
		\Ause_{i,g}
		\, .
	\]

	Next, observe that, by our assumption, 
	for any $g , g' \in G$, we have
	\begin{itemize}
		\item $\alpha_{g} \left( \Ause_{i,g} \right) = \alpha_{g'} \left( \Ause_{i,g'} \right) $ if $g^{-1} g' \in \setstab_\Delta(\Ause_i)$, and 
		\item $\alpha_{g} \left( \Ause_{i,g} \right) \cap \alpha_{g'} \left( \Ause_{i,g'} \right) = \varnothing$ if $g^{-1} g'  \notin \setstab_\Delta(\Ause_i)$. 
	\end{itemize}
	For each $i \in I$, define
	\[
		\Ausehat_{i,g} = \bigcap_{g' \in (\Lo \cup \{\grpid\} ) \cdot g} 
		\Ause_{i,g'} \subseteq \Ause_{i,g} \, .
	\]
	Notice that if $g,h \in G$ satisfy $h^{-1}g \in \setstab_\Delta(\Ause_i)$ 
	then 
	$\alpha_{g} \left( \Ausehat_{i,g} \right) = \alpha_{h} \left( 
	\Ausehat_{i,h} \right)$, so the open set $\alpha_{g} \left( \Ausehat_{i,g} 
	\right)$ depends only on the equivalence class of $g$ in $ G / 
	\setstab_\Delta(\Ause_i)$. We can thus define
	\[
		\Cose_i = \bigsqcup_{[g] \in G / \setstab_\Delta(\Ause_i)} \alpha_{g} 
		\left( \Ausehat_{i,g} \right) \; .
	\]
	Now, define  
	\[ 
	G_{\Cose_i} = \setstab_\Delta(\Ause_i) \in \mathcal{F} \; . 
	\]
	 Finally, define locally constant functions $\La_{\Cose_i} \colon \Cose_i 
	 \to G / G_{\Cose_i}$ which take the value $g \cdot G_{\Cose_i}$ on each 
	 subset 
	 $\alpha_{g} \left( \Ausehat_{i,g} \right)$ (noting that by the above, 
	 this function is indeed well defined). We claim that the collection 
	 $\Coseco = \left\{ \Cose_i \colon i \in I \right\}$, together with the 
	 subgroups $G_{\Cose_i}$ and the functions $\La_{\Cose_i}$, for $i \in I$, 
	 satisfies the conditions in Lemma~\ref{lem:BLR-reformu}. To this end, in 
	 order to prove condition~\eqref{item:lem:BLR-reformu:Lo}, we see that for 
	 any $x \in \Ko$, by our assumption, there exist $i \in I$ and $g \in G$ 
	 such 
	 that $\left( (\Lo \cup \{\grpid\} ) \cdot \Lo \right) \times \{x\} \subset 
	 \Delta_g \left( \Ause_i \right)$, that is, 
 	$\alpha_{g^{-1}} (x) \in \Ause_{i,g'}$ for any $g' \in (\Lo \cup \{\grpid\} 
 	) \cdot \Lo \cdot g$. This implies that $\alpha_{g^{-1}} (x) \in 
 	\Ausehat_{i,g'}$ for any $g' \in \Lo \cdot g$. Hence, for any $g'' \in 
 	\Lo$, we have $\alpha_{g''} (x) \in \alpha_{g'' g} \left( \Ausehat_{i,g'' 
 	g} \right) \subseteq \Cose_i$, as desired. 
	Condition~\eqref{item:lem:BLR-reformu:Mu} in 
	Lemma~\ref{lem:BLR-reformu} follow from its counterpart in 
	Definition~\ref{def:BLR} by a computation similar to the one given in the 
	proof of the other direction. To prove 
	condition~\eqref{item:lem:BLR-reformu:Eq}, we observe that for any $i \in 
	I$, for any $x \in \Cose_i$ and for any $g \in \Lo$ satisfying $\alpha_{g} 
	(x) \in \Cose_i$, if we pick a representative $g' \in G$ for the class 
	$\La_{\Cose_i}(x)$, then, by our 
	construction,
	\[
		x \in \alpha_{g'} \left( \Ausehat_{i,g'} \right) = \bigcap_{g'' \in 
		(\Lo \cup \{\grpid\} )  \cdot  g'} \alpha_{g'} \left( \Ause_{i,g''} 
		\right) \; .
	\]
	Therefore,
	\[
		\alpha_{g} (x) \in \bigcap_{g'' \in (\Lo \cup \{\grpid\} )  \cdot g'} 
		\alpha_{g g'} \left( \Ause_{i,g''} \right) \subseteq \alpha_{g g'} 
		\left( \Ause_{i,gg'} \right) \; .
	\]
	Because our construction implies $\Cose_i \cap \alpha_{g g'} 
	\left( \Ause_{i,gg'} \right) = \alpha_{g g'} \left( \Ausehat_{i,gg'} 
	\right)$, we obtain $\La_{\Cose_i} \left(\alpha_{g} (x) \right) = [g g'] = 
	g \cdot 
	\La_{\Cose_i}(x)$, as required. Finally, 
	condition~\eqref{item:lem:BLR-reformu:Bo} obviously 
	holds with the finite set
	\[
		\Bo = \bigcup_{i \in I} \Fi_i \Fi_i ^{-1} \Subset G \; .
	\]
	Therefore the collection $\Coseco$, together with the subgroups $G_{\Cose_i}$ and the functions $\La_{\Cose_i}$, for $i \in I$, satisfies the conditions (including \eqref{item:lem:BLR-reformu:Bo}) in Lemma~\ref{lem:BLR-reformu}. 
\end{proof}
}

This characterization yields the following relation between equivariant asymptotic dimension and long thin covering dimension. 

\begin{Thm}
	Let $\mathcal{F}$ be a family of subgroups of $G$ closed under conjugation and taking subgroups. Then we have the following dichotomy: 
	\begin{enumerate}
		\item if $\mathcal{F}$ contains all finitely generated subgroups of all point stabilizers of the action $\alpha$, then $\eqasdim (\alpha, \mathcal{F})  \leq \dimltc (\alpha)$; 
		\item otherwise $\eqasdim (\alpha, \mathcal{F}) = \infty $. 
	\end{enumerate}
\end{Thm}

\begin{proof}
	In the first case where $\mathcal{F}$ contains all finitely generated 
	subgroups of all point stabilizers of $\alpha$, it suffices to prove the 
	inequality when $\dimltc (\alpha) = d < \infty$ for some nonnegative 
	integer $d$. This follows directly from observing that, given any finite 
	subset $\Lo \Subset G$ and any compact subset $\Ko \Subset X$, the 
	conditions in Definition~\ref{def:ltc-dim}, in the equivalent formulation 
	given in 
	Proposition~\ref{prop:ltc-dim-new}(\ref{item:prop:ltc-dim-new:bounded}),   
	with $\Thseco = \{X\}$, 
	implies those in Lemma~\ref{lem:BLR-reformu} by 
	converting the near orbit selection functions of Definition~\ref{def:ltc-dim} into the labeling functions of Lemma~\ref{lem:BLR-reformu} in the following way: 
	given a near orbit selection function $\Ne_{\Cose} \colon \Cose \to X$ as 
	specified in Definition~\ref{def:ltc-dim} and fixing a point $x_0 \in 
	\Ne_{\Cose} (\Cose)$, it follows 
	that $\Ne_{\Cose} (\Cose) \subseteq \alpha_{\Bo} \left( x_0 \right)$. By 
	our assumption, the finitely generated group $G_{\Cose} := \left\langle 
	G_{x_0} \cap \left(\Bo^{-1} \left( \Lo \cup \{\grpid\} \right) \Bo\right) 
	\right\rangle$ is in $\mathcal{F}$. Also observe that the map 
	\[
		f_{\Cose} \colon \alpha_{\Bo} \left( x_0 \right) \to G / G_{\Cose} 
	\]
	given by
	\[
	f_{\Cose} ( \alpha_{g} \left( x_0 \right) ) =  g \cdot G_{\Cose} \quad 
	\text{for any } g \in \Bo
	\]
	is well-defined and $\Lo$-equivariant by our construction of $G_{\Cose}$. Hence the composition $f_{\Cose} \circ \Ne_{\Cose} \colon \Cose \to G / G_{\Cose}$ is a labeling function satisfying condition~\eqref{item:lem:BLR-reformu:Eq} of Lemma~\ref{lem:BLR-reformu}.

	In the second case, let $x \in X$ be a point whose point stabilizer $G_x$ 
	contains a finitely generated subgroup that is not in $\mathcal{F}$. Let 
	$\Lo$ be a finite generating set of that group that includes $\grpid$ and 
	let $\Ko = \{x\}$. Observe that for any subgroup $H$ of $G$, if there 
	exists an $\Lo$-equivariant function $f \colon \{x\} \to G/H$, then 
	$\langle \Lo \rangle \leq g H g^{-1}$, where $g \cdot H = f(x)$, and thus 
	$H \notin \mathcal{F}$. This precludes any collection $\Coseco$ of open 
	sets of $X$ that satisfy both conditions~\eqref{item:lem:BLR-reformu:Lo} 
	and~\eqref{item:lem:BLR-reformu:Eq} of Lemma~\ref{lem:BLR-reformu}. 
	Therefore $\eqasdim (\alpha, \mathcal{F}) = \infty $. 
\end{proof}

\begin{Rmk}
	Although we will not need this fact, we point out that using the equivalent characterization in Lemma~\ref{lem:BLR-reformu}, one can establish the equation 
	\[
		\eqasdim(\alpha, \mathcal{F}) = \eqasdim(\alpha, \mathcal{F}_{\mathrm{f.g.}}) \; ,
	\]
	where $\mathcal{F}_{\mathrm{f.g.}}$ consists of all the finitely generated groups in $\mathcal{F}$. 
	
	To see this, we first apply a compactness argument (as in Section~\ref{sec:LTC}) to shrink members of $\Coseco$ so that each $\Ne_{\Cose} (\Cose)$ is finite while \Refcd{Lo} is still satisfied (\Refcd{Mu} and \Refcd{Eq} are automatically satisfied). 
	If we let $H_{\Cose}$ be the finitely generated subgroup $\left\langle G_{\Cose} \cap \left( \left(\Lo \cdot \Ne_{\Cose} (\Cose)  \right)^{-1} \cdot \left(\Lo \cdot \Ne_{\Cose} (\Cose) \right) \right) \right\rangle$, then by Lemma~\ref{Lemma:injective-quotient} below, the quotient map $G / H_{\Cose} \to G / G_{\Cose}$ maps $\Lo \cdot \Ne_{\Cose} (\Cose) \cdot H_{\Cose}  / H_{\Cose}$ injective onto $\Lo \cdot \Ne_{\Cose} (\Cose) \cdot G_{\Cose}  / G_{\Cose}$, which allows us to define an inverse map from $\Lo \cdot \Ne_{\Cose} (\Cose) \cdot G_{\Cose}  / G_{\Cose}$ to $\Lo \cdot \Ne_{\Cose} (\Cose) \cdot H_{\Cose}  / H_{\Cose}$ that is $\Lo$-equivariant on $\Ne_{\Cose} (\Cose) \cdot G_{\Cose}  / G_{\Cose}$, and thus we may replace $G_{\Cose}$ by $H_{\Cose}$ and $\Ne_{\Cose}$ by its composition with this inverse map, so that \Refcd{Eq} still holds. 
\end{Rmk}

The following technical lemma is needed for the proof of Theorem 
\ref{thm:relative-bound-ltc} (the main result of this section), in order to construct refinements of covers. 
The reader is encouraged to jump ahead and read Theorem 
\ref{thm:relative-bound-ltc} as well as the informal discussion thereafter, before coming back to this lemma. 

\begin{Lemma} \label{lem:relative-bound-ltc-inductive}
	Let $d$ be a natural number, let $H$ be a subgroup of $G$, 
	let $\Ko \Subset X$ be a compact subset and let $\Lo, \Fi \Subset G$ be 
	finite subsets. For $g \in G$ and any subgroup $H < G$, we write  $H^{g} = 
	gH g^{-1}$.
	Assume $\dimltc \left( X, {\alpha_{|_{H^{g}}}} \right) \leq d$ for any $g 
	\in \Fi$. 
	Then there exists a 
	natural number $\Binu$ 
	such that the following holds:
	
	For 
	any compact subset $\Ko' \Subset \Ko$, for
	any finite open cover $\Thseco$ of $\Ko'$ in $X$,  
	and for any $\Lo$-equivariant (in the sense of 
	\eqref{item:lem:BLR-reformu:Eq} of Lemma~\ref{lem:BLR-reformu}) locally 
	constant function $\widetilde{\Ne} \colon \bigcup \Thseco \to (\Fi H)/H 
	\subseteq G/H$,  
	there exist 
	\begin{itemize}
		\item a finite collection $\Coseco$ of open subsets of $X$, and
		\item locally constant functions $\Ne_{\Cose} \colon \Cose \to \bigcup \Thseco$ for $\Cose \in \Coseco$, 
	\end{itemize}
	satisfying the conditions~\Refc[\LTCdef]{Lo}{\{e\},\Ko'}, \Refcd[\LTCextra]{Muplus}, \Refcd[\LTCdef]{Eq}, \Refcd[\LTCdef]{Th}, 
	\Refcd[\LTCdef]{Ca}, 
	and $\widetilde{\Ne} \circ \Ne_{\Cose} = \widetilde{\Ne}$ for any $\Cose \in \Coseco$. 
\end{Lemma}

\begin{proof}
	We fix $d$, $H$, $\Lo$ (which we assume, without loss of generality, is symmetric and contains $\grpid$) and $\Ko$, and apply induction on the cardinality of $\Fi$. 
	The initial case of $\Fi = \varnothing$ is trivial since it forces $\bigcup \Thseco = \varnothing$ via the assumption on $\widetilde{\Ne}$. 
	
	Now given a positive integer $k$ and assuming the statement holds for any $\Fi' \Subset G$ with $|\Fi'| < k$, we would like to show it also holds for an arbitrary $\Fi_{} \Subset G$ with $|\Fi_{}| = k$. Without loss of generality, we may assume the map $G \to G / H$ is injective on $\Fi$, as otherwise we would be able to find $\Fi' \subseteq \Fi$ with $|\Fi'| < k$ but $\Fi' H = \Fi H$, which would allow us to directly apply the inductive assumption on $\Fi'$ to obtain the desired conclusion on $\Fi$. 
	
	Before delving into the technical proof, let us summarize the strategy in the next few paragraphs, in an informal style:

	We partition $\Fi$ into an arbitrary singleton $\left\{ g_+ \right\}$ and 
	its complement $\Fi_0$, and, with $\Ko'$, $\Thseco$, and $\widetilde{\Ne}$ 
	given, we split  $\Ko'$ into $\Ko'_+$ and $\Ko'_0$, and $\Thseco$ into 
	$\Thseco_+$ and $\Thseco_0$, essentially so that $\Ko'_+ \subseteq \bigcup 
	\Thseco_+ \subseteq \widetilde{\Ne}^{-1} \left( g_+ \right)$ and $\Ko'_0 
	\subseteq \bigcup \Thseco_0 \subseteq \widetilde{\Ne}^{-1} \left( \Fi_0 
	\right)$. 
	Eventually, our collection $\Coseco$ (together with the near orbit 
	selection functions $\left\{ \Ne_{\Cose} \right\}_{\Cose \in \Coseco}$) 
	will also be constructed by merging two collections $\Coseco_+ = \left\{ 
	\Cose_{+, i} \colon i \in \intervalofintegers{1}{m} \right\}$ and 
	$\Coseco_0$ that cover $\Ko'_+$ and $\Ko'_0$ respectively and, when 
	accompanied by their respective near orbit selection functions $\left\{ 
	\Ne_{+,i} \right\}_{i \in \intervalofintegers{1}{m}}$ and $\left\{ 
	\Ne_{0,\Cose} \right\}_{\Cose \in \Coseco_0}$, are ``thin'' for $\Thseco_+$ 
	and $\Thseco_0$ respectively. 
	Naturally, $\Coseco_+$ will be obtained from the assumption that $\dimltc \left( X, {\alpha_{|_{H^{g_+}}}} \right) \leq d$ (with the parameter $\Lo$ replaced by some $\Lo_+ \subseteq H^{g_+}$, whose choice will be discussed later), while $\Coseco_0$ will be obtained by applying the inductive assumption on $\Fi_0$. 
	
	However, care must be taken when merging $\Coseco_+$ and $\Coseco_0$: 
	if we were to naively define $\Coseco := \Coseco_+ \cup \Coseco_0$, it 
	would typically fail to satisfy \Refcd[\LTCextra]{Muplus}, 
	as some members of $\Coseco_+$ may be $(\alpha, \Lo)$-connected to some 
	members of $\Coseco_0$, possibly resulting in $\mult_{\alpha,\Lo} 
	(\Coseco)$ exceeding the desired bound of $d+1$ (and this excess would 
	accumulate as we continue the induction). 
	
	The approach we take to address this issue is 
	inspired by that of the $r$-saturated union of two covers in 
	\cite[Proposition~24]{BellDranishnikov2008} and involves the following 
	procedures:
	\begin{enumerate}[label=(\alph*)]
		\item \label{lem:relative-bound-ltc-inductive:proof-strategy::components} replace each member $\Cose$ of $\Coseco_0$ by its $\left(\Ne_{0,\Cose}, \Lo\right)$-connected components as in the proof of \eqref{item:prop:ltc-dim-new:orig} $\Rightarrow$ 
		\eqref{item:prop:ltc-dim-new:connected} in Lemma~\ref{prop:ltc-dim-new}; 
		\item \label{lem:relative-bound-ltc-inductive:proof-strategy::capture} 
		let each member $\Cose_{+,i}$ of $\Coseco_+$ capture those members of 
		the new $\Coseco_0$ that are ``$(\alpha, \Lo)$-connected'' to it like a 
		nucleus captures electrons, forming a larger ``atom-like'' open set 
		$\widehat{\Cose}_{i}$; 
		\item \label{lem:relative-bound-ltc-inductive:proof-strategy::Coseco} define $\Coseco$ to consist of these ``atoms'' as well as the ``free electons'', that is, the ``uncaptured'' members of $\Coseco_0$. 
	\end{enumerate}
	Now, with this setup, the aforementioned problem posed by $(\alpha, 
	\Lo)$-connections between members of $\Coseco_+$ and members of 
	$\Coseco_0$ is replaced with that posed by $(\alpha, \Lo)$-connections 
	between the ``atoms'' and the ``free electrons''. 
	The key point here is that the latter kind of $(\alpha, \Lo)$-connections 
	can be effectively controlled by tweaking the parameters controlling the 
	``lengths'' of $\Coseco_+$ and the ``thinness'' of $\Coseco_0$. 
	A bit more specifically, we want to ensure the following: For any 
	$\Cose_{+,i}, \Cose_{+,i'} \in \Coseco_+$, for any $\Cose, \Cose' \in 
	\Coseco_0$, for any $x \in \Cose_{+,i}$ for any $x' \in \Cose_{+,i'}$, for 
	any $y,z \in \Cose$, for any $y',z' \in \Cose'$, and for any $g, g', g'' 
	\in \Lo$ with $y = \alpha_g (x)$, $y' = \alpha_{g'} (x')$ and $z' = 
	\alpha_{g''} (z)$, 
	there exists $h$ in the aforementioned $\Lo_+$ such that 
	$x$ and $\alpha_{h} (x')$ are close to each other relative to the ``thinness'' of $\Coseco_+$. 
	To do this, we need to choose the parameters wisely:  
	\begin{enumerate*}
		\item after Step~\ref{lem:relative-bound-ltc-inductive:proof-strategy::components} above, it is not hard to see that $\Coseco_0$ now satisfies \Refc[\LTCextra]{Coplus}{\Lo, \Binu_0} for some natural number $\Binu_0$; 
		\item thus by making $\Coseco_0$ ``very thin'' relative to the ``thinness'' of $\Coseco_+$, we can ensure in the above that there are $\widehat{g}, \widehat{g}' \in \Lo^{\Binu_0}$ such that $z$ is sufficiently close to $\alpha_{\widehat{g}} (y)$ and $z'$ is sufficiently close to $\alpha_{\widehat{g}'} (y')$, so that $x$ and $\alpha_{ g^{-1} \widehat{g}^{-1} g'' \widehat{g}' g'} (x')$ are close to each other; 
		\item the parameter $\Lo_+$ needs to be large enough to include all such products $g^{-1} \widehat{g}^{-1} g'' \widehat{g}' g'$ (and thus depends on $\Binu_0$). 
	\end{enumerate*}
	
	We conclude this informal discussion by the following diagram that indicates the dependencies among the various parameters used in the proof, and the steps of the proofs in which they are introduced. The letter $\Ko$ with various subscripts and superscripts always refers to compact subsets of $X$, calligraphic capital letters are used to denote collections of open subsets (with the non-caligraphic counterparts, possibly with various subscripts or superscripts, denote elements of those collections), and the letter $\Ne$ denotes orbit selection functions or labeling functions (with values either in $X$ or in a coset space of the group). 
	
	\

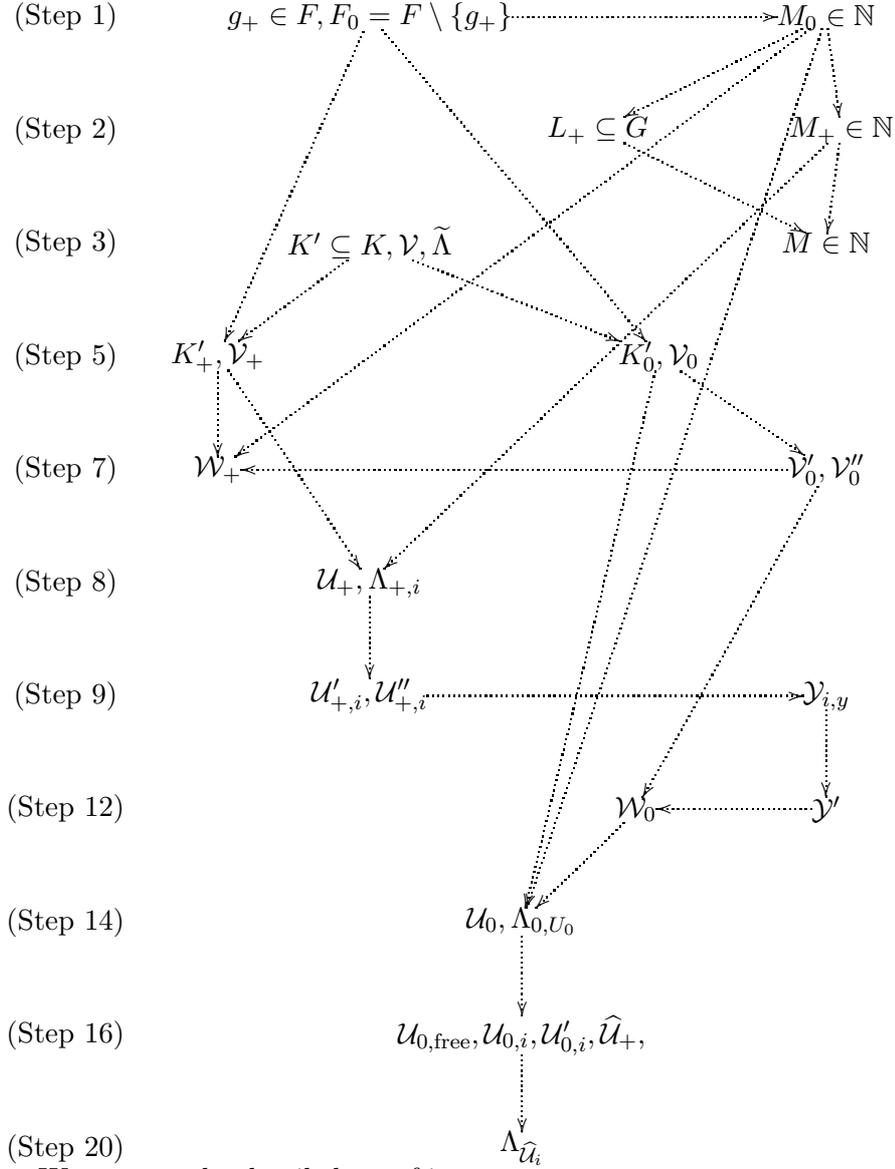
\begin{xy}
	(0,140)*{\textrm{(Step 1)}};
	(40,140)*{\displaystyle g_+ \in F ,  F_0 = F \setminus \{g_+\} }="g_+";
	(100,140)*{\displaystyle M_0 \in \N}="M_0";
	(102,125)*{\displaystyle M_+ \in \N}="M_+";
	(70,125)*{\displaystyle L_+ \subseteq G}="L_+";
	(0,125)*{\textrm{(Step 2)}};
	{\ar@{.>} "g_+";"M_0"};
	{\ar@{.>} "M_0";"M_+"};
	{\ar@{.>} "M_0";"L_+"};
	(100,110)*{\displaystyle M \in \N}="M";
		{\ar@{.>} "L_+";"M"};
	(40,110)*{K' \subseteq K ,  \mathcal{V}  ,  \widetilde{\Ne}}="K'-etc";
	{\ar@{.>} "M_+";"M"};
	(0,110)*{\textrm{(Step 3)}};
	(20,95)*{K_+'  ,  \mathcal{V}_+}="K_+'";
	(78,95)*{K_0' , \mathcal{V}_0}="K_0'";
	(0,95)*{\textrm{(Step 5)}};
	{\ar@{.>} "g_+";"K_0'"};
	{\ar@{.>} "K'-etc";"K_+'"};
	{\ar@{.>} "g_+";"K_+'"};
			{\ar@{.>} "K'-etc";"K_0'"};
	(0,80)*{\textrm{(Step 7)}};
	(100,80)*{\mathcal{V}_0' , \mathcal{V}_0''}="V_0'";
	(20,80)*{\mathcal{W}_+}="W_+";
		{\ar@{.>} "K_0'";"V_0'"};
		{\ar@{.>} "V_0'";"W_+"};
		{\ar@{.>} "K_+'";"W_+"};
		{\ar@{.>} "M_0";"W_+"};
	(0,65)*{\textrm{(Step 8)}};
	(40,65)*{\mathcal{U}_+,\Ne_{+, i}}="U_+";
		{\ar@{.>} "K_+'";"U_+"};
		{\ar@{.>} "M_+";"U_+"};
	(0,50)*{\textrm{(Step 9)}};
	(40,50)*{\mathcal{U}_{+,i}',\mathcal{U}_{+,i}''}="U_+'";
		{\ar@{.>} "U_+";"U_+'"};
	(100,50)*{\mathcal{Y}_{i,y}}="Y_i";
			{\ar@{.>} "U_+'";"Y_i"};
	(0,35)*{\textrm{(Step 12)}};
	(100,35)*{\mathcal{Y}'}="Y'";
	(75,35)*{\mathcal{W}_0}="W_0";
		{\ar@{.>} "Y_i";"Y'"};
		{\ar@{.>} "Y'";"W_0"};
		{\ar@{.>} "V_0'";"W_0"};
	(0,20)*{\textrm{(Step 14)}};
	(60,20)*{\mathcal{U}_0,\Ne_{0, \Cose_0}}="U_0";
		{\ar@{.>} "K_0'";"U_0"};
		{\ar@{.>} "W_0";"U_0"};
		{\ar@{.>} "M_0";"U_0"};
	(0,5)*{\textrm{(Step 16)}};
	(60,5)*{\mathcal{U}_{0,\mathrm{free}} , \mathcal{U}_{0, i} ,  \mathcal{U}_{0, i}' , \widehat{\mathcal{U}}_+,}="U_0-subdivide";
		{\ar@{.>} "U_0";"U_0-subdivide"};
	(0,-10)*{\textrm{(Step 20)}};
	(60,-10)*{\Ne_{\widehat{\mathcal{U}}_i}}="NE20";
	{\ar@{.>} "U_0-subdivide";"NE20"};
\end{xy}

	We present the detailed proof in steps: 
	\begin{enumerate}[itemindent=*,leftmargin=1em,label=\textit{Step~(\arabic*).},ref=\arabic*]
		\item \label{lem:relative-bound-ltc-inductive:proof::Bo_0} We fix $g_+ \in \Fi_{}$ and let $\Fi_0 := \Fi_{} \setminus \left\{ g_+ \right\}$, to which we may apply the inductive assumption in order to obtain a 
		natural number $\Binu_0$ 
		that satisfies the condition in the lemma (with $\Fi$ and $\Binu$ replaced by $\Fi_0$ and $\Binu_0$). 
		
		\item \label{lem:relative-bound-ltc-inductive:proof::Bo_0} 
		Let
		\[
		\Lo_+ := \Lo^{2 \Binu_0 + 3} \cap H^{g_+} \; . 
		\]
		By the assumption that $\dimltc \left( X, {\alpha_{|_{H^{g_+}}}} \right) \leq d$ and in view of Lemma~\ref{prop:ltc-dim-new}\eqref{item:prop:ltc-dim-new:cover-Ko}, 
		there exists a
		natural number $\Binu_+$ such that for 
		any compact subset $\Ko_+$ of $X$ and for
		any finite open cover $\Auseco_+$ of $\Ko_+$ in $X$,  
		there exist a finite collection $\Coseco_+$ of open sets in $X$, and locally constant functions $\Ne_{+, \Cose} \colon \Cose \to X$ for $\Cose \in \Coseco_+$
		satisfying \Refc[\LTCdef]{Lo}{\{e\}, \Ko_+}, 
		\Refc[\LTCextra]{Muplus}{d, \Lo_+}, \Refc[\LTCdef]{Eq}{\Lo_+}, \Refc[\LTCdef]{Th}{\Auseco_+}, and~\Refc[\LTCextra]{Coplus}{\Lo_+,\Binu_+}.  
		
		
		\item \label{lem:relative-bound-ltc-inductive:proof::Bo} 
		Define 
		\[
			\Binu := \max \left\{\left| \Lo_+^{\Binu_+} \right| \left| \Lo^{ \Binu_0 + 1} \right| , \Binu_0 \right\}  
		\]
		We need to prove this choice of $\Binu$ satisfies the condition in the lemma. 
		To do this, we fix 
		a compact subset $\Ko' \Subset \Ko$, a finite open cover $\Thseco_{}$ of $\Ko'$ in $X$, 
		and an $\Lo$-equivariant locally constant function $\widetilde{\Ne}_{} \colon \bigcup \Thseco_{} \to (\Fi_{} H)/H \subseteq G/H$. 
		By replacing each $\Thse \in \Thseco$ by its fibers under the locally 
		constant function $\widetilde{\Ne}_{}|_{\Thse}$, we may assume without 
		loss of generality that $\widetilde{\Ne}_{} ({\Thse})$ is a singleton, 
		for any $\Thse \in \Thseco$. 
		
		\item  \label{lem:relative-bound-ltc-inductive:proof::checkers} 
		The following is an immediate consequence of the $\Lo$-equivariance of 
		$\widetilde{\Ne}_{}$, obtained by induction: for any $x \in \bigcup 
		\Thseco$ and for any $g_0, g_1, \ldots, g_n \in \Lo$, if $\alpha_{g_{j} 
		g_{j+1} \cdots g_{n}} (x) \in \bigcup \Thseco$ for $j = 0,1,\ldots, n$, 
		then $\widetilde{\Ne}_{} \left( \alpha_{g_{0} g_{1} \cdots g_{n}} (x) 
		\right) = g_{0} g_{1} \cdots g_{n} \cdot \widetilde{\Ne}_{} \left( x 
		\right) $.  
		
		\item \label{lem:relative-bound-ltc-inductive:proof::cutout}  
		Define disjoint compact subsets
		\begin{align*}
		\Ko'_+ &:= \Ko' \setminus \widetilde{\Ne}_{}^{-1} \left( (\Fi_0 H)/H \right) = \Ko' \cap \widetilde{\Ne}_{}^{-1} \left( g_+ H \right) \; , \\
		\Ko'_0 &:= \Ko' \cap \widetilde{\Ne}_{}^{-1} \left(  (\Fi_0 H)/H \right) = \Ko' \setminus \widetilde{\Ne}_{}^{-1} \left( g_+ H \right) \; .
		\end{align*}
		Also define collections of open subsets
		\begin{align*}
			\Thseco_+ &:= \left\{  \Thse \in \Thseco \colon \widetilde{\Ne}_{} ({\Thse}) = \left\{ g_+ H \right\} \right\}  \; , \\
			\Thseco_{0} &:= \left\{  \Thse \in \Thseco \colon 
			\widetilde{\Ne}_{} ({\Thse}) \not= \left\{ g_+ H \right\} 
			\right\}   
			\; .
		\end{align*} 
		It follows from our assumption in Step~\eqref{lem:relative-bound-ltc-inductive:proof::Bo} that $\Thseco =  \Thseco_+ \sqcup  \Thseco_0$. 
		We also have $\Ko' = \Ko'_+ \sqcup \Ko'_0$, $\Ko'_+ \subseteq \bigcup \Thseco_+$,  $\Ko'_0 \subseteq \bigcup \Thseco_{0}$  
		and $\bigcup \Thseco = \left(\bigcup \Thseco_+\right) \sqcup \left(\bigcup \Thseco_0\right)$. 

		\item  \label{lem:relative-bound-ltc-inductive:proof::loop-coarse} 
		We record the following fact for later use: For any $y, y' \in \bigcup 
		\Thseco$ and for any $g, g', g'' \in G$, if 
		\begin{itemize}
			\item $g \cdot \widetilde{\Ne}_{} \left( y \right) = \widetilde{\Ne}_{} \left( y' \right) $, 
			\item $\alpha_{g'} \left( y \right) \in \bigcup \Thseco_{+}$ and $\alpha_{g''} \left( y' \right) \in \bigcup \Thseco_{+}$, and 
			\item $\widetilde{\Ne}_{} \left(\alpha_{g'} \left(  y \right) \right) = {g'} \left( \widetilde{\Ne}_{} \left( y \right) \right)$ and $\widetilde{\Ne}_{} \left(\alpha_{g''} \left(  y' \right) \right) = {g''} \left( \widetilde{\Ne}_{} \left( y' \right) \right)$, 
		\end{itemize}
		then we have 
		\[
			g'' g \left(g'\right)^{-1} \in g_+ H g_+^{-1} \; . 
		\]
		
		Indeed, by the definition of $\Thseco_{+}$ in Step~\eqref{lem:relative-bound-ltc-inductive:proof::cutout}, we have  
		\[
			{g'} \left( \widetilde{\Ne}_{} \left( y \right) \right) = \widetilde{\Ne}_{} \left(\alpha_{g'} \left(  y \right) \right) = g_+ H = \widetilde{\Ne}_{} \left(\alpha_{g''} \left(  y' \right) \right) = {g''} \left( \widetilde{\Ne}_{} \left( y' \right) \right)
		\]
		The claim thus follows easily from comparing this equality with our assumption that $g \cdot \widetilde{\Ne}_{} \left( y \right) = \widetilde{\Ne}_{} \left( y' \right) $. 
		

		\item \label{lem:relative-bound-ltc-inductive:proof::Auseco_plus} 
		Applying the barycentric subdivision method of 
		Lemma~\ref{lem:cover-clumping}, we obtain a finite collection 
		$\Thseco'_0$ of open sets covering $\Ko'_0$ such that for any 
		$\Thse'_1, \ldots, \Thse'_n \in \Thseco'_0$, if $\bigcap_{i = 1}^{n} 
		\Thse'_i \not= \varnothing$, then there exists $\Thse \in \Thseco_0$ 
		such 
		that $\bigcup_{i = 1}^{n} \Thse'_i \subseteq \Thse$. 
		By Lemma~\ref{lem:shrunken-cover}, there exist a finite collection 
		$\Thseco''_0$ of open sets covering $\Ko'_0$ such that for any $\Thse'' 
		\in \Thseco''_0$, there exists $\Thse' \in \Thseco'_0$ such that 
		$\overline{\Thse''}  \subseteq \Thse'$. 
		Following Notation~\ref{notation:group-action}, we form a finite 
		collection of open sets
		\[
			\Auseco_{+} := \Thseco_{+} \wedge \alpha^\wedge_{\Lo^{\Binu_0 + 1}} \left( \Thseco'_{0} \cup \left\{ X \setminus \overline{\bigcup \Thseco''_0} \right\} \right)  \; ,
		\]
		which refines $\Thseco_{+}$. Observe that $\Thseco'_{0} \cup \left\{ X \setminus \overline{\bigcup \Thseco''_0} \right\}$ covers $X$, whence we have 
		\[
			\bigcup \Auseco_{+} = \left( \bigcup \Thseco_{+} \right) \cap \alpha^\cap_{\Lo^{\Binu_0 + 1}}  \left( \bigcup \left( \Thseco'_{0} \cup \left\{ X \setminus \overline{\bigcup \Thseco''_0} \right\} \right) \right) = \bigcup \Thseco_{+} \supseteq \Ko'_{+} \; .
		\]
		
		\item \label{lem:relative-bound-ltc-inductive:proof::Coseco_plus} By our choice of $\Binu_+$ in Step~\eqref{lem:relative-bound-ltc-inductive:proof::Bo_0}, 
		there exist a finite collection $\Coseco_+ = \left\{ \Cose_{+,1}, \ldots, \Cose_{+,m} \right\}$ of open sets in $X$, and locally constant functions $\Ne_{+, i} \colon \Cose_{+,i} \to X$ for $i \in \intervalofintegers{1}{m}$
		satisfying \Refc[\LTCdef]{Lo}{\{e\}, \Ko'_+}, 
		\Refc[\LTCextra]{Muplus}{d, \Lo_+}, \Refc[\LTCdef]{Eq}{\Lo_+}, \Refc[\LTCdef]{Th}{\Auseco_+}, and~\Refc[\LTCextra]{Coplus}{\Lo_+,\Binu_+}. 
		As $\Auseco_{+}$ refines $\Thseco_{+}$, it follows from Remark~\ref{rmk:ltc-dim-monotonous} that $\Coseco_{+}$ also satisfies \Refc[\LTCdef]{Th}{\Thseco_+}. In particular, for any $i \in \intervalofintegers{1}{m}$\vphantom{$\Cose \in \Coseco_+$}
		and for any $x \in \Cose_{+, i}$, we have 
		\[
		\widetilde{\Ne}_{} \circ 
		\Ne_{+, i} (x) = g_+ H = \widetilde{\Ne}_{} (x)
		\, .
		\]
		Note that \Refc[\LTCextra]{Coplus}{\Lo_+,\Binu_+} implies that $\Ne_{+, 
		i}  
		\left( \Cose_{+, i} \right)$ is finite.

		\item  \label{lem:relative-bound-ltc-inductive:proof::Lase_plus} 
		By Lemma~\ref{lem:shrunken-cover}, there exist precompact open subsets $\Cose'_{+, i}$ and $\Cose''_{+, i}$ satisfying $\overline{\Cose''_{+, i}} \subseteq {\Cose'_{+, i}} \subseteq \overline{\Cose'_{+, i}} \subseteq \Cose_{+, i}$ for $i \in \intervalofintegers{1}{m}$\vphantom{$\Cose \in \Coseco_+$}, such that $\Ko'_+ \subseteq \bigcup_{i = 1}^{m} \Cose''_{+, i}$. 
		Observe that for any $i \in \intervalofintegers{1}{m}$\vphantom{$\Cose 
		\in \Coseco_+$} and for any $y \in \Ne_{+, i}  \left( \Cose_{+, i} 
		\right)$, since $\Ne_{+, i}^{-1} (y)$ is relatively closed in 
		$\Cose_{+, i}$, the sets $\overline{\Cose'_{+, i}} \cap \Ne_{+, i}^{-1} 
		(y)$ and $\overline{\Cose''_{+, i}} \cap \Ne_{+, i}^{-1} (y)$ are 
		closed and thus 
		\[
			\Aulaseco_{i, y} := \left\{\Ne_{+, i}^{-1} (y) \cap {\Cose'_{+, i}}  , \,  \Ne_{+, i}^{-1} (y) \setminus \overline{\Cose''_{+, i}}  , \, X \setminus \left( \overline{\Cose'_{+, i}} \cap \Ne_{+, i}^{-1} (y) \right) \right\}
		\]
		forms an open cover of $X$. 
		
		\item  \label{lem:relative-bound-ltc-inductive:proof::refine_0} 
		Consider the finite open cover
		\[
			\Aulaseco :=  \bigwedge_{i = 1}^{m} \bigwedge_{y \in \Ne_{+, i}  \left( {\Cose_{+, i}} \right)}  \alpha^\wedge_{\Lo^{\Binu_0 + 1}} 
			\left( \Aulaseco_{i, y} \right)
		\]
		of $X$, 
		where both index sets are finite by Step~\eqref{lem:relative-bound-ltc-inductive:proof::Coseco_plus}. 
		Observe that any $\Aulase \in \Aulaseco$ satisfies that for any $i \in 
		\intervalofintegers{1}{m}$\vphantom{$\Cose \in \Coseco_+$}, for any $y 
		\in \Ne_{+, i}  \left( {\Cose_{+, i}} \right)$, and for any $g \in 
		\Lo^{\Binu_0 + 1}$, we have:  
		\begin{itemize}
			\item if $\alpha_g \left( \Aulase \right) \cap \overline{\Cose'_{+, i}} \cap \Ne_{+, i}^{-1} (y) \not= \varnothing$, then $\alpha_g \left( \Aulase \right) \subseteq \Ne_{+, i}^{-1} (y)$;
			
			\item if $\alpha_g \left( \Aulase \right) \cap \overline{\Cose''_{+, i}} \cap \Ne_{+, i}^{-1} (y) \not= \varnothing$, then $\alpha_g \left( \Aulase \right) \subseteq \Ne_{+, i}^{-1} (y) \cap \Cose'_{+, i}$. 
		\end{itemize}

		\item  \label{lem:relative-bound-ltc-inductive:proof::loop} 
		We record the following fact for later use: For any $y, y' \in \bigcup 
		\Thseco$, for any $n, n' \in \intervalofintegers{0}{\Binu_0}$, for any 
		$g, g_0, g_1, \ldots, g_{n}, g'_0, g'_1, \ldots, g'_{n'} \in \Lo$ and 
		for any $i \in \intervalofintegers{1}{m}$, if 
		\begin{itemize}
			\item there exists $\Aulase \in \Aulaseco$ such that $\left\{ 
			\alpha_g(y), y' \right\} \subseteq \Aulase$, 
			\item $g \cdot \widetilde{\Ne}_{} \left( y \right) = \widetilde{\Ne}_{} \left( y' \right) $, 
			\item $\alpha_{g_j g_{j+1} \cdots g_n} \left( y \right) \in \bigcup \Thseco$ for $j = 1, 2, \ldots, n$, 
			\item $\alpha_{g'_j g'_{j+1} \cdots g'_{n'}} \left( y' \right) \in \bigcup \Thseco$ for $j = 1, 2, \ldots, n'$, 
			\item $\alpha_{g_0 g_{1} \cdots g_n} \left( y \right) \in {\Cose_{+, i}}$,   and 
			\item $\alpha_{g'_0 g'_{1} \cdots g'_n} \left( y' \right) \in \overline{\Cose'_{+, i}}$ (note that $\overline{\Cose'_{+, i}} \subseteq \Cose_{+, i} \subseteq \bigcup \Auseco_{+} \subseteq \bigcup \Thseco$),
		\end{itemize}
		then we have 
		\begin{align*}
			& \alpha_g \circ \alpha_{\left(g_0 g_1 \cdots g_n\right)^{-1}} \left( \Ne_{+, i} \left( \alpha_{g_0 g_1 \cdots g_n} \left( y \right) \right) \right) \\
			& =  
			\alpha_{\left(g'_0 g'_1 \cdots g'_{n'}\right)^{-1}} \left( \Ne_{+, i} \left( \alpha_{g'_0 g'_1 \cdots g'_{n'}} \left( y' \right) \right) \right) \; .
		\end{align*}

		Indeed, let $z := \Ne_{+, i} \left( \alpha_{g'_0 g'_1 \cdots g'_n} \left( y' \right) \right)$. Since we have $\alpha_{g'_0 g'_{1} \cdots g'_n} \left( y' \right) \in \alpha_{g'_0 g'_{1} \cdots g'_n} \left( \Aulase \right) \cap \overline{\Cose'_{+, i}} \cap \Ne_{+, i}^{-1} (z)$, it follows from Step~\eqref{lem:relative-bound-ltc-inductive:proof::refine_0} that $\alpha_{g'_0 g'_{1} \cdots g'_n} \left( \Aulase \right) \subseteq \Ne_{+, i}^{-1} (z)$, 
		whence  
		\[
		\Ne_{+, i} \left(  \alpha_{g'_0 g'_{1} \cdots g'_{n'}  } \circ  \alpha_{g} \left( y \right) \right) = \Ne_{+, i} \left( \alpha_{g'_0 g'_{1} \cdots g'_{n'}  }  \left( y' \right) \right) = z \; . 
		\]
		

		On the other hand, 
		since $\Cose_{+, i} \subseteq \bigcup \Auseco_{+}$ by \Refc[\LTCdef]{Th}{\Auseco_+} in Step~\eqref{lem:relative-bound-ltc-inductive:proof::Coseco_plus} and thus $\Cose_{+, i} \subseteq \bigcup \Thseco_{+}$ by Step~\eqref{lem:relative-bound-ltc-inductive:proof::Auseco_plus}, 
		it follows from 
		Step~\eqref{lem:relative-bound-ltc-inductive:proof::checkers} that 
		$\widetilde{\Ne}_{} \left(  \alpha_{g_0 g_1 \cdots g_n} \left( y  
		\right) \right) = g_0 g_1 \cdots g_n \cdot \widetilde{\Ne}_{} \left( y 
		\right)$ and $\widetilde{\Ne}_{} \left(  \alpha_{g_0' g_1' \cdots 
		g_{n'}'} \left( y'  \right) \right) = g_0' g_1' \cdots g_{n'}' \cdot 
		\widetilde{\Ne} \left(  y'  \right)$. 
		Thus, using 
		Step~\eqref{lem:relative-bound-ltc-inductive:proof::loop-coarse}
		and our construction of $\Lo_+$ in 
		Step~\eqref{lem:relative-bound-ltc-inductive:proof::Bo_0}, it follows
		 that 
		\[
		\left(g_0' \cdots g_{n'}'\right) g \left(g_0 \cdots g_n\right)^{-1} \in g_+ H g_+^{-1} \cap \Lo^{2 \Binu_0 + 3} = \Lo_+ 
		\; .
		\]
		Since $\Ne_{+, i}$ is $\Lo_+$-equivariant by Steps~\eqref{lem:relative-bound-ltc-inductive:proof::Coseco_plus}, we have
		\begin{align*}
		& \alpha_g \circ \alpha_{\left(g_0 \cdots g_n\right)^{-1}} \left( \Ne_{+, i} \left( \alpha_{g_0  \cdots g_n} \left( y \right) \right) \right) \\
		=&\ \alpha_{\left(g_0'  \cdots g_{n'}' \right)^{-1}} \circ \alpha_{\left(g_0' \cdots g_{n'}'\right) g \left(g_0  \cdots g_n\right)^{-1} } \left( \Ne_{+, i} \left( \alpha_{g_0 \cdots g_n} \left( y \right) \right) \right) \\  
		=&\ \alpha_{\left(g_0' \cdots g_{n'}' \right)^{-1}} \left( \Ne_{+, i} \left(  \alpha_{\left(g_0'  \cdots g_{n'}'\right) g \left(g_0 \cdots g_n\right)^{-1} }  \circ \alpha_{g_0 \cdots g_n} \left( y \right) \right) \right) \\ 
		=&\ \alpha_{\left(g_0' \cdots g_{n'}' \right)^{-1}} \left( \Ne_{+, i} \left(  \alpha_{g_0'  \cdots g_{n'}'  } \circ \alpha_g \left( y \right) \right) \right) \\
		=&\ \alpha_{\left(g_0' \cdots g_{n'}'\right)^{-1}} \left( \Ne_{+, i} \left(  \alpha_{g_0'  \cdots g_{n'}'  } \left( y' \right) \right) \right) 
		\; ,
		\end{align*}
		as claimed.  

		\item  \label{lem:relative-bound-ltc-inductive:proof::Auseco_0} 
		By Lemma~\ref{lem:cover-clumping}, there exists a finite open cover 
		$\Aulaseco'$ of $X$ such that for any $\Aulase'_1, \ldots, \Aulase'_n 
		\in \Aulaseco'$, if $\bigcap_{i = 1}^{n} \Aulase'_i \not= \varnothing$, 
		then there exists $\Aulase \in \Aulaseco$ such that $\bigcup_{i = 
		1}^{n} 
		\Aulase'_i \subseteq \Aulase$. 
		We then define a finite collection of open sets
		\[
			\Auseco_{0} := \Thseco''_{0} \wedge \left( \alpha^\wedge_{\Lo} 
			\left( \Aulaseco' \right) \right)  \; .
		\]
		The collection $ \Auseco_{0} $ refines $\Thseco''_{0}$ (and thus also 
		$\Thseco_0$ and $\Thseco$) and $\Aulaseco'$ (and thus also 
		$\Aulaseco$), and satisfies $\bigcup \Auseco_{0} = \bigcup 
		\Thseco''_{0} \supseteq \Ko'_{0}$. 
		
		\item  \label{lem:relative-bound-ltc-inductive:proof::Auseco_0-fact} 
		We record the following fact for later use: 
		For any $g \in \Lo$, for any $\Ause_1, \Ause_2 \in \Auseco_{0}$ 
		satisfying $\Ause_1 \cap \alpha_g \left( \Ause_2 \right) \not= 
		\varnothing$, for any $g' \in \Lo^{\Binu_0 + 1}$, and for any $i \in 
		\intervalofintegers{1}{m}$\vphantom{$\Cose \in \Coseco_+$} satisfying 
		$\overline{\Cose'_{+, i}} \cap \alpha_{g'} \left( \Ause_1 \cup \alpha_g 
		\left( \Ause_2 \right)  \right) \not= \varnothing$, 
		we have 
		\[
		\alpha_{g'} \left( \Ause_1 \right) \cup \alpha_{g' g}  \left( \Ause_2 
		\right)   \subseteq \Cose_{+, i}
		\]
		 and 
		 \[
		 \Ne_{+, i} \left( \alpha_{g'} \left( \Ause_1 \right)  \cup \alpha_{g' 
		 g}  \left( \Ause_2 \right) \right)
		 \]
		  is a singleton.
		Moreover, if $\overline{\Cose''_{+, i}} \cap \alpha_{g'} \left( \Ause_1 \cup \alpha_g \left( \Ause_2 \right)  \right) \not= \varnothing$, 
		then 
		\[
		\alpha_{g'} \left( \Ause_1 \right) \cup \alpha_{g' g}  \left( \Ause_2 
		\right)   \subseteq \Cose'_{+, i}
		\, .
		\]
		Indeed, by our construction above, there exist $\Aulase'_1, \Aulase'_2 
		\in \Aulaseco'$ such that $\Ause_1 \subseteq \Aulase'_1$ and $\alpha_g 
		\left( \Ause_2 \right)  \subseteq \Aulase'_2$. Since $\Aulase'_1 \cap 
		\Aulase'_2 \supseteq \Ause_1 \cap \alpha_g \left( \Ause_2 \right) \not= 
		\varnothing$, it follows from 
		Step~\eqref{lem:relative-bound-ltc-inductive:proof::refine_0} that 
		there exists $\Aulase \in \Aulaseco$ such that 
		\[
		\Aulase \supseteq 
		\Aulase'_1 \cup \Aulase'_2 \supseteq \Ause_1 \cup \alpha_g \left( 
		\Ause_2 \right)
		\, .
		\]
		 Recall from 
		Step~\eqref{lem:relative-bound-ltc-inductive:proof::refine_0} that for 
		any $y \in \Ne_{+, i}  \left( \overline{\Cose'_{+, i}} \right)$, we 
		have a dichotomy: either $\alpha_{g'} \left( \Aulase \right) \subseteq 
		\Ne_{+, i}^{-1} (y)$ or $\alpha_{g'} \left( \Aulase \right) \cap 
		\overline{\Cose'_{+, i}} \cap \Ne_{+, i}^{-1} (y) = \varnothing$. 
		Since by our assumption, 
		\[
		\alpha_{g'} \left( \Aulase \right) \cap 
		\overline{\Cose'_{+, i}} \supseteq \alpha_{g'} \left( \Ause_1 \cup 
		\alpha_g \left( \Ause_2 \right)  \right) \cap \overline{\Cose'_{+, i}} 
		\not= \varnothing
		\, ,
		\]
		 we may choose some $y \in \Ne_{+, i}  \left( 
		\alpha_{g'} \left( \Aulase \right) \cap \overline{\Cose'_{+, i}} 
		\right) \subseteq \Ne_{+, i}  \left( \overline{\Cose'_{+, i}} \right)$. 
		Since $\alpha_{g'} \left( \Aulase \right) \cap \overline{\Cose'_{+, i}} 
		\cap \Ne_{+, i}^{-1} (y) \not= \varnothing$,  by the dichotomy above, 
		we have 
		\[
		\Ne_{+, i}^{-1} (y) \supseteq \alpha_{g'} \left( \Aulase 
		\right) \supseteq \alpha_{g'} \left( \Ause_1 \right) \cup \alpha_{g' 
		g}  \left( \Ause_2 \right)
		\, .
		\]
		 Thus, the first part of the claim 
		follows. 
		The second part is proved in the identical way, but with $\Cose_{+, i}$ and $\Cose'_{+, i}$ replaced by $\Cose'_{+, i}$ and $\Cose''_{+, i}$, respectively. 
		
		\item  \label{lem:relative-bound-ltc-inductive:proof::Coseco_0}
		By the inductive assumption and our choice of $\Binu_0$, there exist 
		a finite collection $\Coseco_{0}$ of open subsets and
		locally constant functions $\Ne_{0, \Cose} \colon \Cose \to \bigcup \Auseco_{0}$ for $\Cose \in \Coseco_{0}$, 
		satisfying the conditions~\Refc[\LTCdef]{Lo}{\{e\},\Ko'_0}, \Refcd[\LTCextra]{Muplus}, \Refcd[\LTCdef]{Eq}, \Refc[\LTCdef]{Th}{\Auseco_{0}}, 
		\Refc[\LTCdef]{Ca}{\Binu_0}, 
		and $\left(\widetilde{\Ne} \middle|_{\bigcup \Auseco_{0}}  \right) \circ \Ne_{0,\Cose} = \left(\widetilde{\Ne} \middle|_{\bigcup \Auseco_{0}}  \right)$ for any $\Cose \in \Coseco_{0}$. 
		
		\item  \label{lem:relative-bound-ltc-inductive:proof::Coseco_0-mod} 
		By replacing each $\Cose \in \Coseco_{0}$ with its $\left( \Ne_{0, 
		\Cose}, \Lo \right)$-connected components as in the proof of 
		\eqref{item:prop:ltc-dim-new:orig} $\Rightarrow$ 
		\eqref{item:prop:ltc-dim-new:connected} in Lemma~\ref{prop:ltc-dim-new} and replacing each $\Ne_{0, \Cose}$ by  the corresponding restrictions to these subsets, we may assume without loss of generality that $\Coseco_0$ and $\left(\Ne_{0, \Cose}\right)_{\Cose \in \Coseco_0}$ also satisfies \Refc[\LTCextra]{Coplus}{\Lo, \Binu_0}. 
		
		Indeed, after such a replacement, it is evident that \Refc[\LTCdef]{Lo}{\{e\},\Ko'_0} still holds, as does the equality $\left(\widetilde{\Ne} \middle|_{\bigcup \Auseco_{0}}  \right) \circ \Ne_{0,\Cose} = \left(\widetilde{\Ne} \middle|_{\bigcup \Auseco_{0}}  \right)$ for any $\Cose \in \Coseco_{0}$, while it follows from Remark~\ref{rmk:ltc-dim-monotonous} that \Refcd[\LTCdef]{Eq}, \Refc[\LTCdef]{Th}{\Auseco_{0}}, and 
		\Refc[\LTCdef]{Ca}{\Binu_0} also hold. 
		The fact that \Refcd[\LTCextra]{Muplus} is preserved follows from the observation that any $(\alpha,\Lo)$-close subset of $X$ cannot have nonempty intersections with two or more $\left( \Ne_{0, \Cose}, \Lo \right)$-connected components of a single $\Cose$ from the original $\Coseco_{0}$. 
		Finally, \Refc[\LTCextra]{Coplus}{\Lo, \Binu_0} follows from 
		\Refc[\LTCdef]{Ca}{\Binu_0} and the fact that now each $\Cose$ from the 
		new collection $\Coseco_{0}$ is $\left( \Ne_{0, \Cose}, \Lo 
		\right)$-connected.

		\item  \label{lem:relative-bound-ltc-inductive:proof::Coseco} 
		Define subcollections of $\Coseco_0$ as follows: 
		\begin{align*}
		\Coseco_{0, \operatorname{free}} &:= \left\{ \Cose' \in \Coseco_0 \colon \alpha^\cup_{\Lo} \left( \Ne_{0, \Cose'}  \left(\Cose'\right) \right) \cap \overline{\Cose'_{+, i}}  = \varnothing \text{ for any } i \in \intervalofintegers{1}{m} \right\}  \; , \\ 
		\Coseco_{0, i} &:=  \left\{ \Cose' \in \Coseco_0 \colon \alpha^\cup_{\Lo} \left( \Ne_{0, \Cose'}  \left(\Cose'\right) \right) \cap  \overline{\Cose'_{+, i}}   \not= \varnothing \right\}  \qquad  \text{for }  i \in \intervalofintegers{1}{m}  \; , \\
		\Coseco'_{0, i} &:=  \Coseco_{0, i} \setminus \left( \bigcup_{j = 1}^{i-1} \Coseco_{0, j} \right) \; .  
		\end{align*}
		Then clearly we have 
		\[
		\Coseco_0 
		= \Coseco_{0, \operatorname{free}} \sqcup \left( \bigcup_{i = 1}^m  \Coseco_{0, i} \right)  
		= \Coseco_{0, \operatorname{free}} \sqcup \left( \bigsqcup_{i = 1}^m  \Coseco'_{0, i} \right) 
		\; .
		\]
		Now for any $i \in \intervalofintegers{1}{m}$\vphantom{$\Cose \in \Coseco_+$}, define an open set 
		\[
			\widehat{\Cose}_i := \Cose''_{+, i}  \cup \left( \bigcup \Coseco'_{0, i} \right) \;,
		\]
		and observe that $\Cose'_{+, i}  \cap \left( \bigcup \Coseco'_{0, i} \right) \subseteq \left( \bigcup \Thseco_+ \right) \cap \left( \bigcup \Thseco_{0} \right) = \varnothing$. 
		Also define the following collections of open sets:  
		\[
			\widehat{\Coseco}_{+} := \left\{ \widehat{\Cose}_i  \colon i \in \intervalofintegers{1}{m} \right\} \quad \text{ and } \quad 
			\Coseco := \widehat{\Coseco}_{+} \cup \Coseco_{0, \operatorname{free}} \; .
		\]

		\item  \label{lem:relative-bound-ltc-inductive:proof::Lo} 
		Observe that 
		\begin{align*}
			\bigcup \Coseco &= \left( \bigcup \widehat{\Coseco}_{+} \right) \cup \left( \bigcup  \Coseco_{0, \operatorname{free}} \right) \\
			&= \left( \bigcup_{i = 1}^m \left(  \Cose''_{+, i}  \cup \left( \bigcup \Coseco'_{0, i} \right) \right) \right) \cup \left( \bigcup  \Coseco_{0, \operatorname{free}} \right) \\
			&= \left( \bigcup_{i = 1}^m   \Cose''_{+, i}   \right)  \cup \left( \bigcup \left(  \left( \bigcup_{i = 1}^m  \Coseco'_{0, i} \right) \cup \Coseco_{0, \operatorname{free}} \right) \right) \\
			&= \left( \bigcup_{i = 1}^m   \Cose''_{+, i}   \right)  \cup \left( \bigcup \Coseco_0 \right) \\ 
			&\supset \Ko'_+ \cup \Ko'_0 
			= \Ko' \; .
		\end{align*}
		Therefore $\Coseco$ satisfies \Refc[\LTCdef]{Lo}{\{e\}, \Ko'}.

		\item \label{lem:relative-bound-ltc-inductive:proof::Muplus} 
		To see that $\Coseco$ satisfies \Refcd[\LTCextra]{Muplus}, 
		it suffices to fix an arbitrary subcollection $\Fi$ of $\Coseco$ such that 
		there is an $\left(\alpha, \Lo\right)$-close subset $\left\{ x_{\Cose} \colon \Cose \in \Fi \right\}$ of $X$ with $x_{\Cose} \in \Cose$ for any $\Cose \in \Fi$ and show $|\Fi| \leq d+1$. To this end, 
		by our construction in Step~\eqref{lem:relative-bound-ltc-inductive:proof::Coseco}, 
		we have $\Fi = \Fi_{\operatorname{free}} \cup \left\{ \widehat{\Cose}_i  \colon i \in \Fi' \right\}$ for some $\Fi_{\operatorname{free}} \subseteq \Coseco_{0, \operatorname{free}}$ and $\Fi' \subseteq \intervalofintegers{1}{m}$. 
		
		Consider the following two complementary possibilities. 
		\begin{enumerate}
			\item If $\Fi_{\operatorname{free}} \not= \varnothing$, i.e., there is some $\Cose_0 \in \Fi_{\operatorname{free}}$, 
			then we first observe that $x_{\Cose} \notin \bigcup_{i = 1}^m   
			\Cose''_{+, i}$ for any $\Cose \in \Fi$. Indeed, if $x_{\Cose} \in 
			\Cose''_{+, i}$ for some $\Cose \in \Fi$ and $i \in 
			\intervalofintegers{1}{m}$, then by $\left(\alpha, 
			\Lo\right)$-closeness there exists $g \in \Lo$ such that $x_{\Cose} 
			= 
			\alpha_{g} \left( x_{\Cose_0} \right)$. 
			By \Refc[\LTCdef]{Th}{\Auseco_{0}} in 
			Step~\eqref{lem:relative-bound-ltc-inductive:proof::Coseco_0} and 
			the fact that $\Auseco_{0}$ refines $\Aulaseco$ by 
			Step~\eqref{lem:relative-bound-ltc-inductive:proof::Auseco_0}, 
			there exists $\Aulase \in \Aulaseco_{0}$ such that $\left\{ \Ne_{0, 
			\Cose_0} \left(x_{\Cose_0} \right), x_{\Cose_0} \right\} \subseteq 
			\Aulase$. Therefore, by 
			Step~\eqref{lem:relative-bound-ltc-inductive:proof::refine_0}, it 
			follows from $\alpha_{g} \left( x_{\Cose_0} \right) \in \Cose''_{+, 
			i}$ that $\alpha_{g} \left( \Ne_{0, \Cose_0} \left(x_{\Cose_0} 
			\right) \right) \in {\Cose'_{+, i}}$; this contradicts the 
			assumption that $\Cose_0 \in \Coseco_{0, \operatorname{free}}$ by  
			Step~\eqref{lem:relative-bound-ltc-inductive:proof::Coseco}. 
			
			It thus follows by Step~\eqref{lem:relative-bound-ltc-inductive:proof::Lo} that 
			for any $i \in \Fi'$, we have $x_{\widehat{\Cose}_i} \in \Cose_i$ for some $\Cose_i \in \Coseco'_{0, i}$. 
			Since we have shown $\Coseco_0 
			= \Coseco_{0, \operatorname{free}} \sqcup \left( \bigsqcup_{i = 
			1}^m  \Coseco'_{0, i} \right)$ in 
			Step~\eqref{lem:relative-bound-ltc-inductive:proof::Coseco}, it 
			follows that $\Fi_{\operatorname{free}} \cup \left\{ \Cose_i  
			\colon i \in \Fi' \right\}$ is a subset of $\Coseco_0$ with the 
			same cardinality as $\Fi$. By \Refcd[\LTCextra]{Muplus} in 
			Step~\eqref{lem:relative-bound-ltc-inductive:proof::Coseco_0-mod}, 
			we have $|\Fi| \leq d+1$ as claimed. 
			
			\item If $\Fi_{\operatorname{free}} = \varnothing$, then 
			for each $i \in \Fi'$, we claim 
			there exists $g_{i} \in \Lo^{\Binu_0 + 1}$ such that 
			\begin{itemize}
				\item $\alpha_{g_{i}} \left( x_{\widehat{\Cose}_i} \right) \in \Cose_{+, i}$, and 
				\item $g_{i} \cdot \widetilde{\Ne}_{} \left( x_{\widehat{\Cose}_i} \right) = \widetilde{\Ne}_{} \left( \alpha_{g_{i}} \left( x_{\widehat{\Cose}_i} \right) \right) $. 
			\end{itemize}
			Indeed, since $\widehat{\Cose}_i = \Cose''_{+, i}  \cup \left( \bigcup \Coseco'_{0, i} \right)$ by definition, there are two complementary cases: 
			\begin{enumerate}
				\item If $x_{\widehat{\Cose}_i} \in \Cose''_{+, i} \subseteq \Cose_{+, i}$, then we simply define $g_{i} = \grpid$. 
				\item If $x_{\widehat{\Cose}_i} \in \Cose'$ for some $\Cose' 
				\in \Coseco'_{0,i} \subseteq \Coseco_{0,i}$, then since 
				\[
				\alpha^\cup_{\Lo} \left( \Ne_{0, \Cose'}  \left(\Cose'\right) 
				\right) \cap  \overline{\Cose'_{+, i}}   \not= \varnothing
				\]				
				 by the definition of $\Coseco_{0,i}$ in 
				 Step~\eqref{lem:relative-bound-ltc-inductive:proof::Coseco}, 
				 there are $z_{i} \in \Ne_{0, \Cose'}  \left(\Cose'\right)$ and 
				 $g_{i,0} \in \Lo$ with 
				$\alpha_{g_{i,0}} \left( z_{i} \right) \in \overline{\Cose'_{+, i}}$. 
				Since $\Ne_{0, \Cose'}  \left(\Cose'\right)$ is $\left(\Lo, \Binu_0\right)$-bounded by \Refc[\LTCextra]{Coplus}{\Lo, \Binu_0} in  Step~\eqref{lem:relative-bound-ltc-inductive:proof::Coseco_0-mod}, 
				there are $g_{i,1}, \ldots, g_{i,\Binu_0} \in \Lo$ such that 
				\[
				\alpha_{g_{i,j}  g_{i,j+1} \cdots g_{i,\Binu_0}} \left( \Ne_{0, 
				\Cose'} \left(x_{\widehat{\Cose}_i} \right) \right) \in \Ne_{0, 
				\Cose'}  \left(\Cose'\right)
				\] 
				for $j = 1, 2, \ldots, \Binu_0$ and $\alpha_{g_{i,1}  g_{i,2} 
				\cdots g_{i,\Binu_0}} \left( \Ne_{0, \Cose'} 
				\left(x_{\widehat{\Cose}_i} \right) \right) = z_{i}$, whence 
				\[
				\alpha_{g_{i,0}  g_{i,1} \cdots g_{i,\Binu_0}} \left( \Ne_{0, 
				\Cose'} \left(x_{\widehat{\Cose}_i} \right) \right) \in 
				\overline{\Cose'_{+, i}}
				\, .
				\]
				Let 
				\[
					g_{i} = g_{i,0}  g_{i,1} \cdots g_{i,\Binu_0} \; .
				\]
				By \Refc[\LTCdef]{Th}{\Auseco_{0}} in 
				Step~\eqref{lem:relative-bound-ltc-inductive:proof::Coseco_0} 
				and the fact that $\Auseco_{0}$ refines $\Aulaseco$ by 
				Step~\eqref{lem:relative-bound-ltc-inductive:proof::Auseco_0}, 
				there exists $\Aulase_{i} \in \Aulaseco_{0}$ such that 
				\[
				\left\{ 
				\Ne_{0, \Cose'} \left(x_{\widehat{\Cose}_i} \right), 
				x_{\widehat{\Cose}_i} \right\} \subseteq \Aulase_{i}
				\, .
				 \]
				Therefore, 
				by 
				Step~\eqref{lem:relative-bound-ltc-inductive:proof::refine_0},
				 because $\alpha_{g_{i}} \left( \Ne_{0, \Cose'} 
				\left(x_{\widehat{\Cose}_i} \right) \right) \in 
				\overline{\Cose'_{+, i}}$, we have $\alpha_{g_{i}} \left( 
				x_{\widehat{\Cose}_i} \right) \in \Cose_{+, i}$. 
				Moreover, by Step~\eqref{lem:relative-bound-ltc-inductive:proof::checkers} and our inductive assumption, we have 
				\begin{align*}
					g_{i} \cdot \widetilde{\Ne}_{} \left( x_{\widehat{\Cose}_i} \right) 
					=& \left( g_{i,0}  g_{i,1} \cdots g_{i,\Binu_0} \right) \cdot \widetilde{\Ne}_{} \left( \Ne_{0, \Cose'} \left(x_{\widehat{\Cose}_i} \right) \right) \\
					=& \widetilde{\Ne}_{} \left( \alpha_{ g_{i,0}  g_{i,1} \cdots g_{i,\Binu_0} } \left(  \Ne_{0, \Cose'} \left(x_{\widehat{\Cose}_i} \right) \right) \right) \\
					=& g_{+} H \\
					=& \widetilde{\Ne}_{} \left( \alpha_{g_{i}} \left( x_{\widehat{\Cose}_i} \right) \right) \; ,
				\end{align*}
				as claimed. 
			\end{enumerate}
			
			Now with this claim established, 
			observe that $\left\{ \alpha_{g_{i}} \left( x_{\widehat{\Cose}_i} 
			\right) \colon i \in \Fi' \right\}$ is $\left(\alpha, \Lo_{+} 
			\right)$-close. To see this, note that for any $i, j \in \Fi'$, it 
			follows from the $\left(\alpha, \Lo\right)$-closeness of $\left\{ 
			x_{\widehat{\Cose}_i} \colon i \in \Fi' \right\}$ that 
			there exists $g_{i j} \in \Lo$ such that $\alpha_{g_{ij}} \left( 
			x_{\widehat{\Cose}_j} \right) = x_{\widehat{\Cose}_i}$. Thus, by 
			Step~\eqref{lem:relative-bound-ltc-inductive:proof::loop-coarse}, 
			we have 
			\[
				g_{i} g_{ij} \left(g_{j}\right)^{-1} \in \left( g_+ H g_+^{-1} \right) \cap \Lo^{ \left( \Binu_0 + 1 \right) + 1 + \left( \Binu_0 + 1 \right)} = \Lo_{+} \; .
			\]
			Hence 
			\[
				\alpha_{g_{i}} \left( x_{\widehat{\Cose}_i} \right) = \alpha_{g_{i} g_{ij} \left(g_{j}\right)^{-1}} \left( \alpha_{g_{j}} \left( x_{\widehat{\Cose}_j} \right) \right) \in \alpha^\cup_{\Lo_{+}} \left( \alpha_{g_{j}} \left( x_{\widehat{\Cose}_j} \right) \right) \; .
			\]
			It thus follows from \Refc[\LTCextra]{Muplus}{d, \Lo_+} in Step~\eqref{lem:relative-bound-ltc-inductive:proof::Coseco_plus} that $|\Fi| \leq d+1$ as desired. 
		\end{enumerate}

		\item \label{lem:relative-bound-ltc-inductive:proof::Ne-Coseco_0} 
		For each $\Cose \in \Coseco_{0, \operatorname{free}} \subseteq 
		\Coseco_0$, we define a locally constant function 
		\[
		\Ne_{\Cose} := \Ne_{0, \Cose} \colon \Cose \to \bigcup \Thseco_{0} 
		\subseteq \bigcup \Thseco
		\, .
		\]
		 By our construction in 
		 step~\eqref{lem:relative-bound-ltc-inductive:proof::Coseco_0},
		 the function $\Ne_{\Cose}$ satisfies \Refcd[\LTCdef]{Eq}, 
		 \Refc[\LTCdef]{Th}{\Auseco_0},  
		\Refc[\LTCdef]{Ca}{\Binu_0} 
		and $\widetilde{\Ne}_{} \circ \Ne_{\Cose} = 
		\widetilde{\Ne}_{}$, and hence also satisfies \Refc[\LTCdef]{Th}{\Thseco_{}}, and 
		\Refcd[\LTCdef]{Ca}
		thanks to Remark~\ref{rmk:ltc-dim-monotonous} and the facts that $\Auseco_0$ refines $\Thseco_{}$ and $\Binu_0 \leq \Binu$.
		
		\item \label{lem:relative-bound-ltc-inductive:proof::Ne-Coseco_plus} 
		For each $i \in \intervalofintegers{1}{m}$\vphantom{$\Cose \in 
		\Coseco_+$}, we follow the notation in 
		Step~\eqref{lem:relative-bound-ltc-inductive:proof::Coseco} and define 
		a function $\Ne_{\widehat{\Cose}_i} \colon \widehat{\Cose}_i = 
		\Cose''_{+, i}  \sqcup \left( \bigcup \Coseco'_{0, i} \right) \to X$ as 
		follows: 
		\begin{itemize}
			\item For any $x \in  \Cose''_{+, i} \subseteq \widehat{\Cose}_i$, since $x \in \Cose_{+, i}$, we may define $\Ne_{\widehat{\Cose}_i} (x) := \Ne_{+, i} (x)$. 
			
			\item For any $x \in \bigcup \Coseco'_{0, i}$, we choose 			
			$\Cose \in \Coseco'_{0, i}$ such that $x \in \Cose$. 
			By the definition of $\Coseco'_{0, i}$, there exist $g_0 \in \Lo$ 
			and $y_0 \in \Ne_{0, \Cose}  \left( \Cose \right)$ such that 
			$\alpha_{g_0} \left( y_0 \right) \in  \overline{\Cose'_{+, i}}$. 
			Furthermore, by the assumption that $\Coseco_0$ satisfies 
			\Refc[\LTCextra]{Coplus}{\Lo, \Binu_0} (see 
			Step~\eqref{lem:relative-bound-ltc-inductive:proof::Coseco_0-mod}), 
			there exist a natural number $n \leq \Binu_0$ and $g_1, \ldots, g_n 
			\in \Lo$ 
			and $y_1, \ldots, y_n \in \Ne_{0, \Cose}  \left( \Cose \right)$ 
			such that $y_n = \Ne_{0, \Cose}  (x)$ and $y_{i-1} = \alpha_{g_i} \left( y_{i} \right)$ for $i = 1, \ldots, n$. 
			In other words,  for $i = 1, 2, \ldots, n$ we have 
			\[
			\alpha_{g_0 g_{1} \cdots g_n} \left( \Ne_{0, \Cose}  \left( x 
			\right) \right) \in \overline{\Cose'_{+, i}}
			 \; \text{ and } \; 
			 \alpha_{g_i g_{i+1} \cdots g_n} \left( \Ne_{0, \Cose}  \left( x 
			 \right) \right) \in \Ne_{0, \Cose}   \left(\Cose\right)
			 \]
			We then define, provisionally,  
			\[
				\Ne_{\widehat{\Cose}_i} (x) := \alpha_{\left(g_0 g_1 \cdots g_n\right)^{-1}} \left( \Ne_{+, i} \left( \alpha_{g_0 g_1 \cdots g_n} \left( \Ne_{0, \Cose}  \left( x \right) \right) \right) \right) \; .
			\]
			We will show in Step~\eqref{lem:relative-bound-ltc-inductive:proof::Ne-Coseco_0-claim} that this is indeed well-defined. 
		\end{itemize}
		Once we show that $\Ne_{\Cose}$ is well defined, we can deduce that 
		since the maps $\Ne_{+, i}$ and $\Ne_{0, \Cose}$ involved in the above definitions are locally constant, 
		$\Ne_{\widehat{\Cose}_i}$ is also locally constant.

		\item \label{lem:relative-bound-ltc-inductive:proof::Ne-Coseco_0-claim} 
		To see that the provisional definition 
		in Step~\eqref{lem:relative-bound-ltc-inductive:proof::Ne-Coseco_plus}   
		does not depend on the choices of $\Cose$ and of $g_0, g_1, \ldots, 
		g_n$,
		we prove the following stronger statement, that will also help us 
		verify \Refcd[\LTCdef]{Eq} later: 
		For any $g \in \Lo$, for any $\Cose' \in \Coseco'_{0, i}$ containing 
		$\alpha_g(x)$, for any natural number $n' \leq \Binu_0$ and for any 
		$g_0', 
		g_1', \ldots, g_{n'}' \in \Lo$ that satisfy 
		\[
		\alpha_{g'_0 g'_{1} \cdots g'_{n'}} \left( \Ne_{0, \Cose'}  \left( 
		\alpha_g(x) \right) \right) \in \overline{\Cose'_{+, i}}
		\; \text{ and } \;`
		\alpha_{g'_i g'_{i+1} \cdots g'_{n'}} \left( \Ne_{0, \Cose'}  \left( 
		\alpha_g(x) \right) \right) \in {\Ne}_{0, \Cose'} \left(\Cose'\right)
		\]
		 for $i = 1, 2, \ldots, {n'}$, 
		we have 
		\begin{align*}
		& \alpha_g \circ \alpha_{\left(g_0 g_1 \cdots g_n\right)^{-1}} \left( \Ne_{+, i} \left( \alpha_{g_0 g_1 \cdots g_n} \left( \Ne_{0, \Cose}  \left( x \right) \right) \right) \right)  \\
		&= \alpha_{\left(g'_0 g'_1 \cdots g'_{n'}\right)^{-1}} \left( \Ne_{+, i} \left( \alpha_{g'_0 g'_1 \cdots g'_{n'}} \left( \Ne_{0, \Cose'}  \left( \alpha_g(x) \right) \right) \right) \right) \; .
		\end{align*}
		Note that the case when $g = \grpid$ justifies the provisional definition in Step~\eqref{lem:relative-bound-ltc-inductive:proof::Ne-Coseco_plus}.    
		
		Indeed, since $\Coseco_{0}$ satisfies \Refc[\LTCdef]{Th}{\Auseco_{0}} 
		by Step~\eqref{lem:relative-bound-ltc-inductive:proof::Coseco_0}, and 
		thus also \Refc[\LTCdef]{Th}{\Thseco} by 
		Step~\eqref{lem:relative-bound-ltc-inductive:proof::Auseco_0}, we have 
		\[
		{\Ne}_{0, \Cose} \left(\Cose\right) \cup {\Ne}_{0, \Cose'} 
		\left(\Cose'\right) \subseteq \bigcup \Thseco
		\]
		 and there exist $\Ause, \Ause' \in \Auseco_{0}$ such that 
		 \[
		 \left\{ x, {\Ne}_{0, \Cose} \left( x \right) \right\} \subseteq \Ause
		 \; \text{ and } \; 
		 \left\{ \alpha_g(x), {\Ne}_{0, \Cose'} \left( \alpha_g(x) \right) 
		 \right\} \subseteq \Ause'
		 \, .
		 \]
		Therefore there exists $\Aulase' \in \Aulaseco'$ such that
		\[
		\left\{ \alpha_g(x), \alpha_g \left ( {\Ne}_{0, \Cose} ( x ) \right ) 
		\right\} \subseteq \alpha_g(\Ause) \subseteq
		\Aulase'
		\, ,
		\]
		and thus, by 
		Step~\eqref{lem:relative-bound-ltc-inductive:proof::Auseco_0} again, 
		there exists $\Aulase \in \Aulaseco$ such that 
		\[
		\left\{ \alpha_g \left( {\Ne}_{0, \Cose} \left( x \right) \right), 
		{\Ne}_{0, \Cose'} \left( \alpha_g(x) \right) \right\} \subseteq \Aulase 
		\, .
		\]
		Note that we also have 
		\[
		\widetilde{\Ne}_{} \left( {\Ne}_{0, \Cose'} \left( \alpha_g(x) \right) \right) 
		= \widetilde{\Ne}_{} \left(  \alpha_g(x) \right) 
		= g \cdot \widetilde{\Ne}_{} \left(  x \right) 
		= g \cdot \widetilde{\Ne}_{} \left( {\Ne}_{0, \Cose} \left( x \right) \right) \; .
		\]
		The claimed statement then follows directly from applying Step~\eqref{lem:relative-bound-ltc-inductive:proof::loop} to $y = {\Ne}_{0, \Cose} \left( x \right)$ and $y' = {\Ne}_{0, \Cose'} \left( \alpha_g(x) \right)$. 

		\item \label{lem:relative-bound-ltc-inductive:proof::Ne-Coseco_0-Eq} 
		We prove that for each $i \in 
		\intervalofintegers{1}{m}$\vphantom{$\Cose \in \Coseco_+$}, the 
		function $\Ne_{\widehat{\Cose}_i}$ satisfies \Refcd[\LTCdef]{Eq}. To 
		this end, we fix $x \in \widehat{\Cose}_i$ and $g \in \Lo$ satisfying 
		$\alpha_g (x) \in \widehat{\Cose}_i$. We consider the following three 
		possibilities: 
		\begin{itemize}
			\item If both $x$ and $\alpha_g (x)$ are in $\Cose''_{+, i} \subseteq \Cose_{+, i}$, this follows from the fact that $\Ne_{+, i}$ satisfies \Refc[\LTCdef]{Eq}{\Lo_+} (see Step~\eqref{lem:relative-bound-ltc-inductive:proof::Coseco_0-mod}), 
			since 
			\[
			g g_+ H = g \widetilde{\Ne} \left( x \right) =  \widetilde{\Ne} 
			\left( \alpha_g (x) \right) = g_+ H 
			\, ,
			\]
			and thus $g \in g_+ H g_+^{-1} \cap \Lo \subseteq \Lo_+$. 
			
			\item If both $x$ and $\alpha_g (x)$ are in $\bigcup \Coseco'_{0, i}$, 
			we apply the statement in Step~\eqref{lem:relative-bound-ltc-inductive:proof::Ne-Coseco_0-claim} to conclude that 
			$\alpha_{g} \circ \Ne_{\widehat{\Cose}_i} (x) = \Ne_{\widehat{\Cose}_i} \left( \alpha_g (x) \right)$. 
			
			\item If one of $x$ and $\alpha_g (x)$ is in $\Cose''_{+, i}$ and the other is in $\bigcup \Coseco'_{0, i}$, without loss of generality, 
			we assume $x \in \bigcup \Coseco'_{0, i}$ and $\alpha_g (x) \in \Cose''_{+, i}$. 
			Choose $\Cose \in \Coseco'_{0, i}$ with $x \in \Cose$. 
			Since $\Coseco_{0}$ satisfies \Refc[\LTCdef]{Th}{\Auseco_{0}} by 
			Step~\eqref{lem:relative-bound-ltc-inductive:proof::Coseco_0}, we 
			can choose $\Ause \in \Auseco_{0}$ such that $\left\{ x, {\Ne}_{0, 
			\Cose} \left( x \right) \right\} \subseteq \Ause$. 
			Since $\alpha_g (x) \in \Cose''_{+, i} \cap \alpha_g (\Ause)$, 
			by Step~\eqref{lem:relative-bound-ltc-inductive:proof::Auseco_0-fact}, we have 
			\[
				\left\{ \alpha_g (x), \alpha_g \left( {\Ne}_{0, \Cose} \left( x \right) \right) \right\} \subseteq \alpha_g (\Ause) \subseteq \Cose'_{+, i} \; 
			\]
			and $\Ne_{+, i} \left( \alpha_g (x) \right) = \Ne_{+, i} \left( \alpha_g \left( {\Ne}_{0, \Cose} \left( x \right) \right) \right)$. 
			Hence it follows from the definitions in Step~\eqref{lem:relative-bound-ltc-inductive:proof::Ne-Coseco_plus} that 
			\begin{align*}
				\alpha_{g} \circ \Ne_{\widehat{\Cose}_i} (x) &= \alpha_{g} \circ \alpha_{ g^{-1}} \left( \Ne_{+, i} \left( \alpha_{g} \left( \Ne_{0, \Cose}  \left( x \right) \right) \right) \right) \\
				&= \Ne_{+, i} \left( \alpha_g (x) \right) = \Ne_{\widehat{\Cose}_i} \left( \alpha_g (x) \right)
			\end{align*}
		as claimed.
		\end{itemize}

		\item \label{lem:relative-bound-ltc-inductive:proof::Ne-Coseco_0-prep} 
		In preparation for the proof of \Refcd[\LTCdef]{Th}, 
		we fix some notation: 
		Fix $z \in X$ and $i \in \intervalofintegers{1}{m}$\vphantom{$\Cose \in \Coseco_+$}.  
		For any $\Cose \in \Coseco'_{0, i}$ and for any $g_0, g_1, \ldots, 
		g_{\Binu_0} \in G$, 
		set
		\begin{align*}
			\Aulase_{z, i, \Cose; \left( g_0, g_1, \ldots, g_{\Binu_0} \right)}
			& = \\
			\Bigg\{ y \in \Ne_{0, \Cose}   \left(\Cose\right) \colon 
			&\ \alpha_{g_i g_{i+1} \cdots g_{\Binu_0}} \left( y \right) \in \Ne_{0, \Cose}   \left(\Cose\right) \quad \text{ for } i = 1, 2, \ldots, g_{\Binu_0} \; , \\
			&\ \alpha_{g_0 g_{1} \cdots g_{\Binu_0}} \left( y \right) \in \overline{\Cose'_{+, i}} \; , \\
			& \text{ and } \quad \alpha_{\left(g_0 g_1 \cdots g_{\Binu_0}\right)^{-1}} \left( \Ne_{+, i} \left( \alpha_{g_0 g_1 \cdots g_{\Binu_0}} \left( y \right) \right) \right) =  z  
			\Bigg\} \; . 
		\end{align*}
		By \Refc[\LTCdef]{Ca}{\Binu_0} in 
		Step~\eqref{lem:relative-bound-ltc-inductive:proof::Coseco_0}, this is 
		a finite set. 
		Also define, for any $\Cose \in \Coseco'_{0, i}$, the finite set 
		\[
			\Aulase_{z, i, \Cose}  
			:=  \bigcup_{ g_0, g_1, \ldots, g_{\Binu_0} \in \Lo  } \Aulase_{z, i, \Cose; \left( g_0, g_1, \ldots, g_{\Binu_0} \right)} \; . 
		\]
		Comparing the construction above with the definition in Step~\eqref{lem:relative-bound-ltc-inductive:proof::Ne-Coseco_plus} and recalling that $\grpid \in \Lo$, we have the equality 
		\[
			\left( \bigcup \Coseco'_{0, i} \right) \cap \Ne_{\widehat{\Cose}_i}^{-1} \left( z \right) = 
			\bigcup_{\Cose \in \Coseco'_{0, i}}  \Ne_{0, \Cose}^{-1} \left( \Aulase_{z, i, \Cose} \right) \; .
		\]

		\item \label{lem:relative-bound-ltc-inductive:proof::Ne-Coseco_0-Th-prep} 
		We claim that for any $z \in X$ and for any $i \in 
		\intervalofintegers{1}{m}$\vphantom{$\Cose \in \Coseco_+$}, if 
		\[
			\bigcup_{\Cose \in \Coseco'_{0, i}}   \Aulase_{z, i, \Cose} \not= \varnothing \; ,
		\]
		then there exists $\Thse \in \Thseco_{0}$ such that 
		\[
			\{z\} \cup 
			\left( \bigcup_{\Cose \in \Coseco'_{0, i}} \Ne_{0, \Cose}^{-1} \left( \Aulase_{ z, i, \Cose} \right)  \right)
			\subseteq \Thse \; . 
		\]
		
		Indeed, in view of the construction of $\Thseco'_{0}$ in 
		Step~\eqref{lem:relative-bound-ltc-inductive:proof::Auseco_plus} and 
		the definition of the finite set $\Aulase_{z, i, \Cose}$ in 
		Step~\eqref{lem:relative-bound-ltc-inductive:proof::Ne-Coseco_0-prep}, 
		it suffices to show that for any $z \in X$, for any $i \in 
		\intervalofintegers{1}{m}$\vphantom{$\Cose \in \Coseco_+$}, for any 
		$\Cose \in \Coseco'_{0, i}$, for any $g_0, g_1, \ldots, g_{\Binu_0} \in 
		\Lo$, and for any $y \in \Aulase_{ z, i, \Cose; \left( g_0, g_1, 
		\ldots, g_{\Binu_0} \right)}$, 
		there exists $\Thse' \in \Thseco'_{0}$ such that 
		\[
		\{z\} \cup  \Ne_{0, \Cose}^{-1} \left( y \right)  \subseteq \Thse'
		\, .
		\]
		To this end, 
		we apply \Refc[\LTCdef]{Th}{\Auseco_0} in Step~\eqref{lem:relative-bound-ltc-inductive:proof::Coseco_0} to obtain $\Ause \in \Auseco_0$ such that 
		$\left\{ y \right\} \cup \Ne_{0, \Cose}^{-1} \left( y \right) \subseteq \Ause$. 
		Since 
		\begin{align*}
			\alpha_{g_0 g_{1} \cdots g_{\Binu_0}} \left( y \right) & \in \alpha_{g_0 g_{1} \cdots g_{\Binu_0}} \left( \Aulase_{ z, i, \Cose; \left( g_0, g_1, \ldots, g_{\Binu_0} \right)} \cap \Ause \right) \\
			& \subseteq \overline{\Cose'_{+, i}} \cap \alpha_{g_0 g_{1} \cdots 
			g_{\Binu_0}} \left( \Ause \right) \, ,
		\end{align*}
		applying Step~\eqref{lem:relative-bound-ltc-inductive:proof::Auseco_0-fact} with $g = \grpid$, 
		it follows that $\alpha_{g_0 g_{1} \cdots g_{\Binu_0}} \left( \Ause 
		\right) \subseteq \Cose_{+, i}$, and the function $\Ne_{+, i}$ maps 
		$\alpha_{g_0 g_{1} \cdots 
		g_{\Binu_0}}$ to a singleton. 
		Thus, by \Refc[\LTCdef]{Th}{\Auseco_+} in 
		Step~\eqref{lem:relative-bound-ltc-inductive:proof::Coseco_plus}, there 
		exists $\Ause' \in \Auseco_+$ such that 
		\[
		\alpha_{g_0 g_{1} \cdots g_{\Binu_0}} \left( \Ause \right) \cup \Ne_{+, 
		i} \left( \alpha_{g_0 g_{1} \cdots g_{\Binu_0}} \left( \Ause \right) 
		\right) \subseteq \Ause'
		\, .
		\]
		By Step~\eqref{lem:relative-bound-ltc-inductive:proof::Auseco_plus}, 
		there exists $\Thse' \in \Thseco'_{0} \cup \left\{ X \setminus 
		\overline{\bigcup \Thseco''_0} \right\}$ such that 
		\[
		\alpha_{\left( g_0 g_{1} \cdots g_{\Binu_0} \right)^{-1}} \left( \Ause' 
		\right) \subseteq \Thse'
		\, .
		\]
		Hence 
		\begin{align*}
			&\ \{z\} \cup  \Ne_{0, \Cose}^{-1} \left( y \right)  \\
			= &\ \alpha_{\left( g_0 g_{1} \cdots g_{\Binu_0} \right)^{-1}} \left( \left\{\Ne_{+, i} \left( \alpha_{g_0 g_{1} \cdots g_{\Binu_0}} \left( y \right) \right) \right\} \cup \alpha_{g_0 g_{1} \cdots g_{\Binu_0}} \left( \Ne_{0, \Cose}^{-1} \left( y \right) \right) \right) \\
			\subseteq &\ \alpha_{\left( g_0 g_{1} \cdots g_{\Binu_0} \right)^{-1}} \left( \Ne_{+, i} \left( \alpha_{g_0 g_{1} \cdots g_{\Binu_0}} \left( \Ause \right) \right) \cup \alpha_{g_0 g_{1} \cdots g_{\Binu_0}} \left( \Ause \right) \right) \\
			\subseteq &\ \alpha_{\left( g_0 g_{1} \cdots g_{\Binu_0} \right)^{-1}} \left( \Ause' \right) 
			\subseteq  \Thse' \; . 
		\end{align*}
		Since $\Ne_{0, \Cose}^{-1} \left( y \right) \subseteq \Cose \subseteq \bigcup \Auseco_{0} = \bigcup \Thseco''_{0}$ by \Refc[\LTCdef]{Th}{\Auseco_{0}} in Step~\eqref{lem:relative-bound-ltc-inductive:proof::Coseco_0} and by Step~\eqref{lem:relative-bound-ltc-inductive:proof::Auseco_0}, we have $\Thse' \not= X \setminus \overline{\bigcup \Thseco''_0}$ and thus $\Thse' \in \Thseco'_{0}$, which proves the claim.

		\item \label{lem:relative-bound-ltc-inductive:proof::Ne-Coseco_0-Th} 
		We prove, for any $i \in \intervalofintegers{1}{m}$\vphantom{$\Cose \in 
		\Coseco_+$}, that the function $\Ne_{\widehat{\Cose}_i}$ satisfies 
		\Refcd[\LTCdef]{Th}. This implies $\Ne_{\widehat{\Cose}_i} \left( 
		\widehat{\Cose}_i \right) \subseteq \bigcup \Thseco$ and, by our 
		assumption in Step~\eqref{lem:relative-bound-ltc-inductive:proof::Bo}, 
		$\widetilde{\Ne} \circ \Ne_{\Cose} = \widetilde{\Ne}$. 
		
		To this end, we fix $z \in \Ne_{\widehat{\Cose}_i} \left( \widehat{\Cose}_i \right)$ and consider two complementary possibilities stemming from the decomposition $\widehat{\Cose}_i = \Cose''_{+, i}  \sqcup \left( \bigcup \Coseco'_{0, i} \right)$: 
		\begin{itemize}
			\item If $\Ne_{\widehat{\Cose}_i}^{-1} \left( z \right) \subseteq 
			\Cose''_{+, i}$, then $\Ne_{\widehat{\Cose}_i}^{-1} \left( z 
			\right) = \Ne_{+, i}^{-1} \left( z \right)$ by 
			Step~\eqref{lem:relative-bound-ltc-inductive:proof::Ne-Coseco_plus},
			 so \Refcd[\LTCdef]{Th} for $\Ne_{\widehat{\Cose}_i}$ follows from 
			\Refc[\LTCdef]{Th}{\Thseco_+} for $\Ne_{+, i}$ in 
			Step~\eqref{lem:relative-bound-ltc-inductive:proof::Coseco_plus}. 
			
			\item If $ \Ne_{\widehat{\Cose}_i}^{-1} \left( z \right) \cap \left( \bigcup \Coseco'_{0, i} \right) \not= \varnothing$, then by the last equality in Step~\eqref{lem:relative-bound-ltc-inductive:proof::Ne-Coseco_0-prep} and the conclusion of Step~\eqref{lem:relative-bound-ltc-inductive:proof::Ne-Coseco_0-Th-prep}, there exists $\Thse \in \Thseco_{0}$ such that 
			\[
				\{z\} \cup 
				\left( \left( \bigcup \Coseco'_{0, i} \right) \cap \Ne_{\widehat{\Cose}_i}^{-1} \left( z \right) \right)
				\subseteq \Thse \; . 
			\]
			In particular, $z \notin \bigcup \Thseco_+$ by 
			Step~\eqref{lem:relative-bound-ltc-inductive:proof::cutout} and 
			thus 
			\[
			\Ne_{\widehat{\Cose}_i}^{-1} \left( z \right) \cap \Cose''_{+, i} = 
			\Ne_{+, i}^{-1} \left( z \right) \cap \Cose''_{+, i} = \varnothing
			\, .
			\]
			Hence $\{z\} \cup 
			\Ne_{\widehat{\Cose}_i}^{-1} \left( z \right)   
			\subseteq \Thse$, as desired. 
		\end{itemize}

		\item \label{lem:relative-bound-ltc-inductive:proof::Ne-Coseco_0-Ca} 
		We prove for any $i \in \intervalofintegers{1}{m}$\vphantom{$\Cose \in \Coseco_+$}, the function $\Ne_{\widehat{\Cose}_i}$ satisfies \Refc[\LTCdef]{Ca}{\left| \Lo_+^{\Binu_+} \right| \left| \Lo^{ \Binu_0 + 1} \right|}. 
		Indeed, by Step~\eqref{lem:relative-bound-ltc-inductive:proof::Ne-Coseco_plus}, we have 
		\[
			\Ne_{\widehat{\Cose}_i} \left( \widehat{\Cose}_i \right) 
			\subseteq \bigcup_{ g_0, g_1, \ldots, g_{\Binu_0} \in \Lo  } \alpha_{\left(g_0 g_1 \cdots g_n\right)^{-1}} \left( \Ne_{+, i} \left( \Cose_{+, i} \right)  \right) 
			\; .  
		\]
		Since $\left| \Ne_{+, i} \left( \Cose_{+, i} \right)  \right| \leq \left| \Lo_+^{\Binu_+} \right|$ by \Refc[\LTCextra]{Coplus}{\Lo_+, \Binu_+} in Step~\eqref{lem:relative-bound-ltc-inductive:proof::Coseco_plus}, we conclude that 
		$\left| \Ne_{\widehat{\Cose}_i} \left( \widehat{\Cose}_i \right)  \right| \leq \left| \Lo_+^{\Binu_+} \right| \left| \Lo^{ \Binu_0 + 1} \right|$, as desired. 
	\end{enumerate}

	Combining Steps~\eqref{lem:relative-bound-ltc-inductive:proof::Lo},~\eqref{lem:relative-bound-ltc-inductive:proof::Muplus},~\eqref{lem:relative-bound-ltc-inductive:proof::Ne-Coseco_0},~\eqref{lem:relative-bound-ltc-inductive:proof::Ne-Coseco_0-Eq},~\eqref{lem:relative-bound-ltc-inductive:proof::Ne-Coseco_0-Th}, and~\eqref{lem:relative-bound-ltc-inductive:proof::Ne-Coseco_0-Ca}, 
	we have completed the inductive proof. 
\end{proof}

We are now ready to state and prove our main theorem in this section.

\begin{Thm} \label{thm:relative-bound-ltc}
	Let $\mathcal{F}$ be a family of subgroups of $G$ closed under conjugation 
	and taking subgroups. Then 
	\[
	\dimltc \left( \alpha \right) + 1  \leq  \left( \eqasdim (\alpha, 
	\mathcal{F}) + 1 \right) 
	\cdot \sup_{H \in \mathcal{F}} \left( \dimltc \left( {\alpha_{|_H}} 
	\right) + 1  \right) 
	\; . 
	\]
\end{Thm}

Before presenting the proof, let us summarize the strategy:  
to bound $\dimltc \left( \alpha \right)$, we construct, after given parameters 
as in Definition~\ref{def:ltc-dim}, a desired open cover of $\Ko$ in $X$ in two 
steps: first constructing an open cover $\Coseco$ of $\Ko$ in $X$ using 
$\eqasdim (\alpha, \mathcal{F})$ and Lemma~\ref{lem:BLR-reformu} and then 
refining it to obtain the final cover. To get an intuition of the refining 
step, it is helpful to imagine the easy case when each $\Cose \in 
\Coseco$ can be decomposed as a finite disjoint union $\bigsqcup_{[g] \in (\Fi 
H)/H} \alpha_g \left( \Cose_0 \right)$
with an open subset $\Cose_0$ in $X$, $H \in \mathcal{F}$ and $\Fi \Subset G$ 
satisfying $\alpha_g \left( \Cose_0 \right) = \alpha_{g'} \left( \Cose_0 
\right)$ if $gH = g'H$,  and the labeling function $\La_{\Cose} \colon 
\bigsqcup_{[g] \in (\Fi H)/H} \alpha_g \left( \Cose_0 \right) \to G/H$ given by 
the index $[g]$ in the union. Let's suppose first, to simplify the situation 
even more, that $\mathcal{F}$ consists just of the trivial group; this 
corresponds to the case of free actions. In this case, the situation described 
above simply gives us an open set $U_0$ which gets translated to disjoint 
copies of it by group elements in some large finite subset of the group. What's 
missing for the definition of long thin covering dimension is that the set 
$U_0$ need not be of small diameter (that is, that the cover isn't `thin'). In 
this case, this is simply remedied by subdividing $U_0$, where the number of 
overlaps is controlled by the (compactly supported) covering dimension of the space; note that the 
(compactly supported) covering dimension is the same as the long thin covering dimension associated to the 
action of the trivial group by Remark~\ref{rmk:ltc-dim-trivial-group}. This illustrates why we get a product of those 
quantities in the statement. Continuing with the simplified case in question when 
$\mathcal{F}$ is not necessarily trivial, we see that those subdivisions of $U_0$ will need 
to be done more carefully, in a way which respects the actions of the subgroups 
in $\mathcal{F}$, and that is where the factor $\dimltc^{+1} \left( X, 
{\alpha_{|_H}} 
\right) $ comes in. 
The actual situation we deal with is much more delicate, because we do not 
necessarily have nice towers of open sets which get mapped exactly one into the 
other. This requires making careful refinements and choices, which is where 
Lemma \ref{lem:relative-bound-ltc-inductive} is used.

\begin{proof}
	We assume that the right hand side is finite, otherwise there's nothing to 
	prove. Denote
	\[
	\sup_{H \in \mathcal{F}} \left( \dimltc \left( X, {\alpha_{|_H}} \right)  
	\right) = d < \infty \quad \text{ and } \quad \eqasdim (\alpha, 
	\mathcal{F}) = d' < \infty
	\]
	and show $\dimltc^{+1} \left( \alpha \right) \leq (d + 1) (d' + 1)$. 
	To this end, it suffices to verify 
	Proposition~\ref{prop:ltc-dim-new}\eqref{item:prop:ltc-dim-new:Lominus}. To 
	do this, we fix arbitrary $\Lo$ and $\Ko$ 
	as in 
	Proposition~\ref{prop:ltc-dim-new}\eqref{item:prop:ltc-dim-new:Lominus}. 
	Without loss of generality, we may assume that $\Lo$ is symmetric and 
	contains 
	$\grpid$. 
	
	By Lemma~\ref{lem:BLR-reformu},  
	there exist 
	\begin{itemize}
		\item a finite subset $\Bo \Subset G$, 
		\item a collection $\Auseco$ of open sets in $X$, 
		\item subgroups $G_{\Ause} \in \mathcal{F}$ for $\Ause \in \Auseco$, and
		\item locally constant functions $\widehat{\Ne}_{\Ause} \colon \Ause 
		\to G / G_{\Ause}$ for $\Ause \in \Auseco$
	\end{itemize}
	that 
	satisfy \Refcd{Lo}, \Refc{Mu}{d'}, and \Refcd{Eq}. 
	Using \Refcd{Lo},
	we see that the collection $\Auseco' := \left\{ \alpha^\cap_{\Lo} (\Ause) 
	\colon \Ause \in \Auseco \right\}$ covers $\Ko$. Since $\Ko$ is compact, 
	there exists a finite collection $\Aulaseco$ of precompact open subsets of 
	$X$ 
	such that $\Aulaseco$ covers $\Ko$ and $\left\{ \overline{\Aulase} \colon 
	\Aulase \in \Aulaseco \right\}$ refines $\Auseco'$. 
	For each $\Aulase \in \Aulaseco$, we choose $\Ause_{\Aulase} \in \Auseco$ 
	satisfying $\overline{\Aulase} \subseteq \alpha^\cap_{\Lo} 
	(\Ause_{\Aulase})$. Define $H_{\Aulase} := G_{\Ause_{\Aulase}}$ and 
	$\widetilde{\Ne}_{\Aulase} := \widehat{\Ne}_{\Ause_{\Aulase}}$. Notice 
	that $\widetilde{\Ne}_{\Aulase} \left( \Aulase \right)$ is finite, because 
	$\overline{\Aulase}$ is compact. 
	It is immediate that $\Bo$, 
	$\Aulaseco$, $\left( H_{\Aulase} \right)_{\Aulase \in \Aulaseco}$ and 
	$\left( \widetilde{\Ne}_{\Aulase} \right)_{\Aulase \in \Aulaseco}$ satisfy 
	\Refc[\BLRlem]{Lo}{\{\grpid\}, \Ko}, \Refc[\LTCextra]{Muplus}{d',\Lo}, and 
	\Refcd[\BLRlem]{Eq}.
	Choose a collection $\left\{\Lase_{\Aulase} \colon \Aulase \in \Aulaseco 
	\right\}$ of open sets such that $\overline{\Lase_{\Aulase}} \subseteq 
	\Aulase$ for any $\Aulase \in \Aulaseco$ and $K \subseteq \bigcup_{\Aulase 
	\in \Aulaseco} \Lase_{\Aulase}$.

	Now we fix an arbitrary $\Aulase \in \Aulaseco$. 
	Since $ \widetilde{\Ne}_{\Aulase} (\Aulase)$ is finite, there exists 
	$\Fi_{\Aulase} \Subset G$ such that $ \widetilde{\Ne}_{\Aulase} (\Aulase) = 
	(\Fi_{\Aulase} H_{\Aulase}) / H_{\Aulase} \subseteq G / H_{\Aulase}$. 
	Recall that we write $H_{\Aulase}^{g} = 
	gH_{\Aulase} g^{-1}$.
	By Lemma~\ref{lem:relative-bound-ltc-inductive}, since $H_{\Aulase}^{g} \in 
	\mathcal{F}$ and thus $\dimltc \left( X, {\alpha_{|_{H_{\Aulase}^{g}}}} 
	\right) \leq d$ for any $g \in \Fi_{\Aulase}$, thus there exists a natural 
	number $\Binu_{\Aulase}$ 
	satisfying: for 
	any finite open cover $\Thseco_{\Aulaseco}$ of $\overline{\Lase_{\Aulase}}$ 
	in $X$ satisfying $\Aulase = \bigcup \Thseco_{\Aulaseco}$,  
	there exist 
	\begin{itemize}
		\item a finite collection $\Coseco_{\Aulase}$ of open subsets of $X$, 
		and
		\item locally constant functions $\Ne_{\Aulase,\Cose} \colon \Cose \to 
		\bigcup \Thseco_{\Aulaseco} = \Aulase$ for $\Cose \in 
		\Coseco_{\Aulase}$, 
	\end{itemize}
	satisfying the 
	conditions~\Refc[\LTCdef]{Lo}{\{e\},\overline{\Lase_{\Aulase}}}, 
	\Refcd[\LTCextra]{Muplus}, \Refcd[\LTCdef]{Eq}, 
	\Refc[\LTCdef]{Th}{\Thseco_{\Aulaseco}}, 
	\Refc[\LTCdef]{Ca}{\Binu_{\Aulase}}, 
	and $\widetilde{\Ne}_{\Aulase} \circ \Ne_{\Aulase,\Cose} = 
	\widetilde{\Ne}_{\Aulase}$ for any $\Cose \in \Coseco_{\Aulase}$. Note that 
	\Refc[\LTCdef]{Th}{\Thseco_{\Aulaseco}} implies that $\bigcup 
	\Coseco_{\Aulase} \subseteq \Aulase$. 
	
	Let 
	\[
	\Binu := \max_{\Aulase \in \Aulaseco} \Binu_{\Aulase} \; .
	\]
	To see that $\Binu$ witnesses 
	Proposition~\ref{prop:ltc-dim-new}\eqref{item:prop:ltc-dim-new:Lominus}, we 
	fix an arbitrary finite open cover $\Thseco$ of $X$. 
	
	For any $\Aulase \in \Aulaseco$, we let $\Thseco_{\Aulase} := \Thseco \vee 
	\{\Aulase\}$. Since $\bigcup \Thseco_{\Aulaseco} = (\bigcup \Thseco) \cap 
	\Aulase = \Aulase \supseteq \overline{\Lase_{\Aulase}}$, thus we obtain 
	$\Coseco_{\Aulase}$ and $\left( \Ne_{\Aulase,\Cose} \right)_{\Cose \in 
	\Coseco_{\Aulase}}$ as above. 
	Let 
	\[
	\Coseco := \bigcup_{\Aulase \in \Aulaseco} \Coseco_{\Aulase} \; .
	\]
	Moreover, we choose, for each $\Cose \in \Coseco$, some $\Aulase_{\Cose} 
	\in \Aulaseco$ such that $\Cose \in \Coseco_{\Aulase_{\Cose}}$, and define 
	$\Ne_{\Cose} := \Ne_{\Aulase_{\Cose}, \Cose} \colon \Cose \to 
	\Aulase_{\Cose} \subseteq X$. 
	
	It remains to verify that $\Coseco$ and $\left( \Ne_{\Cose} \right)_{\Cose 
	\in \Coseco}$ satisfy 
	\Refc[\LTCdef]{Lo}{\{\grpid\},\Ko}, 
	\Refc[\LTCextra]{Muplus}{(d + 1) (d' + 1) - 1, \Lo}, \Refcd[\LTCdef]{Eq}, 
	\Refcd[\LTCdef]{Th}, and  \Refcd[\LTCdef]{Ca}. 
	
	To verify \Refc[\LTCdef]{Lo}{\{\grpid\},\Ko}, we observe  $\bigcup \Coseco 
	= \bigcup_{\Aulase \in \Aulaseco} \left(\bigcup \Coseco_{\Aulase}\right) 
	\supseteq \bigcup_{\Aulase \in \Aulaseco} \overline{\Lase_{\Aulase}} 
	\supseteq \Ko$. 
	
	To verify \Refc[\LTCextra]{Muplus}{(d + 1) (d' + 1) - 1, \Lo}, we fix a 
	finite subcollection $\Fi$ of $\Coseco$ such that 
	there exists an $\left(\alpha, \Lo\right)$-close subset $\left\{ x_{\Cose} 
	\colon \Cose \in \Fi \right\}$ of $X$ with $x_{\Cose} \in \Cose$ for any 
	$\Cose \in \Fi$. Observe that for any $\Aulase \in \Aulaseco$, we have 
	$\left| \Fi \cap \Coseco_{\Aulase} \right| \leq d+1$ since 
	$\Coseco_{\Aulase}$ satisfies \Refcd[\LTCextra]{Muplus}. 
	On the other hand, if we let $\Fi' := \left\{ \Aulase \in \Aulaseco \colon 
	\Fi \cap \Coseco_{\Aulase} \not= \varnothing \right\}$ and, for any 
	$\Aulase \in \Fi'$, choose $y_{\Aulase} := x_{\Cose}$ for some $\Cose \in 
	\Fi \cap \Coseco_{\Aulase}$, then we see that since $\left\{ y_{\Aulase} 
	\colon \Aulase \in \Fi' \right\}$ is $\left(\alpha, \Lo\right)$-close and 
	$y_{\Aulase} \in  \bigcup \Coseco_{\Aulase} \subseteq \Aulase$ for any 
	$\Aulase \in \Fi'$, thus we have $|\Fi'| \leq d'+1$ as $\Aulaseco$ 
	satisfies \Refc[\LTCextra]{Muplus}{d',\Lo}. Combining these estimates 
	yields the desired bound 
	\[
	|\Fi| = \sum_{\Aulase \in \Fi'} \left| \Fi \cap \Coseco_{\Aulase} \right| 
	\leq (d + 1) (d' + 1) \; .
	\]
	
	To verify \Refcd[\LTCdef]{Eq}, we simply recall that each 
	$\Coseco_{\Aulase}$ satifies \Refcd[\LTCdef]{Eq}. 
	
	To verify \Refcd[\LTCdef]{Th}, we notice each $\Coseco_{\Aulase}$ refines 
	$\Thseco_{\Aulase}$, which then refines $\Thseco$. 
	
	To verify \Refcd[\LTCdef]{Ca}, we notice each $\Coseco_{\Aulase}$ satisfies 
	\Refc[\LTCdef]{Ca}{\Binu_{\Aulase}} and thus \Refcd[\LTCdef]{Ca}. 
	
	This completes the proof. 
\end{proof}

\begin{Rmk} \label{rmk:relative-bound-asdim}
	Using similar (but simpler) techniques, one can also prove 
	\[
	\asdim \left( X, \Enseco_\alpha \right) + 1 \leq \left( \eqasdim (\alpha, 
	\mathcal{F}) + 1 \right)\cdot \sup_{H \in \mathcal{F}} \left( \asdim \left( X, 
	\Enseco_{\alpha_{|_H}} \right) + 1 \right) \; .
	\]
	Since we do not make use of this (unless one wishes to improve the bound in 
	Corollary~\ref{cor:dimnuc-hyperbolic}), we leave the proof to the reader, 
	who may find this exercise helpful in order to understand the proof of 
	Theorem~\ref{thm:relative-bound-ltc}. 
\end{Rmk}

We refer the reader to \cite[Part III.H3]{bridson-haefliger-book} for a 
discussion of the boundary and compactification of hyperbolic spaces. The 
Gromov boundary of a finitely generated hyperbolic group is a special case, 
when the hyperbolic space is the group itself with the word metric.

\begin{Cor} \label{cor:hyperbolic-ltc}
	Suppose $\alpha$ is a simplicial proper cocompact action of a hyperbolic 
	group $G$ on a hyperbolic complex $X$. Denote by $\overline{X} = X \cup 
	\partial X$ the compactification, and use $\alpha$ to denote the induced 
	action of $G$ on $\overline{X}$. Then $\asdim^{+1} \left( \overline{X}, 
	\Enseco_\alpha \right)$ and $\dimltc^{+1} \left( \alpha \right)$ are both 
	finite. The same holds for the restriction of $\alpha$ to the boundary 
	$\partial X$.
\end{Cor}

\begin{proof}
	Let $\mathcal{F}$ be the family of all virtually cyclic subgroups of 
	$G$. In 
	\cite[Theorem 1.2]{BartelsLueckReich2008Equivariant}, it is shown in 
	particular 
	that the action $\alpha$ of $G$ on $\overline{G}=  G \cup 
	\partial G$ satisfies $\eqasdim(\alpha, \mathcal{F}) < \infty$. From 
	Theorems 
	\ref{thm:lsp-vnil} and \ref{thm:lsp-ltc}, 
	it follows that for any $H \in \mathcal{F}$ we have $\dimltc^{+1} \left( 
	\alpha|_H \right) \leq 3\dim(X) + 2$. The corollary now follows from 
	Theorem 
	\ref{thm:relative-bound-ltc}.
	
	That the same holds for the restriction to the boundary $\partial X$ 
	follows 
	from the fact that both the equivariant asymptotic dimension and the long 
	thin 
	covering dimension do not increase when restricting to closed invariant 
	subspaces.
\end{proof}

\section{Nuclear dimension for crossed products}
\label{sec:dimnuc} 
\renewcommand{\sectionlabel}{dimnuc}

The goal of this section is to prove our main theorem, Theorem 
\ref{thm:dimnuc-main}, which provides a bound on the nuclear dimension of 
crossed products. As mentioned in the introduction, we work in the setting of 
$C_0(X)$-algebras rather than just actions on commutative $C^*$-algebras. Most 
of the terminology was covered in Subsection \ref{subsection:prelim:C_0(X)}.

\begin{Notation}
	\label{Notation:G-X-algebra}
Let $X$ be a locally compact Hausdorff space, and let $A$ be a 
$C_0(X)$-algebra. 
 Let $\alpha \colon G \to \aut(A)$ be an action, 
which is compatible with an action, also denoted by $\alpha$, of $G$ on $C(X)$, 
i.e. $\alpha (f \cdot a) = \alpha(f) \cdot \alpha(a)$.  We also use $\alpha$ to 
denote the induced 
action on $X$, as before. 

For any subgroup $H \leq G$, we write $\alpha_{|_H}$ for the restriction of $\alpha$ to $H$. 
Notice that if a closed subset $Y$ in $X$ is invariant under $\alpha_{|_H}$, then 
$\alpha_{|_H}$ 
induces an action on the restriction algebra $A_Y$. We denote this action by 
$\alpha_{|_H}^Y$, and if $Y = \{x\}$ is a singleton, then we write $\alpha_{|_H}^x$ for 
short.

For $x \in 
X$, we denote the stabilizer group of the point $x$ by $G_x$. If 
$\alpha$ is an action of $G$ on a set $X$ and $x \in X$, then the map $G/G_x \to 
X$ given by $g G_x \mapsto \alpha_g(x)$ is well defined and injective.
\end{Notation}

The goal of this section is to prove the following theorem. 
\begin{Thm}
	\label{thm:dimnuc-main}
	Let  $X$ be a locally compact Hausdorff space with finite covering 
	dimension, and let $A$ be a $C_0(X)$-algebra. Let $G$ be a countable 
	discrete group. Let $\alpha \colon G \to \aut(A)$ be an action, 
	which is compatible with an action of $G$ on $C_0(X)$, 
	that is, $\alpha (f \cdot a) = \alpha(f) \cdot \alpha(a)$. Set
	\[
	\dstab := \sup \left\{ \dimnuc \left(A_x \rtimes_{\alpha_{|_H}^x} H \right) \colon x \in X, H 
	\leq G_x \right\} \; .
	\]
	Then we have
	\[
	\dimnuc(A \rtimes_{\alpha} G) +1 \leq (\asdim(X, \Enseco_\alpha) +1) 
	\cdot 
	(\dimltc(\alpha) +1) \cdot (\dstab + 1)
	\; .
	\]
\end{Thm}

Note that $\asdim(X, \Enseco_\alpha) \leq \dimltc(\alpha) $ by Theorem~\ref{thm:dimltc-asdim}. 

The proof is somewhat long, and is deferred until we have established a few lemmas and constructions. 

A crucial point in the proof is that 
the collection of all stabilizer subgroups of points in $X$ may be too large 
for us to handle, so, given a finite set $\Lo \subseteq G$, we define a modified 
stabilizer group ``as seen by $\Lo$'', namely the subgroup generated by 
elements of $\Lo$ which stabilize $x$. This is a systematic way to generalize 
the short orbit - long orbit split which we 
used in \cite{Hirshberg-Wu16}. More precisely, we use the following notation.

\begin{Notation}
	\label{Notation:modified stabilizer}
	For any $x \in X$, and any finite subset $P \subseteq G$, we denote 
	\[ 
	G^{P}_x = \left\langle G_x \cap P \right\rangle \, , 
	\]
	that is, the subgroup of $G_x$ generated by $G_x \cap P$. 
	We 
	denote the collection of all such subgroups by 
	\[ 
		\short^{P } = \left 
		\{G^{P}_x \colon x 
		\in X \right \} .
	\]
	Note that $\short^{P }$ is a finite collection of subgroups. 
	We also write 
	\[
		\short^{P }_y = \left 
		\{ H \in \short^{P} \colon H \leq G_y \right \} \quad \text{ for any } y \in X  .
	\]
	Note that for any $y \in X$, any $H \in \short^{P}_y$ automatically satisfies the stronger condition $H \leq G^{P}_y$, 
	since for any $x \in X$, we have $G^{P}_x \leq G_y$ if and only if $G_x \cap P \subseteq G_y$ if and only if $G_x \cap P \subseteq G_y \cap P$ if and only if $G^{P}_x \leq G^{P}_y$. 
\end{Notation}

The following lemma follows immediately from the definition of the quotient map.
\begin{Lemma}
	\label{Lemma:injective-quotient}
	Let $G$ be a group, and let $R \subseteq G$ be a  subset. For any subset 
	$\widetilde{R}$ which contains $R^{-1}R$ and for any subgroup $H < G$, the 
	quotient map $\pi \colon G/\left < H \cap \widetilde{R} \right > \to G/H$ 
	maps $\left( R \cdot \left \langle H \cap \widetilde{R} \right \rangle \right) / \left < H 
	\cap \widetilde{R} \right >$ injectively onto $(R \cdot H) / H$. 
\end{Lemma}


The following elementary inequality is a simple application of convexity; the proof is left to the reader. 

\begin{Lemma}\label{Lemma:sqrt}
	For any $s, t \in [0,\infty)$, we have $\left| \sqrt{s} - \sqrt{t} \right| \leq \sqrt{|s - t|}$. 
	
	The same conclusion holds if we replace $s$ and $t$ by commuting positive elements in a $C^*$-algebra (and replace absolute value by operator norm). 
\end{Lemma}

Now we are ready to prove the main result of the section. 

\begin{proof}[Proof of Theorem~\ref{thm:dimnuc-main}]

We shall assume from now on that all the terms in the right hand side of the 
inequality in Theorem \ref{thm:dimnuc-main} are finite, as otherwise there is 
nothing to prove.
To simplify notation, we set throughout the rest of the proof
\begin{equation}
c = \asdim(X, \Enseco_\alpha)  \quad \text{ and } \quad
d = \dimltc(\alpha) \; .
\end{equation}

We observe that since $C_{\mathrm{c}}(X)$ is dense in $C_0(X)$ and $C_0(X) \cdot A$ is dense in $A$, elements of the form $a u_g$, where $a$ is a contraction in $C_{\mathrm{c}}(X) \cdot A$ and $g \in G$, span a dense subset of $A \rtimes G$, 
Hence in order to obtain the desired bound on $\dimnuc (A \rtimes G)$, 
it suffices to start with an arbitrary finite subset $A_0$ consisting of contractions in $C_{\mathrm{c}}(X) \cdot A$, an arbitrary finite 
subset $\Lo \subseteq G$ and an arbitrary positive number $\eps$ and find a finite dimensional 
$C^*$-algebra $B$ and an approximation
\[
\xymatrix{
A \rtimes_{\alpha} G \ar[dr]_{\Psi} \ar@{.>}[rr]^{\id}	&& A 
\rtimes_{\alpha} G  
\\
& B \ar[ur]_{\Phi} &
}
\]
such that $\Psi$ is a completely positive contraction, $\Phi$ is $(c{+1}) 
\cdot 
(d{+1}) \cdot (\dstab + 1)$-decomposable, and 
we have $\|\Phi \circ \Psi(au_g) - a u_g\| \leq \eps$ for any $a \in A_0$ and for any $g \in \Lo$.  
By enlarging $\Lo$ if necessary, we may also assume without loss of generality that 
\[
\e \in \Lo \textrm{ and } 
\Lo = \Lo^{-1} 
.
\]

Since the construction of $B$, $\Psi$ and $\Phi$ is quite involved, it will be 
presented in five stages (A through E), each further divided into a handful 
steps (labeled (A1), (A2), \ldots, (B1), (B2), etc). Before continuing, let us 
give an overview of these four stages, accompanied by the large diagram below, 
which shall serve as a roadmap to help the reader follow the argument. The 
terms in the diagram have not been defined yet, but the labels at the top of 
the items (e.g., ``\eqref{thm:dimnuc-main:proof::tally}'' above ``$= \Psi$'') 
indicate the steps in which the items are defined. 
In this diagram, we use the symbol $\circlearrowright$ to indicate the maps surrounding it commute, while the symbol $\approx$ indicates the maps ``approximately'' commute (see the corresponding step for the precise parameters). 
The solid arrows represent maps that are defined in Stage~A using explicit formulas, the dashed arrows represent maps that are explicitly derived from the inexplicitly constructed ``partitions of unity'' $\left( \Ponum{\Ause}{l} \right)_{l, \Ause}$ and $\left(\Po[LTC]^{(r)}\right)_{r}$, and the dotted arrows represent maps obtained inexplicitly. 
For brevity, the ranges of the indices under the (direct) sums are not shown; 
let us clarify them here: $l$, $r$, and $m$, respectively, range over 
$\intervalofintegers{0}{c}$, $\intervalofintegers{0}{d }$, and 
$\intervalofintegers{0}{\dstab}$, respectively, while $\Ause$ range over the 
set $\Auseconum{l,r}$, to be defined in 
Step~\eqref{thm:dimnuc-main:proof::Ause-l-r}.

\[\begin{tikzcd}
{A \rtimes_\alpha G} &&&& {A \rtimes_\alpha G} \\
\\
\\
& {\overset{\tiny \text{\eqref{thm:dimnuc-main:proof::equi-Roe}\eqref{thm:dimnuc-main:proof::H-W}}}{\displaystyle \bigoplus_{l,r,W} E_{H_W} \vphantom{\left( \ell^\infty \left(G/G_x^{\widetilde{S}} \right) \otimes A \right) \rtimes_{\lambda \otimes \alpha} G} \vphantom{\ell^\infty \left(G/G_x^{\widetilde{S}}, A \right) \rtimes_{\gamma} G}}} && {\overset{\tiny \text{\eqref{thm:dimnuc-main:proof::equi-Roe}\eqref{thm:dimnuc-main:proof::abbrev}}}{\displaystyle \bigoplus_{l,r,W} D_{W} \vphantom{ E''_H}  \vphantom{ \chi_{SH/H} \left( \ell^\infty(G/H, A) \rtimes_{\gamma^H} G \right) \chi_{SH/H}}}} \\
\\
& {\overset{\tiny \text{\eqref{thm:dimnuc-main:proof::equi-Roe}\eqref{thm:dimnuc-main:proof::abbrev}}}{\displaystyle \bigoplus_{l,r,W} D_{W} \vphantom{ E''_H} \vphantom{ \chi_{SH/H} \left( \ell^\infty(G/H, A) \rtimes_{\gamma^H} G \right) \chi_{SH/H}}}} & {\overset{\tiny \text{\eqref{thm:dimnuc-main:proof::equi-Roe}\eqref{thm:dimnuc-main:proof::abbrev}}}{\displaystyle \bigoplus_{l,r,W} D_{W} \vphantom{ E''_H} \vphantom{ \chi_{SH/H} \left( \ell^\infty(G/H, A) \rtimes_{\gamma^H} G \right) \chi_{SH/H}}}} & {\overset{\tiny \text{\eqref{thm:dimnuc-main:proof::equi-Roe-isomorphism}\eqref{thm:dimnuc-main:proof::abbrev}}}{\displaystyle \bigoplus_{l,r,W} D'_{W} \vphantom{ E'_H} \vphantom{M_{|SH/H|} \left( A \rtimes_{\alpha_{|_H}} H \right)}}} \\
\\
& {\overset{\tiny \text{\eqref{thm:dimnuc-main:proof::equi-Roe-isomorphism}\eqref{thm:dimnuc-main:proof::abbrev}}}{\displaystyle \bigoplus_{l,r,W} D'_{W} \vphantom{ E_H}  \vphantom{ \chi_{SH} \left( \ell^\infty(G, A)^{\beta} \rtimes_{\lambda} G \right) \chi_{SH}}}} & {\overset{\tiny \text{\eqref{thm:dimnuc-main:proof::equi-Roe-isomorphism}\eqref{thm:dimnuc-main:proof::abbrev}}}{\displaystyle \bigoplus_{l,r,W} D'_{W} \vphantom{ E'_H} \vphantom{M_{|SH/H|} \left( A \rtimes_{\alpha_{|_H}} H \right)}}} \\
\\
& {\overset{\tiny \text{\eqref{thm:dimnuc-main:proof::evaluation}\eqref{thm:dimnuc-main:proof::abbrev}}}{\displaystyle \bigoplus_{l,r,W} {C}_{W} \vphantom{D_{H,x}} \vphantom{ \chi_{SH} \left( \ell^\infty(G, A_x)^{\beta^x} \rtimes_{\lambda} G \right) \chi_{SH}}}} & {\overset{\tiny \text{\eqref{thm:dimnuc-main:proof::evaluation}\eqref{thm:dimnuc-main:proof::abbrev}}}{\displaystyle \bigoplus_{l,r,W} {C}_{W} \vphantom{D'_{H,x}} \vphantom{M_{|SH/H|} \left( A_x \rtimes_{\alpha^x_{|_H}} H \right)}}} \\
\\
&& {\overset{\tiny \text{\eqref{thm:dimnuc-main:proof::D-approximation}\eqref{thm:dimnuc-main:proof::abbrev}}}{\displaystyle \bigoplus_{l,r,W} {B}_{W} \vphantom{D_{H,x}} \vphantom{B_{H,x}}}}
\arrow[""{name=0, anchor=center, inner sep=0}, "{\overset{\tiny \text{\eqref{thm:dimnuc-main:proof::evaluation}\eqref{thm:dimnuc-main:proof::abbrev}}}{\underset{l,r,W}{\bigoplus} \mathrm{ev}_{W}}}", from=8-3, to=10-3]
\arrow[""{name=1, anchor=center, inner sep=0}, "{\overset{\tiny \text{\eqref{thm:dimnuc-main:proof::D-approximation}\eqref{thm:dimnuc-main:proof::abbrev}}}{\underset{l,r,W}{\bigoplus} \gamma_W}}"', dotted, from=12-3, to=10-3]
\arrow[""{name=2, anchor=center, inner sep=0}, "{\overset{\tiny \text{\eqref{thm:dimnuc-main:proof::LTC-compression}}}{\underset{l,r,W}{\bigoplus} \mathrm{compr}_{\widetilde{\mu}{'}_W^{(r)}}}}"', shift right=2, dashed, from=8-3, to=6-4]
\arrow["\cong", from=6-2, to=8-2]
\arrow[""{name=3, anchor=center, inner sep=0}, "{\mathrm{id} \vphantom{\theta_{\vec{k}}}}", from=8-2, to=8-3]
\arrow[""{name=4, anchor=center, inner sep=0}, "{\overset{\tiny \text{\eqref{thm:dimnuc-main:proof::evaluation}\eqref{thm:dimnuc-main:proof::abbrev}}}{\underset{l,r,W}{\bigoplus} \mathrm{ev}_{W}}}"', from=8-2, to=10-2]
\arrow["{\mathrm{id} \vphantom{\theta_{x,\vec{k}}}}", from=10-2, to=10-3]
\arrow["{\overset{\tiny \text{\eqref{thm:dimnuc-main:proof::D-approximation}\eqref{thm:dimnuc-main:proof::abbrev}\hphantom{Ste}}}{\underset{l,r,W}{\bigoplus} \kappa_W}}"'{pos=0.3}, shift left=4, curve={height=-6pt}, draw=none, from=10-2, to=12-3]
\arrow[""{name=5, anchor=center, inner sep=0}, "{\overset{\tiny \text{\eqref{thm:dimnuc-main:proof::compression}}}{\underset{l,r,W}{\bigoplus} \mathrm{compr}_{\widetilde{\nu}_W^{(l)}}}}"', dashed, from=4-2, to=6-2]
\arrow[""{name=6, anchor=center, inner sep=0}, "{\overset{\tiny\text{\eqref{thm:dimnuc-main:proof::Pi-Delta}}}{\underset{l,r,W}{\bigoplus} \mathrm{compr}_{\widetilde{\nu}_W^{(l)} \widetilde{\mu}_W^{(r)}}}}", dashed, from=4-2, to=4-4]
\arrow["\cong"'{pos=0.3}, from=6-4, to=4-4]
\arrow[""{name=7, anchor=center, inner sep=0}, "{\overset{\tiny \text{\eqref{thm:dimnuc-main:proof::LTC-compression}}}{\underset{l,r,W}{\bigoplus} \mathrm{compr}_{\widetilde{\mu}_W^{(r)}}}}"{pos=0.2}, shift left=3, dashed, from=6-3, to=4-4]
\arrow[""{name=8, anchor=center, inner sep=0}, "{\mathrm{id} \vphantom{\theta_{\vec{k}}}}", from=6-2, to=6-3]
\arrow[""{name=9, anchor=center, inner sep=0}, "{\overset{\tiny \text{\eqref{thm:dimnuc-main:proof::lifting}\eqref{thm:dimnuc-main:proof::abbrev}}}{\underset{l,r,W,m}{\bigoplus}  \widetilde{\gamma}_W^{(m)}}}"'{pos=0.6}, shift right=5, curve={height=40pt}, dotted, from=12-3, to=8-3]
\arrow["\cong"{pos=0.7}, from=8-3, to=6-3]
\arrow["{\overset{\tiny \text{\eqref{thm:dimnuc-main:proof::equi-Roe-isomorphism}\eqref{thm:dimnuc-main:proof::abbrev}}}{\underset{l,r,W}{\bigoplus} \Theta_{W}^{-1}}}"', from=6-2, to=8-2]
\arrow[""{name=10, anchor=center, inner sep=0}, "{\mathrm{id}_{A \rtimes G}}", curve={height=-30pt}, from=1-1, to=1-5]
\arrow[""{name=11, anchor=center, inner sep=0}, "{\overset{\tiny\text{\eqref{thm:dimnuc-main:proof::Sigma-Pi-Delta}}}{\underset{l,r,W}{\sum} \mathrm{compr}_{\sum \mu^{(r)}_{\nu^{(l,r)}}}}}"', dashed, from=1-1, to=1-5]
\arrow["{\overset{\tiny \text{\eqref{thm:dimnuc-main:proof::Roe}}}{\underset{l,r,W}{\bigoplus} \left( \mathbf{1} \otimes \mathrm{id}_A \right) \rtimes G}}"{pos=0.7}, from=1-1, to=4-2]
\arrow["{\overset{\tiny\text{\eqref{thm:dimnuc-main:proof::Sigma}\eqref{thm:dimnuc-main:proof::H-W}}}{\underset{l,r,W}{\sum}  \Sigma_{H_W}}}"{pos=0.3}, from=4-4, to=1-5]
\arrow["{\overset{\tiny\text{\eqref{thm:dimnuc-main:proof::row-vector}\eqref{thm:dimnuc-main:proof::abbrev}}}{\underset{l,r,W}{\sum} \mathrm{Ad}_{\vec{u}_{W}}}}"', shift right=2, from=6-4, to=1-5]
\arrow[""{name=12, anchor=center, inner sep=0}, "{\overset{\tiny \text{\eqref{thm:dimnuc-main:proof::tally}}}{= \ \ \Phi}}"'{pos=0.3}, shift right=7, curve={height=100pt}, dotted, from=12-3, to=1-5]
\arrow[""{name=13, anchor=center, inner sep=0}, "{\overset{\tiny \text{\eqref{thm:dimnuc-main:proof::tally}}}{= \ \ \Psi}}"'{pos=0.7}, shift right=7, curve={height=100pt}, dashed, from=1-1, to=12-3]
\arrow["{\overset{\tiny \text{\eqref{thm:dimnuc-main:proof::equi-Roe-isomorphism}\eqref{thm:dimnuc-main:proof::abbrev}}}{\underset{l,r,W}{\bigoplus} \Theta_{W}}}"'{pos=0.7}, from=8-3, to=6-3]
\arrow[""{name=14, anchor=center, inner sep=0}, "{\overset{\tiny \text{\eqref{thm:dimnuc-main:proof::equi-Roe-isomorphism}\eqref{thm:dimnuc-main:proof::abbrev}}}{\underset{l,r,W}{\bigoplus} \Theta_{W}}}"{pos=0.3}, from=6-4, to=4-4]
\arrow[""{name=15, anchor=center, inner sep=0}, shift left=3, dotted, from=10-2, to=12-3]
\arrow["\circlearrowright"{description}, draw=none, from=8, to=3]
\arrow["\circlearrowright"{description}, draw=none, from=10-3, to=9]
\arrow["{\overset{\tiny \text{\eqref{thm:dimnuc-main:proof::LTC-compression}}}{\circlearrowright}}"{description}, draw=none, from=7, to=2]
\arrow["{\overset{\tiny\text{\eqref{thm:dimnuc-main:proof::Sigma-Pi-Delta}}}{\circlearrowright}}"{description}, draw=none, from=11, to=6]
\arrow["{\overset{\tiny \text{\eqref{thm:dimnuc-main:proof::row-vector}}}{\circlearrowright}}"{description, pos=0.4}, draw=none, from=14, to=1-5]
\arrow["{\overset{\tiny\text{\eqref{thm:dimnuc-main:proof::Pi-Delta}}}{\circlearrowright}}"{description, pos=0.2}, draw=none, from=5, to=7]
\arrow["\circlearrowright"{description}, draw=none, from=4, to=0]
\arrow["{\overset{\tiny \text{\eqref{thm:dimnuc-main:proof::sum-identity}}}{\approx}}"{description}, draw=none, from=10, to=11]
\arrow["{\overset{\tiny \text{\eqref{thm:dimnuc-main:proof::Psi}}}{\underset{l,r,W}{\bigoplus}  \Psi_W^{(l)}}}"{pos=0}, draw=none, from=13, to=12-3]
\arrow["{\overset{\tiny \text{\eqref{thm:dimnuc-main:proof::approx-order-zero}}}{\underset{l,r,m,W}{\sum}  \Phi_W^{(r,m)}}}"{pos=1}, draw=none, from=12-3, to=12]
\arrow["{\overset{\tiny \text{\eqref{thm:dimnuc-main:proof::Phi-sum}}}{\approx}}"'{pos=0.3}, shift right=5, draw=none, from=8-3, to=12]
\arrow["{\overset{\tiny \text{\eqref{thm:dimnuc-main:proof::Psi}}}{\circlearrowright}}"{description}, draw=none, from=13, to=4-2]
\arrow["{\overset{\tiny \text{\eqref{thm:dimnuc-main:proof::D-approximation}}}{\approx}}"{pos=0.2}, draw=none, from=15, to=1]
\end{tikzcd}\]


\label{thm:dimnuc-main:proof:big_diagram}

In Stage~A, we construct ``prototypes'' for the algebras $E_{H_{\Ause}}$, 
$D_{\Ause}$, $D'_{\Ause}$, and $C_{\Ause}$, as well as for those maps among 
them that are represented by solid arrows above. These ``prototypes'' are 
denoted by $E_H$, $D_{H,\Fi}$, $D'_{H,\Fi}$, $C_{z, H, \Fi}$, 
$\Theta_{H,\tkrep[]}$, $\Sigma_H$, $\mathrm{Ad}_{u_{\tkrep[]}}$, and 
$\mathrm{ev}_{z, H, \Fi}$. The ``final products'' (namely $E_{H_{\Ause}}$, 
$D_{\Ause}$, etc.) will be obtained later when, in Stage~C, we make the 
parameters in the subscripts (namely $H$, $\Fi$, $z$, etc) depend on a single 
index, $\Ause$. 
Among these ``prototypes'', $E_H$ is the crossed product $\left(c_0(G/H)^+ 
\otimes A\right) \rtimes G$, and $D_{H,\Fi}$ is a corner of $E_H$ associated to 
a finite subset $\Fi$ of $G/H$. By a variant of Green's imprimitivity theorem, 
$D_{H,\Fi}$ is isomorphic (via $\Theta_{H,\tkrep[]}$) to a matrix amplification 
of $A \rtimes H$, denoted by $D'_{H,\Fi}$. This more concrete matrix form 
brings into play the nuclear dimension bound $\dstab$ in the statement 
of 
Theorem~\ref{thm:dimnuc-main}. 
Now $\Sigma_H$ is a map that basically sums up the values of a function in the tensor factor $c_0(G/H)$ (this is well-defined on the corner $D_{H,\Fi}$), and $\mathrm{Ad}_{u_{\tkrep[]}}$ is the reincarnation of $\Sigma_H$ on $D'_{H,\Fi}$. Finally, $C_{z, H, \Fi}$ is the fiber of $D'_{H,\Fi}$ at $z \in X$ and $\mathrm{ev}_{z, H, \Fi}$ is the associated quotient map. 
The construction of these ``prototypes'' set the stage for the main players (namely, compressions by ${\tPo[LTC]}^{(r)}_{\Ause}$ and $\tPonum{\Ause}{l}$, as well as maps derived from them, represented by dashed arrows in the diagram above) to enter in Stage~D and~E. 

In Stage~B, we put to use the three dimension bounds on the right-hand side of the inequality in the statement of Theorem~\ref{thm:dimnuc-main}: first in Step~\eqref{thm:dimnuc-main:proof::Auseco}, we use $\asdim(X, \Enseco_\alpha)$ to obtain a ``partition of unity'' $\left( \Ponum{\Ause}{l} \right)_{l, \Ause}$ as in Proposition~\ref{prop:orbit-asdim}, then in Step~\eqref{thm:dimnuc-main:proof::D-approximation}, we use $\dstab$ to obtain prototypes for the nuclear approximations $C_{\Ause} \xrightarrow{\kappa_{\Ause}}  B_{\Ause} \xrightarrow{\gamma_{\Ause}}  C_{\Ause}$, and finally in Step~\eqref{thm:dimnuc-main:proof::LTC}, we use $\dimltc(\alpha)$ to obtain a ``partition of unity'' $\left(\Po[LTC]^{(r)}\right)_{r}$ as in Proposition~\ref{prop:ltc-dim} \textemdash\ later in Stage~D and~E, these ``partitions of unity'' will give rise to the ``main players'' mentioned above. 
This whole process in Stage~B hinges on a series of delicate choices of 
parameters, which involve several finiteness or compactness arguments. Those 
depend, ultimately, on the finiteness of $A_0$ and $\Lo$, the compactness of 
unit balls in finite-dimensional $C^*$-algebras, and
the finiteness of the collection $\short^{P}$ (Notation~\ref{Notation:modified 
stabilizer}) for any finite subset $P$ in $G$. We remark that the construction of $\short^{P}$ reduces the intractable 
task of dealing 
with an infinite number of stabilizers to a manageable one of tracking only a 
finite number of ``nearsighted stabilizers''.

In Stage~C, we convert the ``prototypes'' from Stage~A into ``final products'' 
by first determining in Step~\eqref{thm:dimnuc-main:proof::Ause-l-r} the range 
of the index $\Ause$ \textemdash\ very roughly speaking, we let $\Ause$ range 
over all the ``partial near orbits'' relevant for our approximation 
\textemdash\ and then choosing for each $\Ause$ suitable parameters $x_{\Ause}$ 
(a chosen ``base point'' in $\Ause$), $H_{\Ause}$ (the ``nearsighted 
stabilizer'' at $x_{\Ause}$), $z_{\Ause}$ (a point at which the evaluation 
$\mathrm{ev}_{\Ause}$ is performed), etc. 
In order to handle the case of $C_0(X)$-algebras, we need to allow $x_{\Ause}$ 
to be different from (but still close to) $z_{\Ause}$ in general.

In Stage~D, we focus on the ``upward'' compressions by 
${\tPo[LTC]}^{(r)}_{\Ause}$, constructed using the ``partition of unity'' 
$\left(\Po[LTC]^{(r)}\right)_{r}$ associated with the long thin covering 
dimension from Stage~B. After constructing the compressions in 
Step~\eqref{thm:dimnuc-main:proof::LTC-compression}, the rest of Stage~D is 
devoted to verifying two crucial properties they enjoy, the first being that 
the compositions of these compressions with the other adjacent upward maps in 
the above diagram (namely, $\Sigma_{H_{\Ause}} \circ 
\operatorname{compr}_{{{\tPo[LTC]}}^{(r)}_{\Ause}} \circ \Theta_{{\Ause}} \circ 
\widetilde{\gamma}_{\Ause}^{(m)}$ or equivalently $\mathrm{Ad}_{\vkrep} \circ 
\operatorname{compr}_{{\tPo[LTC]'{}}^{(r)}_{\Ause}} \circ 
\widetilde{\gamma}_{\Ause}^{(m)}$) are approximately order zero (and thus can 
be perturbed to bona-fide order zero maps $\Phi_{\Ause}^{(r,m)}$ in 
Step~\eqref{thm:dimnuc-main:proof::approx-order-zero}) \textemdash\ this relies 
on the approximate invariance of the functions $\Po[LTC]^{(r)}$ and thus 
ultimately on the ``long'' part in the definition of the long thin covering 
dimension \textemdash\ and the second being that the compressions by 
${\tPo[LTC]}^{(r)}_{\Ause}$, in a certain approximate sense, ``factor through'' 
the evaluation maps $\mathrm{ev}_{\Ause}$ (see 
Step~\eqref{thm:dimnuc-main:proof::fine}) \textemdash\ this relies ultimately 
on the ``thin'' part in the definition of the long thin covering dimension. 

In Stage~E, we introduce the ``downward'' compressions by $\tPonum{\Ause}{l}$, constructed using the ``partition of unity'' $\left( \Ponum{\Ause}{l} \right)_{l, \Ause}$ associated with the asymptotic dimension from Stage~B. 
These compressions are needed to control the matrix sizes of the algebras $B_{\Ause}$; without this control, we would not be able to perturb the approximate order zero maps in Stage~D \textemdash\ this is why, in Stage~B, we need to fix $\left( \Ponum{\Ause}{l} \right)_{l, \Ause}$ before fixing $\left(\Po[LTC]^{(r)}\right)_{r}$. 
When we compose these ``downward'' compressions with the ``upward'' ones from Stage~D, the effect, roughly speaking, amounts to overlaying the two ``partitions of unity'' (see Step~\eqref{thm:dimnuc-main:proof::Sigma-Pi-Delta}), where the functions $\Ponum{\Ause}{l}$ are used to ``modulate'' the functions $\Po[LTC]^{(r)}$ with the help of the near orbit selection functions (made precise in Step~\eqref{thm:dimnuc-main:proof::mu_f}). 
This fact eventually allows us to show in Step~\eqref{thm:dimnuc-main:proof::tally} that $\Phi \circ \Psi$ is a sufficiently good approximation of the identity map on $A \rtimes G$. 

We now proceed with the details of the proof.

\begin{enumerate}[itemindent=*,leftmargin=1em,label=\textit{Step~(A\arabic*).},ref=A\arabic*]

	\item \label{thm:dimnuc-main:proof::Roe} 
	For any subgroup $H \leq G$, we write $\lambda_H$ for the action of $G$ on 
	$c_0(G/H)^+$ by left translation. We consider the diagonal action 
	$\lambda_H \otimes \alpha$ of $G$ on $c_0(G/H)^+ \otimes A$. Since the 
	unital embedding $\mathbf{1} \colon \mathbb{C} \hookrightarrow c_0(G/H)^+$ 
	is $\lambda_H$-equivariant, it induces a $(\lambda_H \otimes 
	\alpha)$-equivariant $*$-homomorphism 
	\[
	\mathbf{1} \otimes \id_A \colon A 
	\cong \mathbb{C} \otimes A \to c_0(G/H)^+ \otimes A
	\]
	 and thus also a 
	$*$-homomorphism 
	\[
		(\mathbf{1} \otimes \id_A) \rtimes G \colon A \rtimes_\alpha G \to \left( c_0(G/H)^+ \otimes A \right) \rtimes_{\lambda_H \otimes \alpha} G \; .
	\]



	\item \label{thm:dimnuc-main:proof::equi-Roe}\label{DMP3}
	For any subgroup $H \leq G$, 
	we fix the following notations:
	\begin{align*}
	\label{dimnuc:def-of-E_H}
	E_{H} &= \left( c_0(G/H)^+ \otimes A \right) \rtimes_{\lambda_H \otimes \alpha} G \\
	\underline{E}_{H} &= \left( c_0(G/H)^+ \otimes C_0(X) \right) \rtimes_{\lambda_H \otimes \alpha} G 
	\end{align*}
	with a canonical $*$-homomorphism $\underline{E}_{H} \to M(E_{H})$ induced by the equivariant $*$-homomorphism $C_0(X) \to M(A)$. 
	
	In addition, for any finite subset $\Fi$ of $G / H$, 
	we consider the indicator function 
	$\chi_{\Fi} \in c_0(G/H)$, viewed as an element 
	in the multiplier algebra of $c_0(G/H)^+ \otimes C_0(X)$, and therefore of $\underline{E}_{H}$ and $E_{H}$.
	This gives rise to $C^*$-subalgebras
	\begin{equation*}
	\label{dimnuc:def-of-D_H}
		D_{H, \Fi} = \chi_{\Fi \vphantom{\krep{j}}} \, E_{H} \, \chi_{\Fi \vphantom{\krep{j}}} \subseteq  E_{H}
		\quad \text{ and } \quad 
		\underline{D}_{H, \Fi} = \chi_{\Fi \vphantom{\krep{j}}} \, \underline{E}_{H} \, \chi_{\Fi \vphantom{\krep{j}}} \subseteq  \underline{E}_{H}
	\end{equation*}
	with a canonical $*$-homomorphism $\underline{D}_{H, \Fi} \to M(D_{H, \Fi})$.

	\item \label{thm:dimnuc-main:proof::equi-Roe-isomorphism}
	For any $H \leq G$ and for any finite subset $\Fi \subseteq G / H$, 
	we give a more concrete description of $D_{H, \Fi}$. 
	For the ease of our exposition, we will only focus on the case with $[\grpid] \in \Fi$, which suffices for our purpose. 
	To this end, 
	we fix a tuple 
	\[
		\vec{k} = \left( \krep[]{1},\krep[]{2},\ldots,\krep[]{|\Fi|} \right) \in G^{|\Fi|}
	\]
	of representatives 
	for $\Fi \subseteq G / H$, with $\krep[]{1} = \grpid$, that is, we have $\Fi = \left\{ [\krep[]{1}],\ldots,[\krep[]{|\Fi|}] \right\}$, 
	where we use $[\krep[]{i}]$ to denote the image of $\krep[]{i}$ in $G/H$. 
	Writing 
	$\chi_{[\krep[]{i}]}$ for the characteristic function of the singleton 
	$\left\{ [\krep[]{i}] \right\}$, viewed as a projection in $c_0(G/H)^+$, 
	and using $\{u_g : g \in G\}$ to denote the canonical unitaries in the 
	crossed product, we
	define, for $i, j \in \intervalofintegers{1}{|\Fi|} $, partial isometries 
	\[
		s_{H, \Fi}^{(i,j)} = u_{\krep[]{i}} \chi_{[\grpid\vphantom{\krep[]{j}}]} u_{\krep[]{j}}^* = \chi_{[\krep[]{i}]} \, u_{\krep[]{i}} u_{\krep[]{j}}^* \, \chi_{[\krep[]{j}]} \in M \left( E_{H} 
		\right)
	\]
	and linear subspaces
	\[
		D_{H, \Fi}^{(i,j)} = \chi_{[\krep[]{i}]} \, E_{H} \, \chi_{[\krep[]{j}]} \subseteq  E_{H} \; .
	\]
	Observe that for any $i, j ,k , l \in \intervalofintegers{1}{|\Fi|} $, we have 
	\[
		D_{H, \Fi}^{(i,j)} D_{H, \Fi}^{(k,l)} = D_{H, \Fi}^{(i,j)} s_{H, \Fi}^{(k,l)} = s_{H, \Fi}^{(i,j)} D_{H, \Fi}^{(k,l)} = 
		\begin{cases}
			D_{H, \Fi}^{(i,k)} \, , & \text{ if } j = k \\
			0  \, , & \text{ if } j \not= k
		\end{cases}
		\, , 
	\]
	\[
		s_{H, \Fi}^{(i,j)} s_{H, \Fi}^{(k,l)} = 
		\begin{cases}
		s_{H, \Fi}^{(i,k)} \, , & \text{ if } j = k \\
		0  \, , & \text{ if } j \not= k
		\end{cases}
		\, , 
	\]
	\[
		\left( D_{H, \Fi}^{(i,j)} \right)^* = D_{H, \Fi}^{(j,i)} \, , \quad \text{ and } \quad  \left( s_{H, \Fi}^{(i,j)} \right)^* = s_{H, \Fi}^{(j,i)} \, ,
	\]
	and the equation $\chi_{\Fi} = \sum_{i = 1 } ^{|\Fi|} \chi_{[\krep[]{i}]}$ leads to a vector-space isomorphism
	\[
		D_{H, \Fi} \cong \bigoplus_{i, j = 1 } ^{|\Fi|} D_{H, \Fi}^{(i,j)} = \bigoplus_{i, j = 1 } ^{|\Fi|} s_{H, \Fi}^{(i,1)} D_{H, \Fi}^{(1,1)} s_{H, \Fi}^{(1,j)}  \; .
	\]
	To identify $D_{H, \Fi}^{(1,1)}$, we observe that since 
	$\chi_{[\grpid\vphantom{\krep[]{j}}]}$ is $\lambda_{|_H}$-invariant, it 
	induces an $H$-equivariant $*$-homomorphism 
	\[
	\left( 
	\chi_{[\grpid\vphantom{\krep[]{j}}]} \otimes \id_A \right) \colon A 
	\hookrightarrow c_0(G/H)^+ \otimes A \, ,
	\] 
	which is split injective via an $H$-invariant $*$-homomorphism
	\[
		\left(\mathrm{ev}_{[\grpid\vphantom{\krep[]{j}}]} \otimes \id_A \right)  \colon  c_0(G/H)^+ \otimes A   \to A 
	\]
	induced by evaluation at $[\grpid\vphantom{\krep[]{j}}]$. 	
	Hence we have an embedding 
	\[
		\left( \chi_{[\grpid\vphantom{\krep[]{j}}]} \otimes \id_A \right) \rtimes H \colon  A \rtimes_{\alpha_{|_H}} H  \hookrightarrow \left( c_0(G/H)^+ \otimes A \right) \rtimes_{(\lambda_H \otimes \alpha)_{|_H}} H \; .
	\]	
	Since 
	\[
		\chi_{[\grpid\vphantom{\krep[]{j}}]} u_{g} \chi_{[\grpid\vphantom{\krep[]{j}}]} = 
		\begin{cases}
			0 , & g \in G \setminus H \\
			\chi_{[\grpid\vphantom{\krep[]{j}}]} u_{g} , & g \in  H
		\end{cases}
		\; ,
	\]
	it follows that 
	\begin{align*}
		D_{H, \Fi}^{(1,1)} &= \chi_{[\grpid\vphantom{\krep[]{j}}]} \, \left(\left( c_0(G/H)^+ \otimes A \right) \rtimes_{\lambda_H \otimes \alpha} G\right) \,\chi_{[\grpid\vphantom{\krep[]{j}}]} \\
		&= \overline{\left\{ \sum_{\text{finitely many $i$} } \chi_{[\grpid\vphantom{\krep[]{j}}]} \, (f_i \otimes a_i) u_{g_i} \, \chi_{[\grpid\vphantom{\krep[]{j}}]} \colon f_i \in c_0(G/H)^+ , \ a_i \in A , \ g_i \in G \right\}} \\
		&= \overline{\left\{ \sum_{\text{finitely many $i$} } (\chi_{[\grpid\vphantom{\krep[]{j}}]} \otimes a_i) u_{g_i}  \colon  a_i \in A , \ g_i \in H \right\}} \; ,
	\end{align*}
	which is precisely the image of the embedding $\left( \chi_{[\grpid\vphantom{\krep[]{j}}]} \otimes \id_A \right) \rtimes H$ above. 
	Hence we have a $*$-isomorphism 
	\[
		\left( \chi_{[\grpid\vphantom{\krep[]{j}}]} \otimes \id_A \right) 
		\rtimes H \colon  A \rtimes_{\alpha_{|_H}} H \xrightarrow{\cong} D_{H, 
		\Fi}^{(1,1)}
	\]
	given by
	\[  
		a u_h  \mapsto \left( \chi_{[\grpid\vphantom{\krep[]{j}}]} \otimes a \right)  u_h 
		\; .
	\]
	Define
	\[
	D'_{H, \Fi} := M_{|\Fi|} \left(A \rtimes_{\alpha_{|_H}} H \right) \, .
	\]
	Combining the above facts, we obtain a $*$-isomorphism
	\[
		\Theta_{H,\vec{k}} \colon D'_{H, \Fi} \xrightarrow{\cong}  D_{H, \Fi} 
	\]
	given by
	\begin{align*}
		e_{i,j} \otimes ( a u_h ) \mapsto&\ s_{H, \Fi}^{(i,1)} \left( \left( \chi_{[\grpid\vphantom{\krep[]{j}}]} \otimes a \right)  u_h \right) s_{H, \Fi}^{(1,j)} \\
		&= u_{\krep[]{i}}  \left( \chi_{[\grpid\vphantom{\krep[]{j}}]} \otimes a \right)  u_h  u_{\krep[]{j}}^* \; .
	\end{align*}

	\item \label{thm:dimnuc-main:proof::equi-Roe-isomorphism-multiplier}
	In Step~\eqref{thm:dimnuc-main:proof::equi-Roe-isomorphism}, 
	if we replace $A$ by $C_0(X)$, we define
	\[
	 \underline{D}'_{H,\Fi} := M_{|\Fi|} \left(C_0(X) \rtimes_{\alpha_{|_H}} H 
	 \right)
	\]
	and obtain an analogous $*$-isomorphism 
	\[
		\underline{\Theta}_{H,\vec{k}} \colon \underline{D}'_{H,\Fi} 
		\xrightarrow{\cong}  
		\underline{D}_{H,\Fi} 
	\]
	given by
		\begin{align*}
		e_{i,j} \otimes ( f u_h ) \mapsto&\ s_{H, \Fi}^{(i,1)} \left( \left( \chi_{[\grpid\vphantom{\krep[]{j}}]} \otimes f \right)  u_h \right) s_{H, \Fi}^{(1,j)} \\
		&= u_{\krep[]{i}}  \left( \chi_{[\grpid\vphantom{\krep[]{j}}]} \otimes 
		f \right)  u_h  u_{\krep[]{j}}^* \ ,.
	\end{align*}
	This fits into a commutative diagram 
	\[
		\xymatrix{
		\underline{D}'_{H,\Fi}
		\ar[r]^-{\underline{\Theta}_{H,\vec{k}}}_-{\cong} \ar[d] & \underline{D}_{H,\Fi}
		\ar[d] ,  \\
		M \left(D'_{H,\Fi}\right) \ar[r]^{\Theta_{H,\vec{k}}}_{\cong} & M \left(D_{H,\Fi}\right) ,  \\
		}
	\]
	where the vertical maps are induced from the $C_0(X)$-algebra structure of $A$ and the bottom map is extended to the multiplier algebras.

	\item \label{thm:dimnuc-main:proof::Sigma-pre}
	For any $H \leq G$, we view $C_c(G/H, A^+)$ as a $G$-invariant $*$-subalgebra of $c_0(G/H)^+ \otimes A^+$ 
	and 
	define a linear map 
	\[
	\Sigma_H \colon C_c(G/H, A^+) \rtimes_{\mathrm{alg}} G \to A^+ \rtimes G 
	\]
	given by
	\[
	 f u_g \mapsto \left( \sum_{[g'] \in G/H} f([g']) \right) u_g \; .
	\]
	It is straightforward to see that $\Sigma_H$ maps $C_c(G/H, A^+)$ onto 
	$A^+$ and satisfies, for any  
	$g,g' \in G $ and for any $f \in C_c(G/H, A^+)$,  
	\[
	\Sigma_H (u_g f u_{g'}) = u_g \Sigma_H (f) u_{g'} \, .
	\]
	We claim that $\Sigma_H$ is completely positive, in the sense that for any 
	element $b \in M_n \left( C_c(G/H, A^+) \rtimes_{\mathrm{alg}} G \right)$, 
	we have $\left(\id_{M_{n}} \otimes \Sigma_H\right) (b^* b) \geq 0$ in $M_n 
	(A^+ \rtimes G)$. Indeed, identifying $M_n ( C_c(G/H, A^+) 
	\rtimes_{\mathrm{alg}} G )$ with $( C_c(G/H) \odot M_n(A^+) ) 
	\rtimes_{\mathrm{alg}} G$ in the obvious way and writing $b$ as a finite 
	sum 
	\[
	\sum_{g \in F \subseteq G} \sum_{[g'] \in F' \subseteq G/H} \left( \chi_{[g']} \otimes b_{g, [g']} \right) u_g
	\]
	where 
	$b_{g, [g']} \in M_n(A^+)$, we have 
	\begin{align*}
	b^* b 
	= &\ \sum_{g_1, g_2 \in F \subseteq G} \sum_{[g'] \in F' \subseteq G/H} u_{g_1}^* \left( \chi_{[g']} \otimes b_{g_1, [g']}^* \, b_{g_2, [g']}  \right)  u_{g_2} 
	\end{align*}	
	and thus
	\begin{align*}
	\left(\id_{M_{n}} \otimes \Sigma_H\right) (b^* b) 
	&= \sum_{g_1, g_2 \in F \subseteq G} \sum_{[g'] \in F' \subseteq G/H} u_{g_1}^*  \left( b_{g_1, [g']}^* \, b_{g_2, [g']} \right)  u_{g_2} \\
	&= \sum_{[g'] \in F' \subseteq G/H} \left( \sum_{g_1 \in F \subseteq G} b_{g_1, [g']} u_{g_1}  \right)^* \left( \sum_{g_2 \in F \subseteq G} b_{g_2, [g']} u_{g_2} \right) \\
	&\geq 0
	\end{align*}
	as claimed.

	\item \label{thm:dimnuc-main:proof::Sigma}
	For any $H \leq G$ and for any finite subset $\Fi \subseteq G / H$, observe that 
	\[
	\chi_{\Fi \vphantom{\krep{j}}} \left( \left( c_0(G/H)^+ \otimes A^+ \right) \rtimes_{\lambda_H \otimes \alpha, \mathrm{alg}} G \right) \chi_{\Fi \vphantom{\krep{j}}} \subseteq C_c(G/H, A^+) \rtimes_{\mathrm{alg}} G \; .
	\]
	Since this $*$-subalgebra has a unit, namely $\chi_{\Fi}$, whose image under the completely positive map $\Sigma_H$ is $\left| \Fi \right| \cdot 1_{A^+}$, it follows that the restriction of $\Sigma_H$ to this $*$-subalgebra extends to a completely positive map of norm $\left| \Fi \right|$ on $\chi_{\Fi} \left( \left( c_0(G/H)^+ \otimes A^+ \right) \rtimes_{\lambda_H \otimes \alpha} G \right) \chi_{\Fi}$, which we still denote by $\Sigma_H$. 
	Note that the enlarged domain now contains $D_{H, \Fi}$ from Step~\eqref{thm:dimnuc-main:proof::equi-Roe}. 
	
	\item \label{thm:dimnuc-main:proof::row-vector}
	For any $H \leq G$, for any finite subset $\Fi \subseteq G / H$, and for any $|R|$-tuple $\tkrep[] = \left(\krep[]{1}, \ldots, \krep[]{|\Fi|} \right)$ of elements in $G$ satisfying $\krep[]{1} = \grpid$ and $\Fi = \left\{ [\krep[]{1}],\ldots,[\krep[]{|\Fi|}] \right\}$, 
	we calculate the composition of $\Sigma_H \colon D_{H, \Fi} \to A \rtimes G$ from Step~\eqref{thm:dimnuc-main:proof::Sigma} with the isomorphism $\Theta_{H,\tkrep[]} \colon D'_{H, \Fi} \to D_{H, \Fi}$ from Step~\eqref{thm:dimnuc-main:proof::equi-Roe-isomorphism}. 
	To this end, 
	we define a row vector 
	\[
		u_{\tkrep[]} = \left(  u_{\krep[]{j}} \right)_{j \in \intervalofintegers{1}{|\Fi|}}
	\]
	with values in $C_0(X)^+ \rtimes_{\alpha} G \subseteq M(A \rtimes_{\alpha} 
	G)$  
	and claim that 
	\begin{equation} \label{thm:dimnuc-main:proof::eq:Sigma-k}
	\Sigma_{H} \circ \Theta_{H,\tkrep[]} = \mathrm{Ad}_{u_{\tkrep[]}} \colon 
	D'_{H, \Fi} \to A \rtimes_{\alpha} G \; .
	\end{equation}
	
	Indeed, for any $i, j \in \intervalofintegers{1}{|\Fi|}$, for any $a \in A$ and for any $h \in H$, by Steps~\eqref{thm:dimnuc-main:proof::equi-Roe-isomorphism} and~\eqref{thm:dimnuc-main:proof::Sigma-pre}, we have 
	\begin{align*}
	&\ \Sigma_{H} \circ \Theta_{H,\tkrep[]} (e_{i,j} \otimes ( a u_h )) \\
	= &\ \Sigma_{H}  \left( u_{\krep[]{i}}  \left( \chi_{[\grpid\vphantom{\krep{j}}]} \otimes a \right)  u_h  u_{\krep[]{j}}^* \right)  \\
	= &\ u_{\krep[]{i}}  a  u_h  u_{\krep[]{j}}^*  \\
	= &\ u_{\tkrep[]} (e_{i,j} \otimes ( a u_h )) u_{\tkrep[]}^* 
	\; .
	\end{align*}
	This is sufficient to prove the claim since both $\Sigma_{H} \circ 
	\Theta_{H,\tkrep[]}$ and $\mathrm{Ad}_{u_{\tkrep[]}}$ are linear maps.

	\item \label{thm:dimnuc-main:proof::evaluation}
	For any $z \in X$,  
	for any $H \leq G_z$, 
	and for any finite subset $\Fi \subseteq G / H$ with $[\grpid] \in \Fi$, 
	we define the $C^*$-algebra 
	\[
		C_{z,H,\Fi} := M_{|\Fi|} \left(A_{z} \rtimes_{\alpha_{|_H}^z} H \right)  \; .
	\]
	Since $z$ is fixed by $\alpha_{|_H}$, the evaluation map $A \to A_{z}$ intertwines the actions $\alpha_{|_H}$ and $\alpha_{|_H}^z$ and thus induces 
	a surjective $*$-homomorphism
	\[
		\mathrm{ev}_{z,H,\Fi} \colon D'_{H, \Fi} \to C_{z,H,\Fi} 
	\]
	given by
	\[
	 e_{i,j} \otimes ( a u_h ) \mapsto e_{i,j} \otimes ( a_z u_h ) \; .
	\]
	Observe that since the  $\mathrm{ker}(A \to A_z) = C_0 (X \setminus \{z\}) 
	\, A$, it follows that 
	\begin{equation} \label{thm:dimnuc-main:proof::eq:ker-z}
		\mathrm{ker} \left( \mathrm{ev}_{z,H,\Fi} \right) = C_0 (X \setminus \{z\}) \, M_{|\Fi|} \left(A \rtimes_{\alpha_{|_H}} H \right)
	\end{equation}
	where we identify $M_{|\Fi|} \left(A \rtimes_{\alpha_{|_H}} H \right)$ with $M_{|\Fi|} \left(A \right) \rtimes_{\alpha_{|_H}} H$ and view $M_{|\Fi|} \left(A \right)$ as a $C_0(X)$-algebra. 
	
\end{enumerate}	
	
\begin{enumerate}[itemindent=*,leftmargin=1em,label=\textit{Step~(B\arabic*).},ref=B\arabic*]
	
	\item \label{thm:dimnuc-main:proof::setup}
	Since we have assumed at the beginning that $A_0$ is a finite subset in 
	$C_{\mathrm{c}}(X) \cdot A$, there exists a compact subset $\Ko$ of $X$ 
	such that $f \cdot a = a$ for any $a \in A_0$ and for any $f \in C_0(X)$ 
	satisfying $f|_{\Ko} = 1$. We use the local compactness of $X$ to 
	choose another compact subset $\Ko'$ of $X$ whose interior contains $\Ko$. 
	Set 
	\begin{equation*}
	\Er = \frac{ \eps }{ (c+1)(d+1)(\dstab +8) } \; .
	\end{equation*}

	\item \label{thm:dimnuc-main:proof::Auseco}
	We use  
	Proposition \ref{prop:orbit-asdim} to pick
	\begin{itemize}
		\label{dimnuc:choice-of-asdim-data}
		\item a finite subset $\Bo \subseteq G$,
		\item collections $\Auseconum{0} , \ldots , \Auseconum{c}$ of disjoint 
		finite subsets of $X$, with $\Auseco = \Auseconum{0} \cup \ldots 
		\cup \Auseconum{c}$, and
		\item finitely supported functions $\Ponum{\Ause}{l} \colon X \to 
		[0,1]$ for $l = 0, \ldots, c$ 
		and for $\Ause \in \Auseconum{l}$ 
	\end{itemize}
	satisfying the following conditions:
	{
		\renewcommand{\theenumi}{}	
		\begin{enumerate}
			
			\renewcommand{\labelenumii}{\textup{(\theenumii)}} 
			\renewcommand{\theenumii}{Bo}
			\item[\namedlabel{item:main:use:prop:orbit-asdim:pou:Bo}{(Bo-asdim)}]
			For any $\Ause \in \Auseco$ and for any $x , y \in \Ause$, we have $x 
			\in 
			\alpha_{\Bo} (y)$.

			\renewcommand{\labelenumii}{\textup{(\theenumii)}} 
			\renewcommand{\theenumii}{Su}
			\item[\namedlabel{item:main:use:prop:orbit-asdim:pou:Su}{(Su-asdim)}]
			For any $l \in \intervalofintegers{0}{d}$, for any $\Ause 
			\in 
			\Auseconum{l}$, and for any $x \in X$, 
			we have either $\Ponum{\Ause}{l}(x) = 0 $ or $\alpha_{\Lo}(x) 
			\subseteq \Ause$.  
			
			\renewcommand{\labelenumii}{\textup{(\theenumii)}} 
			\renewcommand{\theenumii}{In}
			\item[\namedlabel{item:main:use:prop:orbit-asdim:pou:Li}{(In-asdim)}]
			For any $l \in \intervalofintegers{0}{c}$ and for any $\Ause \in 
			\Auseconum{l}$, the function $\Ponum{\Ause}{l}$ is \emph{$(\Lo, 
				\Er^2)$-invariant} in the sense that 
			\[ \left| \Ponum{\Ause}{l} (x) - 
			\Ponum{\Ause}{l} \left( \alpha_{g^{-1}} (x)  \right) \right| \leq \Er^2 
			\quad \text{ for any } g \in \Lo \text{ and for any } x \in X \; .
			\] 
			
			\renewcommand{\labelenumii}{\textup{(\theenumii)}} 
			\renewcommand{\theenumii}{Un}
			\item[\namedlabel{item:main:use:prop:orbit-asdim:pou:Un}{(Un-asdim)}]
			For any $x \in \Ko'$, we have $\displaystyle \sum_{l = 0}^{c} 
			\sum_{\Ause \in \Auseconum{l}} \Ponum{\Ause}{l} (x) = 1$ while for any 
			other $x \in X$, we have $\displaystyle \sum_{l = 0}^{c} \sum_{\Ause 
				\in \Auseconum{l}} \Ponum{\Ause}{l} (x) \leq 1$. 
		\end{enumerate}
	}

	As we may enlarge $\Bo$ if need be, in order to slightly simplify notation 
	later, we assume that $\Bo = \Bo^{-1}$ and $\grpid \in \Lo \subseteq \Bo$.

	\item \label{thm:dimnuc-main:proof::D-approximation}
	For any $z \in X$, for any $H \in \short^{\tLo}_z$ as in Notation~\ref{Notation:modified stabilizer}, and for any subset $\Fi \subseteq \Bo H / H$ with $[\grpid] \in \Fi$, 
	we define a set 
	\begin{multline*}
		D_{H, \Fi, 0} := \Big\{  (f \otimes 1)^{1/2} \, ((1 \otimes a) u_g) \, (f \otimes 1)^{1/2} \in D_{H, \Fi} \colon \\ 
		a \in A_0, g 
		\in \Lo, f \in \ball  \ell^\infty(\Fi) _{+, \leq 1} \Big\} \; .
	\end{multline*}
	This is a compact subset of $\left( D_{H, \Fi} \right)_{\leq 1}$ since 
	$\ell^\infty(\Fi)$ is finite dimensional and all the elements in the 
	product above are contractive. 
	It induces the following compact sets: 
	\[
		D'_{H, \Fi, 0} := 
		\bigcup_{\left \{ \vec{k} \in \Bo^{|\Fi|} \colon \Fi = \left\{ 
		[\krep[]{1}],\ldots,[\krep[]{|\Fi|}] \right\} \right \} } 
		\Theta_{H,\vec{k}}^{-1} (D_{H,\Fi,0}) \subseteq \left( D'_{H, \Fi} 
		\right)_{\leq 1}
	\]
	and 
	\[
		C_{z,H,\Fi, 0} := \mathrm{ev}_{z,H,\Fi}  (D'_{H,\Fi,0}) \subseteq \left( C_{z,H,\Fi} \right)_{\leq 1} \; .
	\]
	(The definition of $D'_{H, \Fi, 0}$ produces a compact set which is bigger 
	than what we really need, because the definition involves taking a union 
	over 
	the images of the isomorphisms $\Theta_{H,\vec{k}}^{-1}$ under all possible 
	choices of tuples $\vec{k}$, whereas 
	we really use just one. However, the tuple $\vec{k}$ will only be fixed  
	in Step~(\ref{thm:dimnuc-main:proof::H-W}) below, and there is no harm in 
	defining this larger set.)
	
	By Proposition~\ref{prop:dimnuc-basic}, 
	we have
	\[
	\dimnuc(C_{z,H,\Fi}) =\dimnuc \left(A_{z} \rtimes_{\alpha_{|_H}^z}  H \right) \leq 
	\dstab
	\; . 
	\]
	By Definition~\ref{def:dimnuc}, we may therefore choose a finite dimensional $C^*$-algebra $B_{z,H,\Fi}$ and completely 
	positive maps
	\begin{equation*}
		\label{dimnuc:def-of-B_H-and-gamma-kappa}
	\xymatrix{
		C_{z,H,\Fi} \ar[r]^{\kappa_{z,H,\Fi}} & B_{z,H,\Fi} \ar[r]^{\gamma_{z,H,\Fi}} & C_{z,H,\Fi}
	}
	\end{equation*}
	such that 
	\begin{equation*}
	\label{dimnuc:eqn:gamma-kappa}
		\|\gamma_{z,H,\Fi} \circ \kappa_{z,H,\Fi} (a) -  a \| < \ErB \quad \text{ for any } a \in C_{z,H,\Fi, 0} \; ,
	\end{equation*}
	$\kappa_{z,H,\Fi}$ is contractive, and the map $\gamma_{z,H,\Fi}$ is 
	$(\dstab + 1)$-decomposable, that is, there is a decomposition
	\[
		\gamma_{z,H,\Fi} = \gamma_{z,H,\Fi}^{(0)} + \ldots + \gamma_{z,H,\Fi}^{(\dstab)} \colon B_{z,H,\Fi} \to C_{z,H,\Fi}
	\] 
	into $(\dstab + 1)$ completely positive order zero contractions.

	\item \label{thm:dimnuc-main:proof::lifting}
	For any $z \in X$, for any $H \in \short^{\tLo}_z$, and for any finite subset $\Fi \subseteq G$ with $\grpid \in \Fi$, 
	by Lemma~\ref{lem:lifting-order-zero}, there is a
	$(\dstab + 1)$-decomposable map 
	\[
	\widetilde{\gamma}_{z,H,\Fi} = \widetilde{\gamma}_{z,H,\Fi}^{(0)} + \ldots + \widetilde{\gamma}_{z,H,\Fi}^{(\dstab)} \colon B_{z,H,\Fi} \to D'_{H, \Fi}
	\]
	such that for any $m \in \intervalofintegers{0}{\dstab}$, we have $\mathrm{ev}_{z,H,\Fi} \circ \widetilde{\gamma}_{z,H,\Fi}^{(m)} = \gamma_{z,H,\Fi}^{(m)}$ and $\widetilde{\gamma}_{z,H,\Fi}^{(m)}$ is a completely positive order zero contraction. 
	
	It follows that for any $a \in  D'_{H,\Fi, 0}$ as defined in Step~\eqref{thm:dimnuc-main:proof::D-approximation}, we have 
	\begin{align*}
	\left\| \mathrm{ev}_{z,H,\Fi} \circ \widetilde{\gamma}_{z,H,\Fi} \circ \kappa_{z,H,\Fi} \circ \mathrm{ev}_{z,H,\Fi} (a) - \mathrm{ev}_{z,H,\Fi}  (a) \right\| < \ErB \; ,
	\end{align*}
	that is, $\left( \widetilde{\gamma}_{z,H,\Fi} \circ \kappa_{z,H,\Fi} \circ \mathrm{ev}_{z,H,\Fi} (a) - a \right)$ is in the $\ErB$-neighborhood of the kernel $\mathrm{ker} \left( \mathrm{ev}_{z,H,\Fi} \right)$. 
	By the compactness of $D'_{H,\Fi, 0}$, there is a finite subset $D'_{z, H,\Fi, \mathrm{diff}}$ of $\mathrm{ker} \left( \mathrm{ev}_{z,H,\Fi} \right)$ such that $\left( \widetilde{\gamma}_{z,H,\Fi} \circ \kappa_{z,H,\Fi} \circ \mathrm{ev}_{z,H,\Fi} (a) - a \right)$ is in the $2 \ErB$-neighborhood of $D'_{z, H,\Fi, \mathrm{diff}}$ for any $a \in  D'_{H,\Fi, 0}$.

	\item \label{thm:dimnuc-main:proof::approx-unit}
	For any $z \in X$, using the finiteness of $\short^{\tLo}_z$, $\Bo$, and 
	$D'_{z, H,\Fi, \mathrm{diff}}$ 
	as well as 
	Equation \eqref{thm:dimnuc-main:proof::eq:ker-z}, 
	we apply Lemma~\ref{lem:quasicentral-approximate-unit} to 
	\[
		\bigoplus_{H \in \short^{\tLo}_z} \bigoplus_{^{\Fi \subseteq \Bo \colon}_{\grpid \in \Fi}} \mathrm{ker} \left( \mathrm{ev}_{z,H,\Fi} \right) = C_0 (X \setminus \{z\}) \left( \bigoplus_{H \in \short^{\tLo}_z} \bigoplus_{^{\Fi \subseteq \Bo \colon}_{\grpid \in \Fi}} M_{|\Fi|} \left(A \rtimes_{\alpha_{|_H}} H \right) \right)
	\]
	to obtain $f_{z} \in C_c(X \setminus \{z\})_{+, \leq 1}$ such that 
	for any $H \in \short^{\tLo}_z$, for any $\Fi \subseteq \Bo$ with $\grpid \in \Fi$,  and for any $b\in D'_{z, H,\Fi, \mathrm{diff}}$, 
	\[
		b \left( 1 - f_{z} \right) \approx_{\ErB} 0 \approx_{\ErB} \left( 1 - f_{z} \right) b \; ,
	\] 
	and thus by Step~\eqref{thm:dimnuc-main:proof::lifting}, we have, for any $a \in  D'_{H,\Fi, 0}$,  
	\begin{align*}
		\left( \widetilde{\gamma}_{z,H,\Fi} \circ \kappa_{z,H,\Fi} \circ \mathrm{ev}_{z,H,\Fi} (a) - a \right) \left( 1 - f_{z} \right) &\approx_{3\ErB} 0  \text{ and } \\
	\left( 1 - f_{z} \right) \left( \widetilde{\gamma}_{z,H,\Fi} \circ \kappa_{z,H,\Fi} \circ \mathrm{ev}_{z,H,\Fi} (a) - a \right) &\approx_{3\ErB} 0 \; .
	\end{align*}

	\item \label{thm:dimnuc-main:proof::thin-cover}
	For any $z \in X$, since $f_z$ is compactly supported in $X \setminus \{z\}$, there is an open neighborhood $V_z$ of $z$ such that $f_z |_{\overline{ V_z }} = 0$. 
	Since $\Ko$ is contained in the interior of $\Ko'$, we may assume without loss of generality that 
	\begin{align*}
		\text{either } \quad V_z \subseteq \Ko' \quad \text{ or } \quad V_z \subseteq X \smallsetminus \Ko
	\end{align*}
	Since fixed-point subsets of the form $X^H$ are closed, we may also make the following assumption without loss of generality: 
	\[
		V_z \cap X^H = \varnothing \quad \text{ for any subgroup } H \in \short^{\tLo} \text{ satisfying } z \not \in X^H \; .
	\]

	We fix a 
	finite subset $Z \subseteq X$ such that
	$ \{ V_z \colon z \in Z \}$ covers $\alpha^\cup_{\Bo} ( 
	\Ko )$.

	\item \label{thm:dimnuc-main:proof::order-zero-stability}
	We  use Lemma~\ref{lem:perturb-order-zero} to find $\ErPreLTC > 0 $ 
	such that for any $z \in Z$, for any $H \in \short^{\tLo}_z$, and for any $\Fi \subseteq \Bo$ with $\grpid \in \Fi$, whenever $C$ is a $C^*$-algebra and 
	whenever $\varphi \colon B_{z,H,\Fi} \to C$ is a $*$-linear contraction satisfying 
	\begin{equation} \label{thm:dimnuc-main:proof::eq:approx-order-zero}
		\|\varphi(b)\varphi(b') - \varphi(bb')\varphi(1)\| \leq \ErPreLTC \quad \text{ for any } b, b' \in \left( B_{z,H,\Fi} \right)_{\leq 1} \; ,
	\end{equation}
	there exists an order zero map $\tilde{\varphi} \colon B_{z,H,\Fi} \to C$ such that 
	$\| \tilde{\varphi} - \varphi \| < \ErOrd$. 
	Indeed, as $Z$, $\short^{\tLo}$ and $\Bo$ are finite, 
	we can obtain such a positive number $\ErPreLTC$ as the minimum of the 
	ones 
	which occur for each algebra $B_{z,H,\Fi}$ as $z$ ranges over $Z$, $H$ ranges over $\short^{\tLo}_z$, and $R$ ranges over all subsets of $\Bo$ containing $\grpid$.

	\item \label{thm:dimnuc-main:proof::lift-image}
	For any $H \in \short^{\tLo}$ and for any $z \in Z$ fixed by $H$, the unit 
	ball $\left( B_{z,H,\Fi} \right)_{\leq 1}$ is compact. Since $Z$ is finite, 
	it follows that the union 
	\[
		\bigcup_{z \in Z \cap X^H} \bigcup_{m = 0}^{\dstab} \widetilde{\gamma}_{z,H,\Fi}^{(m)} \left( \left( B_{z,H,\Fi} \right)_{\leq 1} \right)
	\]
	is also a compact subset of $\left( D'_{H, \Fi} \right)_{\leq 1}$. 
	For each $z \in Z$ fixed by $H$, $A_{z} \rtimes_{\alpha_{|_H}^z} H$ is the 
	closure of $A_{z} \rtimes_{\mathrm{alg},\alpha_{|_H}^z} H$, which consists 
	of elements with finite supports in $H$. We can thus choose 
	a finite subset $D'_{H, \Fi, 1}$ of $\left( D'_{H, \Fi} \right)_{\leq 1}$ 
	that is $\frac{\ErPreLTC}{6}$-dense in the above union and consists of 
	matrices 
	with 
	entries in $A_{z} \rtimes_{\mathrm{alg},\alpha_{|_H}^z} H$. 
	Hence there is a finite symmetric subset $\Aubo_H$ of $H$ such that $D'_{H, \Fi, 1}$ consists of matrices with entries supported on $\Aubo_H$. 
	
	\item \label{thm:dimnuc-main:proof::T-eta}
	Define the finite symmetric subset 
	\[
		\Aubo = \Bo \cup \left( \bigcup_{H \in  \short^{\tLo}} \Aubo_{H} 
		\right) \subseteq G \, .
	\]
	Set 
	\[
		\ErLTC = \min \left\{ \ErB , \sqrt{\frac{\ErPreLTC}{18 |\Bo|^2 |\Aubo|}} \right\}  \; .
	\]

	\item \label{thm:dimnuc-main:proof::LTC}
	Since $ \{ X \smallsetminus \alpha^\cup_{\Bo} \left( \Ko \right), V_z \colon z \in Z \}$ is an open cover of $X$ by Step~\eqref{thm:dimnuc-main:proof::thin-cover}, 
	we use
	Proposition 
	\ref{prop:ltc-dim} 
	to pick
	\begin{itemize}
		\item collections $\Coseconum{0} , \ldots , \Coseconum{d}$ of disjoint open 
		subsets of $X$, together with $\Coseco = \Coseconum{0} \cup \ldots \cup 
		\Coseconum{d}$ (and we write $\Cose^{(r)} = \bigcup \Coseconum{r}$, for  $r \in \intervalofintegers{0}{d}$), 
		\item locally constant functions $\Ne^{(r)} \colon \Cose^{(r)} \to X$ for $r \in \intervalofintegers{0}{d}$, 
		\item  continuous functions $\Po[LTC]^{(r)} \colon X \to [0,1]$ for 
		$r 
		\in \intervalofintegers{0}{d}$, 
		and 
		\item a finite subset $\Aubo' \subseteq G$, 
	\end{itemize}
	satisfying the conditions of  Proposition 
	\ref{prop:ltc-dim}, with $\Aubo$ in place of $\Lo$ and with $\ErLTC$ in place 
	of 
	$\Er$; in particular, we have:
			{
		\begin{enumerate}
			\renewcommand{\labelenumi}{\textup{(\theenumi)}}
			\renewcommand{\theenumi}{Lo}
			\item[\namedlabel{item:prop:dimnuc:Lo}{(Lo-LTC)}]
			For any $x \in \Ko$, there exists $\Cose \in \Coseco$ such that 
			$\alpha_{\Aubo} (x) \subseteq \Cose$.

			\renewcommand{\labelenumi}{\textup{(\theenumi)}} 
			\renewcommand{\theenumi}{Eq}
			\item[\namedlabel{item:prop:dimnuc:Eq}{(Eq-LTC)}]
			For any $r \in \intervalofintegers{0}{d}$,  $\Ne^{(r)}$ is {$\Aubo$-equivariant}.		
			
			\renewcommand{\labelenumi}{\textup{(\theenumi)}} 
			\renewcommand{\theenumi}{Th}
			\item[\namedlabel{item:prop:dimnuc:Th}{(Th-LTC)}]
			For any $y \in \bigcup_{r = 0}^{d} \Ne^{(r)} \left( \Cose^{(r)} \right)$, 
			either 
			\[
			\left( \{y\} \cup \left( \bigcup_{r = 0}^{d} 
			\left(\Ne^{(r)}\right)^{-1}(y) \right) \right) \cap 
			\alpha^\cup_{\Bo} \left( \Ko \right) = \varnothing
			\] 
			or  there exists $z \in Z$ such that 
			\[
			\{y\} \cup \left( \bigcup_{r = 0}^{d} 
			\left(\Ne^{(r)}\right)^{-1}(y) \right)
			\subseteq V_z \, .
			\]
			\renewcommand{\labelenumi}{\textup{(\theenumi)}} 
			\renewcommand{\theenumi}{Bo}
			\item[\namedlabel{item:prop:dimnuc:Bo}{(Bo-LTC)}]
			For any $r \in \intervalofintegers{0}{d}$, for any $\Cose \in \Coseconum{r}$ and for any $x, y \in \Cose$, we have $\Ne^{(r)}(x) \in \alpha_{\Aubo'} \left( \Ne^{(r)}(y) \right)$. 
			
			\renewcommand{\labelenumi}{\textup{(\theenumi)}} 
			\renewcommand{\theenumi}{Fi}
			\item[\namedlabel{item:prop:dimnuc:Fi}{(Fi-LTC)}]
			The collection $\Coseco$ is finite. 
			
			\renewcommand{\labelenumi}{\textup{(\theenumi)}} 
			\renewcommand{\theenumi}{Su}
			\item[\namedlabel{item:prop:dimnuc:Su}{(Su-LTC)}]	
			For any $r \in \intervalofintegers{0}{d}$ and for any $x$ in 
			the support of 
			$\Po[LTC]^{(r)}$, 
			we have $\alpha_{\Aubo} (x) \subseteq \Cose$ for some $\Cose \in \Coseco^{(r)}$. 
			
			\renewcommand{\labelenumi}{\textup{(\theenumi)}} 
			\renewcommand{\theenumi}{In}
			\item[\namedlabel{item:prop:dimnuc:Li}{(In-LTC)}]		
			For each $r \in \intervalofintegers{0}{d}$,
			the function 
			$\Po[LTC]^{(r)}$ is 
			\emph{$(\Aubo, \ErLTC^2)$-invariant} in the sense that $\left| 
			\Po[LTC]^{(r)} 
			(x) - 
			\Po[LTC]^{(r)} \left( \alpha_{g^{-1}} (x)  \right) \right| < 
			\ErLTC^2$ 
			for 
			any 
			$g \in \Aubo$ and $x \in X$. 
			
			\renewcommand{\labelenumi}{\textup{(\theenumi)}} 
			\renewcommand{\theenumi}{Un}
			\item[\namedlabel{item:prop:dimnuc:Un}{(Un-LTC)}]		
			For any $x \in \Ko$, we have 
			\[
			\sum_{r = 0}^{d} 
			\Po[LTC]^{(r)} (x)= 
			1
			\]
			while for any other $x \in X$, we have 
			\[
			\sum_{r = 0}^{d} 
			\Po[LTC]^{(r)} (x) \leq 1
			\, .
			\]
		\end{enumerate}
	}

\end{enumerate}	

\begin{enumerate}[itemindent=*,leftmargin=1em,label=\textit{Step~(C\arabic*).},ref=C\arabic*]

	\item \label{thm:dimnuc-main:proof::Ause-l-r}
	For any $l \in \{0,1,\ldots,c\}$ and for any $r \in \{0,1,\ldots,d\}$, we define
	\begin{equation} \label{thm:dimnuc-main:proof::eq:Ause-l-r}
	\Auseco^{(l,r)} = \left\{\Ause \in  \Auseco^{(l)} \colon 
	\left. \Po[LTC]^{(r)} \right|_{\left( \Ne^{(r)} \right)^{-1} (\Ause) \cap \Ko} \not= 0
	\right\} \, .
	\end{equation}
	We claim that
	\begin{equation} \label{thm:dimnuc-main:proof::eq:Ause-l-r-property}
		\Ause \subseteq \Ne^{(r)} \left( \Cose^{(r)} \cap \alpha^\cup_{\Bo} \left( \Ko \right) \right) \quad \text{ for any } \Ause \in \Auseco^{(l,r)} 
	\end{equation}
	and that $\Auseco^{(l,r)}$ is a finite collection of disjoint finite 
	subsets of $X$. 
	
	To prove the first claim, 
	we fix some $x \in \left( \Ne^{(r)} \right)^{-1} (\Ause) \cap \Ko$ with 
	$\Po[LTC]^{(r)} (x) \not= 0$. Since $\Bo \subseteq \Aubo$ by 
	Step~\eqref{thm:dimnuc-main:proof::T-eta}, using \ref{item:prop:dimnuc:Su}, 
	there exists $\Cose \in \Coseconum{r}$ such that  $\alpha_{\Bo} (x) 
	\subseteq \Cose \cap \alpha^\cup_{\Bo} \left( \Ko \right)$. 
	Thus, by \ref{item:prop:dimnuc:Eq}, we have 
	\[
	\alpha_{\Bo} \left( \Ne^{(r)} (x) \right) \subseteq \Ne^{(r)} \left( 
	\Cose^{(r)} \cap \alpha^\cup_{\Bo} \left( \Ko \right)  \right)
	\, .
	\]
	This proves the first claim, as $\Ause \subseteq \alpha_{\Bo} \left( 
	  \Ne^{(r)} (x) \right)$ by \ref{item:main:use:prop:orbit-asdim:pou:Bo}. 
	
	Now the second claim follows from the first one  
	since $\Ne^{(r)} \left( \Cose^{(r)}  \right)$ is finite 
	by \ref{item:prop:dimnuc:Bo} and \ref{item:prop:dimnuc:Fi}, while $\Auseco^{(l)}$ is disjoint.

	\item \label{thm:dimnuc-main:proof::orbit-map}
	For any $x \in X$, 
	it follows from 
	Lemma~\ref{Lemma:injective-quotient} that the quotient map $\left( \Bo  G_{x}^{\tLo} \right) / G_{x}^{\tLo} \twoheadrightarrow \left( \Bo G_x \right) / G_x$ is bijective. 
	Since the map $\left( \Bo G_x \right) / G_x \to \alpha_{\Bo} (x)$ given by 
	$[g] \mapsto \alpha_g (x)$ is also bijective, 
	we obtain a bijection 
	\[
	\omega_{x} \colon  \left( \Bo  G_{x}^{\tLo} \right) / G_{x}^{\tLo} 
	\xrightarrow{\cong}  \alpha_{\Bo} (x) 
	\]
	given by
	\[
	\omega_{x} ( [g] ) = \alpha_g (x) \, .
	\]
	
	\item \label{thm:dimnuc-main:proof::H-W}
	For any $\Ause \in \bigcup_{l = 0}^{c} \bigcup_{r = 0}^{d} 
	\Auseco^{(l,r)}$, we fix $x_{\Ause} \in \Ause$. To lighten notation, we 
	write 
	\[
	H_{\Ause} = G_{x_{\Ause}}^{\tLo} 
	\, .
	\] 	
	It follows from condition \ref{item:main:use:prop:orbit-asdim:pou:Bo} in Step~\eqref{thm:dimnuc-main:proof::Auseco} that
	$\Ause 
	\subseteq 
	\alpha_{\Bo} (x_{\Ause})$. We can thus choose a tuple
	\[
		\tkrep = \left( \krep[\Ause]{1}, \ldots, \krep[\Ause]{|\Ause|} \right) \in \Bo^{|\Ause|}
	\]
	with $\krep[\Ause]{1} = \grpid$ such that 
	\[
		\Ause = \left\{ \alpha_{\krep[\Ause]{i}} (x_{\Ause}) \colon i \in \intervalofintegers{1}{|\Ause|} \right\} \; .
	\]
	Write $\Fi_{\Ause} = \left\{ \left[ \krep[\Ause]{i} \right] \colon i \in \intervalofintegers{1}{|\Ause|} \right\} \subseteq \Bo H_{\Ause} / H_{\Ause}$. It follows from Step~\eqref{thm:dimnuc-main:proof::orbit-map} that $\Ause = \omega_{x_{\Ause}} (\Fi_{\Ause})$. We thus write 
	\[
		\omega_{\Ause} \colon \Fi_{\Ause} \xrightarrow{\cong} \Ause
	\]
	for the restriction of the bijection $\omega_{x_{\Ause}}$ to $\Fi_{\Ause}$. 
	
	\item \label{thm:dimnuc-main:proof::z-W}
	For any $\Ause \in \bigcup_{l = 0}^{c} \bigcup_{r = 0}^{d} \Auseco^{(l,r)}$, 
	it follows from \eqref{thm:dimnuc-main:proof::eq:Ause-l-r-property} that 
	\[
		\left( \bigcup_{r = 0}^{d} \left(\Ne^{(r)}\right)^{-1} \left(x_{\Ause}\right) \right) \cap \alpha^\cup_{\Bo} \left( \Ko \right) \not= \varnothing
	\]
	and thus by \ref{item:prop:dimnuc:Th}, there exists $z_{\Ause} \in Z$ such that 
	\[
		\{x_{\Ause}\} \cup \left( \bigcup_{r = 0}^{d} \left(\Ne^{(r)}\right)^{-1}(x_{\Ause}) \right)
		\subseteq V_{z_{\Ause}}	\; .
	\]
	By Steps~\eqref{thm:dimnuc-main:proof::thin-cover} 
	and~\eqref{thm:dimnuc-main:proof::H-W}, we have $x_{\Ause} \in 
	V_{z_{\Ause}} \cap X^{H_{\Ause}}$. Thus $ z_{\Ause}$ is fixed by 
	$H_{\Ause}$, whence $H_{\Ause} \in \short^{\tLo}_{z_{\Ause}}$. 
	It follows that we may apply the constructions in Steps~\eqref{thm:dimnuc-main:proof::evaluation}, \eqref{thm:dimnuc-main:proof::D-approximation} and~\eqref{thm:dimnuc-main:proof::lifting} 
	to define $C_{z_{\Ause},H_{\Ause}}$, $B_{z_{\Ause},H_{\Ause}}$, etc. 
	
	\item \label{thm:dimnuc-main:proof::abbrev}
	For any $\Ause \in \bigcup_{l = 0}^{c} \bigcup_{r = 0}^{d} 
	\Auseco^{(l,r)}$, 
	we introduce the following abbreviations to items defined in  
	Steps~\eqref{thm:dimnuc-main:proof::equi-Roe}, 
	\eqref{thm:dimnuc-main:proof::evaluation}, 
	\eqref{thm:dimnuc-main:proof::D-approximation} 
	and~\eqref{thm:dimnuc-main:proof::lifting}: 
	\begin{align*}
		& D_{\Ause} := D_{H_{\Ause}, \Fi_{\Ause}}, && D'_{\Ause} := D'_{H_{\Ause},\Fi_{\Ause}} , && \Theta_{\Ause} = \Theta_{H_{\Ause}, \tkrep}  , \\ 
		& \underline{D}_{\Ause} := \underline{D}_{H_{\Ause}, \Fi_{\Ause}}, && \underline{D}'_{\Ause} := \underline{D}'_{H_{\Ause},\Fi_{\Ause}} , && \underline{\Theta}_{\Ause} = \underline{\Theta}_{H_{\Ause}, \tkrep}  , \\ 
		& C_{\Ause} := C_{z_{\Ause},H_{\Ause}, \Fi_{\Ause}}, && \mathrm{ev}_{\Ause} := \mathrm{ev}_{z_{\Ause},H_{\Ause}, \Fi_{\Ause}}  && \vkrep := u_{\tkrep} , \\ 
		& D_{\Ause, 0} := D_{H_{\Ause}, \Fi_{\Ause}, 0}, && D'_{\Ause,0} := D'_{H_{\Ause}, \Fi_{\Ause}, 0} , && C_{\Ause, 0} := C_{z_{\Ause},H_{\Ause}, \Fi_{\Ause}, 0} ,  && \\ 
		& B_{\Ause} := B_{z_{\Ause},H_{\Ause}, \Fi_{\Ause}} , && D'_{\Ause,\mathrm{diff}} := D'_{z_{\Ause},H_{\Ause}, \Fi_{\Ause}, \mathrm{diff}} , && D'_{\Ause, 1} := D'_{H_{\Ause}, \Fi_{\Ause}, 1} ,  \\ 
		& \kappa_{\Ause} := \kappa_{z_{\Ause},H_{\Ause}, \Fi_{\Ause}} , &&  \gamma_{\Ause} := \gamma_{z_{\Ause},H_{\Ause}, \Fi_{\Ause}} , && \widetilde{\gamma}_{\Ause} := \widetilde{\gamma}_{z_{\Ause},H_{\Ause}, \Fi_{\Ause}} ,  \\ 
		&  \gamma_{\Ause}^{(m)} := \gamma_{z_{\Ause},H_{\Ause}, \Fi_{\Ause}}^{(m)} \text{ and} && \widetilde{\gamma}_{\Ause}^{(m)}  := \widetilde{\gamma}_{z_{\Ause},H_{\Ause}, \Fi_{\Ause}}^{(m)}   && \text{ for } m \in \intervalofintegers{0}{\dstab} . \\
	\end{align*}

\end{enumerate}	

\begin{enumerate}[itemindent=*,leftmargin=1em,label=\textit{Step~(D\arabic*).},ref=D\arabic*]

	\item \label{thm:dimnuc-main:proof::mu_f}
	
	For any $r \in \intervalofintegers{0}{d}$ and for any (not necessarily continuous) function $f \colon X \to \mathbb{C}$, we define 
	\[
		\Po[LTC]^{(r)}_{f} \colon X \to \mathbb{C}
	\]
	by
	\[
		\Po[LTC]^{(r)}_{f} ( x ) = 
		\begin{cases}
			f \left( \Ne^{(r)} (x) \right) \, \Po[LTC]^{(r)} (x) \, , & \text{ for any } x \in \Cose^{(r)} \\
			0 \, , & \text{ for any } x \in X \smallsetminus \Cose^{(r)}
		\end{cases}
		\; .
	\]
	Any such function $\Po[LTC]^{(r)}_{f}$ is continuous 
	because $\Ne^{(r)}$ is locally constant and $\Po[LTC]^{(r)}$ is supported in $\Cose^{(r)}$ by \ref{item:prop:dimnuc:Su}. 
	We claim that the map 
	\[
		\Po[LTC]^{(r)}_{-} \colon \ell^\infty (X) \to C_0(X) 
	\]
	taking $f$ to $\Po[LTC]^{(r)}_{f}$ 
	is a completely positive contractive order zero map that takes $1$ to $\Po[LTC]^{(r)}$ and factors through the restriction homomorphism 
	\[
		\ell^\infty (X) \to \ell^\infty \left( \left\{ y \in X \colon \alpha_{\Aubo} (y) \subseteq \Ne^{(r)} \left( \Cose \right) \text{ for some } \Cose \in \Coseco^{(r)} \right\} \right) \; .
	\]
	Indeed, it is clear that $\Po[LTC]^{(r)}_{-}$ is linear and positive (and thus completely positive as the target is commutative), and satisfies $\Po[LTC]^{(r)}_{1} = \Po[LTC]^{(r)}$, whence it is also contractive. It is immediate from the definition that $\Po[LTC]^{(r)}_{f_1} \Po[LTC]^{(r)}_{f_2} = 0$ whenever $f_1 f_2 = 0$, which shows $\Po[LTC]^{(r)}_{-}$ is order zero. Finally, we have 
	\[
	\left\| \Po[LTC]^{(r)}_{f} \right\| \leq \max \left\{ |f (y)| \colon \alpha_{\Aubo} (y) \subseteq \Ne^{(r)} \left( \Cose \right) \text{ for some } \Cose \in \Coseco^{(r)} \right\} \leq \left\| f  \right\|_\infty 
	\] 
	by \ref{item:prop:dimnuc:Su} and \ref{item:prop:dimnuc:Eq}, 
	whence in particular $\Po[LTC]^{(r)}_{f} = \Po[LTC]^{(r)}_{f'}$ whenever $f (y) = f'(y)$ for any $y$ such that $\alpha_{\Aubo} (y) \subseteq \Ne^{(r)} \left( \Cose \right)$ for some $\Cose \in \Coseco^{(r)}$. 	
	
	\item \label{thm:dimnuc-main:proof::mu_Y}
	For any $r \in \intervalofintegers{0}{d}$ and for any subset $Y \subseteq X$, 
	we write $\Po[LTC]^{(r)}_{Y} = \Po[LTC]^{(r)}_{\chi_Y}$ and 
	\[
		\Cose_Y^{(r)} = \left(\Ne^{(r)}\right)^{-1} \left( Y \right) \; .
	\]
	This set is relatively clopen in $\Cose^{(r)}$. The function 
	$\Po[LTC]^{(r)}_{Y}$ satisfies 
	\[
		\Po[LTC]^{(r)}_{Y} = \Po[LTC]^{(r)} \, 
		\chi_{\Cose^{(r)}_{Y}}^{\vphantom{\int abcd}} \in C_0 \left(\Cose^{(r)}_{Y} \right)_{+, \leq 1} \subseteq C_0(\Cose^{(r)})_{+, \leq 1} \; .
	\]
	If $Y = \{y\}$ is a singleton, we write $\Cose_y^{(r)}$ for 
	$\Cose_{\{y\}}^{(r)}$ and $\Po[LTC]^{(r)}_{y}$ for 
	$\Po[LTC]^{(r)}_{\{y\}}$. 
	Because $\Po[LTC]^{(r)}_{-}$ is linear and order zero,
	for any $f \colon X \to \mathbb{C}$, we have an orthogonal sum 
	\begin{equation}	\label{thm:dimnuc-main:proof::mu_Y::eq:decomposition}
		\Po[LTC]^{(r)}_{f} = \sum_{y \in X} f(y) \, \Po[LTC]^{(r)}_{y} \, .
	\end{equation}

	\item \label{thm:dimnuc-main:proof::LTC-compression}
	For any $r \in \intervalofintegers{0}{d}$ and for any $\Ause \in \bigcup_{l = 0}^{c} \Auseco^{(l,r)}$, 
	using the bijection $\omega_{\Ause}$ from Step~\eqref{thm:dimnuc-main:proof::compression}, 
	we define 
	\begin{align*}
	{\tPo[LTC]}^{(r)}_{\Ause} = 
	\sum_{[g] \in \Fi_{\Ause}} \chi_{[g]}  \otimes \Po[LTC]^{(r)}_{\omega_{\Ause} \left( [g] \right)} 
	\in \left( \chi_{\Fi_{\Ause} \vphantom{\krep{j}}} \cdot c_0\left(G/H_{\Ause}\right)^+ \otimes C_0(X) \right)_{+, \leq 1} & \\
	\subseteq \underline{D}_{{\Ause}}
	\; . &
	\end{align*}
	We also define the diagonal matrix 
	\begin{align*}
	\tPo[LTC]'{}{}^{(r)}_{\Ause} = 
	\text{diag} \left (
	 \alpha_{\krep[{\Ause}]{i}}^{-1} \left({\Po[LTC]}^{(r)}_{\omega_{\Ause} 
	 \left( \left[ \krep[{\Ause}]{i} \right] \right)}\right)  \right)_{i \in 
	 \intervalofintegers{1}{|\Ause|}} 
	&\in \left( M_{|\Ause|} \left(C_0(X) \right) \right)_{+, \leq 1} \\
	&\subseteq 
	\underline{D}'_{\Ause}
	\; .
	\end{align*}
	Under the isomorphism established in Step~\eqref{thm:dimnuc-main:proof::equi-Roe-isomorphism-multiplier}, we have 
	\begin{align*}
	\Theta_{\Ause} \left( \tPo[LTC]'{}{}^{(r)}_{\Ause} \right) 
	&= \sum_{i = 1}^{|\Ause|} \Theta_{H_{\Ause}, \tkrep} \left( e_{i,i} \otimes \alpha_{\krep{i}}^{-1} \left({{\Po[LTC]}^{(r)}_{\omega_{\Ause} \left( \left[ \krep{i} \right] \right)}}\right) \right) \\
	&= \sum_{i = 1}^{|\Ause|} u_{\krep{i}}  \left( \chi_{[\grpid\vphantom{\krep{j}}]} \otimes \alpha_{\krep{i}}^{-1} \left({{\Po[LTC]}^{(r)}_{\omega_{\Ause} \left( \left[ \krep{i} \right] \right)}}\right) \right)   u_{\krep{i}}^* \\
	&= \sum_{i = 1}^{|\Ause|} \left( \chi_{[\krep{i}]} \otimes {{\Po[LTC]}^{(r)}_{\omega_{\Ause} \left( \left[ \krep{i} \right] \right)}} \right) \\
	&= {\tPo[LTC]}^{(r)}_{\Ause} \; ,
	\end{align*}
	where the last step follows from Step~\eqref{thm:dimnuc-main:proof::orbit-map}. 
	We thus have a commutative diagram involving the isomorphism established in Step~\eqref{thm:dimnuc-main:proof::equi-Roe-isomorphism}: 
	\begin{equation} \label{thm:dimnuc-main:proof::eq:mu-compressions-square}
		\xymatrix{
			D_{{\Ause}} \ar[rr]^{\operatorname{compr}_{{\tPo[LTC]}^{(r)}_{\Ause}}} \ar[d]_{\Theta_{{\Ause}}}^{\cong} && D_{{\Ause}} \ar[d]_{\Theta_{{\Ause}}}^{\cong} \\
			D'_{{\Ause}} \ar[rr]^{\operatorname{compr}_{\tPo[LTC]'{}{}^{(r)}_{\Ause}}} && D'_{{\Ause}} 
		}
	\end{equation}


	\item \label{thm:dimnuc-main:proof::Sigma-Pi}
	For any $r \in \intervalofintegers{0}{d}$ and for any $\Ause \in \bigcup_{l = 0}^{c} \Auseco^{(l,r)}$, 
	using the constructions in 
	Steps~\eqref{thm:dimnuc-main:proof::LTC-compression} 
	and~\eqref{thm:dimnuc-main:proof::Sigma}, 
	we consider the composition 
	\begin{align*}
	& \chi_{\Fi_{\Ause} \vphantom{\krep{j}}} \left( \left( c_0\left(G/H_{\Ause}\right)^+ \otimes A^+ \right) \rtimes_{\lambda_{H_{\Ause}} \otimes \alpha} G \right) \chi_{\Fi_{\Ause} \vphantom{\krep{j}}} \\
	& \quad \xrightarrow{\operatorname{compr}_{{\tPo[LTC]}^{(r)}_{\Ause}}} 
	\chi_{\Fi_{\Ause} \vphantom{\krep{j}}} \left( \left( c_0\left(G/H_{\Ause}\right)^+ \otimes A^+ \right) \rtimes_{\lambda_{H_{\Ause}} \otimes \alpha} G \right) \chi_{\Fi_{\Ause} \vphantom{\krep{j}}} \\
	& \qquad \xrightarrow{\Sigma_{H_{\Ause}}} A^+ \rtimes G \; ,
	\end{align*}
	which is completely positive. It is also contractive since it maps the unit $\chi_{\Fi_{\Ause} \vphantom{\krep{j}}}$ of the domain to 
	\[
	\Sigma_{H_{\Ause}} \left( {\tPo[LTC]}^{(r)}_{\Ause} \right) = \sum_{[g] \in \Fi_{\Ause}}  \Po[LTC]^{(r)}_{\omega_{\Ause} \left( [g] \right)} 
	= \Po[LTC]^{(r)}_{ \Ause } \quad \text{ by Steps~\eqref{thm:dimnuc-main:proof::orbit-map},  \eqref{thm:dimnuc-main:proof::mu_f} and~\eqref{thm:dimnuc-main:proof::mu_Y} } \; , 
	\]
	whose norm is no more than $1$.

	\item \label{thm:dimnuc-main:proof::adjunctions}
	
	For any $r \in \intervalofintegers{0}{d}$ and for any $\Ause \in \bigcup_{l = 0}^{c} \Auseco^{(l,r)}$, using notations from Steps~\eqref{thm:dimnuc-main:proof::row-vector}, \eqref{thm:dimnuc-main:proof::abbrev} and~\eqref{thm:dimnuc-main:proof::LTC-compression}, 
	it follows from \eqref{thm:dimnuc-main:proof::eq:mu-compressions-square} and \eqref{thm:dimnuc-main:proof::eq:Sigma-k} that 
	\begin{align*}
		\Sigma_{H_{\Ause}} \circ \operatorname{compr}_{{{\tPo[LTC]}}^{(r)}_{\Ause}} \circ \Theta_{{\Ause}} 
		& = \mathrm{Ad}_{\vkrep} \circ \operatorname{compr}_{{\tPo[LTC]'{}}^{(r)}_{\Ause}} \\
		& = \mathrm{Ad}_{\vkrep \left( {\tPo[LTC]'{}}^{(r)}_{\Ause} \right)^{1/2}}  \colon D'_{{\Ause}} \to A \rtimes G 
		\; .
	\end{align*}
	Thus, by Step~\eqref{thm:dimnuc-main:proof::Sigma-Pi}, we have
	\begin{equation*}
		\left( \vkrep \left( {\tPo[LTC]'{}}^{(r)}_{\Ause} \right)^{1/2} \right) \left( \vkrep \left( {\tPo[LTC]'{}}^{(r)}_{\Ause} \right)^{1/2} \right)^* =  \Sigma_{H_{\Ause}} \circ \operatorname{compr}_{{{\tPo[LTC]}}^{(r)}_{\Ause}} (1) = \Po[LTC]^{(r)}_{ \Ause } 
		\; .
	\end{equation*}
	Therefore, the range of the positive contraction $\mathrm{Ad}_{\vkrep} 
	\circ \operatorname{compr}_{{\tPo[LTC]'{}}^{(r)}_{\Ause}}$ is contained in 
	the hereditary subalgebra $\overline{\Po[LTC]^{(r)}_{\Ause} (A \rtimes G) 
	\Po[LTC]^{(r)}_{\Ause}}$ generated by $\Po[LTC]^{(r)}_{\Ause}$. 
	
	Using the orthogonality of the functions $\Po[LTC]^{(r)}_{\omega_{\Ause} 
	\left( \left[ \krep[{\Ause}]{j} \right] \right)}$, for $j = 1, \ldots, 
	|\Ause|$, we have 
	\begin{align*}
		&\ \left( \vkrep \left( {\tPo[LTC]'{}}^{(r)}_{\Ause} \right)^{1/2} \right)^* \left( \vkrep \left( {\tPo[LTC]'{}}^{(r)}_{\Ause} \right)^{1/2} \right) \\
		= & \  \left(  u_{\krep[{\Ause}]{i}}^* \,  \sqrt{\Po[LTC]^{(r)}_{\omega_{\Ause} \left( \left[ \krep[{\Ause}]{i} \right] \right)}} \,  \sqrt{\Po[LTC]^{(r)}_{\omega_{\Ause} \left( \left[ \krep[{\Ause}]{j} \right] \right)}} \, u_{\krep[{\Ause}]{j}}  \right)_{i,j \in \intervalofintegers{1}{|\Ause|}} \\
		= & \ \text{diag} \left(  \alpha_{\krep[{\Ause}]{i}}^{-1} 
		\left({\Po[LTC]^{(r)}_{\omega_{\Ause} \left( \left[ \krep[{\Ause}]{i} 
		\right] \right)}}\right)  \right)_{i \in 
		\intervalofintegers{1}{|\Ause|}} \\
		= & \  {\tPo[LTC]'{}}^{(r)}_{\Ause}
	\end{align*}
	in $M_{n_{x_{\Ause}}} \left(C_0(X) \right)$. In particular, $\mathrm{Ad}_{\vkrep} \circ \operatorname{compr}_{{\tPo[LTC]'{}}^{(r)}_{\Ause}}$ is a completely positive contraction.

	\item \label{thm:dimnuc-main:proof::mu-invariant}
	We claim that for any $g \in \Aubo$ and for any $f \in \ell^\infty(X)$, we have
	\begin{equation*}
	\label{eqn:main:LTC-PoU-almost-equiv}
	\left \| \alpha_g(\Po[LTC]^{(r)}_{f\vphantom{,\Cose}}) - \Po[LTC]^{(r)}_{\alpha_{g} (f) \vphantom{,\Cose}} 
	\right \| < 
	\ErLTC^2  \left\| f  \right\|_\infty \, .
	\end{equation*}
	and thus by Lemma~\ref{Lemma:sqrt}, we also have
	\begin{equation*}
	\label{eqn:main:LTC-PoU-almost-equiv-sqrt}
	\left \| \alpha_g \left ( \sqrt{ \Po[LTC]^{(r)}_{f\vphantom{,\Cose}} } \right ) - \sqrt{ 
		\Po[LTC]^{(r)}_{\alpha_{g}(f) \vphantom{,\Cose}} } \right \| 
	< \ErLTC  
	\quad \text{ if } \left\| f  \right\|_\infty  \leq 1 
	.
	\end{equation*}
	
	To check this claim, we need to show that for any $x \in X$, we have
	\[
	\left | \alpha_g(\Po[LTC]^{(r)}_{f\vphantom{,\Cose}}) (x) - \Po[LTC]^{(r)}_{\alpha_{g} (f) \vphantom{,\Cose}} (x) \right |
	< 
	\ErLTC^2 \, .
	\]
	We may assume at least one term in this difference is nonzero (or else there is nothing to prove). 
	By symmetry, we may further assume without loss of generality that $\alpha_g(\Po[LTC]^{(r)}_{f\vphantom{,\Cose}}) (x) = \Po[LTC]^{(r)}_{f\vphantom{,\Cose}} \left(\alpha_{g^{-1}}(x)\right) \not= 0$. 
	It then follows from Step~\eqref{thm:dimnuc-main:proof::mu_f} that 
	$\alpha_{g^{-1}}(x) \in \Cose^{(r)}$ and, setting $y = \Ne^{(r)} 
	(\alpha_{g^{-1}}(x))$, we have
	\[
	\Po[LTC]^{(r)}_{f\vphantom{,\Cose}} \left(\alpha_{g^{-1}}(x)\right) = f 
	\left( y \right) \Po[LTC]^{(r)} \left(\alpha_{g^{-1}}(x)\right)
	\, .
	\]
	It follows from \ref{item:prop:dimnuc:Su} that $x \in \Cose^{(r)}$ and then 
	from \ref{item:prop:dimnuc:Eq} that $\Ne^{(r)} (x) = \alpha_{g}(y)$, whence 
	we have 
	\[
	\Po[LTC]^{(r)}_{\alpha_{g} (f) \vphantom{,\Cose}} (x) =  (\alpha_{g} (f)) 
	\left( \alpha_{g}(y) \right) \Po[LTC]^{(r)} (x) = f(y) \Po[LTC]^{(r)} (x)
	\, .
	\] 
	The claimed inequality thus follows from \ref{item:prop:dimnuc:Li}.

	\item \label{thm:dimnuc-main:proof::almost-scalar}
	For any $r \in \intervalofintegers{0}{d}$ and for any $\Ause \in \bigcup_{l = 0}^{c} \Auseco^{(l,r)}$, 
	it follows from Steps~\eqref{thm:dimnuc-main:proof::mu_Y}, \eqref{thm:dimnuc-main:proof::mu-invariant}, \eqref{thm:dimnuc-main:proof::H-W}, and~\eqref{thm:dimnuc-main:proof::orbit-map} that 
	\[
	\alpha_{\krep[{\Ause}]{i}}^{-1} \left( {{\Po[LTC]}^{(r)}_{\omega_{\Ause} \left( \left[ \krep[{\Ause}]{i} \right] \right)}} \right) 
	\approx_{\ErLTC^2}
	{ {\Po[LTC]}^{(r)}_{\alpha_{\krep[{\Ause}]{i}}^{-1} \left( \alpha_{\krep[{\Ause}]{i}} \left( x_{\Ause} \right) \right)} }
	=
	{{\Po[LTC]}^{(r)}_{x_{\Ause}}}
	\]
	and similarly
	\[
	\alpha_{\krep[{\Ause}]{i}}^{-1} \left( \sqrt{{\Po[LTC]}^{(r)}_{\omega_{\Ause} \left( \left[ \krep[{\Ause}]{i} \right] \right)}} \right) 
	\approx_{\ErLTC}
	\sqrt{{\Po[LTC]}^{(r)}_{x_{\Ause}}}
	\]
	for any $i \in \intervalofintegers{1}{|\Ause|}$. Hence by Step~\eqref{thm:dimnuc-main:proof::LTC-compression}, we have  
	\[
	\tPo[LTC]'{}{}^{(r)}_{\Ause} 
	\approx_{\ErLTC^2} 
	{\Po[LTC]}^{(r)}_{x_{\Ause}}  \cdot 1_{|\Ause|}  
	\quad \text{ and } \quad 
	\left( \tPo[LTC]'{}{}^{(r)}_{\Ause} \right) ^{1/2}
	\approx_{\ErLTC} 
	\left( {\Po[LTC]}^{(r)}_{x_{\Ause}} \right)^{1/2} \cdot 1_{|\Ause|}  
	\]
	in $\left( M_{|\Ause|} \left(C_0(X) \right) \right)_{+, \leq 1}$, 
	where $1_{|\Ause|}$ denotes the identity matrix. 
	
	Since $\Aubo_{H_{\Ause}} \subseteq 
	G_{x_{\Ause}}$ by Steps~\eqref{thm:dimnuc-main:proof::lift-image} and~\eqref{thm:dimnuc-main:proof::H-W}, 
	a similar computation also yields 
	\[
	\alpha_{h} \left( {{\Po[LTC]}^{(r)}_{x_{\Ause}}} \right) 
	\approx_{\ErLTC^2}
	{ {\Po[LTC]}^{(r)}_{\alpha_{h} \left( x_{\Ause} \right)} }
	=
	{{\Po[LTC]}^{(r)}_{x_{\Ause}}} 
	\quad \text{ for any } h \in \Aubo_{H_{\Ause}} \; .
	\]
	\item \label{thm:dimnuc-main:proof::fine}
	For any $r \in \intervalofintegers{0}{d}$ and for any $\Ause \in \bigcup_{l = 0}^{c} \Auseco^{(l,r)}$, 
	we claim that  for any $ a' \in  D'_{\Ause, 0} $ we have
	\[
	\left\| \operatorname{compr}_{\tPo[LTC]'{}^{(r)}_{\Ause}} \left( \widetilde{\gamma}_{\Ause} \circ \kappa_{\Ause} \circ \mathrm{ev}_{\Ause} (a') - a' \right) \right\| \leq 5 \ErB
	\; .
	\]
	In particular, by Steps~\eqref{thm:dimnuc-main:proof::D-approximation} 
	and~\eqref{thm:dimnuc-main:proof::abbrev}, the above is true if we replace 
	$a'$ by $\Theta_{{\Ause}}^{-1} \circ \mathrm{compr}_{\tPonum{\Ause}{l}}   
	\left(\left( 1_{c_0\left(G/H_{\Ause}\right)^+} \otimes a \right) u_g 
	\right)$ for any $a \in A_0$ and for any $g \in \Lo$. 
	
	To prove this claim,  
	observe that by Steps~\eqref{thm:dimnuc-main:proof::mu_Y} and~\eqref{thm:dimnuc-main:proof::z-W}, 
	we have 
	\[
	{\Po[LTC]}^{(r)}_{x_{\Ause}} \in C_0 \left( \Cose_{x_{\Ause}}^{(r)} \right) \subseteq C_0 \left( V_{z_{\Ause}} \right) \; ,
	\]
	whence by Step~\eqref{thm:dimnuc-main:proof::thin-cover}, we have $\left( {\Po[LTC]}^{(r)}_{x_{\Ause}} \right)^{1/2} f_{z_{\Ause}} = 0$. 
	It now follows from Step~\eqref{thm:dimnuc-main:proof::almost-scalar} and then Step~\eqref{thm:dimnuc-main:proof::approx-unit} that for any $a' \in  D'_{\Ause, 0} \subseteq \left( D'_{\Ause} \right)_{\leq 1}$, we have
	\begin{align*}
	&\ \operatorname{compr}_{\tPo[LTC]'{}^{(r)}_{\Ause}} \left( \widetilde{\gamma}_{\Ause} \circ \kappa_{\Ause} \circ \mathrm{ev}_{\Ause} (a') - a' \right) \\
	= &\ \left( \tPo[LTC]'{}{}^{(r)}_{\Ause} \right) ^{1/2}   \left( \widetilde{\gamma}_{\Ause} \circ \kappa_{\Ause} \circ \mathrm{ev}_{\Ause} (a') - a' \right) \left( \tPo[LTC]'{}{}^{(r)}_{\Ause} \right) ^{1/2}   \\
	\approx_{2 \ErLTC} &\ \left( {\Po[LTC]}^{(r)}_{x_{\Ause}} \right)^{1/2} \cdot  \left( \widetilde{\gamma}_{\Ause} \circ \kappa_{\Ause} \circ \mathrm{ev}_{\Ause} (a') - a' \right) \left( \tPo[LTC]'{}{}^{(r)}_{\Ause} \right) ^{1/2}     \\
	= &\ \left( {\Po[LTC]}^{(r)}_{x_{\Ause}} \right)^{1/2} \left( 1 - f_{z_{\Ause}} \right) \cdot  \left( \widetilde{\gamma}_{\Ause} \circ \kappa_{\Ause} \circ \mathrm{ev}_{\Ause} (a') - a' \right) \left( \tPo[LTC]'{}{}^{(r)}_{\Ause} \right) ^{1/2}    \\
	\approx_{3 \ErB} &\ 0 \; .
	\end{align*}  
	This proves the claim, as $\ErLTC \leq \ErB$ by 
	Step~\eqref{thm:dimnuc-main:proof::T-eta}.

	\item \label{thm:dimnuc-main:proof::commutator}
	
	For any $r \in \intervalofintegers{0}{d}$, for any $\Ause \in \bigcup_{l = 0}^{c} \Auseco^{(l,r)}$, and for any $m \in \intervalofintegers{0}{\dstab}$, using notations from Steps~\eqref{thm:dimnuc-main:proof::LTC-compression} and ~\eqref{thm:dimnuc-main:proof::lifting}, 
	we claim that 
	\[
		\left\| \left[ {\tPo[LTC]'{}}^{(r)}_{\Ause} , \widetilde{\gamma}_{\Ause}^{(m)} (b) \right] \right\| \leq \ErPreLTC / 2 \quad \text{ for any } b \in \left( B_{\Ause} \right)_{\leq 1} \; .
	\]
	
	Indeed, by Steps~\eqref{thm:dimnuc-main:proof::lift-image} 
	and~\eqref{thm:dimnuc-main:proof::abbrev}, there exists $b' \in D'_{\Ause, 
	1} \subseteq \left( D'_{\Ause} \right)_{\leq 1}$ such that 
	\begin{align*}
		\left\| b' - \widetilde{\gamma}_{\Ause}^{(m)} (b) \right\| \leq \ErPreLTC / 6
		\; \text{ and thus } \;
		\left\| \left[ {\tPo[LTC]'{}}^{(r)}_{\Ause} , \widetilde{\gamma}_{\Ause}^{(m)} (b) \right] \right\| 
		& \leq \left\| \left[ {\tPo[LTC]'{}}^{(r)}_{\Ause} , b' \right] \right\| + 2 \ErPreLTC / 6 
	\end{align*}
	Hence it suffices to show 
	\[
		\left\| \left[ {\tPo[LTC]'{}}^{(r)}_{\Ause} , b' \right] \right\| \leq \ErPreLTC / 6 \quad \text{ for any } b' \in D'_{\Ause, 1} \; .
	\]
	To this end, 
	by our choice of $\Aubo_{H_{\Ause}}$ in Step~\eqref{thm:dimnuc-main:proof::lift-image}, each $b' \in D'_{\Ause, 1}$ can be written as a sum of $|\Ause|^2 \cdot  |\Aubo_{H_{\Ause}}|$ elements of the form $e_{i,j} \otimes \left( b'' u_h \right)$ with $i, j \in \intervalofintegers{1}{|\Ause|}$, $b'' \in A_{\leq 1}$, and $h \in \Aubo_{H_{\Ause}}$, where $e_{i,j}$ is the elementary matrix with a single nonzero entry at $(i,j)$ and $b''$ is the $h$-th Fourier coefficient of the $(i,j)$-th entry of $b'$. 
	Since by Step~\eqref{thm:dimnuc-main:proof::T-eta} and \ref{item:main:use:prop:orbit-asdim:pou:Bo}, we have 
	\[
		\ErPreLTC / 6 = {3 \ErLTC^2}{|\Bo|^2 |\Aubo|} \geq {3 \ErLTC^2}{|\Ause|^2 |\Aubo_{H_{\Ause}}|} \; , 
	\]
	the desired inequality then follows, by the (bi-)linearity of the commutator bracket, from the following computation that uses Step~\eqref{thm:dimnuc-main:proof::almost-scalar} twice: 
	\begin{align*}
		\left[ \tPo[LTC]'{}{}^{(r)}_{\Ause} , e_{i,j} \otimes \left( b'' u_h \right) \right] 
		& \approx_{2 \ErLTC^2} \left[ 1_{|\Ause|} \otimes  {\Po[LTC]}^{(r)}_{x_{\Ause}}  , e_{i,j} \otimes \left( b'' u_h \right) \right] \\
		& = e_{i,j} \otimes \left( {\Po[LTC]}^{(r)}_{x_{\Ause}} b'' u_h - b'' u_h {\Po[LTC]}^{(r)}_{x_{\Ause}} \right) \\
		& = e_{i,j} \otimes b'' u_h  \left( \alpha_{ h^{-1}} \left( {\Po[LTC]}^{(r)}_{x_{\Ause}}  \right) - {\Po[LTC]}^{(r)}_{x_{\Ause}}  \right) \\
		& \approx_{\ErLTC^2} 0 \; . 
	\end{align*}

	\item \label{thm:dimnuc-main:proof::approx-order-zero}
	For any $r \in \intervalofintegers{0}{d}$, for any $\Ause \in \bigcup_{l = 0}^{c} \Auseco^{(l,r)}$, and for any $m \in \intervalofintegers{0}{\dstab}$, we 
	claim that the completely positive contraction  
	\[
		\mathrm{Ad}_{\vkrep} \circ \operatorname{compr}_{{\tPo[LTC]'{}}^{(r)}_{\Ause}} \circ \widetilde{\gamma}_{\Ause}^{(m)} \colon B_{\Ause} \to \overline{\Po[LTC]^{(r)}_{\Ause} (A \rtimes G) \Po[LTC]^{(r)}_{\Ause}} \subseteq A \rtimes G
	\]
	satisfies \eqref{thm:dimnuc-main:proof::eq:approx-order-zero} (see also Step~\eqref{thm:dimnuc-main:proof::adjunctions} for restriction on the range). Consequently by Step~\eqref{thm:dimnuc-main:proof::order-zero-stability}, there exists an order zero map
	\[
		\Phi_{\Ause}^{(r,m)} \colon B_{\Ause} \to \overline{\Po[LTC]^{(r)}_{\Ause} (A \rtimes G) \Po[LTC]^{(r)}_{\Ause}} \subseteq A \rtimes G
	\] 
	such that 
	\[
		\left\| \Phi_{\Ause}^{(r,m)} - \mathrm{Ad}_{\vkrep} \circ \operatorname{compr}_{{\tPo[LTC]'{}}^{(r)}_{\Ause}} \circ \widetilde{\gamma}_{\Ause}^{(m)} \right\| < \ErOrd \; .
	\]
	
	To prove the claim, we apply Steps~\eqref{thm:dimnuc-main:proof::adjunctions} and~\eqref{thm:dimnuc-main:proof::commutator} as well as the order-zero-ness of $\widetilde{\gamma}_{\Ause}^{(m)}$ to see that for any $b, b' \in \left( B_{\Ause} \right)_{\leq 1}$, we have 
	\begin{align*}
		& \ \left( \mathrm{Ad}_{\vkrep} \circ \operatorname{compr}_{{\tPo[LTC]'{}}^{(r)}_{\Ause}} \circ \widetilde{\gamma}_{\Ause}^{(m)} \right) (b) \  \left( \mathrm{Ad}_{\vkrep} \circ \operatorname{compr}_{{\tPo[LTC]'{}}^{(r)}_{\Ause}} \circ \widetilde{\gamma}_{\Ause}^{(m)} \right) (b') \\
		= & \ \left( \vkrep \left( {\tPo[LTC]'{}}^{(r)}_{\Ause} \right)^{1/2} \right) \left( \widetilde{\gamma}_{\Ause}^{(m)} (b) \right) \left( {\tPo[LTC]'{}}^{(r)}_{\Ause} \right) \left( \widetilde{\gamma}_{\Ause}^{(m)}  (b') \right) \left( \vkrep \left( {\tPo[LTC]'{}}^{(r)}_{\Ause} \right)^{1/2} \right)^* \\
		\approx_{\ErPreLTC / 2} & \ \left( \vkrep \left( {\tPo[LTC]'{}}^{(r)}_{\Ause} \right)^{1/2} \right) \left( \widetilde{\gamma}_{\Ause}^{(m)} (b) \right) \left( \widetilde{\gamma}_{\Ause}^{(m)}  (b') \right) \left( {\tPo[LTC]'{}}^{(r)}_{\Ause} \right) \left( \vkrep \left( {\tPo[LTC]'{}}^{(r)}_{\Ause} \right)^{1/2} \right)^*  \\
		= & \  \left( \vkrep \left( {\tPo[LTC]'{}}^{(r)}_{\Ause} \right)^{1/2} \right) \left( \widetilde{\gamma}_{\Ause}^{(m)} (b b') \right) \left( \widetilde{\gamma}_{\Ause}^{(m)}  (1) \right) \left( {\tPo[LTC]'{}}^{(r)}_{\Ause} \right) \left( \vkrep \left( {\tPo[LTC]'{}}^{(r)}_{\Ause} \right)^{1/2} \right)^*  \\
		\approx_{\ErPreLTC / 2} & \  \left( \vkrep \left( {\tPo[LTC]'{}}^{(r)}_{\Ause} \right)^{1/2} \right) \left( \widetilde{\gamma}_{\Ause}^{(m)} (b b') \right) \left( {\tPo[LTC]'{}}^{(r)}_{\Ause} \right) \left( \widetilde{\gamma}_{\Ause}^{(m)}  (1) \right) \left( \vkrep \left( {\tPo[LTC]'{}}^{(r)}_{\Ause} \right)^{1/2} \right)^*  \\
		= & \  \left( \mathrm{Ad}_{\vkrep} \circ \operatorname{compr}_{{\tPo[LTC]'{}}^{(r)}_{\Ause}} \circ \widetilde{\gamma}_{\Ause}^{(m)} \right) (b b') \  \left( \mathrm{Ad}_{\vkrep} \circ \operatorname{compr}_{{\tPo[LTC]'{}}^{(r)}_{\Ause}} \circ \widetilde{\gamma}_{\Ause}^{(m)} \right) (1) \; , 
	\end{align*}
	as claimed.


	\item \label{thm:dimnuc-main:proof::Phi-sum}
	For any $r \in \intervalofintegers{0}{d}$ and for any $l \in \intervalofintegers{0}{c}$, 
	it follows from the disjointness of $\Auseco^{(l,r)}$ in Step~\eqref{thm:dimnuc-main:proof::Ause-l-r} and 	
	the order-zero-ness of $\Po[LTC]^{(r)}_{-}$ in Step~\eqref{thm:dimnuc-main:proof::mu_f} 
	that the functions $\Po[LTC]^{(r)}_{\Ause}$, for $\Ause \in \Auseco^{(l,r)}$, are orthogonal.   It then follows from Step~\eqref{thm:dimnuc-main:proof::adjunctions} that the maps $\mathrm{Ad}_{\vkrep} \circ \operatorname{compr}_{{\tPo[LTC]'{}}^{(r)}_{\Ause}}$, for $\Ause \in \Auseco^{(l,r)}$, have orthogonal ranges. 
	
	Hence it follows from Step~\eqref{thm:dimnuc-main:proof::fine} that 
	for any tuple $\left( a_{\Ause} \right)_{\Ause \in \Auseco^{(l,r)}}$ with 
	$a_{\Ause} \in D'_{\Ause, 0}$, we have
	\[
		\left\| \sum_{\Ause \in \Auseco^{(l,r)}} \mathrm{Ad}_{\vkrep} \circ \operatorname{compr}_{\tPo[LTC]'{}^{(r)}_{\Ause}} \left( \widetilde{\gamma}_{\Ause} \circ \kappa_{\Ause} \circ \mathrm{ev}_{\Ause} (a_{\Ause}) - a_{\Ause} \right) \right\| \leq 5 \ErB
		\; .
	\]
	
	Likewise, applying this orthogonality to the construction of Step~\eqref{thm:dimnuc-main:proof::approx-order-zero}, we see
	that for any $m \in \intervalofintegers{0}{\dstab}$, the completely positive map 
	\[
		\Phi^{(l,r,m)} := 
		\sum_{\Ause \in \Auseco^{(l,r)}}  \Phi_{\Ause}^{(r,m)} \colon \bigoplus_{\Ause \in \Auseco^{(l,r)}} B_{\Ause} \to  A \rtimes G
	\]
	is contractive and order zero and satisfies 
	\[
		\left\| \Phi^{(l,r,m)} - \left( \sum_{\Ause \in \Auseco^{(l,r)}} \mathrm{Ad}_{\vkrep} \circ \operatorname{compr}_{{\tPo[LTC]'{}}^{(r)}_{\Ause}} \circ \widetilde{\gamma}_{\Ause}^{(m)} \right) \right\| < \ErOrd \; .
	\]
	
\end{enumerate}	

\begin{enumerate}[itemindent=*,leftmargin=1em,label=\textit{Step~(E\arabic*).},ref=E\arabic*]

	\item \label{thm:dimnuc-main:proof::compression}
	For any $l \in \intervalofintegers{0}{c}$ and for any $\Ause \in \bigcup_{r = 0}^{d} \Auseco^{(l,r)}$, 
	it follows from Step~\eqref{thm:dimnuc-main:proof::H-W} that the bijection 
	$\omega_{\Ause} \colon  \Fi_{\Ause} \to \Ause$ 
	induces a $*$-isomorphism
	\[
	\omega_{\Ause}^* \colon \ell^\infty\left( \Ause \right) 
	\xrightarrow{\cong} \ell^\infty\left(\Fi_{\Ause} \right) \;.
	\]
	It follows from \ref{item:main:use:prop:orbit-asdim:pou:Su} that $\Ponum{\Ause}{l}$ in Step~\eqref{thm:dimnuc-main:proof::Auseco} is a positive contraction in $\ell^\infty\left( \Ause \right)$ and thus we may 
	define a positive contraction 
	\[
	\tPonum{\Ause}{l} = \omega_{\Ause}^* \left( \Ponum{\Ause}{l} \right) \otimes 1_{C_0(X)^+} \in \ell^\infty\left(\Fi_{\Ause} \right) \otimes C_0(X)^+ \subseteq c_0\left(G/H_{\Ause}\right)^+ \otimes C_0(X)^+ \; .
	\]
	Using notations in Steps~\eqref{thm:dimnuc-main:proof::equi-Roe} and~\eqref{thm:dimnuc-main:proof::abbrev} and the observation that $\tPonum{\Ause}{l} \chi_{\Fi_{\Ause}  \vphantom{\krep{j}}} = \tPonum{\Ause}{l}$, 
	we define a completely positive contraction
	\[
	\compr_{\tPonum{\Ause}{l}}  \colon E_{H_{\Ause}} \to \chi_{\Fi_{\Ause}  
	\vphantom{\krep{j}}} E_{H_{\Ause}} \chi_{\Fi_{\Ause}  \vphantom{\krep{j}}} 
	= D_{\Ause} \; 
	\]
	given by
	 \[ 
	 b \mapsto \left(\tPonum{\Ause}{l}\right)^{1/2} \, b \, 
	 \left(\tPonum{\Ause}{l}\right)^{1/2} \; . 
	\]
	
	\item \label{thm:dimnuc-main:proof::Pi-Delta}
	For any $l \in \{0,1,\ldots,c\}$, for any $r \in \{0,1,\ldots,d\}$, and for any $\Ause \in \Auseco^{(l,r)}$, 
	we observe that the following equations hold in $\ell^\infty\left(\Fi_{\Ause} \right) \otimes C_0(X) \subseteq c_0\left(G/H_{\Ause}\right)^+ \otimes C_0(X)^+$: 
	\[
	\tPonum{\Ause}{l} {\tPo[LTC]}^{(r)}_{\Ause} = {\tPo[LTC]}^{(r)}_{\Ause} \tPonum{\Ause}{l} 
	= \sum_{[g] \in \Fi_{\Ause}}  \left( \chi_{[g]}  \otimes \left( \Ponum{\Ause}{l} \left( \omega_{\Ause} \left( [g] \right) \right) \cdot \Po[LTC]^{(r)}_{\omega_{\Ause} \left( [g] \right)} \right)   \right)
	\]
	and the summands of the last term are mutually orthogonal. 
	It follows that $\tPonum{\Ause}{l} {\tPo[LTC]}^{(r)}_{\Ause}$ is also a positive contraction and 
	\[
	\operatorname{compr}_{{\tPo[LTC]}^{(r)}_{\Ause}} \circ \operatorname{compr}_{\tPonum{\Ause}{l}} =  \operatorname{compr}_{\tPonum{\Ause}{l} {\tPo[LTC]}^{(r)}_{\Ause}} 
	\colon E_{H_{\Ause}} \to D_{H_{\Ause}}  \; .
	\]

	\item \label{thm:dimnuc-main:proof::Sigma-Pi-Delta}
	For any $l \in \{0,1,\ldots,c\}$, for any $r \in \{0,1,\ldots,d\}$, for any $a \in A_0$ and for any $g \in \Lo$, 
	we write 
	\begin{equation} \label{thm:dimnuc-main:proof::eq:nu-l-r}
	\Ponum{}{l, r} = \sum_{\Ause \in \Auseco^{(l,r)}} \Ponum{\Ause}{l} \, .
	\end{equation}
	We claim, using notations from 
	Steps~\eqref{thm:dimnuc-main:proof::compression}, 
	\eqref{thm:dimnuc-main:proof::mu_f}, 
	\eqref{thm:dimnuc-main:proof::LTC-compression}, 
	\eqref{thm:dimnuc-main:proof::Pi-Delta} 
	and~\eqref{thm:dimnuc-main:proof::Sigma},  
	that
	\begin{align*}
	\sum_{\Ause \in \Auseco^{(l,r)}} \Sigma_{H_{\Ause}} \circ \operatorname{compr}_{\tPonum{\Ause}{l} {\tPo[LTC]}^{(r)}_{\Ause}} \left( (1_{c_0\left(G/H_{\Ause}\right)^+} \otimes a) u_g \right) 
	=  \operatorname{compr}_{\Po[LTC]^{(r)}_{ \Ponum{}{l, r} }}  \left( a u_g \right) \; .
	\end{align*}
	
	Indeed, we first observe that for any $\Ause, \Ause' \in \Auseco^{(l,r)}$ and 
	for any $y \in {\Ause}$ and $y' \in {\Ause'}$, unless $\Ause = \Ause'$ and 
	$y = \alpha_g (y')$, we have 
	\begin{equation} \label{thm:dimnuc-main:proof::Sigma-Pi-Delta::eq:orthogonal}
	\Ponum{\Ause}{l} \left( y \right) \,  \Po[LTC]^{(r)}_{y} \,  \Ponum{\Ause'}{l} \left( y' \right) \, \alpha_g \left( \Po[LTC]^{(r)}_{y'} \right) = 0 \; 
	\end{equation}
	in $C_0(X)$. To see this, note that whenever this product is nonzero at 
	some point $\alpha_g 
	(x) \in X$, we have $\Ne^{(r)} (\alpha_g (x)) = y$ and $\Ne^{(r)} (x) = y'$ 
	with $\Ponum{\Ause}{l} \left( y \right) \not= 0$ and $\Ponum{\Ause'}{l} 
	\left( y' \right) \not= 0$, 
	which implies $\Ause \ni y = \alpha_g \left( y' \right) \in \Ause'$ by \ref{item:main:use:prop:orbit-asdim:pou:Su} and \ref{item:prop:dimnuc:Eq} and thus also $\Ause = \Ause'$ by the disjointness of $\Auseco^{(l)}$, proving the observation. 
	
	It then follows that 
	\begin{align*}
	&\ \operatorname{compr}_{\Po[LTC]^{(r)}_{ \Ponum{}{l, r} }}  \left( a u_g \right) \\
	= &\ \left( \sum_{\Ause \in \Auseco^{(l,r)}} \Po[LTC]^{(r)}_{\Ponum{\Ause}{l}} \right)^{1/2}  a u_g  \left(  \sum_{\Ause' \in \Auseco^{(l,r)}} \Po[LTC]^{(r)}_{\Ponum{\Ause'}{l}} \right)^{1/2} & \text{(linearity in Step~\eqref{thm:dimnuc-main:proof::mu_f})} \\
	= &\ \mathrlap{ \left( \sum_{\Ause \in \Auseco^{(l,r)}} \sum_{y \in {\Ause} } \Ponum{\Ause}{l}(y) \, \Po[LTC]^{(r)}_{y} \right)^{1/2}  a u_g  \left(  \sum_{\Ause' \in \Auseco^{(l,r)}} \sum_{y' \in {\Ause'} } \Ponum{\Ause'}{l}(y') \, \Po[LTC]^{(r)}_{y'} \right)^{1/2} } \\&& \mathllap{\text{(\ref{item:main:use:prop:orbit-asdim:pou:Su}, \ref{item:main:use:prop:orbit-asdim:pou:Bo}, and \eqref{thm:dimnuc-main:proof::mu_Y::eq:decomposition})}} \\
	= &\ \mathrlap{ \sum_{\Ause, \Ause' \in \Auseco^{(l,r)}} \sum_{y \in {\Ause} } \sum_{y' \in {\Ause'} } \left( \Ponum{\Ause}{l}(y) \, \Po[LTC]^{(r)}_{y} \right)^{1/2}  a u_g  \left(  \Ponum{\Ause'}{l}(y') \, \Po[LTC]^{(r)}_{y'} \right)^{1/2} } \\&& \mathllap{\text{(orthogonality among the $\Po[LTC]^{(r)}_{y}$'s, by \eqref{thm:dimnuc-main:proof::mu_Y::eq:decomposition})}} \\
	= &\ \mathrlap{ \sum_{\Ause \in \Auseco^{(l,r)}} \sum_{{y \in {\Ause} \cap \alpha_g (\Ause)} }  a \left( \Ponum{\Ause}{l} \left( y \right) \,  \Po[LTC]^{(r)}_{y} \,  \Ponum{\Ause}{l} \left( \alpha_g^{-1} (y) \right) \, \alpha_g \left( \Po[LTC]^{(r)}_{\alpha_g^{-1} (y)} \right) \right)^{1/2}  u_g } \\&& \mathllap{\text{(\eqref{thm:dimnuc-main:proof::Sigma-Pi-Delta::eq:orthogonal})}} \\
	= &\ \mathrlap{ \sum_{\Ause \in \Auseco^{(l,r)}} \Sigma_{H_{\Ause}} \Bigg(  
	\Bigg( \sum_{{[g'] \in \Fi_{\Ause} \text{ satisfying }} {[g^{-1} g'] \in 
	\Fi_{\Ause}} }  \chi_{[g']} \otimes } \\ && \mathllap{  \left( a \left( 
	\Ponum{\Ause}{l} \left( \omega_{\Ause} \left( [g'] \right) \right) \,  
	\Po[LTC]^{(r)}_{\omega_{\Ause} \left( [g'] \right)} \,  \Ponum{\Ause}{l} 
	\left( \omega_{\Ause} \left( [g^{-1} g'] \right) \right) \, \alpha_g \left( 
	\Po[LTC]^{(r)}_{\omega_{\Ause} \left( [g^{-1} g'] \right)} \right) 
	\right)^{1/2} \right) \Bigg) u_g \Bigg) }  \\&& 
	\mathllap{\text{(Steps~\eqref{thm:dimnuc-main:proof::orbit-map} 
	and~\eqref{thm:dimnuc-main:proof::Sigma-pre})}} \\
	= &\ \mathrlap{ \sum_{\Ause \in \Auseco^{(l,r)}} \Sigma_{H_{\Ause}}  \Bigg( \Bigg( \sum_{[g'] \in \Fi_{\Ause}}  \chi_{[g']}  \otimes \tPonum{\Ause}{l} ([g']) \Po[LTC]^{(r)}_{\omega_{\Ause} \left( [g'] \right)} \Bigg)^{1/2} } \\ && \mathllap{   (1_{c_0\left(G/H_{\Ause}\right)^+} \otimes a) u_g    \Bigg( \sum_{[g''] \in \Fi_{\Ause}}  \chi_{[g'']}  \otimes \tPonum{\Ause}{l} ([g'']) \Po[LTC]^{(r)}_{\omega_{\Ause} \left( [g''] \right)} \Bigg)^{1/2} \Bigg) }   \\&& \mathllap{\text{(setting $[g''] = [g^{-1} g']$ and using orthogonality among the $\chi_{[g']}$'s)}} \\
	= &\ \sum_{\Ause \in \Auseco^{(l,r)}} \Sigma_{H_{\Ause}} \circ \operatorname{compr}_{\tPonum{\Ause}{l} {\tPo[LTC]}^{(r)}_{\Ause}} \left( (1_{c_0\left(G/H_{\Ause}\right)^+} \otimes a) u_g \right)  & \text{(Step~\eqref{thm:dimnuc-main:proof::Pi-Delta})} 
	\end{align*}
	as claimed. 
	
	\item \label{thm:dimnuc-main:proof::Delta-invariant}
	For any $l \in \{0,1,\ldots,c\}$, for any $r \in \{0,1,\ldots,d\}$, and for any $g \in \Lo$, 
	we claim that 
	\[
	\left\| \alpha_g \left( \Ponum{}{l, r} \right) - \Ponum{}{l, r} \right\| \leq \Er^2 \; .
	\]
	
	Indeed, this follows from the equation 
	\[
	\alpha_g \left( \Ponum{}{l, r} \right) - \Ponum{}{l, r} = \sum_{\Ause \in \Auseco^{(l,r)}} \left( \alpha_g \left( \Ponum{\Ause}{l} \right) - \Ponum{\Ause}{l} \right)
	\]
	as each summand has norm bounded above by $\Er^2$ thanks to \ref{item:main:use:prop:orbit-asdim:pou:Li}, and the summands are orthogonal, since by \ref{item:main:use:prop:orbit-asdim:pou:Su}, each function $\left( \alpha_g \left( \Ponum{\Ause}{l} \right) - \Ponum{\Ause}{l} \right)$ vanishes outside $\Ause$, while $\Auseco^{(l,r)}$ is a disjoint family.

	\item \label{thm:dimnuc-main:proof::Delta-sum}
	For any $r \in \{0,1,\ldots,d\}$ and any $x \in \Ko$, 
	we claim that  
	\[
	\sum_{l = 0}^{c} \Po[LTC]^{(r)}_{ \Ponum{}{l, r} } (x) = \Po[LTC]^{(r)} (x) 
	\]
	Indeed, since the equality holds trivially when $\Po[LTC]^{(r)} (x) = 0$, we may assume without loss of generality that $\Po[LTC]^{(r)} (x) \not= 0$, which implies that $x \in \Cose^{(r)}$ by \ref{item:prop:dimnuc:Su}. It then follows from \ref{item:prop:dimnuc:Th} that there exists $z \in Z$ such that $\left\{ x, \Ne^{(r)} (x)  \right\} \subseteq V_z$, which implies $V_z \cap \Ko \not= \varnothing$ and thus $\Ne^{(r)} (x) \in V_z \subseteq \Ko'$ by Step~\eqref{thm:dimnuc-main:proof::thin-cover}, whence 
	by \ref{item:main:use:prop:orbit-asdim:pou:Un}, we have
	\[
	\sum_{l = 0}^{c} \sum_{\Ause \in \Auseco^{(l)}} \Ponum{\Ause}{l} \left( \Ne^{(r)} (x) \right) = 1 \; .
	\]
	Now observe that for any $l \in \{0,1,\ldots,c\}$ and for any $\Ause \in \Auseco^{(l)}$, 
	if $\Ponum{\Ause}{l} \left( \Ne^{(r)} (x) \right)  \Po[LTC]^{(r)} (x)  \not= 0$, 
	then $\Ne^{(r)} (x) \in \Ause$ by \ref{item:main:use:prop:orbit-asdim:pou:Su} and thus $x$ witnesses $\left. \Po[LTC]^{(r)} \right|_{\left( \Ne^{(r)} \right)^{-1} (\Ause) \cap \Ko} \not= 0$, whence $\Ause \in \Auseco^{(l,r)}$ by \eqref{thm:dimnuc-main:proof::eq:Ause-l-r}. 
	Combining this observation with Step~\eqref{thm:dimnuc-main:proof::mu_f} and \eqref{thm:dimnuc-main:proof::eq:nu-l-r}, we see that  
	\begin{align*}
	\Po[LTC]^{(r)} (x) &= \sum_{l = 0}^{c} \sum_{\Ause \in \Auseco^{(l)}} \Ponum{\Ause}{l} \left( \Ne^{(r)} (x) \right)  \Po[LTC]^{(r)} (x) & \\
	&= \sum_{l = 0}^{c} \sum_{\Ause \in \Auseco^{(l, r)}} \Ponum{\Ause}{l} \left( \Ne^{(r)} (x) \right)  \Po[LTC]^{(r)} (x) & \\
	&= \sum_{l = 0}^{c} \sum_{\Ause \in \Auseco^{(l, r)}} \Po[LTC]^{(r)}_{ \Ponum{\Ause}{l} } (x)  
	\quad = \sum_{l = 0}^{c} \Po[LTC]^{(r)}_{ \Ponum{}{l, r} } (x)  \; ,
	\end{align*}
	which proves the claim.

	\item \label{thm:dimnuc-main:proof::sum-identity}
	For any $a \in A_0$ and for any $g \in \Lo$, we claim
	that
	\begin{align*}
	\left\| a u_g - \left( \sum_{r = 0}^d \sum_{l = 0}^{c} \operatorname{compr}_{\Po[LTC]^{(r)}_{ \Ponum{}{l, r} }}  \left( a u_g \right) \right) \right\|
	\leq 2 (c+1) (d+1) \ErB \; .
	\end{align*}
	
	Indeed, we compute, for any $l \in \{0,1,\ldots,c\}$ and for any $r \in \{0,1,\ldots,d\}$, 
	\begin{align*}
	&\ \operatorname{compr}_{\Po[LTC]^{(r)}_{ \Ponum{}{l, r} }}  \left( a u_g \right) \\
	= &\ \left( \Po[LTC]^{(r)}_{ \Ponum{}{l, r} } \right)^{1/2} \, \left( \alpha_g \left( \Po[LTC]^{(r)}_{ \Ponum{}{l, r} } \right) \right)^{1/2} \, a u_g \\
	\approx_{\ErLTC} &\ \left( \Po[LTC]^{(r)}_{ \Ponum{}{l, r} } \right)^{1/2} \, \left(   \Po[LTC]^{(r)}_{ \alpha_g \left( \Ponum{}{l, r} \right)}  \right)^{1/2} \, a u_g & \text{(Step~\eqref{thm:dimnuc-main:proof::mu-invariant})\hphantom{.}} \\
	\approx_{\ErB} &\  \Po[LTC]^{(r)}_{ \Ponum{}{l, r} }  \, a u_g & \mathllap{\text{(Steps~\eqref{thm:dimnuc-main:proof::Delta-invariant} and~\eqref{thm:dimnuc-main:proof::mu_f}, and Lemma~\ref{Lemma:sqrt}).}}
	\end{align*}
	Recalling that $a$ is supported in $\Ko$, we obtain that 
	\begin{align*}
	& \ \sum_{r = 0}^d \sum_{l = 0}^{c} \operatorname{compr}_{\Po[LTC]^{(r)}_{ \Ponum{}{l, r} }}  \left( a u_g \right) \\
	\approx_{(c+1) (d+1) ( \ErB + \ErLTC )} & \ \sum_{r = 0}^d \sum_{l = 0}^{c} \Po[LTC]^{(r)}_{ \Ponum{}{l, r} }  \, a u_g  \\
	= & \ \sum_{r = 0}^d \Po[LTC]^{(r)}  \, a u_g  & \text{(Step~\eqref{thm:dimnuc-main:proof::Delta-sum})} \\
	= & \ a u_g  & \text{(\ref{item:prop:dimnuc:Un})} 
	\end{align*}
	which proves the claim, as $\ErLTC \leq \ErB$ by Step~\eqref{thm:dimnuc-main:proof::T-eta}.

	\item \label{thm:dimnuc-main:proof::Psi}
	For any $l \in \intervalofintegers{0}{c}$ and for any $\Ause \in \bigcup_{r = 0}^{d} \Auseco^{(l,r)}$, 
	define $\Psi_{\Ause}^{(l)}$ to be the composition 
	\[
		A \rtimes G \xrightarrow{\left( \mathbf{1} \otimes \mathrm{id}_A \right) \rtimes G} E_{H_{\Ause}} \xrightarrow{\compr_{\tPonum{\Ause}{l}}} D_{\Ause} \xrightarrow{\Theta_{\Ause}^{-1}} D'_{\Ause} \xrightarrow{\mathrm{ev}_{\Ause}} C_{\Ause} \xrightarrow{\kappa_{\Ause}} B_{\Ause}
	\]
	where 
	the maps are defined in Steps~\eqref{thm:dimnuc-main:proof::Roe}, \eqref{thm:dimnuc-main:proof::compression}, \eqref{thm:dimnuc-main:proof::equi-Roe-isomorphism}, \eqref{thm:dimnuc-main:proof::evaluation}, and~\eqref{thm:dimnuc-main:proof::D-approximation}, respectively (see also Step~\eqref{thm:dimnuc-main:proof::abbrev}). 
	We also write 
	\[
		\Psi^{(l,r)} := \bigoplus_{\Ause \in \Auseco^{(l,r)}} \Psi_{\Ause}^{(l)} \colon A \rtimes G \to \bigoplus_{\Ause \in \Auseco^{(l,r)}} B_{\Ause} \; ,
	\]
	which is again a completely positive contraction. 
	
	\item \label{thm:dimnuc-main:proof::tally}
	We define completely positive maps 
	\begin{align*}
		\Psi &:= \bigoplus_{l = 0}^{c} \bigoplus_{r = 0}^{d} \Psi^{(l,r)}  \colon A \rtimes G \to \bigoplus_{l = 0}^{c} \bigoplus_{r = 0}^{d} \left( \bigoplus_{\Ause \in \Auseco^{(l,r)}} B_{\Ause} \right) \; ,\\
		\Phi &:= \sum_{l = 0}^{c} \sum_{r = 0}^{d} \left( \sum_{m = 0}^{\dstab} \Phi^{(l,r,m)} \right)
		\colon  \bigoplus_{l = 0}^{c} \bigoplus_{r = 0}^{d} \left( \bigoplus_{\Ause \in \Auseco^{(l,r)}} B_{\Ause} \right) \to A \rtimes G \; . 
	\end{align*}
	Note that $\Psi$ is contractive by Step~\eqref{thm:dimnuc-main:proof::Psi}, while $\Phi$ is $(c+1) (d+1) (\dstab +1)$-decomposable by Step~\eqref{thm:dimnuc-main:proof::Phi-sum}. 
	Finally, for any $a \in A_0$ and for any $g \in \Lo$, we write 
	\[
	\Er' = (c+1)(d+1)\Er = \frac{\eps}{ \dstab + 8 }
	\]
	 and compute that (the reader is encouraged to follow the diagram on 
	 page~\pageref{thm:dimnuc-main:proof:big_diagram})
	\begin{align*}
		&\ \Phi \circ \Psi (a u_g) \\
		=&\ \sum_{l = 0}^{c} \sum_{r = 0}^{d}  \sum_{m = 0}^{\dstab} \Phi^{(l,r,m)}  \circ \Psi^{(l,r)} (a u_g)	\\ 
		\approx_{(\dstab+1)\Er'} &\ \sum_{l = 0}^{c} \sum_{r = 0}^{d}  \sum_{m = 0}^{\dstab} \sum_{\Ause \in \Auseco^{(l,r)}} \mathrm{Ad}_{\vkrep} \circ \operatorname{compr}_{{\tPo[LTC]'{}}^{(r)}_{\Ause}} \circ \widetilde{\gamma}_{\Ause}^{(m)} \circ \Psi_{\Ause}^{(l)}  (a u_g)		\\ &&& \mathllap{\text{(Steps~\eqref{thm:dimnuc-main:proof::Phi-sum} and~\eqref{thm:dimnuc-main:proof::Psi})}} \\ 
		= &\ \sum_{l = 0}^{c} \sum_{r = 0}^{d}  \sum_{\Ause \in \Auseco^{(l,r)}} \mathrm{Ad}_{\vkrep} \circ \operatorname{compr}_{{\tPo[LTC]'{}}^{(r)}_{\Ause}} \circ \widetilde{\gamma}_{\Ause} \circ \\
		& \hphantom{\ \sum_{l = 0}^{c} \sum_{r = 0}^{d}   } \kappa_{\Ause} \circ \mathrm{ev}_{\Ause} \circ \Theta_{{\Ause}}^{-1} \circ \mathrm{compr}_{\tPonum{\Ause}{l}}   \left(\left( 1_{c_0\left(G/H_{\Ause}\right)^+} \otimes a \right) u_g \right)		\\ &&& \mathllap{\text{(Steps~\eqref{thm:dimnuc-main:proof::lifting} and~\eqref{thm:dimnuc-main:proof::Psi})}} \\ 
		\approx_{5\Er'} &\ \sum_{l = 0}^{c} \sum_{r = 0}^{d}  \sum_{\Ause \in \Auseco^{(l,r)}} \mathrm{Ad}_{\vkrep} \circ \operatorname{compr}_{{\tPo[LTC]'{}}^{(r)}_{\Ause}} \circ \\
		& \hphantom{\ \sum_{l = 0}^{c} \sum_{r = 0}^{d}  }  \Theta_{{\Ause}}^{-1} \circ \mathrm{compr}_{\tPonum{\Ause}{l}}   \left(\left( 1_{c_0\left(G/H_{\Ause}\right)^+} \otimes a \right) u_g \right)	&& \mathllap{\text{(Step~\eqref{thm:dimnuc-main:proof::fine})}} \\ 
		= &\ \sum_{l = 0}^{c} \sum_{r = 0}^{d}  \sum_{\Ause \in \Auseco^{(l,r)}} \mathrm{Ad}_{\vkrep} \circ \Theta_{{\Ause}}^{-1} \circ \\
		& \hphantom{\ \sum_{l = 0}^{c} \sum_{r = 0}^{d}   }  \operatorname{compr}_{{{\tPo[LTC]}}^{(r)}_{\Ause}} \circ \mathrm{compr}_{\tPonum{\Ause}{l}}   \left(\left( 1_{c_0\left(G/H_{\Ause}\right)^+} \otimes a \right) u_g \right)	&& \mathllap{\text{(\eqref{thm:dimnuc-main:proof::eq:mu-compressions-square})}} \\ 
		= &\ \sum_{l = 0}^{c} \sum_{r = 0}^{d} \sum_{\Ause \in \Auseco^{(l,r)}} \Sigma_{H_{\Ause}} \circ   \operatorname{compr}_{\tPonum{\Ause}{l} {{\tPo[LTC]}}^{(r)}_{\Ause} }   \left(\left( 1_{c_0\left(G/H_{\Ause}\right)^+} \otimes a \right) u_g \right)		\\ &&& \mathllap{\text{(\eqref{thm:dimnuc-main:proof::eq:Sigma-k} and Step~\eqref{thm:dimnuc-main:proof::Pi-Delta})}} \\ 
		= &\ \sum_{l = 0}^{c} \sum_{r = 0}^{d}   \operatorname{compr}_{\Po[LTC]^{(r)}_{ \Ponum{}{l, r} }}  \left( a u_g \right)		&& \mathllap{\text{(Step~\eqref{thm:dimnuc-main:proof::Sigma-Pi-Delta})}} \\ 
		\approx_{2\Er'} &\  a u_g \; ,	&& \mathllap{\text{(Step~\eqref{thm:dimnuc-main:proof::sum-identity})}}
	\end{align*}
	that is, we have $\|\Phi \circ \Psi(au_g) - a u_g\| \leq \eps$. 
\end{enumerate}
This completes the proof. 
\end{proof}

\begin{Cor} \label{cor:dimnuc-virnil-topological}
	Let $G$ be a finitely generated virtually nilpotent group and let $X$ be a locally compact Hausdorff space with finite covering dimension. 
	Then for any action $\alpha$ of $G$ on $X$, we have 
	\[
	\dimnuc(C_0(X) \rtimes_{\alpha} G) \leq 2 \cdot \hirsch(G)! \cdot 9^{d(G)}\cdot (\dim(X^+)+1)^2.
	\]
	
\end{Cor}

\begin{proof}
	This follows from combining Theorem~\ref{thm:dimnuc-main} with \cite[Theorem 5.1]{EGM}, Theorem~\ref{thm:lsp-vnil},  Theorem~\ref{lem:thin-neighborhood} and Theorem~\ref{thm:dimltc-asdim}.
\end{proof}

\begin{Cor} \label{cor:dimnuc-hyperbolic}
	Suppose $\alpha$ is a simplicial proper cocompact action of a hyperbolic 
	group $G$ on a hyperbolic complex $X$. Denote by $\overline{X} = X \cup 
	\partial X$ the compactification, and use $\alpha$ to denote the induced 
	action of $G$ on $\overline{X}$. Then $\dimnuc(C_0(\overline{X}) 
	\rtimes_{\alpha} G) < \infty$. In particular, also $\dimnuc(C_0(\partial 
	X) 
	\rtimes_{\alpha} G) < \infty$. 
\end{Cor}
\begin{proof}
	This follows from Corollary~\ref{cor:hyperbolic-ltc}.
\end{proof}

\section{Allosteric actions} \label{sec:allosteric actions}

Suppose $\Gamma$ is a discrete group acting on a compact metrizable space with an 
ergodic Borel probabilty measure $\mu$. The action is said to be 
\emph{essentially free} if the set of points with nontrivial stabilizer has 
measure zero. The action is said to be \emph{topologically free} if for any nontrivial group element, the set of points not fixed by it is dense; when $X$ is metrizable and $\Gamma$ is countable, this is equivalent to saying that the set of points with nontrivial stabilizer is meager. An action is said to be 
\emph{allosteric} if it is topologically free but not essentially free. Very 
recently, Joseph (\cite{Joseph2023}) constructed examples of allosteric 
actions of amenable groups. Actions which are almost finite, including proposed 
generalizations to the topologically free setting, have to be essentially free, 
and therefore cannot be allosteric. In this section, we generalize the 
construction of Joseph in order to produce further examples of allosteric 
actions, including ones which can be handled using our techniques. 

We recall the general construction of a profinite action. 
Suppose that $\Gamma$ is a countable, infinite residually finite group, and let 
$\Gamma_1 > 
\Gamma_2 > \Gamma_3 \ldots$ a sequence of subgroups with finite index such that 
$\bigcap_{n=1}^{\infty} \Gamma_n = \{e\}$. The group $\Gamma$ acts on each finite 
coset space $\Gamma/\Gamma_n$ by left translation. Let $X = 
\underset{\longleftarrow}{\lim} ( \Gamma/\Gamma_n )$ be the inverse limit, where the 
connecting 
maps from $\Gamma/\Gamma_{n+1}$ to $\Gamma/\Gamma_n$, whose notation is suppressed, are the 
quotient maps and thus $\Gamma$-invariant. The set $X$ is 
homeomorphic to the Cantor set, and the actions of $\Gamma$ on the sets $\Gamma/\Gamma_n$ 
define an action of $\Gamma$ on $X$.

The following is a generalization of \cite[Theorem 3.2]{Joseph2023}, which 
provides a criterion for a profinite action to be allosteric.

\begin{Thm}
	\label{thm:allosteric-condition}
	Let $\Gamma$ be a residually finite countable group, and let 
	$s \in \Gamma$ be a nontrivial element. 
	Fix $\delta \in (0,1)$. Suppose $\Gamma_1 < 
	\Gamma_2 < \ldots$ is a sequence of finite index subgroups such that 
	$\bigcap_{n=1}^{\infty} \Gamma_n = \{\e\}$. We let $\Gamma$ act on 
	$\Gamma/\Gamma_n$ by left translation, and let $\mu_n$ be the uniform 
	probability measure on $\Gamma/\Gamma_n$. Let $X = 
	\underset{\longleftarrow}{\lim} \Gamma/\Gamma_n$, with the associated 
	profinite action of $\Gamma$. Denote the induced ergodic measure by $\mu$.
	Suppose that for all $n \in \N$ we have  
	\[
	| \{ q \in \Gamma/\Gamma_n \mid sq = q \} | \geq \delta \cdot | 
	\Gamma/\Gamma_n | 
	\, .
	\]
	Then the profinite action of $\Gamma$ on $X$ is topologically free and 
	minimal but not essentially free (with regard to $\mu$).
\end{Thm}
\begin{proof}
    For any action of a discrete group, the set of points with trivial stabilizer is always a $G_{\delta}$ set, because for any $g \in \Gamma$, the set $\{ x \in X \mid gx \neq x \}$ is open. Thus, to see that this action is topologically free, one needs to show that the set of points with trivial stabilizer is dense.
	Indeed element of $X$ can be written as a sequence 
	$ ( h_1 \Gamma_1 , h_2 \Gamma_2 , h_3 \Gamma_3 , \ldots )$ where $h_1,h_2,h_3, \ldots$ are elements in $\Gamma$ satisfying $h_{n+1} \Gamma_n = h_n \Gamma_n$ for all $n$. 
	We have
	\[
	g ( h_1 \Gamma_1 , h_2 \Gamma_2 , h_3 \Gamma_3 , \ldots ) = ( h_1 \Gamma_1 , h_2 \Gamma_2 , h_3 \Gamma_3 , \ldots )
	\]
	if and only if $h_n^{-1} g h_n \in \Gamma_n$ for all $n$. In particular, if we set $\bar{e} = ( e \Gamma_1 , e \Gamma_2 , e \Gamma_3 , \ldots )$ then we have $\Gamma_{\bar{e}} = \{e\}$. Because the action of $\Gamma$ on $X$ is minimal, all the points in the orbit of $\bar{e}$ have trivial stabilizer as well, so the points with trivial stabilizer are dense. 
	
	Since for any $n$ we have 
	\[
	\frac{ | \{ q \in \Gamma/\Gamma_n \mid sq = q \} | }{ | \Gamma/\Gamma_n | } \geq \delta 
	\, ,
	\]
	 it follows that 
	$\mu ( \{ x \in X \mid s \cdot x = x \} ) \geq \delta$, 
	hence the action is not essentially free.
\end{proof}

We now fix some notation for wreath products. Suppose $G$ and $H$ are groups. 
We denote by $c_{00}(H,G)$ the group of finitely supported $G$-valued functions 
from $H$, that is, for any $f \in c_{00}(H,G)$ the support $\supp(f) = \{h \in 
H 
\mid 
f(h) \neq \e \}$ is finite. (We use $\e$ to denote both the trivial element of 
$H$ and of $G$, by slight abuse of notation, as it should be clear which one is 
referred to from context.) 
This is a group with pointwise multiplication, and 
it is equipped with an action of $H$ given by left translation $(\lambda_h(f) 
(h') = f(h^{-1}h')$. The wreath product $ G \wr H $ is the semidirect product $ 
c_{00}(H,G) \rtimes_{\lambda} H$. We think of each element in $ G \wr H $  as a 
pair $(f,h)$ with $f \in c_{00}(H,G)$ and $h \in H$, with $(f,h) \cdot (f',h') 
= 
( f \lambda_h(f') , hh')$. As we consider the case in which $G$ is abelian, we 
shall use additive notation for the group operations in $G$ and in 
$c_{00}(H,G)$. 
\begin{Thm}
	\label{thm:allosteric-profinite}
	Let $G$ be a nontrivial countable residually finite abelian group and let $H$ be a 
	countable residually 
	finite group. Set $\Gamma = G \wr H \cong c_{00}(H,G) \rtimes_{\lambda} H$. 
	
	Fix 
	$\delta \in (0,1)$ and an increasing sequence $E_1 \subseteq E_2 \subseteq E_3 \subseteq \ldots$ of finite subsets of $H$ such that $H = \bigcup_{n=1}^{\infty} E_n$. 
	Suppose we have finite index subgroups 
	$H \geq H_1 \geq H_2 \geq H_3 \geq \ldots$ 
	satisfying, for each $n$, that 
	\begin{enumerate}
		\item $E_n^{-1}E_n \cap H_n \subseteq \{\e\}$, and
		\item $\left| \left(\bigcup_{j = 1}^n E_j H_j \right) / H_n \right| < (1-\delta) \cdot | H/H_n |$. 
	\end{enumerate}
	Also suppose we have finite index subgroups 
$G \geq G_1 \geq G_2 \geq G_3 \geq \ldots$ such that  $\bigcap_{n=1}^{\infty} G_n = \{\e\}$.

 	Define a sequence $N_1 \geq N_2 \geq N_3 \geq \ldots$ of subgroups in $c_{00}(H,G)$ by 
 	\[
 		N_n = \left \{ f \in c_{00}(H,G) \colon  
 		\sum_{h \in H_j} f(hq) \in G_j \text{ for any } j \in \intervalofintegers{1}{n} \text{ and }  q \in 
 		E_j \right\}
 		\, .
 	\]
 	Note that $N_n$ is invariant under the left action of $H_n$. Set $\Gamma_n 
 	= N_n \rtimes_{\lambda} H_n$. 
 	
 	Then $\Gamma_1 \geq \Gamma_2 \geq \ldots$ are 
 	finite index subgroups of $\Gamma$, we have $\bigcap_{n=1}^{\infty} 
 	\Gamma_n = 
 	\{\e\}$, and the profinite action of $\Gamma$ on 
 	$\underset{\longleftarrow}{\lim} 
 	\Gamma/\Gamma_n$ is allosteric.
\end{Thm}
\begin{proof}
	We first verify that indeed the subgroups $N_n$ have finite index. For any 
	$n \in \N$ and any $q \in E_n$, set 
	\[
	K_{n,q} = \left \{ f \in c_{00}(H,G) \mid 
	\sum_{h \in H_n} f(hq) \in G_n  \right  \}
	\, .
	\]
	 Note that $N_n = 
	\bigcap_{j=1}^n \bigcap_{q \in E_j} K_{j,q}$, so it suffices to show that 
	each $K_{n,q}$ has finite 
	index. Let $Z$ be a finite set of coset representatives for $G_n$ in $G$, 
	and for any $z \in Z$, let $f_z \in c_{00}(H,G)$ be the function given by 
	$f_z(q) = z$ and $f_z(h) = 0$ for any $h \neq q$. Given $f \in 
	c_{00}(H,G)$, 
	pick $z \in Z$ such that $z - \sum_{h \in H_n} f(hq)  \in G_n$, so $f_z - f 
	\in K_{n,q}$. Thus, $\{f_z\}_{z \in Z}$ is a finite set of coset 
	representatives of $K_{n,q}$ in $c_{00}(H,G)$, as required.
	
	We now verify that $\bigcap_{n=1}^{\infty} 
	\Gamma_n = 
	\{\e\}$. Indeed, let $(f,h) \in \bigcap_{n=1}^{\infty} 
	\Gamma_n$. The assumptions imply that $\bigcap_{n=1}^{\infty} 
	H_n = 	\{\e\}$, and therefore $h = e$. To show that $f = 0$, 
	let $D$ be the support of $f$ and let $R$ be the image of $f$. Pick $n$ 
	such that $D \subseteq E_n$,  $DD^{-1} \cap H_n \subseteq \{e\}$ and $R 
	\cap G_n = \{0\}$. 
	Notice that for any $q$, the condition  $DD^{-1} \cap H_n \subseteq \{e\}$ 
	implies that for any $q \in H$, the set $\{h \in H_n \mid hq \in \D\}$ has 
	at 
	most one element, so $\{ \sum_{h \in H_n} f(hq) \mid q \in E_n \} = R$. 
	However $(f,e) \in N_n$, so $\{ \sum_{h \in H_n} f(hq) \mid q \in E_n \} 
	\subseteq G_n$. Thus $R \subseteq G_n$, and therefore, by the choice of 
	$G_n$, we have $R = \{0\}$. This shows that  $\bigcap_{n=1}^{\infty} 
	\Gamma_n = 	\{\e\}$.

	To show the action is allosteric, we fix a nontrivial element $f_0 \in  
	c_{00}(H,G)$ with $\supp(f_0) = \{\e\}$ and set $s = (f_0,e)$. We  
	claim that the conditions of 
	Theorem 
	\ref{thm:allosteric-condition} 
	hold for $s$, that is, for any $n \in \N$ we have 
	\[
	| \{ q \in \Gamma/\Gamma_n \mid sq = q \} | \geq \delta \cdot | 
	\Gamma/\Gamma_n | 
	\, .
	\]
	Fix $n \in \N$. Let $q \in \Gamma/\Gamma_n$. Denote $q = [(f,h)]$ with  $f 
	\in c_{00}(H,G)$ and $h \in H$. 
	Then $sq = q$ 
	holds if and only if $(f,h)^{-1} (f_0,\e) (f,h) \in \Gamma_n$. Note that 
	\[
	(f,h)^{-1} (f_0,\e) (f,h)  = ( \lambda_{h^{-1}}(f_0),\e ) 
	\, ,
	\]
	 and $\supp 
	\lambda_{h^{-1}}(f_0) = \{ h^{-1} \}$. Thus, if $h^{-1} \not \in \bigcup_{j = 1}^n E_j H_j$ 
	then $sq = q$. So, 
	\[
	\{ [(f,h)] \in \Gamma/\Gamma_n \mid s [(f,h)] = [(f,h)] 
	\} \supseteq \{ [(f,h)] \in \Gamma/\Gamma_n \mid h \not \in \bigcup_{j = 1}^n E_j H_j \}
		\, .
	\]
	Because $\left| \left(\bigcup_{j = 1}^n E_j H_j \right) / H_n \right| < (1-\delta) \cdot | H/H_n |$, we have 
	\[
		\left| \left\{ [h] \in H/H_n \mid h \not \in \bigcup_{j = 1}^n E_j H_j \right\} \right| \geq  \delta  \cdot |H/H_n| \; .
	\]
	Since the quotient map $\Gamma / \Gamma_n \to H / H_n$, given by $[(f,h)] 
	\mapsto [h]$, pushes forward the uniform measure on $\Gamma/\Gamma_n$ to 
	the uniform measure on $H/H_n$, we have 
	\[
		\frac{| \{ [(f,h)] \in \Gamma/\Gamma_n \mid h \not \in \bigcup_{j = 1}^n E_j H_j \} | }{|\Gamma/\Gamma_n |} = \frac{| \{ [h] \in H/H_n \mid h \not \in  \bigcup_{j = 1}^n E_j H_j \} | }{|H/H_n |} \geq  \delta \; ,
	\]
	as required. 
\end{proof}

\begin{Cor}
	\label{cor:allosteric-construction}
	Let $G$ be a nontrivial countable residually finite abelian group and let $H$ be a
	countably infinite residually finite group. 
	Then $G \wr H$ has a profinite action that is allosteric. 
\end{Cor}

\begin{proof}
	Fix a decreasing sequence 
	$G \geq G_1 \geq G_2 \geq G_3 \geq \ldots$ of finite index subgroups such that  $\bigcap_{n=1}^{\infty} G_n = \{\e\}$. 
	Fix an increasing sequence $E_1 \subseteq E_2 \subseteq E_3 \subseteq 
	\ldots$ of finite subsets of $H$ satisfying $H = \bigcup_{n=1}^{\infty} 
	E_n$. 
	Write $E_0 = \varnothing$ and $E'_n = E_n \setminus E_{n-1}$ for $n = 1, 2, 
	\ldots$. 
	
	Now 
	we fix an arbitrary $\delta \in (0,1)$. 
	Define inductively a decreasing sequence 
	$H_1 \geq H_2 \geq H_3 \geq \ldots$ of finite index subgroups of $H$
	satisfying the conditions in Theorem~\ref{thm:allosteric-profinite}, as 
	follows: 
	\begin{enumerate}
		\item For convenience, we start by defining $H_0 = H$. We 
		trivially have
		 \[
		E_0^{-1}E_0 \cap H_0 = \varnothing \subseteq \{\e\}
		\] and 
		\[
		\left| \left(\bigcup_{j = 1}^0 E_j H_j \right) / H_0 \right| = 0 < 
		(1-\delta) \cdot | H/H_0 |
		\, . 
		\]
		\item Now suppose $H_0, H_1, \ldots, H_{n-1}$ have been defined for 
		some positive integer $n$. We set 
		\[
			\varepsilon_n = \frac{\left| \left(\bigcup_{j = 1}^{n-1} E_j H_j 
			\right) / H_{n-1} \right|}{| H/H_{n-1} | }  \in \left[ 0, 1- \delta 
			\right) \; .
		\]
		Since $H$ is infinite and residually finite, we can choose a finite 
		index 
		subgroup $H_n$ of $H$ such that $E_n^{-1}E_n \cap H_n = \varnothing$ 
		and $| H/H_{n} | > \frac{|E'_n|}{1 - \delta - \varepsilon_n} $. It follows that
		\begin{align*}
			\frac{\left| \left(\bigcup_{j = 1}^n E_j H_j \right) / H_n \right|}{| H/H_n |} &\leq \frac{\left| \left(\bigcup_{j = 1}^{n-1} E_j H_j \right) / H_n \right|}{| H/H_n |} + \frac{\left| E'_n H_n / H_n \right|}{| H/H_n |} \\
			&\leq \frac{\left| \left(\bigcup_{j = 1}^{n-1} E_j H_j \right) / H_{n-1} \right|}{| H/H_{n-1} |} + \frac{\left| E'_n \right|}{| H/H_n |} \\
			&< \varepsilon_n + 1 - \delta - \varepsilon_n = 1- \delta
		\end{align*}
		as desired. 
	\end{enumerate}
	Hence we may follow the construction in Theorem~\ref{thm:allosteric-profinite} to obtain a profinite action that is allosteric. 
\end{proof}

The following proposition uses our techniques to show that profinite actions of 
wreath products of a finite group $H$ by $\Z^d$ give rise to crossed products 
with finite nuclear dimension. By Corollary~\ref{cor:allosteric-construction}, 
some 
of those examples are allosteric, if the choice of subgroups is done according 
to the conditions in  Theorem \ref{thm:allosteric-profinite}.
\begin{Prop}
	\label{prop:allosteric-finite-nuc-dim}
	Let $H$ be a finite abelian group (with more than one element). 
	Let $G = H \wr \Z^d$. 
	Suppose $G_1 > 
	G_2 > G_3 \ldots$ a sequence of subgroups with finite index such that 
	$\bigcap_{n=1}^{\infty} G_n = \{e\}$, let $X = 
	\underset{\longleftarrow}{\lim} G/G_n$ be the inverse limit. Denote by 
	$\alpha$  the action of $G$ on $X$, defined as the inverse limit of the 
	left translation actions. Then 
	$\dimnuc (C(X) \rtimes_{\alpha} G)$ is finite, with a uniform bound for all 
	such actions, depending only on $d$. 
\end{Prop}
\begin{proof}
	Denote by $c_{00}(\Z^d,H)$ the finitely supported $H$-valued 
	functions from $\Z^d$, and recall that $H \wr \Z^d \cong c_{00}(\Z^d,H) 
	\rtimes_\lambda 
	\Z^d$, 
	where $\Z^d$ acts on $c_{00}(\Z^d,H)$ by translation. Recall that 
	\[
	C(X) \rtimes_{\alpha} G \cong \left ( C(X) \rtimes_{\alpha|_{c_{00}(\Z^d,H)}} 
	c_{00}(\Z^d,H) \right ) \rtimes \Z^d \; ,
	\]
	where $\Z^d$ acts on $C(X)$ by $\alpha|_{\Z^d}$ and on $c_{00}(\Z^d,H)$ by $\lambda$. 
	
	Recall furthermore that if $\Gamma$ is an abelian group acting on a finite 
	set $S$, then 
	\[
	C(S) \rtimes \Gamma = \bigoplus_{[x] \in S/\Gamma} 
	M_{|\Gamma 
		\cdot x|} \otimes 
	C^*(\Gamma_x) \cong \bigoplus_{x \in S/\Gamma} M_{|\Gamma 
		\cdot x|} \otimes 
	C(\widehat{\Gamma}_x)
	\]
	(where $\Gamma_x$ is the stabilizer of a point 
	$x$ in 	the 	orbit; as $\Gamma$ is abelian, the subgroup does not 
	depend on the choice 
	of $x$, and $\Gamma \cdot x$ is the orbit). 
	When $\Gamma = c_{00}(\Z^d,H)$, the group $\Gamma_x$ for any $x$ is a 
	countable infinite abelian torsion group, and therefore the Pontryagin dual 
	$\widehat{\Gamma_x}$ is homeomorphic to the Cantor set. 
	Thus, for any $n \in \N$, if we denote by $Y_n$ the disjoint union of 
	copies of the Pontryagin duals of the stabilizer groups of the action of 
	$c_{00}(\Z^d,H)$ on $G/G_n$, then 
	the crossed product $ C(G/G_n) \rtimes_{\alpha|_{c_{00}(\Z^d,H)}} 
	c_{00}(\Z^d,H)$ has 
	the structure of a $C(Y_n)$-algebra with fibers isomorphic to matrix 
	algebras. Now $\dim(Y_n) = 0$ and the fibers all have zero nuclear 
	dimension. Consider the action $\alpha|_{\Z^d}$ restricted to the 
	invariant subalgebra  
	 \[
	  B_n = C(G/G_n) 
	\rtimes_{\alpha|_{c_{00}(\Z^d,H)}} c_{00}(\Z^d,H) 
	\, .
	\]
	 Let $y \in Y_n$. 
	Let $N$ be a subgroup of the stabilizer group of $y$. If $N$ is 
	nontrivial, it is isomorphic to $\Z^m$ for some $0 < m \leq d$. The 
	fiber $B_y$ is isomorphic to a matrix algebra $M_k$ for some $k$, and 
	therefore the action of $N$ is pointwise inner. Let $u_1,u_2\ldots,u_m 
	\in U_k$ be unitaries which implement the actions of generators of $N$ 
	(for some choice of generators). Each commutator $u_ju_lu_j^*u_l^*$ 
	acts trivially on $M_k$, and therefore it is of the form $\lambda 1_k$ 
	for some scalar $\lambda$ in the unit circle. However, 
	$\det(u_ju_lu_j^*u_l^*) 
	= 1$, so $\lambda^k = 1$, and thus $\lambda$ is a $k$'th root of 
	unity. Let $M$ be the group generated by $u_1,u_2,\ldots,u_m$, then 
	$[M,M]$ is a a subgroup of the roots of unity, and $[M,M] \subseteq 
	Z(M)$. 
	It follows in particular that $M$ is nilpotent, and $B_y \rtimes N$ is 
	a quotient of $B_y \rtimes M$ (where we omit the notation for the 
	action). Because the action of $M$ on $B_y$ is inner, we have $B_y 
	\rtimes M \cong M_k \otimes C^*(M)$. As the Hirsch length of $M$ is at most 
	$d$, 
	by  \cite[Theorem 5.1]{EGM}, we have $\dimnuc ( C^*(M) ) \leq 
	2\cdot d!$, and therefore also $\dimnuc (  B_y \rtimes N ) \leq 
	10^{d-1}d!$.

	As $\dim(Y_n) = 0$, it follows that $\dimltc(\alpha|_{\Z^d})$ is finite, 
	and bounded by a constant depending only on $d$ (and not on $n$).  
	Thus, 
	for any $n$, it follows from Theorem \ref{thm:dimnuc-main} that $\dimnuc 
	(B_n \rtimes_{\alpha|_{\Z^d}} 
	\Z^d)$ is bounded by a constant depending only on $d$.   
	Now, $C(X) \rtimes_{\alpha} G$ is an inductive limit of the subalgebras 
	$B_n \rtimes_{\alpha|_{\Z^d}} \Z^d$. We thus have $\dimnuc (C(X) 
	\rtimes_{\alpha} G) < \infty$ (with a bound depending only 
	on $d$).
\end{proof}
\begin{Rmk}
	In Proposition~\ref{prop:allosteric-finite-nuc-dim}, we assumed 
	that $H$ is finite and abelian. It 
	suffices, in fact, at the cost of complicating notation somewhat, to assume 
	that $H$ is abelian, locally finite and residually finite. As the 
	modification is straightforward, we do not include further details. 
\end{Rmk}

\newcounter{sectiontotal}
\setcounter{sectiontotal}{\value{section}}

\newcounter{appendix}
\setcounter{appendix}{0}
\renewcommand{\thesection}{\Alph{appendix}}

{
	\renewcommand{\thesection}{Appendix~\Alph{appendix}} \refstepcounter{appendix}
\section{Nuclear dimension and covering dimension} \label{sec:dimc} 
}
\renewcommand{\sectionlabel}{LTC}
\ref{sectionlabel=LTC}

In this appendix, we clarify the relationship between the nuclear dimension of $C_0(X)$ (together with its close relative, the decomposition rank $\dr \left( C_0 (X) \right)$) and the covering dimension of $X$, for a general locally compact Hausdorff space $X$. Under some extra assumptions, the following equality is well-known. 

\begin{Thm}[\cite{winter-covering-II, winter-zacharias}] 
\label{thm:dimnuc-dim-cite}
	If $X$ is a locally compact metrizable space or a compact Hausdorff space, 
	then we have 
	\[
		\dimnuc \left( C_0 (X) \right) = \dr \left( C_0 (X) \right) = \dim(X) \; .
	\]
	\qed
\end{Thm}

To deal with general locally compact and Hausdorff spaces, it is convenient to introduce the following notion. 

\begin{Def} \label{def:dimc}
	Let $X$ be a Hausdorff space. Then the \emph{compactly supported covering dimension} of $X$, denoted by $\dimc (X)$, is the infimum of all nonnegative integers $d$ such that for any finite open cover $\Thseco$ of $X$ and any compact subset $\Ko \Subset X$, there is an open cover $\Coseco$ of $\Ko$ in $X$ that refines $\Thseco$ and has multiplicity at most $d+1$. 
\end{Def}

\begin{Rmk} \label{rmk:dimc-basics}
	The following properties follow directly from Definition~\ref{def:dimc}: 
	\begin{enumerate}
		\item We have $\dimc(X) \leq \dim(X)$, with equality if $X$ is compact. 
		\item For any $Y \subseteq X$, we have $\dimc(Y) \leq \dimc(X)$. 
	\end{enumerate}
\end{Rmk}

\begin{Lemma} \label{lem:dimc-compactification}
	Let $X$ be a locally compact Hausdorff space. Then we have: 
	\begin{enumerate}
		\item $\dimc(X) = \dim(X^+)$, and 
		\item $\dimc(X) = \sup_{\Ko \Subset X} \dim(\Ko)$, where $\Ko$ ranges over all compact subsets of $X$. 
	\end{enumerate}
\end{Lemma}

\begin{proof}
	These statements are trivial if $X$ is compact; thus we may assume $X$ is noncompact. 
	Since $X^+$ is compact Hausdorff and thus normal, the first equality follows from applying \cite[Ch.~3, 5.6]{Pears75} to the closed subset $\{ \infty \}$ in $X^+$. 
	To prove the second equality, we first note that $\dimc(X) \geq \sup_{\Ko \Subset X} \dim(\Ko)$ by Remark~\ref{rmk:dimc-basics}. For the other direction, we observe that for any finite open cover $\Thseco$ of $X$ and any compact subset $\Ko \Subset X$, by the local compactness of $X$, we may choose a compact set $\Ko' \Subset X$ such that $\Ko$ is contained in the interior of $\Ko'$. Restricting $\Thseco$ to $\Ko'$, we may use the definition of $\dim(\Ko')$ to find an open cover $\Coseco$ of $\Ko'$ that refines $\Thseco$ and has multiplicity at most $\dim(\Ko')+1$. Restricting $\Coseco$ to the interior of $\Ko'$, we obtain an open cover of $\Ko$ in $X$ that has multiplicity at most $\dim(\Ko')+1$. Therefore we also have $\dimc(X) \leq \sup_{\Ko \Subset X} \dim(\Ko)$. 
\end{proof}

\begin{Thm} \label{thm:dimnuc-dim-general}
	If $X$ is a locally compact Hausdorff space, then we have 
	\[
		\dimnuc \left( C_0 (X) \right) = \dr \left( C_0 (X) \right) = \dimc(X) = \dim(X^+) \; .
	\]
\end{Thm}

\begin{proof}
	While it is possible to give a direct proof along the same lines as Theorem~\ref{thm:dimnuc-dim-cite}, we opt to give a short proof that makes use of Theorem~\ref{thm:dimnuc-dim-cite} instead. 
	In view of Lemma~\ref{lem:dimc-compactification}, 
	it suffices to show that $\dimnuc \left( C_0 (X) \right) = \sup_{\Ko 
	\Subset X} \dimnuc(C(\Ko))$ and $\dr \left( C_0 (X) \right) = \sup_{\Ko 
	\Subset X} \dr(C(\Ko))$, where $\Ko$ ranges over all compact subsets of 
	$X$. We prove the first equality; the second can be proved in the 
	same way. To this end, observe that $\dimnuc \left( C_0 (X) \right) \geq 
	\dimnuc(C(\Ko))$ for any compact subset $\Ko$ of $X$, since $C(\Ko)$ is a 
	quotient of $C_0(X)$. On the other hand, since 
	\[
		C_0(X) = \varinjlim_{	^{\hphantom{open and } U \subset X}  _{\text{open and precompact}}} C_0(U)
	\]
	with the indices directed by set containment, 
	and each $C_0(U)$ is an ideal in $C(\overline{U})$, we have 
	\begin{align*}
		\dimnuc \left( C_0 (X) \right) &\leq \liminf_{	^{\hphantom{open and } U \subset X}  _{\text{open and precompact}}} \dimnuc \left( C_0 (U) \right) & \\
		& \leq  \liminf_{	^{\hphantom{open and } U \subset X}  _{\text{open and precompact}}} \dimnuc(C(\overline{U})) &\leq  \sup_{\Ko \Subset X} \dimnuc(C(\Ko)) \; ,
	\end{align*}
	as desired. 
\end{proof}

Finally, we remark that the conclusion of Theorem~\ref{thm:dimnuc-dim-cite} breaks down for general locally compact Hausdorff spaces. 

\begin{Exl}
	Let $Y$ be the Tychonoff plank, i.e., $Y = [0, \omega_0] \times [ 0, \omega_1]$, where $\omega_0$ is the first inifnite ordinal, $\omega_1$ is the first uncountable ordinal, the intervals are given the order topology (and thus compact and Hausdorff), and $Y$ is endowed with the product topology. Let $X = Y \setminus \left\{ \left( \omega_0, \omega_1 \right) \right\}$, i.e., the \emph{deleted Tychonoff plank}. Then $Y$ is the one-point compactification of $X$. 
	By \cite[Ch.~4, 3.1]{Pears75}, we have 
	\[
		\dim(X) = 1 > 0 = \dim(X^+) = \dimc(X) = \dimnuc \left( C_0 (X) \right) = \dr \left( C_0 (X) \right) \; .
	\]
\end{Exl}

{
	\renewcommand{\thesection}{Appendix~\Alph{appendix}} \refstepcounter{appendix}
\section{Asymptotic dimension via finite subspaces} \label{sec:asdim-local} 
}
\renewcommand{\sectionlabel}{OCS}
\ref{sectionlabel=OCS}
In this appendix, we establish the fact that the asymptotic dimension of a 
coarse space $(X, \Enseco)$ can be detected in a finitary way, a fact that is 
needed in showing $\asdim \left( X, \Enseco_\alpha \right) \leq 
\dimltc(\alpha)$. More precisely, $\asdim(X, \Enseco)$ is equal to the 
asymptotic dimension of the coarse disjoint union of all finite subspaces of 
$X$, a quantity which is also known as the uniform asymptotic dimension of all 
finite subspaces of $X$. In fact, a more general statement holds where one 
replaces the collection of all finite subspaces of $X$ by any net of subspaces 
that exhausts $X$. This is probably known to experts (see 
\cite[Lemma~2.7]{DelabieTointon2017}), but as we did not find in the literature 
the exact statements we need, we spell out the details here. 
Logically speaking, the content of this appendix may be inserted right before Lemma~\ref{lem:asdim-finitary}. 
Throughout the appendix, $(X, \Enseco)$ denotes a coarse space. 

To get started, we first make precise what we mean by coarse disjoint union of subspaces. 

\begin{Def} \label{def:coarse-disjoint-union}
	Let $\left( X_i \right)_{i \in I}$ be a collection of subsets of $X$. Let us write $X_{\bigsqcup I}$ for the disjoint union $\bigsqcup_{i \in I} X_i$, whose elements are denoted by $(i, x)$ with $i \in I$ and $x \in X_i$. 
	For any $\Ense \subseteq X \times X$, we write 
	\[
		\Ense_{\bigsqcup I} := \bigsqcup_{i \in I} \Ense \cap \left( X_i \times X_i \right) \subseteq 
		\bigsqcup_{i \in I} X_i \times X_i  \subseteq
		X_{\bigsqcup I} \times X_{\bigsqcup I} \; . 
	\]
	We observe that 
	\begin{itemize}
		\item $\left(\Delta_X\right)_{\bigsqcup I} = \Delta_{X_{\bigsqcup I}}$;
		\item $\Ense_{\bigsqcup I} \subseteq \Auense_{\bigsqcup I}$ if $\Ense 
		\subseteq \Auense$;
		\item $\Ense_{\bigsqcup I}^{-1} = \left( \Ense^{-1} \right)_{\bigsqcup 
		I}$;
		\item $\Ense_{\bigsqcup I} \circ \Auense_{\bigsqcup I} \subseteq \left( 
		\Ense \circ \Auense \right)_{\bigsqcup I}$; 
		\item $\Ense_{\bigsqcup I} \cup \Auense_{\bigsqcup I} = \left( \Ense 
		\cup \Auense \right)_{\bigsqcup I}$ for any $\Ense, \Auense \subseteq X 
		\times X$.
	\end{itemize}
	Define a coarse structure 
	\[
		\Enseco_{\bigsqcup I} := \left\{ 
		\Auense \colon \Auense \subseteq 
		\Ense_{\bigsqcup I}
		\text{ for some } \Ense \in \Enseco \right\}
	\]
	on $X_{\bigsqcup I}$. The coarse space $\left( X_{\bigsqcup I}, \Enseco_{\bigsqcup I} \right)$ is called the \emph{coarse disjoint union} of the collection $\left( X_i \right)_{i \in I}$ of subsets of $X$. 
\end{Def}

This construction has appeared widely in the coarse geometry literature, at least in the case of metric coarse geometry: 

\begin{Exl}
	In the context of Definition~\ref{def:coarse-disjoint-union}, if $\Enseco$ is the metric coarse structure induced from an extended metric $d$ on $X$ as in Example~\ref{exl:metric-coarse}, then it induces an extended metric $d_{\bigsqcup I}$ on $X_{\bigsqcup I}$ such that for any $i, j \in I$, $x \in X_i$ and $y \in X_j$, we have 
	\[
		d_{\bigsqcup I} \left( (i, x), (j, y) \right) = 
		\begin{cases}
			d(x,y) \, , & \text{if } i = j \\
			\infty \, , & \text{if } i \not= j
		\end{cases}
		\; .
	\]
	It is clear that the metric coarse structure induced by $d_{\bigsqcup I}$ agrees with $\Enseco_{\bigsqcup I}$. 
\end{Exl}

\begin{Prop} \label{prop:asdim-directed-limit}
	Let $\left( X_i \right)_{i \in I}$ be a directed net of subsets of $X$ 
	ordered by inclusion such that $X = \bigcup_{i \in I} X_i$. 
	Then we have 
	\[
		\asdim \left( X , \Enseco\right) = \asdim \left(  X_{\bigsqcup I} ,  \Enseco_{\bigsqcup I}  \right) \; .
	\]
\end{Prop}

We deduce this proposition from the following more general fact concerning 
ultraproducts of coarse spaces. Let $(\Lambda,\leq)$ be a directed set, that 
is, for any two elements $i,j \in \Lambda$, we require that there 
exists $l \in \Lambda$ such that $i \leq l$ and $j \leq 
l$. 
Recall that a \emph{filter} $\mathcal{F}$ on the set $\Lambda$ is a nonempty collection of nonempty subsets of $\Lambda$ that is closed under taking finite intersections and taking supersets. It is an \emph{ultrafilter} if for any $J \subseteq \Lambda$, we have either $J \in \mathcal{F}$ or $\Lambda \setminus J \in \mathcal{F}$, or equivalently, for any partition $\Lambda = \Lambda_1 \sqcup \ldots \sqcup \Lambda_n$, exactly one of the sets $\Lambda_1 , \ldots , \Lambda_n$ is in $\mathcal{F}$.

We say that a filter $\mathcal{F} \subseteq P(\Lambda)$ is 
\emph{cofinal} if for any $i \in \Lambda$ we have $\{j \mid i 
\leq j \} \in \mathcal{F}$. 
It is clear that $\mathcal{F}$ is {cofinal} if and only if it contains the \emph{filter of eventually upward closed sets}
\[
\left\{ J \subseteq \Lambda \colon \text{there is } i \in \Lambda \text{ such that } j \in J \text{ for any } j \geq i  \right\} \; .
\]
By Tarski's result on the existence of ultrafilters (using Zorn's lemma), there exists a cofinal ultrafilter on $(\Lambda, \leq)$. 

Now, 
fix a cofinal ultrafilter $\mathcal{F}$ on $(\Lambda,\leq)$ and
suppose we are given a coarse space $(X,\Enseco)$, and for any $i \in \Lambda$ 
we are we are given a subset $X_i \subseteq X$. Set $\Enseco_i = 
\Enseco|_{X_i}$. We define a coarse space $(X_{\mathcal{F}} , 
\Enseco_{\mathcal{F}} )$ as follows. The elements of $X_{\mathcal{F}}$ are 
equivalence classes of elements of the product $\prod_{i \in \Lambda} X_i$, 
with $(x_i)_{i \in \Lambda} \sim (x'_i)_{i \in \Lambda}$ if $\{i \mid x_i = 
x'_i \} \in \mathcal{F}$. Let $\pi \colon \prod_{i \in \Lambda} X_i \to 
X_{\mathcal{F}}$ be the quotient map.
  For any $E \in \Enseco$ we define $\Enseco_{\mathcal{F}}^{(E)}$ to be 
defined as the images of all products $\prod_{i \in \Lambda} (E_i)$ such that 
for all $i \in \Lambda$ we have $E_i \subseteq E$ inside $X_{\mathcal{F}} 
\times X_{\mathcal{F}}$ under the 
map $\pi \times \pi$. Set $\Enseco_{\mathcal{F}} = \bigcup_{E\in \Enseco} 
\Enseco_{\mathcal{F}}^{(E)}$.
It is immediate to check that this is indeed a coarse 
space. This construction is somewhat more general than the ultraproduct 
construction of metric spaces discussed in \cite[Section 7.4]{Roe2003Lectures}.

\begin{Lemma} \label{lem:asdim-ultraproducts}
	Let $(X,\Enseco)$ be a coarse space. Set $d = \asdim(X,\Enseco)$. Suppose 
	$(\Lambda,\leq)$ is a directed set, and for each $i \in \Lambda$ we 
	are given a coarse subset $(X_i,\Enseco_i)$ of $(X,\Enseco)$ as above. Let 
	$\mathcal{F}$ be a cofinal 
	ultrafilter on $\Lambda$. Then 
	$\asdim(X_{\mathcal{F}} , 
	\Enseco_{\mathcal{F}} ) \leq d$.
\end{Lemma}
\begin{proof}
	Let $F \in E_{\mathcal{F}}$. Choose $E \in \Enseco$ and $E \supseteq E_i 
	\in \Enseco_i$ for all $i$ such 
	that $F$ is the image of $\prod_{i \in \Lambda} E_i$ in  $X_{\mathcal{F}} 
	\times X_{\mathcal{F}}$ under the equivalence relation defining the 
	ultraproduct. For each $i$, we choose a cover $\Auseco_i$ of $X_i$ as in 
	Definition \ref{def:asdim} such that 
	 $\Auseco_i$ is $\Enseco_i$-uniformly bounded and
	$\Auseco_i$ can be written as a union $\Auseco_i^{(0)} \cup \ldots 
		\cup \Auseco_i^{(d)}$ with each $\Auseco^{(l)}$ being 
		\emph{$E_i$-separated}.
	
	For $l = 0,1,\ldots d$, we define $\Auseco_\mathcal{F}^{(l)}$ to be the 
	image under the equivalence relation of all the products of the form 
	$\prod_{i \in \Lambda} \Ause_i$ for $\Ause_i \in \Auseco_i^{(0)}$ for all 
	$i$. Note that if $\prod_{i \in \Lambda} \Ause_i$ and $\prod_{i \in 
	\Lambda} \Ause'_i$ are two such sets, then $\{ i \mid \Ause_i \neq \Ause'_i 
	\}$ coincides with the set of $i$ for which $\Ause_i$ and $\Ause'_i$ are 
	$E_i$-separated, and therefore if two such sets are distinct in the 
	ultraproduct, then they are $F$-separated. The fact each set in this 
	collection is uniformly bounded is immediate from the construction. To see 
	that this is a cover, suppose  we are given $(x_i)_{i \in \Lambda}$, then 
	for each $i$ we can choose $l(i) \in \{0,1,\ldots,d\}$ such that 
	$\Ause_{i,l(i)} \in \Auseco_i^{(l(i))}$ with $x_i \in \Ause_{i,l(i)}$. 
	Since $\mathcal{F}$ is an ultrafilter, there exists exactly one $m \in 
	\{0,1,\ldots,d\}$ such that $\{i \mid l(i) = m\} \in \mathcal{F}$. It 
	follows that the image of $\prod_i  \Auseco_i^{(l(i))}$, which contains the 
	point define by $(x_i)_{i \in \Lambda}$ is in $\Auseco_\mathcal{F}^{(m)}$, 
	and thus this is indeed a cover.
\end{proof}

\begin{proof}[Proof of Proposition~\ref{prop:asdim-directed-limit}]
	The ``$\geq$'' 
	direction can be easily verified through Definition~\ref{def:asdim} and 
	inducing covers of $X_{\bigsqcup I}$ by restricting covers of $X$ to each 
	$X_i$, as in Lemma \ref{lem:asdim-passes-to-subspaces}. It remains to prove 
	the ``$\leq$'' direction.
	
	Let $\mathcal{F}$ be a cofinal ultrafilter on $I$. By Lemma 
	\ref{lem:asdim-passes-to-subspaces}, we have $\asdim(X) \geq \asdim(X_i)$ 
	for all $i$. Notice that we can view each $X_i$ as a subset both of $X$ and 
	of $X_{\bigsqcup I}$, and we obtain the same ultrafilter, where we replace 
	each $E \in \Enseco$ by $\bigsqcup_{i \in I} E|_{X_i} \in 
	\Enseco_{\bigsqcup I}$ for the definition of $\Enseco^{(E)}_{\mathcal{F}}$.
	
	Therefore, by Lemma \ref{lem:asdim-ultraproducts}, we have 
	$\asdim(X) \geq \asdim(X_{\mathcal{F}})$ and also 	$\asdim(X_{\bigsqcup 
	I}) \geq \asdim(X_{\mathcal{F}})$.  Because $X = \bigcup_{i \in I} 
	X_i$, for any $x \in X$ there exists $i \in \Lambda$ such that $x \in X_j$ 
	for any $j \geq i_x$. We can thus define an inclusion $\iota \colon X \to 
	X_{\mathcal{F}}$, such that $\iota(x)$ is the equivalence class of any 
	element in $\prod_{i \in \Lambda} X_i$ such that the $j$'th coordinate is 
	$x$ for all $j$ greater than some index $i$. (Any two such elements in the 
	product are equivalent, and thus this map is well defined.) Obviously this 
	map is one-to-one. Thinking of $X$ is a subset of $X_{\mathcal{F}}$, we 
	claim now that $\Enseco = \Enseco_{\mathcal{F}}|_{X}$.  Indeed, if $F \in 
	\Enseco_{\mathcal{F}}$ then let $E \in \Enseco$ such that  $F \in 
	\Enseco^{(E)}_{\mathcal{F}}$, and then $F \cap X \times X \subseteq E$ and 
	thus $F \cap X \times X \in \Enseco$. Conversely, for $E \in \Enseco$, we 
	have $E = \iota \times \iota (E) \cap X \times X$, and $\iota \times \iota 
	(E) \in \Enseco_{\mathcal{F}}$. Thus, by Lemma 
	\ref{lem:asdim-passes-to-subspaces}, we have 
	$\asdim(X_{\mathcal{F}}) \geq \asdim(X)$, hence $\asdim(X_{\bigsqcup 
		I}) \geq \asdim(X)$ as required.
\end{proof}

\begin{Cor} \label{cor:asdim-finitary}
	Let $\left( X_{\operatorname{Fin}} , \Enseco_{\operatorname{Fin}} \right)$ be the coarse disjoint union of all the finite subsets of $X$. 
	Then we have 
	\[
		\asdim \left( X , \Enseco\right) = \asdim \left( X_{\operatorname{Fin}} , \Enseco_{\operatorname{Fin}}  \right) \; .
	\]
	More concretely, $\asdim \left( X , \Enseco\right)$ is the infimum of natural numbers $d$ such that for any controlled set $\Ense$, there is a controlled set $\Auense$ such that for any finite subset $Y \subseteq X$, there is a cover $\Auseco$ of $Y$ such that 
	\begin{enumerate}
		\item $\bigcup_{\Ause \in \Auseco} \Ause \times \Ause \subseteq \Auense$,and
		\item the $\Ense$-multiplicity of $\Auseco$ is at most $d+1$. 
	\end{enumerate}
\end{Cor}

\begin{proof}
	The first statement follows directly from applying Proposition~\ref{prop:asdim-directed-limit} to the net of all finite subsets of $X$. 
	The second statement follows by 
	applying Lemma~\ref{lem:asdim-multiplicity-wide} to 
	$\asdim \left( X_{\operatorname{Fin}} , \Enseco_{\operatorname{Fin}}  \right)$. 
\end{proof}

Compared to the original definition of asymptotic dimension (Definition~\ref{def:asdim}), the characterization in Corollary~\ref{cor:asdim-finitary} is of a more finitary nature and is sometimes easier to verify, at the cost of an added layer of a universal quantifer.

\bibliographystyle{alpha}
\bibliography{nilpotent-group-actions-bib}
\end{document}